\documentclass{amsart}

\usepackage{cancel}
\usepackage{soul}
\usepackage[normalem]{ulem}

\usepackage{amssymb,amsmath,graphicx}
\usepackage{imakeidx}
\usepackage{xcolor}
\usepackage{enumerate}

\makeindex[title=Subject index]

\newtoks\prt
 \newtheorem{thma}{Theorem}

\newtheorem{thm}{Theorem}[section]

\newtheorem{ques}[thm]{Question}
\newtheorem{lemma}[thm]{Lemma}
\newtheorem{prop}[thm]{Proposition}
\newtheorem{cor}[thm]{Corollary}

\newtheorem{obs}[thm]{Observation}
\newtheorem{example}[thm]{Example}

\newtheorem{claim}{Claim}

\theoremstyle{definition}
\newtheorem{notation}[thm]{Notation}
\newtheorem{example2}[thm]{Example}
\newtheorem{remark}[thm]{Remark}

\newtheorem{definition}[thm]{Definition}
\newtheorem{remarks}[thm]{Remarks}

\def\eqn#1$$#2$${\begin{equation}\label#1#2\end{equation}}

\def\1{\boldsymbol{1}}
\def\A{\mathcal A}
\def\B{\mathcal B}
\def\C{\mathcal C}
\def\D{\mathcal D}
\def\E{\mathcal E}
\def\R{\mathcal R}
\def\F{\mathcal F}

\def\ms{\mathcal S}

\def\H{\mathcal H}

\def\Z{\mathcal Z}

\def\ce{\mathbb C}

\def\frd{\mathfrak{D}}

\def\lin{Lindel\"of}
\def\diam{\operatorname{diam}}
\def\co{\operatorname{conv}}
\def\ep{\varepsilon}

\def\en{\mathbb N}
\def\er{\mathbb R}
\def\qe{\mathbb Q}

\def\TT{\mathbb T}

\def\II{\mathbb I}
\def\wt{\widetilde}

\def\ov{\overline}
\def \Ch {\operatorname{Ch}}

\def \Bo {\operatorname{Bo}}
\def \Fr {\operatorname{Fr}}
\def \Lb {\operatorname{Lb}}
\def \Ba {\operatorname{Ba}}

\def \ext {\operatorname{ext}}

\def\spt{\operatorname{spt}}
\def\span{\operatorname{span}}

\def\wh{\widehat}
\def \reg {\partial _{\kern1pt\text{reg}}}

\def\Coz{\operatorname{Coz}}
\def\Zer{\operatorname{Zer}}
\def\FpG{(F\wedge G)}
\def\FsG{(F\vee G)}

\def\di{\,\mbox{\rm d}}

\newcommand{\norm}[1]{\left\|#1\right\|}

\newcommand{\wscl}[1]{\overline{#1}^{w^*}}

\newcommand{\nesubset}{\rotatebox{45}{$\,\subset\ $}}

\newcommand{\sesubset}{\rotatebox{315}{$\,\subset\ $}}

\newcommand{\abs}[1]{\left|#1\right|}
\newcommand{\setsep}{;\,}

\numberwithin{equation}{section}

\usepackage[automake,symbols]{glossaries}

\newglossaryentry{AcX}{%
  name={\ensuremath{A_c(X)}},
  description={the space of affine continuous functions on $X$},
  sort={Ac(x)},
  type=symbols
}

\newglossaryentry{ZAcX}{%
  name={\ensuremath{Z(A_c(X))}},
  description={the center of $A_c(X)$},
  sort={Z(Ac(X))},
  type=symbols
}

\newglossaryentry{AbX}{%
  name={\ensuremath{A_b(X)}},
  description={the space of bounded affine functions on $X$},
  sort={Ab(X)},
  type=symbols
}

\newglossaryentry{extX}{%
  name={\ensuremath{\ext X}},
  description={the set of extreme points of $X$},
  sort={extX},
  type=symbols
}

\newglossaryentry{C(X)}{%
  name={\ensuremath{C(X)}},
  description={the space of real-valued continuous functions on $X$},
  sort={C(X)},
  type=symbols
}

\newglossaryentry{M1(X)}{%
  name={\ensuremath{M_1(X)}},
  description={the space of Radon probability measures on $X$},
  sort={M1(X)},
  type=symbols
}

\newglossaryentry{1B}{%
  name={\ensuremath{1_B}},
  description={the characteristic function of the set $B$},
  sort={1B},
  type=symbols
}

\newglossaryentry{Fsigma}{%
  name={\ensuremath{F^\sigma}},
  description={the smallest sequentially closed family containing $F$},
  sort={Fsigma},
  type=symbols
}

\newglossaryentry{Fmu}{%
  name={\ensuremath{F^\mu}},
  description={the monotone sequential closure of $F$},
  sort={Fmu},
  type=symbols
}

\newglossaryentry{Fup}{%
  name={\ensuremath{F^\uparrow}},
  description={the set of limits of bounded non-decreasing sequences from $F$},
  sort={Fup},
  type=symbols
}

\newglossaryentry{Fdown}{%
  name={\ensuremath{F^\downarrow}},
  description={the set of limits of bounded non-increasing sequences from $F$},
  sort={Fdown},
  type=symbols
}

\newglossaryentry{Adelta}{%
  name={\ensuremath{\A_\delta}},
  description={the family of countable intersections of elements of $\A$},
  sort={Adelta},
  type=symbols
}

\newglossaryentry{Asigma}{%
  name={\ensuremath{\A_\sigma}},
  description={the family of countable unions of elements of $\A$},
  sort={Asigma},
  type=symbols
}

\newglossaryentry{sigma(A)}{%
  name={\ensuremath{\sigma(\A)}},
  description={the $\sigma$-algebra generated by $\A$},
  sort={sigma(A)},
  type=symbols
}

\newglossaryentry{AwedgeB}{%
  name={\ensuremath{\A\wedge\B}},
  description={the family $\{A\cap B\setsep A\in\A, B\in \B\}$},
  sort={AwedgeB},
  type=symbols
}

\newglossaryentry{AveeB}{%
  name={\ensuremath{\A\vee\B}},
  description={the family $\{A\cup B\setsep A\in\A, B\in \B\}$},
  sort={AveeB},
  type=symbols
}

\newglossaryentry{Ba1(Y)}{%
  name={\ensuremath{\Ba_1(Y)}},
  description={the space of all Baire-one functions on $Y$},
  sort={Ba1(Y)},
  type=symbols
}

\newglossaryentry{Ba1b(Y)}{%
  name={\ensuremath{\Ba_1^b(Y)}},
  description={the space of all bounded Baire-one functions on $Y$},
  sort={Ba1b(Y)},
  type=symbols
}

\newglossaryentry{Ba(Y)}{%
  name={\ensuremath{\Ba(Y)}},
  description={the space of all Baire functions on $Y$},
  sort={Ba(Y)},
  type=symbols
}

\newglossaryentry{Bab(Y)}{%
  name={\ensuremath{\Ba^b(Y)}},
  description={the space of all bounded Baire functions on $Y$},
  sort={Bab(Y)},
  type=symbols
}

\newglossaryentry{Bo(Y)}{%
  name={\ensuremath{\Bo(Y)}},
  description={the space of all Borel functions on $Y$},
  sort={Bo(Y)},
  type=symbols
}

\newglossaryentry{Bo1(Y)}{%
  name={\ensuremath{\Bo_1(Y)}},
  description={the space of all functions of the first Borel class on $Y$},
  sort={Bo1(Y)},
  type=symbols
}

\newglossaryentry{Coz}{%
  name={\ensuremath{\Coz}},
  description={the family of cozero sets},
  sort={Coz},
  type=symbols
}

\newglossaryentry{Zer}{%
  name={\ensuremath{\Zer}},
  description={the family of zero sets},
  sort={Zer},
  type=symbols
}

\newglossaryentry{H}{%
  name={\ensuremath{\H}},
  description={the family of resolvable sets},
  sort={H},
  type=symbols
}

\newglossaryentry{Fr(X)}{%
  name={\ensuremath{\Fr(X)}},
  description={the space of all fragmented functions on $X$},
  sort={Fr(X)},
  type=symbols
}

\newglossaryentry{Frb(X)}{%
  name={\ensuremath{\Fr^b(X)}},
  description={the space of all bounded fragmented functions on $X$},
  sort={Frb(X)},
  type=symbols
}

\newglossaryentry{r(mu)}{%
  name={\ensuremath{r(\mu)}},
  description={the barycenter of the measure $\mu$},
  sort={r(mu)},
  type=symbols
}

\newglossaryentry{mu(f)}{%
  name={\ensuremath{\mu(f)}},
  description={a shortcut for $\int f\di\mu$},
  sort={mu(f)},
  type=symbols
}

\newglossaryentry{muprecnu}{%
  name={\ensuremath{\mu\prec\nu}},
  description={dilation order on $M_1(X)$},
  sort={muprecnu},
  type=symbols
}

\newglossaryentry{lambdaA}{%
  name={\ensuremath{\lambda_A}},
  description={the affine function associated to a split (or parallel) face $A$},
  sort={lambdaA},
  type=symbols
}

\newglossaryentry{A1(X)}{%
  name={\ensuremath{A_1(X)}},
  description={the space of all affine Baire-one functions on $X$},
  sort={A1(X)},
  type=symbols
}

\newglossaryentry{Al(X)}{%
  name={\ensuremath{A_l(X)}},
  description={the cone of all affine lower semicontinuous functions on $X$},
  sort={Al(X)},
  type=symbols
}

\newglossaryentry{As(X)}{%
  name={\ensuremath{A_s(X)}},
  description={the space $A_l(X)-A_l(X)$},
  sort={As(X)},
  type=symbols
}

\newglossaryentry{Af(X)}{%
  name={\ensuremath{A_f(X)}},
  description={the space of all affine fragmented functions on $X$},
  sort={Af(X)},
  type=symbols
}

\newglossaryentry{Asa(X)}{%
  name={\ensuremath{A_{sa}(X)}},
  description={the space of all strongly affine functions on $X$},
  sort={Asa(X)},
  type=symbols
}

\newglossaryentry{S(A)}{%
  name={\ensuremath{S(A)}},
  description={the state space of $A$},
  sort={S(A)},
  type=symbols
}

\newglossaryentry{D(W)}{%
  name={\ensuremath{\frd(W)}},
  description={the ideal center of $L(W)$},
  sort={D(W)},
  type=symbols
}

\newglossaryentry{Z(A)}{%
  name={\ensuremath{Z(A)}},
  description={the center of $A$},
  sort={Z(A)},
  type=symbols
}

\newglossaryentry{phi}{%
  name={\ensuremath{\phi}},
  description={the evaluation mapping from a compact space into the state space of a function space},
  sort={phi},
  type=symbols
}

\newglossaryentry{Phi}{%
  name={\ensuremath{\Phi}},
  description={the canonical isometry of a function space $E$ onto $A_c(S(E))$},
  sort={Phi},
  type=symbols
}

\newglossaryentry{Mx(E)}{%
  name={\ensuremath{M_x(E)}},
  description={the set of probabilities representing $x$},
  sort={Mx(E)},
  type=symbols
}

\newglossaryentry{ChE(K)}{%
  name={\ensuremath{\Ch_E K}},
  description={the Choquet boundary of $E$},
  sort={ChEK}
  type=symbols
}

\newglossaryentry{iota}{%
  name={\ensuremath{\iota}},
  description={the evaluation map from $X$ into $S(H)$},
  sort={iota},
  type=symbols
}

\newglossaryentry{pi}{%
  name={\ensuremath{\pi}},
  description={the canonical surjection of $S(H)$ onto $X$},
  sort={pi},
  type=symbols
}

\newglossaryentry{V}{%
  name={\ensuremath{V}},
  description={the isometry of $\ell^\infty(K)\cap E^{\perp\perp}$ onto the space of strongly affine functions on $S(E)$},
  sort={V},
  type=symbols
}

\newglossaryentry{M(H)}{%
  name={\ensuremath{M(H)}},
  description={the space of multipliers on $H$},
  sort={M(H)},
  type=symbols
}

\newglossaryentry{Ms(H)}{%
  name={\ensuremath{M^s(H)}},
  description={the space of strong multipliers on $H$},
  sort={Ms(H)},
  type=symbols
}

\newglossaryentry{m(X)}{%
  name={\ensuremath{m(X)}},
  description={the set of those points at which the multiplication formula holds},
  sort={m(X)},
  type=symbols
}

\newglossaryentry{f*}{%
  name={\ensuremath{f^*}},
  description={the upper envelope of $f$},
  sort={f*},
  type=symbols
}

\newglossaryentry{tauCh}{%
  name={\ensuremath{\tau_{\Ch}}},
  description={the Choquet topology},
  sort={tauCh},
  type=symbols
}

\newglossaryentry{Sigma'}{%
  name={\ensuremath{\Sigma'}},
  description={the $\sigma$-algebra on $\ext X$ generated by traces of closed extremal sets and traces of Baire sets},
  sort={Sigma'},
  type=symbols
}

\newglossaryentry{mu'}{%
  name={\ensuremath{\mu'}},
  description={the measure on $\Sigma'$ associated to a maximal measure $\mu$},
  sort={mu'},
  type=symbols
}

\newglossaryentry{R}{%
  name={\ensuremath{R}},
  description={the restriction map $H\to\ell^\infty(\ext X)$},
  sort={R},
  type=symbols
}

\newglossaryentry{AH}{%
  name={\ensuremath{\A_H}},
  description={the family of subsets of $\ext X$ defined using $M(H)$},
  sort={AH},
  type=symbols
}

\newglossaryentry{AHs}{%
  name={\ensuremath{\A_H^s}},
  description={the family of subsets of $\ext X$ defined using $M^s(H)$},
  sort={AHs},
  type=symbols
}

\newglossaryentry{SH}{%
  name={\ensuremath{\ms_H}},
  description={the family of traces of split faces $F$ with $\lambda_F\in H^\uparrow$},
  sort={SH},
  type=symbols
}

\newglossaryentry{ZH}{%
  name={\ensuremath{\Z_H}},
  description={the family of traces of sets $[f=1]$ for $f\in H^\uparrow$, $f$ attaining just the values $0,1$ on $\ext X$},
  sort={ZH},
  type=symbols
}

\newglossaryentry{Upsilon}{%
  name={\ensuremath{\Upsilon}},
  description={the extension operator from a split face},
  sort={Upsilon},
  type=symbols
}

\newglossaryentry{Lb(K)}{%
  name={\ensuremath{\Lb(K)}},
  description={the space of differences of bounded lower semicontinuous functions on $K$},
  sort={Lb(K)},
  type=symbols
}

\newglossaryentry{KLA}{%
  name={\ensuremath{K_{L,A}}},
  description={the compact space with the porcupine topology induced by $L$ and $A$},
  sort={KLA},
  type=symbols
}

\newglossaryentry{ELA}{%
  name={\ensuremath{K_{L,A}}},
  description={the canonical function space on $K_{L,A}$},
  sort={ELA},
  type=symbols
}

\newglossaryentry{XLA}{%
  name={\ensuremath{X_{L,A}}},
  description={Stacey simplex, the state space of $E_{L,A}$},
  sort={XLA},
  type=symbols
}

\newglossaryentry{jmath}{%
  name={\ensuremath{\jmath}},
  description={the canonical injection of $L$ into $K_{L,A}$},
  sort={jmath},
  type=symbols
}

\newglossaryentry{psi}{%
  name={\ensuremath{\psi}},
  description={the canonical surjection of $K_{L,A}$ onto $L$},
  sort={psi},
  type=symbols
}

\newglossaryentry{ftilde}{%
  name={\ensuremath{\widetilde{f}}},
  description={the modification of a function $f$ on $K_{L,A}$ satisfying the averages condition from the definition of $E_{L,A}$},
  sort={ftilde},
  type=symbols
}

\makeglossaries

\definecolor{green}{rgb}{0,0.5,0}
\title[Boundary integral representation of multipliers]{Boundary integral representation of multipliers of fragmented affine functions and other intermediate function spaces}
\author{Ondřej F.K. Kalenda, Jakub Rondo\v s and Ji\v r\'\i\ Spurn\'y}

\address{Ondřej F.K. Kalenda\\
Charles University\\
Faculty of Mathematics and Physics\\
Department of Mathematical Analysis \\
Sokolovsk\'{a} 83, 186 \ 75\\Praha 8, Czech Republic}
\email{kalenda@karlin.mff.cuni.cz}

\address{Jakub Rondo\v s\\
Charles University\\
Faculty of Mathematics and Physics\\
Department of Mathematical Analysis \\
Sokolovsk\'{a} 83, 186 \ 75\\Praha 8, Czech Republic}
\email{jakub.rondos@gmail.com}

\address{Ji\v r\'\i\ Spurn\'y\\
Charles University\\
Faculty of Mathematics and Physics\\
Department of Mathematical Analysis \\
Sokolovsk\'{a} 83, 186 \ 75\\Praha 8, Czech Republic}
\email{spurny@karlin.mff.cuni.cz}



\thanks{O.F.K. Kalenda and J. Spurn\'y were partially supported by the Research grant GA\v{C}R 23-04776S}
\begin{document}

\begin{abstract} We develop a theory of abstract intermediate function spaces on a compact convex set $X$ and study the behaviour of multipliers and centers of these spaces. In particular, we
provide some criteria for coincidence of the center with the space of multipliers and a general theorem on boundary integral representation of multipliers.
We apply the general theory in several concrete cases, among others to strongly affine Baire functions, to the space $A_f(X)$ of fragmented affine functions, to the space 
$(A_f(X))^\mu$, the monotone sequential closure of $A_f(X)$, to their natural subspaces formed by Borel functions, or, in some special cases, to the space of all strongly affine functions. 
In addition, we prove that the space $(A_f(X))^\mu$ is determined by extreme points and provide a large number of illustrating examples and counterexamples.
\end{abstract}

\maketitle

\newpage

\tableofcontents
\clearpage

\section{Introduction}

We investigate centers, spaces of multipliers and their integral representation for distinguished spaces of affine functions on compact convex sets. The story starts by results on affine continuous functions.
If $X$ is a compact convex set in a (Hausdorff) locally convex space, we denote by \gls{AcX} the space of all real-valued continuous affine functions on $X$. This space equipped with the supremum norm and pointwise order is an example of a unital $F$-space (see the next section for the notions unexplained here). Therefore, it possesses the \emph{center}\index{center!of $A_c(X)$} \gls{ZAcX}, i.e., the set of all elements of the form $T1_X$, where $T\colon A_c(X)\to A_c(X)$ is an operator satisfying $-\alpha I\le T\le \alpha T$ for some $\alpha\ge 0$. This notion may be viewed as a generalization of the center of a unital C$^*$-algebra. Indeed, if $A$ is a unital C$^*$-algebra, its state space $S(A)$ is a compact convex set and $A_c(S(A))$ is canonically identified with the self-adjoint part of $A$. Moreover, by \cite[Lemma 4.4]{wils} the center $Z(A_c(S(A)))$ is in this way identified with the self-adjoint part of the standard center of $A$.

It is shown in \cite[Theorem II.7.10]{alfsen} that the center of $A_c(X)$ has two more natural representations. On one hand it coincides with the spaces of \emph{multipliers}\index{multiplier!in $A_c(X)$} in $A_c(X)$, i.e., the functions $m\in A_c(X)$ such that for each $a\in A_c(X)$ there exists $b\in A_c(X)$ with $m(x)a(x)=b(x)$ for each $x\in \ext X$, where \gls{extX} stands for the set of all extreme points of $X$. On the second hand, $Z(A_c(X))$ may be identified with the space of functions on  $\ext X$, that are continuous in the facial topology. This identification is given by the restriction mapping. 

In \cite{edwards} these notions were studied for the monotone sequential closure $(A_c(X))^\mu$ of $A_c(X)$ and analogous results were proved. In \cite{smith-london} these results were extended to the sequential closure $(A_c(X))^\sigma$ of $A_c(X)$ and also to the monotone closure of the space of differences of semicontinuous affine functions (this space is denoted by $(A_s(X))^\mu$). It is shown in these papers that all these spaces are $F$-spaces with unit such that their center coincides with the space of multipliers. A crucial property of these spaces is that they are determined by the values on $\ext X$, i.e., that every function $u$ from these spaces satisfies $$ \inf u(\ext X)\le u\le \sup u(\ext X).$$

In addition, a central integral representation is available for these spaces.
More precisely, the following result is proved in \cite[Theorem 5.5]{smith-london} (see also \cite[Theorem 1.2]{smith-pacjm})

\begin{thma}\label{T:smith} There exists a $\sigma$-algebra $\R$ of subsets of $\ext X$ such that the restriction mapping $\pi$ maps the center $Z((A_s(X))^\mu)$ onto the space of all bounded $\R$-measurable functions on $\ext X$. Moreover, there exists a unique affine map $x\mapsto \nu_x$ from $X$ into the set of probability measures on $\R$ satisfying, for $z\in Z((A_s(X))^\mu)$, $z(x)=\int_{\ext X} \pi(z)\di\nu_x$.
\end{thma}

These results have profound consequences in the theory of $C^*$-algebras, where in fact this line of research has originated (see \cite{davies,pedersen-scand,pedersen-weak,pedersen-appl} or \cite[Section 5]{edwards}). Later on, the concept of the center was considered for arbitrary Banach spaces, see \cite{edwards-banach} and \cite{behrends}.

The original motivation of our research contained in the present paper was to extend the theory of \cite{edwards} and \cite{smith-london} to the case of the so-called \emph{fragmented} affine functions on $X$ (denoted by $A_f(X)$). 
Fragmented functions on a compact space $K$ form a natural space (denoted by $\Fr(K)$) with very interesting properties. If $K$ is metrizable, $\Fr(K)$ coincide with the functions of the first Baire class (see Theorem~\ref{T:b} below). If $K$ is not metrizable, this space may be much larger. It includes both functions of the first Baire class and semicontinuous functions and also the functions of the first Borel class. In particular, fragmented functions are not necessarily obtained from continuous ones using countable or monotone limit procedures.
However, they still enjoy a continuity property (see Theorem~\ref{T:a} below)
and hence they may be used as a starting space of a whole hierarchy.

In the affine setting the situation is even more interesting. Affine fragmented functions are automatically bounded and strongly affine (see \cite[Theorem 4.21]{lmns}). A surprising result obtained in  \cite{dostal-spurny} says that the space $A_f(X)$ is determined by extreme points. This result was rather surprising because it is not possible to derive it from the result on continuous affine functions by a limit procedure, but a new method using directly continuity properties of fragmented affine functions needed to be developed.  One of our results is the extension of this theorem to the space $(A_f(X))^\mu$, the monotone sequential closure  of $A_f(X)$ in Theorem~\ref{extdeter} below.

Moreover, it has appeared that it is natural and useful to develop a more abstract theory of \emph{intermediate functions spaces}\index{intermediate function space} which are closed spaces $H$ satisfying $A_c(X)\subset H\subset A_b(X)$ (here \gls{AbX} stands for the space of all affine bounded functions on $X$). In this abstract context we investigate the relationship of the center and the space of multipliers, preservation of extreme points, characterizations by a measurability condition on $\ext X$ and related results on integral representation. We also collect a number of examples and counterexamples.

We point out that the proofs of Theorem~\ref{T:smith} and its predecessors from \cite{edwards} use the
same abstract scheme. We unify this approach in Theorem~\ref{T:integral representation H} and extend it in Theorems~\ref{t:aha-regularita} and~\ref{T: meritelnost H=H^uparrow cap H^downarrow} to a more general context. This not only covers the results of \cite{edwards,smith-london} but can be applied in many other cases, in particular to fragmented functions and their limits. Moreover, in some cases (for example for strongly affine Baire functions) we provide a more concrete description of the analogues of the $\sigma$-algebra $\R$ from Theorem~\ref{T:smith}.

The paper is organized as follows:

In the following section we collect basic definitions, facts and notation from the Choquet theory, topology and descriptive set theory needed throughout the paper. In Section~\ref{s:IFS} 
we introduce intermediate function spaces and collect their basic properties and representations.

Section~\ref{s:multi-atp} is devoted to centers and multipliers. We introduce and compare these notions. In addition, we introduce a new notion of strong multipliers and relate it to the other notions. It should be stressed that for $A_c(X)$, the space of affine continuous functions, all multipliers are strong and both generalizations are natural.

In Section~\ref{sec:mult-acx} we provide a characterization of multipliers on $A_c(X)$ which is, in many cases, easy to apply.

Section~\ref{s:preserveext} is devoted to the investigation of preservation of extreme points to larger intermediate function spaces. This is closely related to the coincidence of the center and the spaces of multipliers. We provide some sufficient conditions and also some counterexamples.

In Section~\ref{s:determined} we use the results of \cite{teleman} on boundary integral representation of semicontinuous functions to show that the space $(A_s(X))^\sigma$ is determined by extreme points. We also prove the above-mentioned result that the space $(A_f(X))^\mu$ is determined by extreme points.

Section~\ref{s:reprez-abstraktni} contains an abstract version of Theorem~\ref{T:smith} on integral representation of spaces of multipliers and also its generalization to a more general context.

In Section~\ref{s:meas sm} we give a better version of some results of Section~\ref{s:reprez-abstraktni} for strong multipliers, using suitable families of split faces.

Section~\ref{sec:splifaces} is devoted to extension results for strongly affine functions on split faces which are later used to formulate
concrete versions of the abstract results from Section~\ref{s:reprez-abstraktni} in a more descriptive way. 

In Section~\ref{sec:baire} we give concrete versions of the abstract results for spaces of strongly affine Baire functions and Section~\ref{sec:beyond} is devoted to the results on larger spaces (defined using semicontinuous, Borel or fragmented functions).

The last two sections are devoted to examples and counterexamples. Section~\ref{sec:strange} contains several
strange counterexamples showing that some natural inclusions are not automatic. Finally, the last section is devoted to an analysis of a special class of simplices defined using the porcupine topology and considered by Stacey \cite{stacey}. This analysis illustrates that some sufficient conditions for certain representations are natural and inspires several open problems.

\section{Basic notions}

In this section we fix the notation, collect definitions of basic notions and provide some background results needed in the paper. We start by pointing out that we will deal only with real vector spaces. In particular, we consider only real-valued affine functions.

\subsection{Compact spaces and important classes of sets and functions}
\label{ssc:csp}

If $X$ is a compact  (Hausdorff) topological space, we denote by \gls{C(X)} the Banach space of all real-valued continuous functions on $X$ equipped with the sup-norm. The dual of $C(X)$ will be identified (by the Riesz representation theorem) with $M(X)$, the space of signed (complete) Radon measures on $X$ equipped with the total variation norm and the respective weak$^*$ topology. Let  \gls{M1(X)} stand for the set of all probability Radon measures on $X$. A set $B\subset X$ is \emph{universally measurable}\index{set!universally measurable} if it is measurable with respect to any probability Radon measure on $X$. If $B\subset X$ is universally measurable, we write $M_1(B)$ for the set of all $\mu\in M_1(X)$ with $\mu(X\setminus B)=0$. Let \gls{1B} denote  the characteristic function of $B\subset X$, we often write 
$1=1_X$ for the constant function $1$.

Before coming to specific classes of functions, let us introduce some general notation for families of functions. Let $F$ be a family of bounded functions on an~abstract set $\Gamma$. We set
\begin{itemize}
    \item \gls{Fsigma} to be the smallest family of functions on $\Gamma$ that contains $F$ and is closed with respect to taking limits of pointwise converging bounded sequences;
    \item \gls{Fmu} to be the smallest family of functions on $X$ that contains $F$ and is closed with respect to taking pointwise limits of monotone bounded sequences;
    \item \gls{Fdown} to be the set of all pointwise limits of bounded non-increasing sequences from $F$;
    \item \gls{Fup} to be the set of all pointwise limits of bounded non-decreasing sequences from $F$.
\end{itemize}
Obviously $F^\mu\subset F^\sigma$. The converse holds whenever $F$ is a lattice (with pointwise operations), but not in general (see, e.g., Proposition~\ref{P:Baire-srovnani}(ii) below). We will also repeatedly use the following easy observation.

\begin{lemma}\label{L:muclosed is closed}
 Let $\Gamma$ be a set and let $F\subset \ell^\infty(\Gamma)$ be a linear subspace containing constant functions such that $F^\mu=F$. Then $F$ is uniformly closed.   
\end{lemma}

\begin{proof} This follows from the beginning of the proof of \cite[Lemma 3.5]{edwards}. But the quoted lemma has stronger assumptions (even though they are not used at this step), so we provide the easy argument.

It is enough to show that, given sequence $(f_n)$ in $F$ such that $\sum_n\norm{f_n}<\infty$, the sum $f=\sum_n f_n$ (which exists in $\ell^\infty(\Gamma)$) belongs to $F$. So, let $(f_n)$ be such a sequence
and denote $k=\sum_n\norm{f_n}\in[0,\infty)$ and $f=\sum_n f_n\in\ell^\infty(\Gamma)$. Set $g_n=f_n+\norm{f_n}$. By our assumptions $g_n\in F$ and $g_n\ge0$ for each $n\in\en$. The sequence of partial sums $(\sum_{k=1}^n g_k)_n$ is a non-decreasing sequence in $F$ uniformly bounded by $2k$, hence
$g=\sum_{k=1}^{\infty} g_k\in F$. Then $f=g-k\in F$ as well.
    
\end{proof}

If $\A$ and $\B$ are two families of subsets of a set $\Gamma$, we denote, as usually, by
\begin{itemize}
    \item \gls{Asigma} the family of sets which may be expressed as a countable union of elements of $\A$;
    \item \gls{Adelta} the family of sets which may be expressed as a countable intersection of elements of $\A$;
    \item \gls{sigma(A)} the $\sigma$-algebra generated by $\A$;
    \item \gls{AwedgeB} the system $\{A\cap B\setsep A\in\A,B\in\B\}$;
    \item  \gls{AveeB} the system $\{A\cup B\setsep A\in\A,B\in\B\}$.
    %\item $\A^c$ the system made by complements of the elements of $\A$.
\end{itemize}

If $Y$ is a Tychonoff (i.e., completely regular and Hausdorff) space, we denote by \gls{Ba1(Y)} the space of all \emph{Baire-one functions}\index{function!Baire-one} on $Y$ (i.e., pointwise limits of sequences of continuous functions) and by \gls{Ba1b(Y)} the space of all bounded Baire-one functions on $Y$. The space of \emph{Baire functions}\index{function!Baire} on $Y$ (i.e., the smallest space containing continuous functions and closed under pointwise limits of sequences) is denoted by \gls{Ba(Y)}, bounded Baire functions are denoted by \gls{Bab(Y)}.

Baire functions may be characterized using a measurability condition. Recall that a \emph{zero set}\index{set!zero} is a set $Z\subset Y$ of the form
$$Z=[f=0]=\{y\in Y\setsep f(y)=0\}$$ for a continuous function $f$. Complements of zero sets are called \emph{cozero sets}\index{set!cozero}. The families of zero and cozero sets will be denoted by \gls{Zer} and \gls{Coz}, respectively. It is known that $f\in \Ba_1(Y)$ if and only if it is $\Zer_\sigma$-measurable (see \cite[Exercise 3.A.1]{lmz}).
Further, the $\sigma$-algebra generated by zero (or cozero) sets is called \emph{Baire $\sigma$-algebra}\index{sigma-algebra@$\sigma$-algebra!Baire} and its elements are \emph{Baire sets}\index{set!Baire}. Another known result says that $f\in\Ba(Y)$ if and only if it is Baire measurable (i.e., $f^{-1}(U)$ is a Baire set for any open set $U\subset\er$).

As usually, \emph{Borel sets}\index{set!Borel} are elements of the $\sigma$-algebra generated by open sets and a function is called \emph{Borel}\index{function!Borel} if the inverse image of any open set is a Borel set. We write \gls{Bo(Y)} for the family of all Borel functions on $Y$. An~important subclass of Borel functions is the family \gls{Bo1(Y)} of  functions of the \emph{first Borel class}\index{function!of the first Borel class}, i.e., $\FpG_\sigma$-measurable functions, where $F\wedge G$ denote the sets obtained as the intersection of a closed and open set.
(This class of functions coincide with the family $\operatorname{Bof}_1(Y,\er)$ from \cite{spurny-amh}.)

Now we come to fragmented functions and related classes of functions and sets. A function $f:Y\to\er$ 
\begin{itemize}
    \item is called \emph{fragmented}\index{function!fragmented}, if for any $\epsilon>0$ and nonempty closed set $F\subset Y$ there exists a relatively open nonempty set $U\subset F$ such that $\diam f(U)<\epsilon$;
    \item is said to have the \emph{point of continuity property}, if $f|_F$ has a point of continuity for any $F\subset X$ nonempty closed.
\end{itemize}
Further, a set $A\subset Y$ is \emph{resolvable}\index{set!resolvable} if the characteristic function $1_A$ is fragmented. This definition is used in \cite[Section 2]{koumou}, where such sets are also called \emph{$H$-sets}. The term resolvable is used in \cite[Chapter I, \S12]{kuratowski} where an alternative definition is used. The two definitions are equivalent by \cite[Lemma 2.2]{koumou} or \cite[Chapter I, \S12.V]{kuratowski}. We will denote the family of resolvable sets by \gls{H}.

The following theorem revealing relationships of these classes follows from \cite[Theorem 2.3]{koumou}.

\begin{thma}\label{T:a} Let $Y$ be a Tychonoff space and  let $f:Y\to\er$ be a function. Then we have the following. 
\begin{enumerate}[$(a)$]
\item $f$ has the point of continuity property $\Longrightarrow$
$f$ is fragmented $\Longrightarrow$ $f$ is $\H_\sigma$-measurable
\item   If $Y$ is compact (or, more generally, hereditarily Baire), then the three conditions from $(a)$ are equivalent.
\end{enumerate}
\end{thma}

With few exceptions we will consider fragmented functions only on a compact space $X$. 
 In such a case we write \gls{Fr(X)} (or \gls{Frb(X)}) for the space of all  fragmented (or bounded fragmented) functions on $X$.
It easily follows from Theorem~\ref{T:a} that fragmented functions form a Banach lattice and algebra, see \cite[Theorem~5.10]{lmns}. It also readily follows that any semicontinuous function $f\colon X\to \er$ is fragmented.

Further, it follows from \cite[Lemma 4.1]{koumou} that any resolvable set is universally measurable. Therefore also any fragmented function on a compact space is universally measurable (in the usual sense that the preimage of any open set is universally measurable).

We further note that fragmented functions may be viewed as a generalization of Baire-one functions. This is witnessed by the following known result:

\begin{thma}\label{T:b}
Let $X$ be a compact Hausdorff space and let $f:X\to\er$ be a function. Then we have the following.
\begin{enumerate}[$(a)$]
    \item $f$ is Baire-one $\Longrightarrow$ $f$ is of the first Borel class $\Longrightarrow$ $f$ is fragmented.
    \item If $X$ is metrizable, then the three conditions from (a) are equivalent.
\end{enumerate}
\end{thma}

Indeed, any Baire-one function is $\Zer_\sigma$-measurable, and thus $F_\sigma$-measurable, so it is of the first Borel class. Further, resolvable sets form an algebra containing open sets (by \cite[Proposition 2.1(i)]{koumou}), so any $\FpG_\sigma$ set belongs to $\H_\sigma$ and thus any function of the first Borel class is $\H_\sigma$-measurable.
 Further, if $X$ is metrizable, then a subset of $X$ is resolvable if and only if it is both $F_\sigma$ and $G_\delta$ (see, e.g., \cite[Chapter II, \S24.III]{kuratowski}), hence any fragmented function is $F_\sigma$-measurable, thus Baire-one.

We finish this section by collecting known results on transferring of topological properties of mappings by continuous surjections between compact spaces.

\begin{lemma}\label{L:kvocient}
  Let $X$ and $Y$ be compact spaces, let $\varphi:X\to Y$ be a continuous surjection and let $f:Y\to\er$ be a function.
  Then we have the following equivalences:
  \begin{enumerate}[$(a)$]
      \item $f$ is continuous if and only if $f\circ \varphi$ is continuous.
      \item $f$ is lower semicontinuous if and only if $f\circ \varphi$ is lower semicontinuous.
      \item $f\in\Ba_1(Y)\Longleftrightarrow f\circ \varphi\in\Ba_1(X)$.
      \item $f\in\Ba(Y)\Longleftrightarrow f\circ \varphi\in\Ba(X)$.
      \item $f\in\Bo_1(Y)\Longleftrightarrow f\circ \varphi\in\Bo_1(X)$.
      \item $f\in\Bo(Y)\Longleftrightarrow f\circ \varphi\in\Bo(X)$.
      \item $f$ is fragmented if and only if $f\circ \varphi$ is fragmented.
      \item $f$ is $\sigma(\H)$-measurable if and only if $f\circ \varphi$ is $\sigma(\H)$-measurable.
  \end{enumerate}
\end{lemma}

\begin{proof}
    Assertions $(a)$ and $(b)$ follow easily from the fact that $\varphi$ is a closed mapping (and hence a quotient map).
    The remaining assertions follow for example from \cite[Theorem A]{affperf}. (Assertions $(e)$, $(f)$ and $(h)$ follow directly from the quoted theorem. To get assertions $(c)$ and $(d)$ we need to use moreover the above-mentioned facts that Baire-one functions are characterized by $\Zer_\sigma$-measurability and Baire functions are characterized by measurability with respect to the Baire $\sigma$-algebra. To prove $(g)$ we use in addition Theorem~\ref{T:a}$(b)$.)
\end{proof}

\subsection{Compact convex sets}
\label{ssc:ccs}
Let $X$ be a compact convex set in a locally convex (Hausdorff) topological vector space.  Given a Radon probability measure $\mu$ on $X$, we write \gls{r(mu)} for the \emph{barycenter of $\mu$}\index{barycenter}, i.e., the unique point $x\in X$ satisfying $a(x)=\int_X a \di\mu$ for each affine continuous function on $X$ (see \cite[Proposition~I.2.1]{alfsen} or \cite[Chapter 7, \S\,20]{lacey}). Conversely, for a point $x\in X$, we denote by $M_{x}(X)$ the set of all Radon probability measures on $X$ with the barycenter $x$ (i.e., the set of all probabilities \emph{representing}\index{measure!representing a point} $x$).

We recall that  $x \in X$ is an \emph{extreme point}\index{extreme point} of $X$ if whenever $x=\frac12(y+ z)$ for some $y, z \in X$, then $x=y=z$. We write $\ext X$ for the set of extreme points of $X$.

The usual dilation order $\prec$ on the set $M_1(X)$ of Radon probability measures on $X$ is defined as \gls{muprecnu} if and only if $\mu(f)\le \nu(f)$ for any real-valued convex continuous function $f$ on $X$. (Recall that \gls{mu(f)} is a shortcut for $\int f\di\mu$.)  A measure $\mu\in M_1(X)$ is said to be \emph{maximal}\index{measure!maximal} if it is maximal with respect to the dilation order.
If $B\supset \ext X$ is a Baire set, then $\mu(B)=1$ for any maximal measure $\mu\in M_1(X)$ (this follows from \cite[Corollary I.4.12]{alfsen}). Also, maximal measures are supported by $\ov{\ext X}$, see \cite[Theorem 3.79(c)]{lmns}.

Further, in case $X$ is metrizable, maximal probability measures are exactly the probabilities carried by the $G_\delta$ set $\ext X$ of extreme points of $X$ (see, e.g., \cite[p. 35]{alfsen} or \cite[Corollary 3.62]{lmns}).

By the Choquet representation theorem, for any $x\in X$ there exists a maximal representing measure (see \cite[p. 192, Corollary]{lacey} or \cite[Theorem I.4.8]{alfsen}).
A compact convex set $X$ is termed \emph{simplex}\index{simplex} if this maximal measure is uniquely determined for each $x\in X$. It is a \emph{Bauer simplex}\index{simplex!Bauer} if moreover the set of extreme points is closed. In this case, the set $X$ is affinely  homeomorphic with the set $M_1(\ext X)$ (see \cite[Corollary II.4.2]{alfsen}).

We will also need the following generalization of extreme points.
A subset $A\subset X$ is called \emph{extremal}\index{set!extremal} if  $\frac12(y+ z)\in A$ for some $y, z \in X$, then $y,z\in A$. 
A convex extremal set is called \emph{face}\index{face}.

Let $A\subset X$ be a face. The \emph{complementary set}\index{set!complementary to a face} $A'$ is the union of all faces disjoint from $A$.
Then $A'$ is clearly extremal, so it is a face as soon as it is convex. In such a case it is called the \emph{complementary face}\index{face!complementary} of $A$. Further, a face $A$ is said to be
\begin{itemize}
    \item a \emph{parallel face}\index{face!paralel} if $A'$ is convex and for each $x\in X\setminus (A\cup A')$ there is a unique $\lambda\in (0,1)$ such that $x=\lambda a+(1-\lambda)a'$ for some $a\in A$, $a'\in A'$;
   \item a \emph{split face}\index{face!split} if $A'$ is convex and for each $x\in X\setminus (A\cup A')$ there is a unique $\lambda\in (0,1)$ and a unique pair $a\in A$, $a'\in A'$ such that $x=\lambda a+(1-\lambda)a'$.  
\end{itemize}
Clearly, any split face is parallel but the converse is not true. If $A$ is a split face (or at least a parallel face), there is a canonical mapping assigning to each $x\in X\setminus (A\cup A')$ the unique value $\lambda$ from the definition. If we extend it by the values $1$ on $A$ and $0$ on $A'$, we clearly obtain an affine function on $X$. We will denote it by \gls{lambdaA}. We finish this section by noting
that in a simplex any closed face is split by \cite[Theorem II.6.22]{alfsen}.

\subsection{Distinguished spaces of affine functions}\label{ssc:meziprostory}

Given a compact convex set $X$, we denote by \gls{AbX} the space of all real-valued bounded affine functions on $X$. This space equipped with the supremum norm and the pointwise order is an  
ordered Banach space. It will serve as the basic surrounding space for our investigation.
 
We will further consider the following distinguished subsets of $A_b(X)$:

\begin{itemize}
\item \gls{AcX} stands for the space of all affine continuous functions on $X$. It is a closed linear subspace of $A_b(X)$.

\item \gls{A1(X)} stands for the space of all affine Baire-one functions on $X$. It is a closed subspace of $A_b(X)$. We also note that by the Mokobodzki theorem (see, e.g., \cite[Theorem 4.24]{lmns}) any affine Baire-one function is a pointwise limit of a bounded sequence of affine continuous functions, i.e., it is \emph{of the first affine class}.\index{function!of the first affine class} 

\item \gls{Al(X)} denotes the set of all real-valued lower semicontinuous affine functions on $X$. This is not a linear space, but it is a convex cone contained in $A_b(X)$.
\item \gls{As(X)} denotes the space $A_l(X)-A_l(X)$. It is a linear subspace of $A_b(X)$. This space need not be closed in $A_b(X)$ by Proposition~\ref{P:Baire-srovnani}(b) below. So, we will consider also its closure $\overline{A_s(X)}$.

\item $A_b(X)\cap\Bo_1(X)$ is the space of all affine functions of the first Borel class on $X$ (note that any such function is automatically bounded). It is a closed subspace of $A_b(X)$ which contains $A_1(X)\cup\overline{A_s(X)}$.

\item \gls{Af(X)} stands for the space of all fragmented affine functions on $X$. Recall that any fragmented affine function on $X$ has a point continuity (due to Theorem~\ref{T:a}) and thus it is bounded on $X$ by\cite[Lemma 4.20]{lmns}. It easily follows from Theorem~\ref{T:a} that $A_f(X)$ is a closed subspace of $A_b(X)$. Moreover, by Theorem~\ref{T:b} we deduce that $A_b(X)\cap \Bo_1(X)\subset A_f(X)$.
\item \gls{Asa(X)} denotes the space of all \emph{strongly affine} functions\index{function!strongly affine}, i.e., the space of all universally measurable functions $f\colon X\to\er$ satisfying $\mu(f)=f(r(\mu))$ for each $\mu\in M_1(X)$. It is clear that any strongly affine function is affine and it follows from the proof of \cite[Satz 2.1.(c)]{krause} that any strongly affine function is bounded. Hence, $A_{sa}(X)$ is a closed subspace of $A_b(X)$. Moreover, any fragmented affine function is strongly affine, see \cite[Theorem 4.21]{lmns}.
\end{itemize}
We summarize the above-mentioned inclusions:
\begin{equation}\label{eq:prvniinkluze}
\begin{array}{ccccccccc}
A_c(X)&\subset& A_1(X)&\subset& A_b(X)\cap\Bo_1(X)&\subset&A_f(X)&\subset&A_{sa}(X) \\
\cap&&&&\cup &&&&\cap \\
A_l(X)&\subset& A_s(X) &\subset& \overline{A_s(X)}& && &A_b(X).
\end{array}
\end{equation}

In case $X$ is metrizable, the situation is simpler. More specifically, in this case we have:
\begin{equation}\label{eq:prvniinkluze-metriz}
\begin{aligned}
A_c(X)&\subset A_l(X)\subset A_s(X)\subset\overline{A_s(X)}\subset A_1(X)=A_b(X)\cap \Bo_1(X)=A_f(X)\\&\subset A_{sa}(X)\subset A_b(X).\end{aligned}
\end{equation}

We will further consider spaces 
$$(A_c(X))^\mu, (A_1(X))^\mu, (A_c(X))^\sigma, (A_s(X))^\mu, (A_b(X)\cap\Bo_1(X))^\mu, (A_f(X))^\mu.$$
By Lemma~\ref{L:muclosed is closed} all these families are closed linear subspaces of $A_b(X)$, in particular $\overline{A_s(X)}\subset (A_s(X))^\mu$.

\subsection{$F$-spaces and their centres}
\label{ssc:fspaces}

An \emph{$F$-space with unit}\index{F-space with unit@$F$-space with a unit} is a partially ordered Banach space $A$ with closed positive cone $A^+$ together with an element $e\in A$ of unit norm such that the closed unit ball of $A$ satisfies
\[
B_A=\{a\in A\setsep -e\le a\le e\}.
\]
The element $e$ is then called \emph{the unit} of $A$. The norm in such a space may be computed by the formula
$$\norm{a}=\min \{ c\ge 0\setsep -ce\le a\le ce\}.$$
An abstract theory of $F$-spaces and related types of ordered Banach spaces is developed in \cite{perdrizet}.
Examples of $F$-spaces with unit include spaces $A_c(X)$, $A_b(X)$ and other spaces defined in Section~\ref{ssc:meziprostory} above. The role of the element $e$ plays the constant function equal to $1$.

In fact, any $F$-space with unit may be represented as $A_c(X)$ for a suitable $X$. Indeed, let $A$ be an $F$-space with unit $e$. Denote
$$\gls{S(A)}=\{\varphi\in A^*\setsep \norm{\varphi}\le 1\ \&\ \varphi(e)=1\}.$$
Then $S(A)$ is obviously a convex weak$^*$-compact subset of $A^*$. It is called the \emph{state space}\index{state space!of an $F$-space} of $A$. Then $A$ is identified with $A_c(S(A))$ by the following lemma which is essentially known in the theory of $F$-spaces.

\begin{lemma}\label{L:F-spaces}
 Let $A$ be an $F$-space with unit $e$. Then the following assertions hold.
 \begin{enumerate}[$(i)$]
     \item $B_{A^*}=\co(S(A)\cup (-S(A)))$.
     \item The operator $T:A\to A_c(S(A))$ defined by
    $$T(a)(\varphi)=\varphi(a),\quad a\in A, \varphi\in S(A)$$
    is a linear order-preserving isometry of $A$ onto $A_c(S(A))$.
 \end{enumerate}
  \end{lemma}

\begin{proof}
     $(i)$: Inclusion `$\supset$' is obvious. Further, the set on the right-hand side is clearly a convex symmetric weak$^*$-compact set. So, to prove the equality it is enough to show that $S(A)\cup (-S(A))$ is a $1$-norming subset of $B_{A^*}$. To this end fix any $a\in A\setminus\{0\}$. We may find maximal $\alpha\in\er$ and minimal $\beta\in\er$ such that $\alpha e\le a\le \beta e$. Then $\norm{a}=\max\{-\alpha,\beta\}$. Up to passing to $-a$ if necessary we may and shall assume that $\norm{a}=\beta>0$. We get $\norm{a+e}=\beta+1$. The Hahn-Banach theorem yields some $\varphi\in B_{A^*}$ with $\varphi(a+e)=\beta+1$. It follows that
     $$\beta=\norm{a}\ge\varphi(a)\ge \beta+1-\varphi(e),$$
     hence $\varphi(e)=1$ (thus $\varphi\in S(A)$) and $\varphi(a)=\beta=\norm{a}$. This completes the argument.

   $(ii)$:  It is clear that $T$ is a well-defined linear operator on $A$ with values in $A_c(S(A))$ and that it is  order-preserving. By the very definition we get $\norm{T}\le 1$. In the proof of (i) we have showed that
   $S(A)\cup (-S(A))$ is a $1$-norming subset of $B_{A^*}$. It follows that $T$ is an isometry. It remains to prove that $T$ is surjective. To this end fix any $f\in A_c(S(A))$. Then we define $\widetilde{f}:B_{A^*}\to \er$ by
 $$\widetilde{f}(t\varphi_1-(1-t)\varphi_2)=t f(\varphi_1)-(1-t) f(\varphi_2),\quad \varphi_1,\varphi_2\in S(A), t\in[0,1].$$
 By (i) we know that each $\varphi\in B_{A^*}$ may be expressed as $t\varphi_1-(1-t)\varphi_2$ for some $\varphi_1,\varphi_2\in S(A)$ and $t\in[0,1]$. Further,
$\widetilde{f}$ is well defined as whenever
$$t\varphi_1-(1-t)\varphi_2=s\psi_1-(1-s)\psi_2,$$
we get
$$t\varphi_1+(1-s)\psi_2=s\psi_1+(1-t)\varphi_2.$$
By \cite[Part II, Theorem 6.2]{alfsen-effros-ann} the norm is additive on the positive cone of $A^*$, so $t+1-s=s+1-t$, in other words, $s=t$. Now, using the affinity of $f$, we easily deduce that
$$t f(\varphi_1)+(1-s)f(\psi_2)=sf(\psi_1)+(1-t)f(\varphi_2).$$
It follows that
$\widetilde{f}$ is a well-defined affine function. Moreover, $\widetilde{f}(0)=0$ and $\widetilde{f}$ is weak$^*$-continuous by a quotient argument. (Indeed, the mapping $q:S(A)\times S(A)\times [0,1]\to B_{A^*}$ defined by $q(\varphi_1,\varphi_2,t)=t\varphi_1-(1-t)\varphi_2$ is a quotient mapping and $\widetilde{f}\circ q$ is continuous.) It follows by the Banach-Dieudonn\'e theorem that $f=T(a)$ for some $a\in A$. 
\end{proof}

If $W$ is a partially ordered Banach space, we denote by $\frd(W)$ the \emph{ideal center}\index{ideal center} of the ordered algebra $L(W)$ of bounded linear operators on $W$, i.e., 
$$\gls{D(W)}=\{T\in L(W)\setsep \exists\lambda \ge0\colon -\lambda I\le T\le \lambda I\}.$$
Then $\frd(W)$ is an algebra.
 We refer the reader to \cite{alfsen-effros-ann} or \cite{wils} for a detailed discussion on properties of these algebras. We note that in \cite[\S7]{alfsen}, elements of $\frd(W)$ are called \emph{order-bounded} operators, but we do not use this term in order to avoid confusion with one of the standard notions from the theory of Banach lattices.

If $A$ is an $F$-space with unit $e$, the \emph{center}\index{center!of an $F$-space} $Z(A)$ of $A$ is defined by 
$$\gls{Z(A)}=\{Te\setsep T\in\frd(A)\}$$ 
(see \cite[page 159]{alfsen} for details).

\begin{lemma}\label{L:extenze na bidual}
    Let $A$ be a Banach space. For any bounded linear operator $T:A\to A^{**}$ we define 
    $$\widehat{T}=\left(T^*\circ \kappa_{A^*}\right)^*,$$
    where $\kappa_E$ denotes the canonical embedding of a Banach space $E$ into $E^{**}$.  Then the following assertions hold.
    \begin{enumerate}[$(a)$]
        \item $\widehat{T}$ is a bounded linear operator on $A^{**}$ such that $\widehat{T}\circ \kappa_A=T$ (i.e., $\widehat{T}$ extends $T$).
        \item $\widehat{\kappa_A}=I_{A^{**}}$.
        \item If $A$ is an ordered Banach space and $T\ge0$, then $\widehat{T}\ge0$.
        \item Assume that $A$ is an $F$-space with unit $e$ and $T$ satisfies $-\lambda\kappa_A\le T\le\lambda\kappa_A$ for some $\lambda>0$. Then $\widehat{T}\in\frd(A^{**})$ and hence $T(e)\in Z(A^{**})$.
    \end{enumerate}
\end{lemma}

\begin{proof}
    $(a)$: It is clear that $\widehat{T}$ is a bounded linear operator on $A^{**}$. To prove the equality fix any $a\in A$ and $a^*\in A^*$. Then
    $$\begin{aligned}
        \widehat{T}(\kappa_A(a))(a^*)&=(T^*\circ\kappa_{A^*})^*(\kappa_A(a))(a^*)=\kappa_A(a)(T^*(\kappa_{A^*}(a^*)))\\&=T^*(\kappa_{A^*}(a^*))(a)=\kappa_{A^*}(a^*)(T(a))=T(a)(a^*), \end{aligned}$$
    which completes the argument.
    
    $(b)$: Fix $a^{**}\in A^{**}$ and $a^*\in A^*$ and compute:
    $$\widehat{\kappa_A}(a^{**})(a^*)=((\kappa_A)^*\circ \kappa_{A^*})^*(a^{**})(a^*)
    =a^{**}((\kappa_A)^*(\kappa_{A^*}(a^*))).$$
    Given $a\in A$ we have
    $$(\kappa_A)^*(\kappa_{A^*}(a^*))(a)
    =\kappa_{A^*}(a^*)(\kappa_A(a))=
    \kappa_A(a)(a^*)=a^*(a),$$
    so $(\kappa_A)^*(\kappa_{A^*}(a^*))=a^*$ and the proof is complete.

    $(c)$: Assume $T\ge 0$. Fix $a^{**}\in A^{**}$ such that $a^{**}\ge0$ and $a^*\in A^*$ with $a^*\ge0$. Then
    $$\widehat{T}(a^{**})(a^*)=(T^*\circ \kappa_{A^*})^*(a^{**})(a^*)
    =a^{**}(T^*(\kappa_{A^*}(a^*))).$$
    Given $a\in A$ with $a\ge0$ we have
    $$T^*(\kappa_{A^*}(a^*))(a)=\kappa_{A^*}(a^*)(T(a))=T(a)(a^*)\ge0,$$
    hence $T^*(\kappa_{A^*}(a^*))\ge0$ and thus also $a^{**}(T^*(\kappa_{A^*}(a^*)))\ge0$.

    $(d)$: Under the given assumptions we deduce from $(b)$ and $(c)$ that $-\lambda I_{A^{**}}\le \widehat{T}\le\lambda I_{A^{**}}$. Since $\kappa_A(e)$ is the unit of $A^{**}$ and by (a) we have $\widehat{T}(\kappa_A(e))=T(e)$, the proof is complete.
\end{proof}

We will use several times the following easy abstract lemma on continuity.

\begin{lemma}\label{L:spojitost monot}
    Let $A$ be an ordered vector space. Let $T:A\to A$ be a linear operator such that $0\le T\le I$.
    If $(x_\nu)$ is a non-decreasing net in $A$ with supremum $x\in A$, then $(Tx_\nu)$ is a non-decreasing net in $A$ with supremum $Tx$.
    I.e., we have
    $$x_\nu\nearrow x\mbox{ in }A\Longrightarrow T(x_\nu)\nearrow T(x) \mbox{ in }A.$$
\end{lemma}

\begin{proof}
    Assume that $(x_\nu)$ is a non-decreasing net with supremum $x$. Since $T\ge0$, $(Tx_\nu)$ is a non-decreasing net and $Tx$ is its upper bound. It remains to prove that it is the least upper bound.
    
    So, let $y$ be any upper bound of $(Tx_\nu)$. 
    For each $\nu$ we have $x-x_\nu\ge0$, hence
    $$0\le T(x)-T(x_\nu)\le x-x_\nu,$$
    so $$y\ge T(x_\nu)\ge T(x)-x+x_\nu.$$
    Hence, $y$ is an upper bound of $(T(x)-x+x_\nu)$. But this net has supremum $T(x)-x+x=T(x)$, so $y\ge T(x)$. 

    Therefore $T(x)$ is the least upper bound of $(T(x_\nu))$ and the proof is complete.
\end{proof}

\subsection{Function spaces}\label{ssc:ch-fs}

An important source of examples of compact convex sets is provided by state spaces of function spaces. Therefore we recall basic facts from the theory of function spaces described in detail in \cite[Chapter 3]{lmns}.

If $K$ is a compact (Hausdorff) space and $E\subset C(K)$ is a subspace of $C(K)$ containing constant functions and separating points of $K$, we consider $E$ to be a \emph{function space}\index{function space}. Function spaces generalize spaces of affine continuous functions -- if $X$ is a compact convex set, then $E=A_c(X)$ is a function space. 

Let 
\[
S(E)=\{\varphi\in E^*\setsep \norm{\varphi}=\varphi(1)=1\}
\]
endowed with the weak$^*$ topology. Then $S(E)$ is a compact convex set. Its elements are called \emph{states on $E$} and $S(E)$ is called the \emph{state space}\index{state space!of a function space} of $E$. We also note that $\varphi\in S(E)$ if and only if $\varphi$ is a positive functional of norm one.

Let $\gls{phi}\colon K\to S(E)$ be defined as the evaluation mapping. Then $\phi$ is a homeomorphic injection. Further, define $\Phi\colon E\to A_c(S(E))$ by 
\[
\gls{Phi}(h)(\varphi)=\varphi(h),\quad \varphi\in S(E), h\in E.
\]
Then $\Phi$ is an isometric isomorphism of $E$ into $A_c(S(E))$. If $E$ is closed, $\Phi$ is moreover surjective (see \cite[Propostion 4.26]{lmns}).

For each $x\in K$, let
\[
\gls{Mx(E)}=\{\mu\in M_1(K)\setsep \mu(h)=h(x) \mbox{ for each } h\in E\}.
\]
Then $M_x(E)$ is nonempty since it contains at least the Dirac measure $\ep_x$.
Further, we define the Choquet boundary $\Ch_E K$\index{Choquet boundary} as 
\[
\gls{ChE(K)}=\{x\in K\setsep M_x(E)=\{\ep_x\}\}.
\]
Then $\phi(\Ch_E K)=\ext S(E)$ (see \cite[Proposition 4.26(d)]{lmns}).

If $K=X$ is a compact convex set and $E=A_c(X)$, we obtain that $\Ch_{E} X=\ext X$ (see \cite[Theorem 2.40]{lmns}).

\begin{remark}
    Let $E$ be a closed function space on a compact space $K$. Then $E$ (equipped with the inherited norm and the pointwise order) is an $F$-space with unit (the unit is the constant function $1_K$).
    The state space defined in the present section coincide with the state space from Section~\ref{ssc:fspaces}. Moreover, the above defined operator $\Phi$ coincide with the operator $T$ from Lemma~\ref{L:F-spaces}(ii).
\end{remark}

\section{Intermediate function spaces}\label{s:IFS}

Let $X$ be a compact convex set. 
Any closed subspace $H\subset A_b(X)$ containing $A_c(X)$ will be called \emph{intermediate function space}\index{intermediate function space}. 
Further, we say that an intermediate function space $H$ is \emph{determined by extreme points}\index{intermediate function space!determined by extreme points} if 
$$\forall u\in H\colon \inf u(\ext X)\le u\le \sup u(\ext X).$$

Examples of intermediate function spaces include the closed subspaces of $A_b(X)$ mentioned in Section~\ref{ssc:meziprostory}.
The presents section is devoted mainly to various representations and general properties of intermediate function spaces. We start by observing that
any intermediate function space is an $F$-space with unit $1_X$ in the terminology from Section~\ref{ssc:fspaces} above. Therefore we may consider the state space $S(H)$\index{state space!of an intermediate function space}. The following lemma is a more precise version of Lemma~\ref{L:F-spaces} in this context.

\begin{lemma}\label{L:intermediate}
Let $H$ be an intermediate function space.
\begin{enumerate}[$(a)$]
    \item The operator $T:H\to A_c(S(H))$ defined by
    $$T(a)(\varphi)=\varphi(a),\quad a\in H, \varphi\in S(H)$$
    is a linear order-preserving isometry of $H$ onto $A_c(S(H))$.
    \item For $x\in X$ define
    $$\gls{iota}(x)(a)=a(x),\quad a\in H.$$
    Then $\iota(x)\in S(H)$. Moreover, $\iota:X\to S(H)$ is a one-to-one affine mapping and $\iota(X)$ is dense in $S(H)$.
    \item If $H=A_c(X)$, then $\iota$ is an affine homeomorphism of $X$ onto $S(A_c(H))$.
    \item $H$ is determined by extreme points if and only if $\ext S(H)\subset\overline{\iota(\ext X)}$.
    \item For any $\varphi\in S(H)$ there is a unique $\gls{pi}(\varphi)\in X$ such that
    $$\varphi(a)=a(\pi(\varphi)),\quad a\in A_c(X).$$
    Moreover, $\pi:S(H)\to X$ is a continuous affine surjection.
    \item $\pi\circ\iota=\mbox{\rm id}_X$.
\end{enumerate}
\end{lemma}

\begin{proof} Assertion $(a)$ follows from Lemma~\ref{L:F-spaces}.
  
 Define $\iota$ as in $(b)$. Obviously $\iota(x)\in S(H)$ for $x\in X$ and $\iota$ is clearly an affine mapping. Since $A_c(X)$ separates points of $X$ by the Hahn-Banach theorem and $H\supset A_c(X)$, we deduce that $\iota$ is one-to-one. Further, given any $f\in A_c(S(H))$, by $(a)$ we find $a\in H$ with $f=T(a)$.
 Since $T$ is an isometry, we deduce
 $$
 \begin{aligned}
    \norm{f}&=\norm{a}=\sup\{\abs{a(x)}\setsep x\in X\}
=\sup\{\abs{\iota(x)(a)}\setsep x\in X\}\\
&=\sup\{\abs{f(\iota(x))}\setsep x\in X\}.
 \end{aligned}
 $$
It follows that $\iota(X)$ is dense in $S(H)$ and the proof of $(b)$ is completed.

  $(c)$: If $H=A_c(X)$, then $\iota$ is continuous (from $X$ to the weak$^*$-topology), hence $\iota(X)$ is compact. We conclude using $(b)$.
 
 $(d)$: Assume first that $H$ is determined by extreme points. It follows that for each $a\in H$ we have
 $$\sup\{ \iota(x)(a)\setsep x\in \ext X\}=\sup\{ \iota(x)(a)\setsep x\in X\}=\sup\{ \varphi(a)\setsep \varphi\in S(H)\},$$
  where the second equality follows from the density of $\iota(X)$ in $S(H)$. Hence, the Hahn-Banach separation theorem implies that
  $S(H)=\overline{\co \iota(\ext X)}$.
  Milman's theorem now yields
  $\ext S(H)\subset \overline{\iota(\ext X)}$.
  
  Conversely, assume  $\ext S(H)\subset \overline{\iota(\ext X)}$. Let $a\in H$ and $c\in\er$ be such that $a\le c$ on $\ext X$. Then $T(a)\le c$ on $\iota(\ext X)$. Since $T(a)\in A_c(S(H))$, we deduce that $T(a)\le c$ on $\overline{\iota(X)}\supset\ext S(H)$. Thus $T(a)\le c$ on $S(H)$  and we deduce that $a\le c$ on $X$.
  
  $(e)$: Let $\varphi\in S(H)$. Then $\varphi|_{A_c(X)}\in S(A_c(X))$ and clearly $\varphi\mapsto \varphi|_{A_c(X)}$ is a continuous affine surjection of $S(H)$ onto $S(A_c(H))$. We may conclude using $(c)$.
  
  Assertion $(f)$ is obvious.
\end{proof}

As a consequence we get the following representation of states on $H$.

\begin{lemma}\label{L:reprezentace}
Let $H$ be an intermediate function space. Then for each $\varphi\in S(H)$ there is a net $(x_\nu)$ in $X$ such that
$$\varphi(a)=\lim_\nu a(x_\nu),\quad a\in H.$$
Moreover, the net $(x_\nu)$ converges in $X$ to $\pi(\varphi)$.

Conversely, if $(x_\nu)$ is a net in $X$ such that $\lim_\nu a(x_\nu)$ exists for each $a\in H$, then 
$$\varphi(a)=\lim_\nu a(x_\nu),\quad a\in H$$
defines a state on $H$ and $\lim_\nu x_\nu=\pi(\varphi)$.
\end{lemma}

\begin{proof}
Let $\varphi\in S(H)$. By Lemma~\ref{L:intermediate}$(b)$ there is a net $(x_\nu)$ in $X$ such that $\iota(x_\nu)\to\varphi$ in $S(H)$. This proves the existence of $(x_\nu)$. Using moreover assertion $(e)$ of Lemma~\ref{L:intermediate} we get that
$$\lim_\nu a(x_\nu)=a(\pi(\varphi)),\quad a\in A_c(X).$$
Thus it follows from assertion $(c)$ that $x_\nu\to \pi(\varphi)$.

The converse is obvious.
\end{proof}

Lemma~\ref{L:intermediate} shows, in particular, that given an intermediate function space $H$ on a compact convex set $X$, there is a compact convex set $Y=S(H)$ and two affine mappings $\iota$ and $\pi$ with certain properties. 
The next lemma provides a converse -- given a pair of compact convex sets $X,Y$ and two affine mappings with certain properties, 
we may reconstruct the respective intermediate function space.
Therefore there is a canonical correspondence between pairs $(X,H)$, where $X$ is a compact convex set and $H$ an intermediate function space, and pairs of compact convex sets accompanied by compatible pairs of affine mappings.

\begin{lemma}\label{L:dva kompakty}
Let $X$ and $Y$ be compact convex sets. Assume there is an affine continuous surjection $\varpi:Y\to X$ and an affine injection $\jmath:X\to Y$ with $\jmath(X)$ dense in $Y$ and such that $\varpi\circ\jmath=\mbox{\rm id}_X$. Then the following assertions are valid.
\begin{enumerate}[$(a)$]
    \item The operator $U:A_c(Y)\to A_b(X)$ defined by $U(f)=f\circ\jmath$ is an isometric injection and $H=U(A_c(Y))$ is an intermediate function space.
    \item $H$ is determined by extreme points if and only if $\ext Y\subset\overline{\jmath(\ext X)}$.
    \item Let $T$, $\pi$ and $\iota$ be the mappings associated to $H$ by Lemma~\ref{L:intermediate}. Further, let $\iota_Y:Y\to S(A_c(Y))$ be the mapping associated to $A_c(Y)$ (in place of $H$) by Lemma~\ref{L:intermediate}. Then the following holds:
    \begin{enumerate}[$(i)$]
        \item The dual operator $U^*$ maps $S(H)$ homeomorphically onto $S(A_c(Y))$;
        \item $\iota=(U^*|_{S(H)})^{-1}\circ \iota_Y\circ\jmath$;
        \item $\pi=\varpi\circ (\iota_Y)^{-1}\circ U^*|_{S(H)}$;
        \item $Tf=U^{-1}(f)\circ(\iota_Y)^{-1}\circ U^*|_{S(H)}$ for $f\in H$.
    \end{enumerate}
\end{enumerate}
\end{lemma}

\begin{proof}
$(a)$: Any $f\in A_c(Y)$ is bounded and $\jmath(X)$ is dense in $Y$, hence $U$ is an isometric embedding of $A_c(Y)$ to $A_b(X)$. Moreover, if $a\in A_c(X)$, then $a\circ\varpi\in A_c(Y)$ and 
$$U(a\circ\varpi)=a\circ\varpi\circ\jmath=a,$$
so $A_c(X)\subset U(A_c(Y))$. Hence, $H=U(A_c(Y))$ is indeed an intermediate function space.

Assertion $(c)$ follows by a straightforward calculation and assertion $(b)$ is a consequence of $(c)$ and Lemma~\ref{L:intermediate}$(d)$.
\end{proof}

There is one more view on intermediate function spaces. We may look at them as  spaces in between $A_c(X)$ and $(A_c(X))^{**}$. The following lemma explains it.

\begin{lemma}\label{L:bidual}
    Let $X$ be a compact convex set. Let $\iota:X\to S(A_c(X))$ be the mapping from Lemma~\ref{L:intermediate}. Set $Y=(B_{(A_c(X))^*},w^*)$.
    Then the following assertions hold.
    \begin{enumerate}[$(a)$]
        \item There is a (unique) linear surjective isometry $\Psi:A_b(X)\to (A_c(X))^{**}$ such that
        $$\Psi(f)(\iota(x))=f(x),\quad x\in X, f\in A_b(X).$$
        Moreover, $\Psi$ is order-preserving and a homeorphism from the topology of pointwise convergence to the weak$^*$ topology.
        \item $\Psi(f)|_Y$ is continuous if and only if $f\in A_c(X)$;
       \item $\Psi(f)|_Y\in (A_c(Y))^\sigma$ if and only if $f\in (A_c(X))^\sigma$;
        \item $\Psi(f)|_Y$ is Baire-one if and only if $f\in A_1(X)$;
        \item $\Psi(f)|_Y$ is a Baire function if and only if $f$ is a Baire function;
         \item $\Psi(f)|_Y$ is of the first Borel class if and only if $f\in A_b(X)\cap\Bo_1(X)$;
         \item $\Psi(f)|_Y$ is fragmented if and only if $f\in A_f(X)$.
          \item  $\Psi(f)|_Y$ is strongly affine if and only if $f$ is strongly affine.
    \end{enumerate}
\end{lemma}

\begin{proof}
    Assertion $(a)$ is just a reformulation of \cite[Proposition 4.32]{lmns}.

Let us continue by proving assertions $(b)$, $(d)$--$(g)$.
    To prove the `only if' parts  it is enough to observe that $f=\Psi(f)\circ \iota$ and that $\iota$ is continuous (see Lemma~\ref{L:intermediate}).
        To prove the `if parts' we will use a quotient argument as in the proof of \cite[Lemma 5.39(a)]{lmns} (which in fact covers all the cases except for $(g)$). Let $q:X\times X\times [0,1]\to B_{(A_c(X))^*}$ be defined by
    $$q(x,y,t)=t\iota(x)-(1-t)\iota(y),\quad x,y\in X, t\in[0,1].$$
    By combining Lemma~\ref{L:intermediate}(c) with Lemma~\ref{L:F-spaces}(i) we see that $q$ is a continuous surjection. Moreover, given any $f\in A_b(X)$ we have
    $$(\Psi(f)\circ q)(x,y,t)=\Psi(f)(t\iota(x)-(1-t)\iota(y))=t f(x)-(1-t) f(y).$$
    Hence, the properties of $f$ easily transfers to $\Psi(f)\circ q$ and then by Lemma~\ref{L:kvocient} to $\Psi(g)|_Y$.

Assertions $(c)$ follows from $(b)$ using assertion $(a)$.

Assertion $(h)$ follows from \cite[Lemma 5.39(b)]{lmns}.
 \end{proof}

In \cite[Proposition 3.2 and Proposition 4.1]{edwards} it is proved that $Z(A_c(X))\subset Z((A_c(X))^\mu)\subset Z(A_b(X))$. This result was extended in \cite[Proposition 4.7 and Proposition 4.11]{smith-london} to get $Z(A_c(X))\subset Z((A_c(X))^\mu)\subset Z((A_c(X))^\sigma)\subset Z(A_b(X))$ and $Z(A_c(X))\subset Z((A_s(X))^\mu)\subset Z(A_b(X))$.  The next proposition says, in particular, that the fact that some centers are contained in $Z(A_b(X))$ is not specific for the concrete spaces but holds for any intermediate function space. 

\begin{prop}\label{P:ZH subset ZAbX}
    Let $X$ be a compact convex set.
    \begin{enumerate}[$(i)$]
        \item Let $T:A_c(X)\to A_b(X)$ be a linear operator satisfying
        $$-\lambda f\le T(f)\le \lambda f,\quad f\in A_c(X), f\ge0,$$
        for some $\lambda>0$. Then $T$ may be extended to an operator $\widetilde{T}\in \frd(A_b(X))$.
        \item  $Z(H)\subset Z(A_b(X))$ for any  intermediate function space $H$ on $X$.
    \end{enumerate} 
\end{prop}

\begin{proof}
    $(i)$: Set $S=\Psi\circ T$ where $\Psi$ is from Lemma~\ref{L:bidual}. Then $S$ satisfies the assumptions of Lemma~\ref{L:extenze na bidual}(d). The just quoted lemma provides an operator $\widehat{S}$. It remains to set $\widetilde{T}=\Psi^{-1}\circ\widehat{S}\circ\Psi$.

    $(ii)$: Let $H$ be any intermediate function space and let $h\in Z(H)$. Then there is $T\in \frd(H)$ with $T(1)=h$. Note that $T|_{A_c(X)}$ satisfies the assumptions of $(i)$, hence there is $S\in\frd(A_b(X))$ with $S|_{A_c(X)}=T|_{A_c(X)}$. In particular $h=T(1)=S(1)\in Z(A_b(X))$.
\end{proof}

We note that in the proof of (ii) we do not claim that $S$ extends $T$, just $S$ extends the restriction of $T$ to $A_c(X)$. We further note that inclusion $Z(A_c(X))\subset Z(H)$ is not automatic as witnessed by Example~\ref{ex:inkluzeZ} below.

We present one more way of constructing intermediate function spaces, starting from a function space in the sense of Section~\ref{ssc:ch-fs}. This will be useful mainly to simplify constructions of concrete examples.
%Later on we present several examples of compact convex sets with various properties. In order to construct them, we need a general construction described in the next lemma.

\begin{lemma}\label{L:function space}
Let $K$ be a compact space and let $E\subset C(K)$ be a closed function space. Denote 
\[
\ell^\infty(K)\cap E^{\perp\perp}=\left\{h\in\ell^\infty(K)\text{ universally measurable}\setsep
\int h\di\mu=0\mbox{ if }\mu\in E^\perp\right\},
\]
where 
$$E^\perp=\{\mu\in M(K)\setsep \int f\di\mu=0\mbox{ for each }f\in E\}$$
is the annihilator of $E$ in $M(K)=C(K)^*$. Let $X=S(E)$ denote the state space of $E$ and $\phi\colon K\to X$ be the evaluation mapping.

Given $f\in \ell^\infty(K)\cap E^{\perp\perp}$ and $\varphi\in X$, we set 
$$\gls{V}(f)(\varphi)=\int f\di\mu\mbox{ whenever }\mu\in M_1(K)\mbox{ and }\varphi(a)=\int a\di\mu\mbox{ for }a\in E.$$

Then the following assertions are valid:
\begin{enumerate}[$(a)$]
    \item  $V$ is a well-defined linear isometry of $\ell^\infty(K)\cap E^{\perp\perp}$ onto $A_{sa}(X)$. Its inverse is given by the mapping $f\mapsto f\circ \phi$, $f\in A_{sa}(X)$.

    Moreover, $V$ is order-preserving and $V(f_n)\to V(f)$ pointwise whenever $(f_n)$ is a bounded sequence in $ \ell^\infty(K)\cap E^{\perp\perp}$ pointwise converging to $f$.
    \item $V|_E=\Phi$, and thus $V(E)=A_c(X)$. 
    Moreover,
    $$\begin{gathered}
           V(\Ba^b_1(K)\cap E^{\perp\perp})=A_1(X), V(\Ba^b(K)\cap E^{\perp\perp})=A_{sa}(X)\cap \Ba(X), \\
    V(\Bo^b_1(K)\cap E^{\perp\perp})=A_b(X)\cap\Bo_1(X),
    V(\Fr^b(K)\cap E^{\perp\perp})=A_f(X),\\ 
    V(\{f\in \ell^\infty(K)\cap E^{\perp\perp}\setsep f\mbox{ is lower semicontinuous on }K\})=A_l(X). \end{gathered}
    $$
    
    \item Let $H\subset \ell^\infty(K)\cap E^{\perp\perp}$ be a closed subspace containing $E$. Then $V(H)$ is an intermediate function space on $X$. If $\iota\colon X\to S(V(H))$ is the mapping provided by Lemma~\ref{L:intermediate} and $\imath\colon K\to S(V(H))$ is defined by $\imath(x)(Vh)=h(x)$ for $h\in H$ and $x\in K$, then
    $$\iota(\phi(x))=\imath(x),\quad x\in K.$$
    \item The following assertions are equivalent:
    \begin{enumerate}[$(i)$]
        \item   $V(H)$ is determined by extreme points of $S(E)$;
        \item  $H$ is determined by the Choquet boundary of $E$, i.e., for each $h\in H$ we have $\inf h(\Ch_E K)\le h\le \sup h(\Ch_E K)$;
        \item  $\ext S(V(H))\subset\overline{\imath (\Ch_E K)}.$
    \end{enumerate}
   
\end{enumerate}
\end{lemma}

\begin{proof} $(a)$: The fact that $V$ is well defined and maps $\ell^\infty\cap E^{\perp\perp}$ onto $A_{sa}(X)$ is proved in \cite[Theorem 5.40 and Corollary 5.41]{lmns}. But we recall the scheme leading to the proof because it will be needed to prove the remaining assertions.

Let $U\colon \ell^\infty(K)\cap E^{\perp\perp}\to A_{b}(M_1(K))$ be defined by
$$U(h)(\mu)=\int h\di\mu,\quad \mu\in M_1(K), h\in H.$$
Since elements of $\ell^\infty(K)\cap E^{\perp\perp}$ are universally measurable and bounded, $U$ is a well-defined linear operator and $\norm{U}\le 1$. Since $M_1(K)$ contains Dirac measures, $U$ is an isometric injection.
By \cite[Proposition 5.30]{lmns} or \cite[Proposition 3.1]{spurnyrepre} the range of $U$ is contained in $A_{sa}(M_1(K))$. 

Let $\rho\colon M_1(K)\to X$ be defined by
$$\rho(\mu)(f)=\int f\di\mu,\quad f\in E, \mu\in M_1(K).$$
Then $\rho$ is an affine continuous surjection of $M_1(K)$ onto $X$. (This is an easy consequence of the Hahn-Banach theorem, cf. \cite[Section 4.3]{lmns}.)

Fix $h\in \ell^\infty(K)\cap E^{\perp\perp}$. If $\mu_1,\mu_2\in M_1(K)$ are such that $\rho(\mu_1)=\rho(\mu_2)$, then $\mu_1-\mu_2\in E^\perp$ and hence $\int h\di\mu_1=\int h\di\mu_2$.
It follows that there is a (unique) function $V(h)\colon X\to\er$ with $U(h)=V(h)\circ \rho$. Since $\rho$ is an affine continuous surjection, we deduce that $V(h)$ is a strongly affine function and $\norm{V(h)}=\norm{U(h)}$ (see \cite[Proposition 5.29]{lmns}). The linearity of $V$ is clear.

Thus $V$ maps $\ell^\infty(K)\cap E^{\perp\perp}$ into $A_{sa}(X)$. For the proof of its surjectivity, let $f\in A_{sa}(X)$ be given. Then $h=f\circ \phi$ is a bounded universally measurable function on $K$. If $\mu_1,\mu_2\in M_1(K)$ satisfy $\mu_1-\mu_2\in E^\perp$, then $\phi(\mu_1),\phi(\mu_2)$ are probability measures on $X$ with the same barycenter $\rho(\mu_1)=\rho(\mu_2)$ (see \cite[Proposition 4.26(c)]{lmns}). Hence
\[
\begin{aligned}
\int_K h\di\mu_1&=\int_K f\circ\phi\di\mu_1=\int_X f\di(\phi(\mu_1))=f(\rho(\mu_1))=f(\rho(\mu_2))=\cdots\\
&=\int_K h\di\mu_2.
\end{aligned}
\]
Thus $h\in E^{\perp\perp}$. 
Further, for $\varphi\in X$ we find a measure $\mu\in M_1(K)$ with $\rho(\mu)=\varphi$. Then 
\[
Vh(\varphi)=Uh(\mu)=\mu(h)=\mu(f\circ \phi)=(\phi(\mu))(f)=f(r(\phi(\mu)))=f(\rho(\mu))=f(\varphi).
\]
Hence $Vh=f$ and $V$ is surjective.

So far we have proved that $V$ is a well defined surjective isometry and the mapping $f\mapsto f\circ \phi$, $f\in A_{sa}(X)$, is its inverse.

 It is clear that $V$ is order-preserving. The sequential continuity follows from the Lebesgue dominated convergence theorem.

$(b)$: The equality $V|_E=\Phi$ follows from the definitions. To prove the remaining equalities 
 we will use the scheme recalled in the proof of $(a)$. 

To prove inclusions `$\supset$' we fix $f$ in the space on the right-hand side. Then $f$ is strongly affine (see \cite[Theorem 4.21]{lmns}) and hence $f\circ \phi\in E^{\perp\perp}$. Moreover, since $\phi$ is a homeomorphic injection, $f\circ \phi$ shares the descriptive properties of $f$.

Conversely, assume $h$ belongs to the space on the left-hand side. Then $V(h)\in A_{sa}(X)$ by $(a)$. Moreover, $U(h)$ belongs to the respective descriptive class on $M_1(K)$ by  \cite[Lemma 3.2]{lusp} in the case of fragmented functions and by \cite[Proposition 5.30]{lmns} in the remaining cases. Hence $V(h)$ belongs to the same class on $X$ by Lemma~\ref{L:kvocient}.

%(a) If $h\in E$ and $\varphi\in X$, let $\mu\in\M_1(K)$ with $\rho(\mu)=\varphi$ be chosen. Then 
%\[
%Vh(\varphi)=Vh(\rho(\mu))=Uh(\mu)=\int h\di\mu=\rho(\mu)(h)=\varphi(h)=(\Phi h)(\varphi).
%\]
%Hence $V|_E=\Phi$, and thus $V(E)=\Phi(E)=A_c(X)$.

\iffalse(b) Let $h\in \ell^{\infty}(K)$ be a pointwise limit of a bounded sequence $\{h_n\}$ from $E$. Then $h\in E^{\perp\perp}$ by the Lebesgue dominated convergence theorem and the functions $\{\Phi(h_n)\}$ converge by the same reasoning pointwise to $Vh$. Hence $Vh\in A_1(X)$.

Conversely, if $Vh\in A_1(X)$ for some $h\in \ell^\infty(K)\cap E^{\perp\perp}$  and $\{f_n\}$ in $A_c(X)$ is bounded sequence converging pointwise to $f$, then $h_n=f_n\circ \phi\in E$ and converge pointwise to $h=Vh\circ \phi$.

[If $h\in \ell^{\infty}(K)$ is fragmented, $Uh$ is fragmented as well by \cite[Lemma 3.2]{lusp}. Since $\rho$ is a continuous mapping of a compact space $M_1(K)$ onto a compact space $X$, $Vh=Uh\circ\rho$ is fragmented by \cite[Theorem 3]{HoSp} and Theorem~\ref{T:a}.]
\fi
%The converse implication is clear since $\phi$ is a homeomorphic embedding.

$(c)$: Let $\iota\colon S(E)\to S(V(H))$ and $\imath\colon K\to S(V(H))$ be as in the statement. Then for $x\in K$ and $h\in H$ we have
\[
\imath(x)(Vh)=h(x)=Uh(\ep_x)= Vh(\rho(\ep_x))=Vh(\phi(x))=(\iota(\phi(x)))(Vh),
\]
and thus $\imath(x)=\iota(\phi(x))$, $x\in K$.

$(d)$: By Lemma~\ref{L:intermediate}(d), $V(H)$ is determined by extreme points of $S(E)$ if and only if $
\ext S(V(H))\subset \ov{\iota (\ext S(E))}$. Since 
\[
\iota(\ext S(E))=\iota(\phi(\Ch_E K))=\imath(\Ch_E K),
\]
we have the equivalence of $(i)$ and $(iii)$.

Finally, for each $x\in K$ and $h\in H$ we have 
\[
Vh(\phi(x))=Vh(\rho(\ep_x))=Uh(\ep_x)=h(x).
\]
Further, $\phi(\Ch_E K)=\ext S(E)$ (see \cite[Proposition 4.26(d)]{lmns}). From these observations it is easy to see that $(i)$ is equivalent to $(ii)$.
\end{proof}

\section{Multipliers and central elements in intermediate function spaces}\label{s:multi-atp}

Let $X$ be a compact convex set and let $H$ be an intermediate function space on $X$. Since $H$ is an $F$-space, we know from Section~\ref{ssc:fspaces} what is $Z(H)$, its center. There is another important subspace of $H$ -- that of multipliers. In this section we investigate the structure of this subspace and its relationship to the center. This is motivated by \cite[Theorem II.7.10]{alfsen} where this relationship is clarified for $H=A_c(X)$ and by its extensions in \cite[Proposition 4.4 and Proposition 4.9]{smith-london} to $H=(A_c(X))^\sigma$ and $H=(A_s(X))^\mu$. 

We start by two definitions.
A function $u\in H$ is a \emph{multiplier}\index{multiplier!of an intermediate function space} if for any $a\in H$ there is some $b\in H$ such that $b=u\cdot a$ on $\ext X$. The set of multipliers is denoted by $M(H)$. I.e., we have
$$\gls{M(H)} =\{ u\in H \setsep \forall a\in H\;\exists b\in H\colon b=u\cdot a \mbox{ on }\ext X\}.$$
Further, we define the family of \emph{strong multipliers}\index{multiplier!strong} of $H$ as 
\begin{equation*}
\begin{aligned}
\gls{Ms(H)} = \{ u\in H \setsep \forall a\in H\;\exists b\in H\colon b&=u\cdot a\quad  \mu\text{-almost everywhere}  \\& \text{ for each maximal } \mu \in M_1(X)\}.
\end{aligned}
\end{equation*}
Since for each $x\in\ext X$ the Dirac measure $\ep_x$ is maximal, $M^s(H) \subset M(H)$ for each intermediate function space $H$. The converse implication is not true in general (see Example~\ref{ex:dikous-mezi-new} below). However, it holds in many important cases, as we show below. 
First we recall a notion of a standard compact convex set.

\begin{definition}
\label{d:standard}
A compact convex set $X$ is called a \emph{standard compact convex set}\index{standard compact convex set}, provided $\mu(A)=1$ for each maximal $\mu\in M_1(X)$ and each universally measurable $A\supset \ext X$.
\end{definition}

 It follows from \cite[Theorem 3.79(b)]{lmns} that a compact convex set $X$ is standard whenever $\ext X$ is \lin.

\begin{prop}\label{P:rovnostmulti}
Let $H$ be an intermediate function space on a compact convex set $X$. Assume that
\begin{itemize}
\item either $H\subset \Ba^b(X)$,
\item or functions in $H$ are universally measurable and $X$ is a standard compact convex set.
\end{itemize}
Then $M^s(H)=M(H)$.
\end{prop}

\begin{proof}
Let $u \in M(H)$ and $a \in H$ be given. We find $b \in H$ such that $b=u \cdot a$ on $\ext X$. We fix a maximal measure $\mu \in M_1(X)$, and we shall show that $b=u \cdot a$ $\mu$-almost everywhere. Assuming that $H$ consists of Baire functions, the set $[b=u \cdot a]$ is a Baire set containing $\ext X$, and thus $\mu([b=u \cdot a])=1$ by \cite[Theorem 3.79(a)]{lmns}. If, on the other hand, $X$ is a standard compact convex set and the functions from $H$ are universally measurable,  then $[b=u \cdot a]$ is a universally measurable set containing $\ext X$, and thus $\mu([b=u \cdot a])=1$ as well. The proof is finished. 
\end{proof}

We note that unlike the notion of a multiplier, the notion of a strong multiplier seems to be completely new. This is not surprising, since it follows from the above proposition that in previous works devoted to centers and multipliers of spaces $(A_c(X))^\mu$ and $(A_c(X))^\sigma$, the notions of a multiplier and a strong multiplier would be the same. In Proposition~\ref{P:multi strong pro As} below we complement the previous proposition by showing equality $M^s(H)=M(H)$ for $H\subset (A_s(X))^\sigma$. However, when dealing with abstract intermediate function spaces, both these notions seem to be useful. As we have already remarked, they may differ by Example~\ref{ex:dikous-mezi-new} below, but we know no counterexample within `natural' intermediate functions spaces (cf. Question~\ref{q:m=ms} below).

If $H$ is determined by extreme points, the function $b$ from the definition of a multiplier is uniquely determined by $a$, so any multiplier defines an operator. This is the content of the following proposition.

\begin{prop}\label{P:mult}
Let $H$ be an intermediate function space on $X$ determined by extreme points. Let $u\in M(H)$. Then there is a unique mapping $T:H\to H$ such that
$$T(a)(x)=u(x)\cdot a(x),\quad x\in\ext X, a\in H.$$
Moreover, $T$ is a linear operator, $T(1)=u$ and
$$\inf u(X)\cdot I\le T\le \sup u(X)\cdot I,$$
where $I$ is the identity operator on $H$.

In particular, $M(H)\subset Z(H)$.
 \end{prop}

\begin{proof} The proof is completely straightforward. 
\end{proof}

%Since any intermediate function space is an $F$-space, the terminology and notation from Section~\ref{ssc:fspaces} applies. In particular, Proposition~\ref{P:mult} says that $M(H)\subset Z(H)$ whenever $H$ is determined by extreme points.

Now we focus on the converse inclusion, i.e., on the validity of equality $Z(H)=M(H)$. It holds in some special cases, for example if $H=A_c(X)$ (see \cite[Theorem II.7.10]{alfsen}), $H=(A_c(X))^\sigma$ or $H=(A_s(X))^\mu$ (see \cite{smith-london}). We will study its validity in the abstract setting. 

The first tool is the following lemma.

\begin{lemma}\label{L:nasobeni}
Let $X$  be a compact convex set and $H_1,H_2$ two intermediate function spaces on $X$. Assume that $H_1\subset H_2$, 
 $H_2$ is determined by extreme points and for any $x\in\ext X$ the functional $\iota_1(x)$ is an extreme point of $S(H_1)$. (Note that $\iota_1:X\to S(H_1)$ is the mapping provided by Lemma~\ref{L:intermediate}.) Then for any  operator $T\in\frd(H_2)$ we have
$$\forall a\in H_1\colon T(a)=T(1)\cdot a\mbox{ on }\ext X.$$
\end{lemma}

\begin{proof}
First assume that $0\le T\le I$.
Set $z=T(1)$. Fix $x\in \ext X$. We wish to show that $T(a)(x)=z(x)\cdot a(x)$ for $a\in H_1$.

Since $0\le T\le I$, we deduce that $z(x)\in [0,1]$. Assume first $z(x)=0$. Given $a\in H_1$, we get
$$(\inf a) T(1)=T(\inf a)\le T(a)\le T(\sup a)=(\sup a)T(1),$$ 
hence $T(a)(x)=0$.
If $z(x)=1$, replace $T$ by $S=I-T$ and deduce that $S(a)(x)=0$, hence $T(a)(x)=a(x)=a(x)z(x)$ for $a\in H_1$.

Assume  $z(x)\in(0,1)$. Then
$$a(x)=Ta(x) + (I-T)(a)(x)=z(x)\cdot\frac{T(a)(x)}{z(x)}+(1-z(x))\frac{(I-T)a(x)}{1-z(x)}, \quad a\in H_1.$$
Set
$$\varphi_1(a)=\frac{T(a)(x)}{z(x)}\quad\mbox{and}\quad\varphi_2(a)=\frac{(I-T)a(x)}{1-z(x)}$$
for $a\in H_1$. Then $\varphi_1,\varphi_2\in S(H_1)$ and
$\iota_1(x)=z(x)\varphi_1+(1-z(x))\varphi_2$. By the assumption we deduce that $\varphi_1=\varphi_2=\iota_1(x)$, hence $T(a)(x)=a(x)\cdot z(x)$ for $a\in H_1$ and the proof is complete.

Now let $T$ be a general operator in $\frd(H_2)$. Then there is a linear mapping $S$ with $0\le S\le I$ and $\alpha,\beta\in\er$ such that $T=\alpha I+\beta S$. Then, given any $a\in H_1$, 
$$T(a)=\alpha a+\beta S(a)= \alpha a +\beta S(1)\cdot a=(\alpha +\beta S(1))a= T(1)\cdot a$$
on $\ext X$.
\end{proof}

For $H_1=H_2=A_c(X)$ we get the above-mentioned known result covered by \cite[Theorem II.7.10]{alfsen}. There are two important more general cases where the previous lemma may be applied covered by the following corollary.

\begin{cor}\label{cor:rovnost na ext}
Let $H$ be an intermediate function space determined by extreme points.
\begin{enumerate}[$(a)$]
    \item If $T\in\frd(H)$ is an operator, then 
    $$\forall a\in A_c(X)\colon T(a)=T(1)\cdot a\mbox{ on }\ext X.$$
    \item If $\iota(x)\in\ext S(H)$ for each $x\in \ext X$, then $Z(H)=M(H)$.
\end{enumerate}
\end{cor}

The next lemma shows that the connection between multipliers and  operators in ideal center is very tight.

\begin{lemma}\label{L:ob-soucin}
Let $H$ be an intermediate function space determined by extreme points.
Let $T\in\frd(H)$ be an operator. Then $$T(1)\in M(H)  
\Longleftrightarrow \forall a\in H\colon T(a)=T(1)\cdot a\mbox{ on }\ext X.$$ 
\end{lemma}

\begin{proof}
Implication `$\Longleftarrow$' is obvious. To prove the converse assume that $T(1)\in M(H)$.
By Proposition~\ref{P:mult}
 we obtain an operator $S\in\frd(H)$ such that
$$S(a)=T(1)\cdot a\mbox{ on }\ext X \mbox{ for }a\in H.$$
Note that $S(1)=T(1)$.

Now we use the identification $H=A_c(S(H))$ provided by Lemma~\ref{L:intermediate}$(a)$. Then both $S$ and $T$ are  operators in $\frd(A_c(S(H)))$. Hence by Corollary~\ref{cor:rovnost na ext},
$$\begin{gathered}T(a)=T(1)\cdot a\mbox{ on }\ext S(H)\mbox{ for }a\in A_c(S(H)),\\
S(a)=S(1)\cdot a\mbox{ on }\ext S(H)\mbox{ for }a\in A_c(S(H)).\end{gathered}$$
Since $S(1)=T(1)$, we deduce $S=T$. Thus $T(a)=T(1)\cdot a$ on $\ext X$ for each $a\in H$
and the proof is complete.
\end{proof}

This lemma, together with the identification from Lemma~\ref{L:intermediate}$(a)$ used in the proof, inspires a characterization of intermediate function spaces $H$ satisfying the equality $Z(H)=M(H)$.
To formulate the characterization we introduce the following notation.
For a compact convex set $X$ we set
\begin{equation}\label{eq:m(X)}   
\gls{m(X)}=\{x\in X\setsep \forall T\in\frd(A_c(X))\,\forall a\in A_c(X)\colon T(a)(x)=T(1)(x)\cdot a(x)\}.\end{equation}
Obviously, $m(X)$ is a closed subset of $X$. By the above we know it contains $\ext X$.

The promised characterization of equality $Z(H)=M(H)$ is contained in the following proposition which follows from Lemma~\ref{L:ob-soucin} using the identifications $H=A_c(S(H))$ provided by Lemma~\ref{L:intermediate}.

\begin{prop}\label{p:zh-mh}
Let $H$ be an intermediate function space determined by extreme points. Then $Z(H)=M(H)$ if and only if $\iota(\ext X)\subset m(S(H))$ (where $\iota$ is the mapping from Lemma~\ref{L:intermediate}).
\end{prop}

The next section will be devoted to a description of $m(X)$ which enables us to apply Proposition~\ref{p:zh-mh} in concrete cases.

We continue by an example pointing out that in some cases there are only trivial multipliers.

\begin{example}\label{ex:symetricka}
 Let $X$ be a centrally symmetric compact convex subset of a locally convex space which is not a segment and let $H$ be an intermediate function space on $X$. Then the following assertions hold.
 \begin{enumerate}[$(a)$]
     \item $Z(H)$ contains only constant functions.
     \item If $H$ is determined by extreme points, then $M(H)$ contains only constant functions.
 \end{enumerate}   
\end{example}

\begin{proof}
 $(b)$: By shifting the set $X$ if necessary we may assume that $X$ is symmetric around $0$. Assume there is a non-constant $u\in M(H)$. Up to adding a constant function, we may assume that $u(0)=0$. Since $H$ is determined by extreme points, there is $x\in\ext X$ with $u(x)\ne0$. Up to multiplying $u$ by a non-zero constant we may assume $u(x)=1$. Then clearly $u(-x)=-1$.

 Let $y\in \ext X\setminus\{x,-x\}$. Such a point exists as $X$ is not a segment. Set $\alpha=u(y)$. Then $u(-y)=-\alpha$. Since $u$ is a multiplier, there is $b\in H$ such that $b=u^2$ on $\ext X$. It follows that $\alpha^2=1$. So, up to replacing $y$ by $-y$ if necessary, we may assume that $u(y)=1$ (and $u(-y)=-1$). 

 Let $a_0$ be the linear functional on the linear span of $\{x,y\}$ such that $a_0(x)=1$ and $a_0(y)=0$. By the Hahn-Banach extension theorem it may be extended to some $a\in A_c(X)$.
 Since $A_c(X)\subset H$, we deduce that $a\in H$. Thus there is $v\in H$ such that $v=ua$ on $\ext X$. It follows that $0=v(0)=1$, a contradiction.

 $(a)$: Since $Z(H)$ is canonically identified with $M(A_c(S(H)))$, it is enough to show that $S(H)$ is centrally symmetric and then use $(b)$. Let $\iota:X\to S(H)$ be the mapping from Lemma~\ref{L:intermediate}. We claim that $\iota(0)$ is the center of symmetry of $S(H)$. To prove this it is enough to observe that $2\iota(0)-\varphi\in S(H)$ whenever $\varphi\in S(H)$. So, take any $\varphi\in S(H)$ and let $\psi=2\iota(0)-\varphi$. Then $\psi$ is clearly a continuous linear functional on $H$. Moreover, $\psi(1_X)=2-\varphi(1_X)=1$. Further, assume $f\in H$, $f\ge0$. Then for any $x\in X$ we have
 $f(x)\le f(x)+f(-x)=2f(0)$, hence $2f(0)-f\ge0$. So,
 $$\psi(f)=2f(0)-\varphi(f)=\varphi(2f(0)-f)\ge0.$$
 We deduce that $\psi\in S(H)$ and the proof is complete.
 \iffalse
 The symmetry is then defined by
 $$\eta(\varphi)=2\iota(0)-\varphi,\quad\varphi\in S(H).$$
 So, take any $\varphi\in S(H)$. Then $\eta(\varphi)$ is clearly a continuous linear functional on $H$. Moreover, $\eta(\varphi)(1_X)=2-\varphi(1_X)=1$. Further, assume $f\in H$, $f\ge0$. Then for any $x\in X$ we have
 $f(x)\le f(x)+f(-x)=2f(0)$, hence $2f(0)-f\ge0$. So,
 $$\eta(\varphi)(f)=2f(0)-\varphi(f)=\varphi(2f(0)-f)\ge0.$$
 We deduce that $\eta(\varphi)\in S(H)$. Since $\eta(\iota(0))=\iota(0)$ and $\eta^{-1}=\eta$, the proof is complete.\fi
\end{proof}

Note, that if $X$ is a segment, necessarily $M(H)=H=A_c(X)=A_b(X)$.

Now we pass to some basic properties of centers and spaces of multipliers. We start by proving they are closed subspaces.

\begin{lemma}\label{L:uzavrenost}
Let $X$ be a compact convex set and let $H$ be an intermediate function space on $X$.
\begin{enumerate}[$(i)$]
    \item $Z(H)$ is a closed linear subspace of $H$.
    \item If $H$ is determined by extreme points, then both $M(H)$ and $M^s(H)$ are closed linear subspaces of $H$.
   
\end{enumerate}
 \end{lemma}

\begin{proof}
    $(ii)$: Assume that $H$ is determined by extreme points. It is clear that $M(H)$ and $M^s(H)$ are linear subspaces of $H$. 
    
    Let $(m_n)$ be a sequence in $M(H)$ converging in $H$ to some $m\in H$. We are going to prove that $m\in M(H)$.
    Let $a\in H$. For each $n\in \en$ there is $b_n\in H$ such that $b_n=m_n \cdot a$ on $\ext X$. Then $b_n|_{\ext X}$ converge uniformly to $m a|_{\ext X}$. Since $H$ is determined by extreme points, we deduce that $(b_n)$ is uniformly Cauchy on $X$. Thus there is $b\in H$ such that $b_n\to b$ in $H$. Clearly $b=m \cdot a$ on $\ext X$, and thus $M(H)$ is closed. 

    To show that $M^s(H)$ is closed, we proceed similarly. Let $(m_n)$ be a sequence in $M^s(H)$ converging in $H$ to $m\in H$. Let $a\in H$ be arbitrary. Let $b_n$ and $b$ be as above. Fix any maximal measure $\mu \in M_1(X)$. Since $m_n\in M^s(H)$ and $H$ is determined by extreme points, we deduce that
    $b_n=m_n \cdot a$ $\mu$-almost everywhere. Hence (using the $\sigma$-additivity of $\mu$) we find a $\mu$-null set $N\subset X$ such that
    $$\forall x\in X\setminus N\;\forall n\in\en\colon b_n(x)=m_n(x)\cdot a(x),$$
    thus $b(x)=m(x)\cdot a(x)$ for $x\in X\setminus N$. Therefore $b=m\cdot a$ $
    \mu$-almost everywhere. 
    This completes the proof.
    \iffalse
        Let $a\in H$, and we fix a maximal measure $\mu \in M_1(X)$. For each $n\in \en$ we find $b_n\in H$ such that $b_n=m_n \cdot a$ $\mu$-almost everywhere. Thus we can find a set $A_n \subset X$ such that $\mu(A_n)=1$ and $b_n=m_n \cdot a$ on $A_n$. Further, as above we find a function $b \in H$ which is the uniform limit of $(b_n)$. Now, it is clear that the set $b=m \cdot a$ contains $\bigcap_{n \in \en} A_n$, hence $b=m \cdot a$ $\mu$-almost everywhere. Thus $m \in M^s(H)$, which completes the proof. \fi

    $(i)$: Let $T:H\to A_c(S(H))$ be the operator from Lemma~\ref{L:intermediate}.
    Then $T(Z(H))=Z(A_c(S(H)))=M(A_c(S(H)))$ which is a closed subspace of $A_c(S(H))$ by $(ii)$.

\end{proof}

When $H$ is an intermediate function space, the spaces $H^\mu$ and $H^\sigma$ are closed subspaces of $A_b(X)$ by Lemma~\ref{L:muclosed is closed}. Hence, they are intermediate function spaces as well. Let us now look at the relationship between multipliers of $H$, $H^\mu$ and $H^\sigma$. The following result partially settles this question.

\begin{prop}
    \label{p:multi-pro-mu}
 Let $H$ be an intermediate function space.
 \begin{enumerate}[$(i)$]
     \item Assume that $H^\mu$ is determined by extreme points. Then $(M(H))^\mu\subset M(H^\mu)$ and $(M^s(H))^\mu\subset M^s(H^\mu)$.
          In particular, if $H=H^\mu$, then $M(H)=(M(H))^\mu$ and  $M^s(H)=(M^s(H))^\mu$.
     \item Assume that $H \subset A_{sa}(X)$ and $H^\sigma$ is determined by extreme points. Then $(M^s(H))^\sigma\subset M^s(H^\sigma)$. In particular,
     \[M^s(H) \subset M^s(H^\mu) \subset M^s(H^\sigma).\] 
     If, additionally, $H=H^\sigma$, then $M^s(H)=(M^s(H))^\sigma$.
      \end{enumerate}
\end{prop}

\begin{proof}
$(i)$: We start by showing that $M(H)\subset M(H^\mu)$. To this end let $m\in M(H)$ be given. We assume that $m\ge0$, otherwise we would just add a suitable constant to it. Let
\[
\F=\{a\in H^\mu\setsep \exists b\in H^\mu\colon b=am\text{ on }\ext X\}.
\]
Then $\F$ is clearly a linear subspace of $H^\mu$ containing $H$. Further, $\F$ is stable with respect to taking pointwise limits of monotone bounded sequences. Indeed, let $(a_n)$ be a non-decreasing bounded sequence in $\F$ with the limit $a\in A_b(X)$. Then we find functions $b_n\in H^\mu$ such that $b_n=a_nm$ on $\ext X$. Then $b_{n+1}-b_n=(a_{n+1}-a_n)m\ge 0$ on $\ext X$. Since $H^\mu$ is determined by extreme points, we deduce that $(b_n)$ is non-decreasing and bounded. It follows that there exists its limit $b=\lim b_n$, which belongs to $H^\mu$. Since $b=am$ on $\ext X$, the function $a$ belongs to $\F$. %Since clearly $-a\in \F$ whenever $a\in \F$, 
We conclude that $\F=H^\mu$ and therefore $m\in M(H^\mu)$. 

To finish the proof that $(M(H))^\mu\subset M(H^\mu)$, it remains to show that  $M(H^\mu)$ is closed with respect to pointwise limits of bounded monotone sequences.
 To this end fix a sequence $(m_n)$ in $M(H^\mu)$ such that $m_n\nearrow m\in A_b(X)$. Clearly $m\in H^\mu$. 

    We continue by proving that $m\in M(H^\mu)$. To this end we fix $a\in H^\mu$ and we are looking for $b\in H^\mu$ such that $b=ma$ on $\ext X$.
   Since all the functions in question are bounded and $1\in H\subset H^\mu$, we may assume without loss of generality that $a\ge 0$ and $m_1\ge0$.
   By the assumptions we find $b_n\in H^\mu$ such that $b_n=m_na$ on $\ext X$. We deduce that 
   $$b_n=m_n a\nearrow m a\mbox{ on }\ext X.$$
   Since $H^\mu$ is determined by extreme points, 
the sequence $(b_n)$ is nondecreasing.
Using one more determinacy by extreme points we deduce that each $b_n$ is bounded from above by $\norm{m}\cdot\norm{a}$, so $b_n\nearrow b$ for some $b\in H^\mu$.  Since $b=ma$ on $\ext X$, we conclude that $m\in M(H^\mu)$.

This settles the case of $M(H)$, including the `in particular' part. The case of $M^s(H)$ is similar:
The first step is again to prove that $M^s(H)\subset M^s(H^\mu)$. We assume that $m\in M^s(H)$ is positive, and we consider the family
\begin{equation*}
\begin{aligned}
\F^s=\{a\in H^\mu\setsep \exists b\in H^\mu\colon b=am \ &\mu\text{-almost everywhere}  \\ &\text{ for each maximal } \mu \in M_1(X)\}.
\end{aligned}
\end{equation*}
Again, $\F^s$ is a linear subspace of $H^\mu$ containing $H$. As above, it is enough to show that the family $\F^s$ is stable with respect to taking pointwise limits of monotone bounded sequences. Thus we pick a non-decreasing bounded sequence $(a_n)$ in $\F$ with the limit $a\in A_b(X)$. 
We find $b_n$ and $b$ as above. Fix any maximal measure $\mu \in M_1(X)$. Since $m\in M^s(H)$ and $H$ is determined by extreme points, we deduce that
    $b_n=m \cdot a_n$ $\mu$-almost everywhere. Hence (using the $\sigma$-additivity of $\mu$) we find a $\mu$-null set $N\subset X$ such that
    $$\forall x\in X\setminus N\;\forall n\in\en\colon b_n(x)=m(x)\cdot a_n(x),$$
    hence $b(x)=m(x)\cdot a(x)$ for $x\in X\setminus N$. Therefore $b=ma$ $\mu$-almost everywhere.
    Hence $a\in\F^s$.

\iffalse
We find functions $b_n\in H^\mu$ such that the equality $b_n=a_nm$ holds $\mu$-almost everywhere for each maximal measure $\mu \in M_1(X)$, in particular it holds on $\ext X$. Thus as above, we can find the limit $b=\lim b_n$, which is in $H^\mu$. It remains to show that for each maximal measure $\mu \in M_1(X)$ and $\mu$-almost every point $x \in X$, $b(x)=a(x)m(x)$. For a fixed such measure $\mu$ and each $n \in \en$ we can find a set $B_n$ such that $\mu(B_n)=1$ and $b_n=a_nm$ on $B_n$. It then follows that the equality $b=am$ holds on $\bigcap_{n \in \en} B_n$, which finishes the argument.
\fi

The second step is again to show that $M^s(H^\mu)$ is closed with respect to taking pointwise limits of bounded monotone sequences. We proceed as above, we just assume that the sequence $(m_n)$ belongs to $M^s(H^\mu)$ instead of $M(H^\mu)$. We fix a maximal measure $\mu \in M_1(X)$. Then  $b_n=m_na$ $\mu$-almost everywhere for each $n\in\en$, 
hence $b=ma$ $\mu$-almost everywhere. The proof is finished.

$(ii)$: Assume that $H$ consists of strongly affine functions and $H^\sigma$ is determined by extreme points. Then $H^{\sigma}$ consists of strongly affine functions as well. As above, the first step is to prove that $M(H)\subset M(H^\sigma)$. 
Let $m\in M(H)$ be given and let
\begin{equation*}
%\nonumber
\begin{aligned}
\F=\{a\in H^\sigma\setsep \exists b\in H^\sigma\colon b=am \ &\mu\text{-almost everywhere} \\& \text{ for each maximal } \mu \in M_1(X) \}.
\end{aligned}
\end{equation*}
Then $\F$ is a linear subspace of $H^\sigma$ containing $H$. Further, $\F$ is closed with respect to taking pointwise limits of bounded sequences. Indeed, let $(a_n)$ be such a sequence  in $\F$ with limit $a\in A_b(X)$. Then clearly $a\in H^\sigma$. For each $n\in\en$ we find $b_n\in H^\sigma$ such that $b_n=a_nm$ $\mu$-almost everywhere for each maximal measure $\mu \in M_1(X)$. 

Further, we fix a point $x\in X$ and a maximal measure $\nu \in M_x(X)$. Since the functions $b_n$ are strongly affine, $b_n(x)=\nu(b_n)$. Since $b_n=a_nm$ $\nu$-almost everywhere, we deduce $b_n(x)=\nu(a_nm)$. By the Lebesgue dominated convergence theorem we deduce that $b_n(x)\to \nu(am)$.
So, the sequence $(b_n)$ pointwise converges to some $b\in A_b(X)$. Then $b\in H^\sigma$ is strongly affine. Further, let $\mu\in M_1(X)$ be any maximal measure. Since $b_n\to b$, $a_nm\to am$ and $b_n=a_nm$ $\mu$-almost everywhere for each $n\in\en$, we deduce  that $b=am$ $\mu$-almost everywhere. 
Hence $a\in \F$. It follows that $\F=H^\sigma$ and thus $m\in M^s(H^\sigma)$. 

The second step is to show that $M(H^\sigma)$ is closed with respect to pointwise limits of bounded sequences.
 To this end fix a bounded sequence $(m_n)$ in $M(H^\sigma)$ such that $m_n\to m\in A_b(X)$. Clearly $m\in H^\sigma$. 
    We continue by proving $m\in M(H^\sigma)$. To this end we fix $a\in H^\sigma$. For each $n\in\en$ there is $b_n\in H^\sigma$ such that $b_n=m_na$ $\mu$-almost everywhere for any maximal measure $\mu\in M_1(X)$. Fix $x\in X$ and take a maximal measure $\nu\in M_x(X)$. Then
    $$b_n(x)=\nu(b_n)=\nu(m_na)\to\nu(ma).$$
    Thus the sequence $(b_n)$ pointwise converges to some $b\in A_b(X)$. Then $b\in H^\sigma$. Further, similarly as above we show that $b=ma$ $\mu$-almost everywhere for any maximal $\mu\in M_1(X)$. Hence $m\in M(H^\sigma)$.

The `in particular' part follows from $(i)$ and the above combined with the fact that $(H^\mu)^\sigma=H^\sigma$. The proof is finished.
\end{proof}

\begin{remarks}\label{rem:Ms(H1)}
(1) In assertion $(i)$ of Proposition~\ref{p:multi-pro-mu} it is not sufficient to assume that just $H$ is determined by extreme points. Indeed, Example~\ref{ex:deter-ext-body} below shows that there is an intermediate function space $H$ determined by extreme points such that the space $H^\mu$ is not.

(2) Assertion $(ii)$ of Proposition~\ref{p:multi-pro-mu} may be refined (using essentially the same proof): Assume that $H\subset A_{sa}(X)$. Let $H_1$ be the space of all limits of bounded pointwise converging sequences from $H$. Then $H_1$ is a linear subspace of $A_{sa}(X)$ (not necessarily closed) and $(M^s(H))_1\subset M^s(H_1)$. 
\end{remarks}

We finish this section by asking the following open question on validity of assertion $(ii)$ of the previous proposition  for $M(H)$ instead of $M^s(H)$.

\begin{ques}
    Assume that $H$ is an intermediate function space such that $H^\sigma$ is determined by extreme points. Is $M(H)^\sigma\subset M(H^\sigma)$?
\end{ques}

\section{On multipliers for $A_c(X)$}\label{sec:mult-acx}

In this section we will deal only with continuous affine functions. Our aim to describe $m(X)$ and to characterize multipliers on $A_c(X)$. A description of $m(X)$ is given in Proposition~\ref{P:m(X)} below, a characterization of multipliers is provided in Proposition~\ref{P:mult-charact} in the general case and in Proposition~\ref{P:simplex mult} in case $X$ is a simplex. These characterizations are related to the characterization from \cite[Theorem II.7.10]{alfsen}  using the continuity with respect to the facial topology, but our version is more understandable and easier to apply in concrete cases (as witnessed by the examples at the end of   
this section). 

We start by the following lemma which is a key step to a description of $m(X)$.

\begin{lemma}\label{L:platnost soucinu}
Let $T\in\frd(A_c(X))$ be an operator and $u=T(1)$.
Let $x\in X$. Then the  following assertions are equivalent.
\begin{enumerate}[$(1)$]
    \item $T(a)(x)=u(x)\cdot a(x)$ for each $a\in A_c(X)$.
    \item $T(u)(x)=u(x)^2$.
    \item $u$ is constant on $\spt\mu$ for any probability measure $\mu$ supported by $\overline{\ext X}$ such that $r(\mu)=x$.
    \item $u$ is constant on $\spt\mu$ for some probability measure $\mu$ supported by $\overline{\ext X}$ such that $r(\mu)=x$.
\end{enumerate}
\end{lemma}

\begin{proof}
$(1)\implies(2)$: This is trivial.

$(2)\implies (3)$: Assume $T(u)(x)=u(x)^2$. Let $\mu$ be a probability measure supported by $\overline{\ext X}$ with $r(\mu)=x$. Then
$$\begin{aligned}
\int (u-u(x))^2\di\mu&= \int u^2\di\mu - 2 \int u(x)u\di\mu+\int u(x)^2\di\mu =\int u^2\di\mu-u(x)^2\\
&=\int T(u)\di\mu -u(x)^2=T(u)(x)-u(x)^2=0.
\end{aligned}$$
In the second equality we used the assumption $r(\mu)=x$ and $u\in A_c(X)$. In the third equality we used the fact that $T(u)=u^2$ on $\ext X$  (by Corollary\ref{cor:rovnost na ext}(a)) and so by continuity also on $\overline{\ext X}$. The fourth equality uses again that $r(\mu)=x$ and the last one follows from (2).

It follows that $u=u(x)$ $\mu$-a.e. Hence, by continuity, $u=u(x)$ on $\spt\mu$.

$(3)\implies(4)$: This is trivial.

$(4)\implies(1)$: Let $\mu$ be such a measure. Then for each $a\in A_c(X)$ we have
$$T(a)(x)=\int T(a)\di\mu=\int ua\di\mu=\int u(x)a\di\mu=u(x)\int a\di\mu=u(x)a(x).$$
In the second equality we used the fact that  $T(a)=T(1)\cdot a=u \cdot a$ on $\ext X$ (Corollary~\ref{cor:rovnost na ext}(a)) and so by continuity also on $\overline{\ext X}$. The third one follows from the assumption $u=u(x)$ on $\spt\mu$.
\end{proof}

As a corollary we get the following characterization of $m(X)$.

\begin{prop}\label{P:m(X)}
Let $X$ be a compact convex set and $x\in X$. Then the  following assertions are equivalent.
\begin{enumerate}[$(1)$]
    \item $x\in m(X)$.
    \item There is a probability measure $\mu$ supported by $\overline{\ext X}$ with barycenter $x$ such that any $u\in M(A_c(X))$ is constant on $\spt\mu$.
    \item If $\mu$ is any probability measure  supported by $\overline{\ext X}$ with barycenter $x$,  then any $u\in M(A_c(X))$ is constant on $\spt\mu$.
\end{enumerate}
\end{prop}

This proposition follows immediately from Lemma~\ref{L:platnost soucinu}.
Although this proposition provides a complete characterization of $m(X)$, its use is limited by our knowledge of $M(A_c(X))$. Multipliers on $A_c(X)$ can be characterized by continuity in the facial topology (see \cite[Theorem II.7.10]{alfsen}) but this is hard to use in concrete cases. In the rest of this section we provide some more accessible necessary conditions and  characterizations.

\begin{lemma}\label{L:mult-nutne}
Let $X$ be a compact convex set and let $u\in M(A_c(X))$. Then the following conditions are satisfied.
\begin{enumerate}[$(i)$]
    \item If $\mu\in M_1(\overline{\ext X})$ has the barycenter in $\overline{\ext X}$, then  $u$ is constant on $\spt\mu$.
  \item  If $\mu,\nu\in M_1(\overline{\ext X})$ have the same barycenter and $B\subset \er$ is a Borel set, then $\mu(u^{-1}(B))=\nu(u^{-1}(B))$ and, moreover, if this number is strictly positive, then the probabilities
$$\frac{\mu_{|u^{-1}(B)}}{\mu(u^{-1}(B))}\mbox{ and }\frac{\nu_{|u^{-1}(B)}}{\nu(u^{-1}(B))}$$
have the same barycenter.
\item If $a,b,c,d\in\overline{\ext X}$ are four distinct points and the segments $[a,b]$ and $[c,d]$ intersect in a point which is internal in both segments, then $u(a)=u(b)=u(c)=u(d)$.
\item  If $a,b,c,d,e\in\overline{\ext X}$ are five distinct points and
$$t_1a+t_2b+t_3c=sd+(1-s)e$$
for some $t_1,t_2,t_3,s\in(0,1)$ such that $t_1+t_2+t_3=1$, then $u(a)=u(b)=u(c)=u(d)=u(e)$.
\end{enumerate}
\end{lemma}

\begin{proof}
$(i)$: Any $x\in\overline{\ext X}$ belongs to $m(X)$, hence the assertion follows from Proposition~\ref{P:m(X)}, implication $(1)\implies(3)$.

$(ii)$: Assume $\mu,\nu\in M_1(\overline{\ext X})$ have the same barycenter $x$.

Set
$$E=\overline{\span}\{u^n\setsep n\ge 0\}\subset C(X).$$
By the Stone-Weierstrass theorem we get
$$E=\{f\circ u\setsep f\in C(u(X))\}.$$
Moreover, since $u$ is a multiplier, we claim that  $(f\cdot g)|_{\overline{\ext X}}$ may be extended to an affine continuous function on $X$ for any $f\in A_c(X)$ and $g\in E$.
Indeed, the mapping
$$R:f\mapsto f|_{\overline{\ext X}}$$
is an isometric linear injection of $A_c(X)$ into $C(\overline{\ext X})$. So, its range is a closed linear subspace of $C(\overline{\ext X})$. Hence, given any $f\in A_c(X)$, the set
$$\{g\in C(X)\setsep (f\cdot g)|_{\overline{\ext X}}\in R(A_c(X))\}$$
is a closed linear subspace of $C(X)$, containing constant functions and stable under the  multiplication by $u$. It follows that this subspace contains $E$.

We deduce
$$\int a\cdot g\di\mu=\int a\cdot g\di\nu,\quad a\in A_c(X), g\in E,$$
hence 
\begin{equation}\label{eq:aaa}
    \int a\cdot f\circ u\di\mu=\int a\cdot f\circ u\di\nu,\quad a\in A_c(X), f\in C(u(X)).\end{equation}
Applying to $a=1$ we get
$$\int  f\circ u\di\mu=\int  f\circ u\di\nu,\quad  f\in C(u(X)),$$
so 
$$\int f\di u(\mu)=\int f\di u(\nu),\quad f\in C(u(X)),$$
i.e., $u(\mu)=u(\nu)$. I.e., $\mu(u^{-1}(B))=\nu(u^{-1}(B))$ for any Borel set $B\subset \er$.
This proves the first part of the assertion.

To prove the second part, observe that \eqref{eq:aaa} implies (using the Lebesgue dominated convergence theorem) that
$$\int a\cdot f\circ u\di\mu=\int a\cdot f\circ u\di\nu,\quad a\in A_c(X), f\in \Ba^b(u(X)).$$
In particular, applying to $f=1_B$ for a Borel set $B\subset \er$ (note that Borel and Baire subsets of $\er$ coincide) we get
$$\int a\cdot 1_B\circ u\di\mu=\int a\cdot 1_B\circ u\di\nu,\quad a\in A_c(X),$$
i.e.,
$$\int_{u^{-1}(B)} a\di\mu=\int_{u^{-1}(B)} a\di\nu,\quad a\in A_c(X).$$
By the first part we know that $\mu(u^{-1}(B))=\nu(u^{-1}(B))$. If this number is zero, the above equality holds trivially; if it is strictly positive, the above equality means exactly that
the probabilities
$$\frac{\mu_{|u^{-1}(B)}}{\mu(u^{-1}(B))}\mbox{ and }\frac{\nu_{|u^{-1}(B)}}{\nu(u^{-1}(B))}$$
have the same barycenter. This completes the proof of assertion $(ii)$.

$(iii)$: By the assumption there are $s,t\in(0,1)$ such that
$$sa+(1-s)b=tc+(1-t)d.$$
Define measures
$$\mu=s\varepsilon_a+(1-s)\varepsilon_b,\quad \nu=t\varepsilon_c+(1-t)\varepsilon_d.$$
These two measures satisfy the assumptions of $(ii)$, so $u(\mu)=u(\nu)$.

If $u(a)=u(b)=\lambda$, then $u(\mu)=\varepsilon_\lambda$. Since then $u(\nu)=\varepsilon_\lambda$ as well, necessarily $u(c)=u(d)=\lambda$.
If $u(c)=u(d)$, we proceed similarly.

So assume $u(a)\ne u(b)$ and $u(c)\ne u(d)$. Up to relabeling the points we may assume $u(a)<u(b)$ and $u(c)<u(d)$. Fix $\omega\in(u(a),u(b))$ and let $B=(-\infty,\omega)$. Then $\mu|_{u^{-1}(B)}=s\varepsilon_a$ and $\mu(u^{-1}(B))=s$. Since $\nu(u^{-1}(B))=\mu(u^{-1}(B))=s$ and simultaneously $\nu(u^{-1}(B))\in\{0,t,1\}$, we deduce $t=s$. But then the barycenters of the normalized measures are $a$ and $c$. By $(ii)$, these barycenters must be equal, which is a contradiction.
This completes the proof.

$(iv)$: Define measures
$$\mu=t_1\varepsilon_a+t_2\varepsilon_b+t_3\varepsilon_c,\quad \nu=s\varepsilon_d+(1-s)\varepsilon_e.$$
These two measures satisfy the assumptions of $(ii)$, so $u(\mu)=u(\nu)$.

If $u(d)=u(e)=\lambda$, then $u(\nu)=\varepsilon_\lambda$. Since then $u(\mu)=\varepsilon_\lambda$ as well, necessarily $u(a)=u(b)=u(c)=\lambda$.

Assume $u(d)\ne u(e)$. Then 
$$
\begin{aligned}
\mu(u^{-1}(\er\setminus \{u(d),u(e)\})&=\nu(u^{-1}(\er\setminus \{u(d),u(e)\})=0,\\
\mu(u^{-1}(\{u(d)\})&=\nu(u^{-1}(\{u(d)\})=s>0,\\
\mu(u^{-1}(\{u(e)\})&=\nu(u^{-1}(\{u(e)\})=1-s>0.
\end{aligned}
$$
It follows that $u(\{a,b,c\})=\{u(d),u(e)\}$.
So, at some two of points $a,b,c$ function $u$ attains the same value and at the third point it attains a different value. Up to relabelling we may assume
$$u(a)=u(b)=u(d)\ne u(c)=u(e).$$ 
Then $t_1+t_2=s$, $t_3=1-s$. Since the normalized measures $\frac{\mu|_{u^{-1}(\{u(e)\}}}{\mu(u^{-1}(\{u(e)\})}$ and 
$\frac{\nu|_{u^{-1}(\{u(e)\}}}{\nu(u^{-1}(\{u(e)\})}$ have the same barycenter (by $(ii)$), we deduce that $c=e$. This contradiction completes the proof.
\end{proof}

\begin{remark}\label{rem:nutne}
(1) Condition $(i)$ of the previous lemma follows from condition $(ii)$ applied to $\mu$ and the Dirac measure carried at the barycenter of $\mu$ in place of $\nu$. However, we formulate this condition separately because it is easier and, moreover, in case $X$ is a simplex it serves as a characterization (see Proposition \ref{P:simplex mult} below).

(2) Condition $(ii)$ of the previous lemma can be reformulated using measures annihilating $A_c(X)$ as follows:
\begin{enumerate}[$(ii')$]
    \item If $\mu\in M(\overline{\ext X})\cap A_c(X)^\perp$ and $B\subset \er$ is a Borel set, then $\mu(u^{-1}(B))=0$ and $\mu|_{u^{-1}(B)}\in A_c(X)^\perp$.
\end{enumerate}

(3) In conditions $(iii)$ and $(iv)$ of the previous lemma it is essential that we consider either two pairs of points or a pair and a triple of points. Example \ref{E:more points}(4) below illustrates that a similar condition does not hold for larger families of points.
\end{remark}

Next we are going to show that condition $(ii)$ from Lemma~\ref{L:mult-nutne} characterizes multipliers. To this end we will use the following lemma.  

\begin{lemma}\label{L:extending}
Assume that $X$ is a compact convex set. Let $f\in C(\overline{\ext X})$. Then $f$ may be extended to an affine continuous function on $X$ if and only if
\[
\forall\mu,\nu\in M_1(\overline{\ext X})\colon r(\mu)=r(\nu) \Rightarrow \int_{\overline{\ext X}} f\di\mu=\int_{\overline{\ext X}} f\di\nu.
\]
\end{lemma}

\begin{proof}
The `only if part' is obvious and the `if part' follows from \cite[Theorem II.4.5]{alfsen}.
\iffalse
Let $F\colon M_1(\ov{\ext X})\to \er$ be defined as $F(\mu)=\int f\di\mu$, $\mu\in M_1(\ov{\ext X})$.
Then $F$ is an affine continuous function on the compact convex set $M_1(\ov{\ext X})$. Let the extension $h$ of $f$ be given by
\[
h(x)=\int f\di\mu,\quad \mu\in M_1(\ov{\ext X}), r(\mu)=x.
\]
Then $h$ is well defined. If $r\colon M_1(\ov{\ext X})\to X$ denotes the barycentric mapping restricted on $M_1(\ov{\ext X})$, then $r$ is an affine continuous surjection (see \cite[Theorem 3.81]{lmns}).
Further, $F=h\circ r$. Now it is easy to observe from the properties of $r$ that $h$, as well as $F$, is an affine continuous function on $X$. Since it extends $f$, the proof is finished.\fi
\end{proof}

\begin{prop}\label{P:mult-charact}
Let $u\in A_c(X)$. Then the  following assertions are equivalent.
\begin{enumerate}[$(1)$]
    \item $u$ is a multiplier for $A_c(X)$.
    \item If $\mu,\nu\in M_1(\overline{\ext X})$ have the same barycenter and $B\subset \er$ is a Borel set, then $\mu(u^{-1}(B))=\nu(u^{-1}(B))$ and, moreover, if this number is strictly positive, then the probabilities
$$\frac{\mu_{|u^{-1}(B)}}{\mu(u^{-1}(B))}\mbox{ and }\frac{\nu_{|u^{-1}(B)}}{\nu(u^{-1}(B))}$$
have the same barycenter.
%\item Assertion (2) is valid if $B$ is an interval.
\end{enumerate}
\end{prop}

\begin{proof}
$(1)\implies(2)$ is proved in Lemma~\ref{L:mult-nutne}$(ii)$.

$(2)\implies(1)$: Let $h\in A_c(X)$. It is enough to prove that $h\cdot u|_{\overline{\ext X}}$ may be extended to an affine continuous function. We will check the condition from Lemma~\ref{L:extending}. Take two probability measures $\mu,\nu$ on $\overline{\ext X}$ with the same barycenter. By $(2)$ we get that
$$\int h\cdot g\di\mu=\int h\cdot g\di\nu,$$
whenever $g$ is the characteristic function of $u^{-1}(B)$, where $B\subset \er$ is a Borel set. It follows that this equality holds for any bounded function $g$ measurable with respect to the $\sigma$-algebra generated by
$$u^{-1}(B), B\subset\er \mbox{ Borel}.$$
In particular, the choice $g=u$ is possible, which completes the proof.
\end{proof}

Now we are going to show that condition $(i)$ from Lemma~\ref{L:mult-nutne} characterizes multipliers on simplices. To this end we need the following lemma:

\begin{lemma}\label{L:simplex extension}
Assume that $X$ is a simplex. Let $f\in C(\overline{\ext X})$. Then $f$ may be extended to an affine continuous function on $X$ if and only if
$$\forall\mu\in M_1(\overline{\ext X})\colon r(\mu)\in\overline{\ext X}\Rightarrow f(r(\mu))=\int f\di\mu.$$
\end{lemma}

\begin{proof} The `only if part' is obvious, let us prove the `if part'. For each $x\in X$, let $\delta_x$ denote the unique maximal measure representing $x$. Let $\tilde{f}\in C(X)$ be an extension of $f$ provided by the Tietze theorem.
Then the function $h(x)=\int_{X} \tilde{f}\di\delta_x$, $x\in X$ is a strongly affine function on $X$ by  \cite[Theorem 6.8(c)]{lmns}.
Since maximal measures are supported by $\ov{\ext X}$ and $\tilde{f}|_{\ov{\ext X}}=f$, we deduce that $h=f$ on $\ov{\ext X}$ by the assumption.
Since $h$ is strongly affine and continuous on $\ov{\ext X}$, $h$ is continuous on $X$ by \cite[Theorem 3.5]{lusp}.
\end{proof}

\begin{prop}\label{P:simplex mult} Assume that $X$ is a simplex. Let $u\in A_c(X)$. Then $u\in M(A_c(X))$ if and only if
$$\forall\mu\in M_1(\overline{\ext X})\colon r(\mu)\in\overline{\ext X}\Rightarrow u\mbox{ is constant on }\spt\mu.$$
\end{prop}

\begin{proof}
Implication $\implies$ is proved in Lemma~\ref{L:mult-nutne}$(i)$. Let us prove the converse implication. Assume that $u$ satisfies the given property and fix $f\in A_c(X)$. We will show that there is $g\in A_c(X)$ coinciding with $u\cdot f$ on $\overline{\ext X}$.

We will check the condition provided by Lemma~\ref{L:simplex extension}.
To this end take any $\mu\in M_1(\overline{\ext X})$ such that $r(\mu)\in\overline{\ext X}$. Then
$$\int uf\di\mu=\int u(r(\mu)) f\di\mu= u(r(\mu)) \int f\di\mu= u(r(\mu)) f(r(\mu)),$$
where in the first equality we used that $u$ is constant on $\spt\mu$. 
This completes the argument.
\end{proof}

We now collect some examples illustrating the use of the above characterizations on concrete compact convex sets. Most of them are defined as the state spaces of certain function space, hence we use the notation from  Section~\ref{ssc:ch-fs}.

\begin{example2}\label{E:more points}
(1) Let $X$ be a Bauer simplex, i.e., $X=M_1(K)$ for some compact topological space $K$.
Then $M(A_c(X))=A_c(X)\, (=C(K))$.  This follows easily from the definitions. Alternatively, it follows from Proposition~\ref{P:simplex mult} as $\ext X$ is closed  and the only probabilities carried by $\overline{\ext X}=\ext X$ having barycenter in $\overline{\ext X}=\ext X$ are the Dirac measures.
Moreover, in this case $m(X)=\ext X\ (=\{\varepsilon_t\setsep t\in K\})$. This follows from Proposition~\ref{P:m(X)} together with the Urysohn lemma.

(2) Let $K=[0,1]$,
$$E=\{f\in C([0,1])\setsep f(\tfrac12)=\tfrac12(f(0)+f(1))\}.$$
Then $E$ is a function space, $\Ch_E K=[0,\frac12)\cup(\frac12,1]$. Moreover, from \cite[Theorem 2]{stacey} it follows that $X=S(E)$ is a simplex. Indeed, it is not difficult to verify that the mapping
\[
t\mapsto\begin{cases}\ep_t,& t\in K\setminus\{\frac12\},\\
\frac12(\ep_0+\ep_1),&t=\frac12,
\end{cases}
\]
satisfies the assumptions of the aforementioned theorem.

Using Proposition~\ref{P:simplex mult} we see that  
\begin{equation}
\label{eq:priklad-multi-konst}
M(A_c(X))=\{\Phi(f)\setsep f\in E, f(0)=f(1)=f(\tfrac12)\}.
\end{equation}
Using Proposition~\ref{P:m(X)} we deduce 
$$m(X)=\{\phi(t)\setsep t\in [0,1]\} \cup [\phi(0),\phi(1)],$$
where $[\phi(0),\phi(1)]$ denotes the respective segment in $X$.

Indeed, the inclusion `$\supset$' follows from \eqref{eq:priklad-multi-konst} and Proposition~\ref{P:m(X)}$(2)$.

To see the inclusion `$\subset$', let $s\in m(X)$ be given. Then there exists a probability measure $\mu\in M_s(X)\cap M_1(\ov{\ext X})$. Then $\mu$ is supported by $\{\phi(t)\setsep t\in [0,1]\}$. By Proposition~\ref{P:m(X)}$(3)$ we obtain that $\mu$ is either the Dirac measure at some $\phi(t)$, $t\in [0,1]\setminus\{\frac12\}$ or $\mu$ is supported by the set $\{\phi(0),\phi(\frac12),\phi(1)\}$.
Then it follows that $s\in [\phi(0),\phi(1)]$.

(3) Let $K=[0,1]$ and 
$$E=\left\{f\in C([0,1])\setsep f(0)=\int_0^1 f\right\}.$$
Then it is easy to check that $E$ a function space and $\Ch_E K=(0,1]$. Using \cite[Theorem 2]{stacey} we obtain that $X=S(E)$ is a simplex. By Proposition~\ref{P:simplex mult} we deduce that $M(A_c(X))$ consists only of constant functions. In particular, $m(X)=X$.

(4) Let $K=[0,5]$,
$$E=\{f\in C([0,5])\setsep f(1)=\tfrac12(f(0)+f(2)),f(4)=\tfrac12(f(3)+f(5))\}.$$
Then $E$ is a function space, $\Ch_E K=[0,5]\setminus\{1,4\}$ and $X=S(E)$ is a simplex (again we use \cite[Theorem 2]{stacey}). Using Proposition~\ref{P:simplex mult} we see that  
$$M(A_c(X))=\{\Phi(f)\setsep f\in E, f(0)=f(1)=f(2)\ \&\  f(3)=f(4)=f(5)\}.$$
Using Proposition~\ref{P:m(X)} we deduce 
$$m(X)=\{\phi(t)\setsep t\in [0,5]\} \cup [\phi(0),\phi(2)]\cup[\phi(3),\phi(5)].$$

This example is just a more complicated variant of example (2) above. But it may be used to illustrate the optimality of assertions $(iii)$ and $(iv)$ of Lemma~\ref{L:mult-nutne}.
Indeed, set
$$\mu_1=\tfrac14(\varepsilon_{\phi(0)}+\varepsilon_{\phi(2)})+\tfrac12\varepsilon_{\phi(4)}, \quad\nu_1=\tfrac14(\varepsilon_{\phi(3)}+\varepsilon_{\phi(5)})+\tfrac12\varepsilon_{\phi(1)}.$$
Then $\mu_1$ and $\nu_1$ are mutually orthogonal probabilities supported by $\overline{\ext X}$ with the same barycenter. Their supports have exactly three points, but there are multipliers which are not constant on the union of these supports.

Further, set
$$
\mu_2=\tfrac14(\varepsilon_{\phi(0)}+\varepsilon_{\phi(2)}+\varepsilon_{\phi(3)}+\varepsilon_{\phi(5)}), \quad 
\nu_2=\tfrac12(\varepsilon_{\phi(1)}+\varepsilon_{\phi(4)}).$$
Again, $\mu_2$ and $\nu_2$ are mutually orthogonal probabilities supported by $\overline{\ext X}$ with the same barycenter. The support of $\mu_2$ has four points and the support of $\nu_2$ has two points, but there are multipliers which are not constant on the union of these supports.

(5) Let $K=[0,1]$ and
$$E=\{f\in C([0,1])\setsep f(0)+f(\tfrac13)=f(\tfrac23)+f(1)\}.$$
Then $E$ is a  function space with $\Ch_E K=[0,1]$ (this is easy to see from the fact that for each $t\in K$ there exists a nonnegative function $f\in E$ that attains $0$ precisely at $t$). The measures $\frac12(\varepsilon_{\phi(0)}+\varepsilon_{\phi(1/3)})$ and $\frac12(\varepsilon_{\phi(2/3)}+\varepsilon_{\phi(1)})$ are then two maximal measures on $X=S(E)$ with the same barycenter in $S(E)$, and hence $X$ is not a simplex.
We claim that
$$M(A_c(X))=\{\Phi(f)\setsep f\in E, f(0)=f(\tfrac13)=f(\tfrac23)=f(1)\}.$$
Inclusion `$\subset$' follows from Lemma~\ref{L:mult-nutne}$(iii)$. To prove the converse we will use Proposition~\ref{P:mult-charact} using the reformulation from Remark~\ref{rem:nutne}(2).

Take any $\mu\in M(K)$ such that $\mu\in E^\perp$. It follows from the definition of $E$ and the bipolar theorem that $\mu$ is a multiple of
$$\varepsilon_0+\varepsilon_{1/3}-\varepsilon_{2/3}-\varepsilon_1.$$
Assume that $f\in E$ is such that $f(0)=f(\frac13)=f(\frac23)=f(1)$. Then clearly $\mu(f)=0$. Since $f$ is constant on the support of $\mu$, given any Borel set $B\subset\er$, the set $f^{-1}(B)$ either contains the support of $\mu$ or is disjoint with the support of $\mu$. In both cases $\mu|_{f^{-1}(B)}\in E^\perp$.

Using Proposition~\ref{P:m(X)} we get
$$m(X)=\{\phi(t)\setsep t\in[0,1]\}\cup\co\{\phi(0),\phi(1/3),\phi(2/3),\phi(1)\}.$$

(6) Let $K=[0,1]$ and $\mu,\nu\in M_1(K)$ be two mutually orthogonal probabilities such that $\mu,\nu\notin \{\ep_t\setsep t\in K\}$. Let
$$E=\left\{f\in C([0,1])\setsep \int f\di\mu=\int f\di\nu\right\}.$$
We claim that $E$ is a function space and $\Ch_E K=[0,1]$. 

Obviously, $E$ contains constant functions. If $t,s\in K$ different satisfy $f(t)=f(s)$ for every $f\in E$, then by the bipolar theorem there exists $c\in\er$ such that $\ep_t-\ep_s=c(\mu-\nu)$. Up to relabelling $s$ and $t$ we may and shall assume that $c\ge0$. Then by the orthogonality of $\mu$ and $\nu$ we have
\[
2=\norm{\ep_t-\ep_s}=c\norm{\mu-\nu}=2c,
\]
i.e., $c=1$ and hence $\ep_t+\nu=\mu+\ep_s$. Then 
\[
1+\nu(\{t\})=\mu(\{t\}).
\]
Since $\mu(\{t\})\le 1$, we have $\nu(\{t\})=0$ and $\mu=\ep_t$. But this contradicts the choice of $\mu$. Hence $E$ separates the points of $K$, so it is a function space.

Further we check that $M_t(E)=\{\ep_t\}$ for every $t\in K$. Indeed, let $t\in K$ be arbitrary and $\lambda\in M_t(E)\setminus\{\ep_t\}$. Then we may assume without loss of generality that $\lambda(\{t\})=0$.  Then $\ep_t-\lambda$ is again by the bipolar theorem a multiple of $\mu-\nu$, hence $\ep_t-\lambda=c(\mu-\nu)$ for some $c\in\er$. Up to relabelling $\mu$ and $\nu$ we may and shall assume that $c\ge0$. As above we have
\[
2=\norm{\ep_t-\lambda}=c\norm{\mu-\nu}=2c,
\]
which gives $c=1$. Hence $\ep_t+\nu=\mu+\lambda$, which yields
\[
1+\nu(\{t\})=\mu(\{t\})\le 1.
\]
We obtain $\nu(\{t\})=0$ and $\mu(\{t\})=1$. Hence $\mu=\ep_t$ and $\nu=\lambda$. But this contradicts our assumptions on measures $\mu$ and $\nu$. This completes the proof that  $\Ch_E K=[0,1]$.

Next we verify that $X=S(E)$ is not a simplex.
Indeed, the measures $\phi(\mu)$ and $\phi(\nu)$ are different measures supported by $\ext X=\phi(K)$ (and hence maximal) with the same barycenter. Thus $X$ is not a simplex.

We claim that
$$M(A_c(X))=\{\Phi(f)\setsep f\in E, f\mbox{ is constant on }\spt\mu\cup\spt\nu\}.$$
Indeed, by the bipolar theorem $E^\perp$ is formed by the multiples of $\mu-\nu$. Inclusion `$\supset$' follows from Proposition~\ref{P:mult-charact} similarly as in example (5) above.
To prove the converse assume that $f\in E$ but $f$ is not constant on $\spt\mu\cup\spt\nu$. It follows that there is 
an open set $U\subset\er$ such that $0<\mu(f^{-1}(U))+\nu(f^{-1}(U))<2$. Then $\lambda=\mu-\nu\in E^\perp$, but  $\lambda|_{f^{-1}(U)}$ is not a multiple of $\mu-\nu$ and hence it does not belong to $E^\perp$. It thus follows from Lemma~\ref{L:mult-nutne} that $\Phi(f)\notin M(A_c(X))$.

Using Proposition~\ref{P:m(X)} we now deduce that
$$m(X)=\{\phi(t)\setsep t\in [0,1]\}\cup \overline{\co}\{\phi(t)\setsep t\in\spt\mu\cup\spt\nu\}.$$

(7) Let $K=[0,7]$ and 
$$E=\{f\in C([0,7])\setsep f(0)+f(1)=f(2)+f(3)\ \&\ f(4)+f(5)=f(6)+f(7)\}.$$
Then $E$  is a function space, $\Ch_E K=[0,7]$ and $X=S(E)$ is not a simplex (the reasoning is similar as in example (6)). 
By the bipolar theorem $E^\perp$ is formed by linear combinations of $\varepsilon_0+\varepsilon_1-\varepsilon_2-\varepsilon_3$ and $\varepsilon_4+\varepsilon_5-\varepsilon_6-\varepsilon_7$. Similarly as in (5) and (6) we get
$$\begin{aligned}
    M(A_c(X))=\{\Phi(f), f\in E, f(0)=f(1)&=f(2)=f(3)\\
    &\&\ f(4)=f(5)=f(6)=f(7)\}.
    \end{aligned}$$
In particular,
$$\varepsilon_{\phi(0)}+\varepsilon_{\phi(1)}+\varepsilon_{\phi(4)}+\varepsilon_{\phi(5)}\quad\mbox{ and } \quad\varepsilon_{\phi(2)}+\varepsilon_{\phi(3)}+\varepsilon_{\phi(6)}+\varepsilon_{\phi(7)}$$
are two mutually orthogonal maximal measures on $X$ with the same barycenter but there are multipliers which are not constant on the supports of these measures.   \end{example2}

\section{On preservation of extreme points}\label{s:preserveext}

In this section we investigate the relationship between $\ext X$ and $\ext S(H)$ for an intermediate function space $H$ on a compact convex set $X$. This is motivated, among others, by the need to clarify when Corollary~\ref{cor:rovnost na ext}$(b)$ may be applied.

Given a compact convex set $X$ and an intermediate function space $H$ on $X$,  $\iota\colon X\to S(H)$ and $\pi\colon S(H)\to X$ will be the mappings provided by Lemma~\ref{L:intermediate}. 

We start by the following easy lemma.

\begin{lemma}\label{L:ext a fibry} Let $X$ be a compact convex set and let $H$ be an intermediate function space.
\begin{enumerate}[$(a)$]
    \item Let $x\in \ext X$. Then $\pi^{-1}(x)$ is a closed face of $S(H)$, so it contains an extreme point of $S(H)$. In particular, $\pi(\ext S(H))\supset\ext X$.
    \item If $H$ is determined by extreme points, then $\pi(\ext S(H))\subset \overline{\ext X}$.
\end{enumerate}
\end{lemma}

\begin{proof}
Assertion $(a)$ is obvious (using the Krein-Milman theorem). Assertion $(b)$ follows easily from Lemma~\ref{L:intermediate}$(d)$.
\end{proof}

\begin{example2}
It may happen that for some $\varphi\in\ext S(H)$ we have $\pi(\varphi)\notin\ext X$, even if $H$ is determined by extreme points:

Let $K=[0,1]$ and
$$E=\{f\in C([0,1])\setsep f(\tfrac{1}{2})=\tfrac{1}{2}(f(0)+f(1))\}.$$
Then $X=S(E)$ is a simplex %(see \cite[Theorem 2]{stacey}) 
and  $\Ch_E K=[0,\tfrac12)\cup(\tfrac12,1]$ (see Example~\ref{E:more points}(2)).

Let
$$\widetilde{H}=\{f\in \Ba_1^b([0,1])\setsep f(\tfrac{1}{2})=\tfrac{1}{2}(f(0)+f(1))\}.$$
Then $\widetilde{H}$ is formed exactly by pointwise limits of uniformly bounded sequences from $E$. In particular, $\widetilde{H}$ is determined by $\Ch_EK$. Clearly, $\widetilde{H}$ fits to the scheme from Lemma~\ref{L:function space}$(c)$. Hence, using notation from the quoted lemma,  $H=V(\widetilde{H})$ is an intermediate function space on $X$. % Let $\imath\colon K\to S(H)=S(V(\widetilde{H}))$ be the mapping from Lemma~\ref{L:function space}(b).
Let 
$$F=\left\{\varphi\in S(H)\setsep \varphi\left(V\left(1_{\left[0,\tfrac12-\delta\right]\cup\left\{\tfrac12\right\}\cup\left[\tfrac12+\delta,1\right]}\right)\right)=0\mbox{ for each }\delta\in(0,\tfrac12)\right\}.$$
Then $F$ is nonempty, it contains for example any cluster point of the sequence $\imath(\frac12-\tfrac{1}{n+1})$ in $S(H)$ (where $\imath\colon K\to S(H)$ is the mapping from Lemma~\ref{L:function space}$(c)$). Further, it is clear that $F$ is a closed face of $S(H)$, so it contains some extreme points of $S(H)$. 

Finally observe that $\pi(F)=\{\phi({\frac12})\}$, where $\phi\colon K\to X$ is the evaluation mapping from Section~\ref{ssc:ch-fs}. 
Indeed, assume $\varphi\in F$. By Lemma~\ref{L:reprezentace} there is a net $(x_\alpha)$ in $X=S(E)$ such that $\iota(x_\alpha)\to\varphi$. By the Hahn-Banach theorem each $x_\alpha$ may be extended to a state on $C([0,1])$, represented by a probability measure $\mu_\alpha$ on $[0,1]$.
Then for each $\delta\in(0,\frac12)$ we have
$$\mu_\alpha([0,\tfrac12-\delta]\cup\{\tfrac12\}\cup[\tfrac12+\delta,1])\to0,$$
i.e.,
$$\mu_\alpha((\tfrac12-\delta,\tfrac12)\cup(\tfrac12,\tfrac12+\delta))\to1.$$
Let $\mu$ be any weak$^*$-cluster point of the net $(\mu_\alpha)$. Then $\mu([\tfrac12-\delta,\tfrac12+\delta])=1$ for each $\delta\in(0,\tfrac{1}{2})$. Thus $\mu=\varepsilon_{\frac12}$. It follows that $\mu_\alpha\to\varepsilon_{\frac12}$ in the weak$^*$-topology. We conclude that $\pi(\varphi)=\phi(\frac12)$.

Since $\phi(\frac12)\notin\ext X=\phi(\Ch_E K)$, the proof is complete.
\end{example2}

In view of Lemma~\ref{L:platnost soucinu} it is important to know when $\iota(x)$ is an extreme point of $S(H)$ whenever $x\in \ext X$. A general characterization is given in the following lemma.

\begin{lemma}\label{L:zachovavani ext}
Let $X$ be a compact convex set and let $H$ be an intermediate function space.
Let $x\in \ext X$. Then the following are equivalent:
\begin{enumerate}[$(1)$]
    \item $\iota(x)\in \ext S(H)$.
    \item $\iota(x)\in\ext\pi^{-1}(x)$.
    \item If $(x_i)$ and $(y_i)$ are two nets in $X$ converging to $x$ such that $$a(\tfrac{x_i+y_i}2)\to a(x) \mbox{ for each }a\in H,$$ then $a(x_i)\to a(x)$ and $a(y_i)\to a(x)$ for each  $a\in H$.
\end{enumerate}
 \end{lemma}

\begin{proof}
Implication $(1)\implies(2)$ is trivial. To show $(2)\implies(1)$ we recall that $\pi^{-1}(x)$ is a closed face of $S(H)$. 

%In the rest of the proof we denote by $\iota_0$ the affine homeomorphism of $X$ onto $S(A_c(X))$ provided by Lemma~\ref{L:intermediate}$(c)$.

$(3)\implies(2)$: Assume that $\iota(x)=\frac12(\varphi_1+\varphi_2)$ for some $\varphi_1,\varphi_2\in \pi^{-1}(x)$. 
By Lemma~\ref{L:reprezentace} there are nets $(x_i)$ and $(y_i)$ in $X$ converging to $x$ such that $\iota(x_i)\to\varphi_1$ and $\iota(y_i)\to\varphi_2$ in $S(H)$.

Further, $$\frac12(\iota(x_i)+\iota(y_i))
\to \frac12(\varphi_1+\varphi_2)=\iota(x)\mbox{ in }S(H),$$
hence
$a(\frac{x_i+y_i}2)\to a(x)$ for each $a\in H$.
The assumption yields $a(x_i)\to a(x)$ and $a(y_i)\to a(x)$ for each $a\in H$, i.e., $\iota(x_i)\to \iota(x)$ and  $\iota(y_i)\to \iota(x)$ in $S(H)$. In other words, $\varphi_1=\varphi_2=\iota(x)$. 
Hence, $\iota(x)\in\ext \pi^{-1}(x)$.

$(1)\implies(3)$: Assume that  $(x_i)$ and $(y_i)$ are nets satisfying the given properties. 
Up to passing to subnets we may assume that $\iota(x_i)\to \varphi_1$ and $\iota(y_i)\to\varphi_2$ in $S(H)$. Then, given $a\in H$ we have
$$\frac12(\varphi_1(a)+\varphi_2(a))=\lim_i a(\tfrac{x_i+y_i}2)=a(x),$$
so $\frac12(\varphi_1+\varphi_2)=\iota(x)$. By the assumption we get that $\varphi_1=\varphi_2=\iota(x)$, hence
$a(x_i)\to a(x)$ and $a(y_i)\to a(x)$ for each $a\in H$.
This completes the proof.
\end{proof}

The following two observations provide a sufficient condition for a point $x\in\ext X$ to satisfy $\iota(x)\in\ext S(H)$.

\begin{lemma}\label{lem:contin}
Let $X$ be a compact convex set and $H$ an intermediate function space. Let $x\in\ext X$. Assume that each $f\in H$ is continuous at $x$. Then $\pi^{-1}(x)=\{\iota(x)\}$, in particular, $\iota(x)\in \ext S(H)$.
\end{lemma}

\begin{proof}
Let $\varphi\in \pi^{-1}(x)$. Let $(x_\nu)$ be a net provided by Lemma~\ref{L:reprezentace}. 
Then for each $f\in H$ we have
$$\varphi(f)=\lim_\nu f(x_\nu)=f(x),$$
so $\varphi=\iota(x)$.
\end{proof}

\begin{cor}\label{c:spoj-ext} Let $X$ be a compact convex set and $H$ an intermediate function space. Assume that each function from $H$ is strongly affine. Let $x\in \ext X$.
\begin{enumerate}[$(i)$]
    \item Assume that for each $f\in H$ its restriction $f|_{\overline{\ext X}}$ is continuous at $x$. Then $\pi^{-1}(x)=\{\iota(x)\}$, in particular, $\iota(x)\in \ext S(H)$.
    \item If $x$ is an isolated extreme point, then  $\pi^{-1}(x)=\{\iota(x)\}$, in particular, $\iota(x)\in \ext S(H)$.
\end{enumerate}
\end{cor}

\begin{proof}
Since $(ii)$ follows from $(i)$, it is enough to verify $(i)$. To this end we need to prove that any function $f\in H$ is continuous at $x\in\ext X$ provided its restriction $f|_{\ov{\ext X}}$ is continuous at $x$. So let $\ep>0$ be given. We find an open neighbourhood $U$ containing $x$ such that $\abs{f(y)-f(x)}<\ep$ for every $y\in U\cap \ov{\ext X}$. By \cite[Lemma 2.2]{rs-fragment} there exists a neighbourhood $V$ of $X$ such that for each $y\in V$ and $\mu\in M_y(X)$ it holds $\mu(U)>1-\ep$. Then for arbitrary $y\in V$ we pick a measure $\mu_y\in M_y(X)\cap M_1(\ov{\ext X})$. Then
\[
\begin{aligned}
\abs{f(y)-f(x)}&=\abs{\mu_y(f)-\mu_y(f(x))}\le \int_{\ov{\ext X}}\abs{f-f(x)}\di\mu_y\\
&\le\int_{\ov{\ext X}\cap U} \abs{f-f(x)}\di\mu_y+2\norm{f}\mu_y(X\setminus U)\\
&\le \ep+2\ep\norm{f}.
\end{aligned}
\]
Hence $f$ is continuous at $x$ and Lemma~\ref{lem:contin} applies.
\end{proof}

\begin{example}\label{ex:Bauer}
    There are metrizable Bauer simplices $X_1, X_2$ and  intermediate function spaces $H_1,H_2$ on $X_1, X_2$, respectively, such that the following properties are satisfied (for $i=1,2$):
    \begin{enumerate}[$(i)$]
         \item $H_i\subset A_1(X_i)$.
         \item $Z(H_i)=H_i$ and $M(H_i)=A_c(X_i)$.
        \item There is $x\in\ext X_1$ with $\iota(x)\notin \ext S(H_1)$ and $\iota(\ext X_1\setminus\{x\})\subset\ext S(H_1)$.
        \item $\iota(\ext X_2)\cap \ext S(H_2)=\emptyset$.
        \item $S(H_i)$ is a Bauer simplex and $S(H_1)$ is metrizable.
    \end{enumerate}
    \end{example}

\begin{proof} We will use the approach of Lemma~\ref{L:function space}, together with the respective notation. We construct first $X_1$ and $H_1$ and then the second pair. 

(1) Let $K_1=[0,1]$ and $E_1=C(K_1)$ and set $X_1=S(E_1)=M_1(K_1)$. Then $X_1$ is a metrizable Bauer simplex and we have $\Ch_{E_1} K_1=[0,1]$.

Set
$$\begin{aligned}
\widetilde{H_1}=\Bigl\{f:[0,1]\to\er\setsep f&\mbox{ is continuous on }[0,\tfrac12)\cup(\tfrac12,1], \\ & \lim_{t\to\frac12-} f(t) \mbox{ and } \lim_{t\to\frac12+} f(t)\mbox{ exist in }\er \\& \mbox{ and }
f(\tfrac12)=\tfrac12(\lim_{t\to\frac12-} f(t)+\lim_{t\to\frac12+} f(t)) \Bigr\}.\end{aligned}$$
Clearly $\widetilde{H_1}\subset\Ba_1^b(K_1)$. Further,
$\widetilde{H_1}$ fits into the scheme from Lemma~\ref{L:function space} and $H_1=V(\widetilde{H_1})\subset A_1(X_1)$. In particular, $H_1$ is determined by 
$$\ext X_1=\phi(\Ch_{E_1} K_1)=\{\ep_t\setsep t\in [0,1]\}.$$

Let us characterize $\ext S(H_1)$.  We fix $t\in [0,1]$ and describe $\pi^{-1}(\phi(t))$. If $t\ne\frac12$, then any $f\in H$ is continuous at $t$ when restricted to $\ov{\ext X}$, hence $\pi^{-1}(\phi(t)) =\{\imath(t)\}$ by Corollary~\ref{c:spoj-ext}. In particular, in this case $\imath(t)=\iota(\ep_t)\in\ext S(H_1)$.

Next assume that $t=\frac12$. Let $\varphi$ be an extreme point of $\pi^{-1}(\phi(t))$. 
 It follows from Lemma~\ref{L:function space}$(d)$ that there is a net $(t_\nu)$ in $[0,1]$ converging to $\frac{1}{2}$ such that $\varphi=\lim_\nu\imath(t_\nu)$. Up to passing to a subnet we may assume that one of the following possibilities takes place:
 \begin{itemize}
     \item $t_\nu=\frac12$ for all $\nu$. Then $\varphi=\imath(\frac12)$.
     \item $t_\nu<\frac12$ for all $\nu$. Then
     $$\varphi(Vf)=\lim_{s\to\frac12-} f(s),\quad f\in \widetilde{H_1}.$$
     Denote this functional by $\varphi_-$.
     \item $t_\nu>\frac12$ for all $\nu$. Then
     $$\varphi(Vf)=\lim_{s\to\frac12+} f(s),\quad f\in \widetilde{H_1}.$$
     Denote this functional by $\varphi_+$.
 \end{itemize}
Thus, $\ext\pi^{-1}(\phi(\frac12))\subset\{\imath(\frac12),\varphi_-,\varphi_+\}$. Since $\imath(\frac12)=\frac12(\varphi_-+\varphi_+)$, we deduce that  $\ext\pi^{-1}(\phi(\frac12))=\{\varphi_-,\varphi_+\}$ and $\pi^{-1}(\phi(\frac12))$ is the segment connecting $\varphi_-$ and $\varphi_+$.

We deduce that 
$$\ext S(H_1)=\{\imath(t)\setsep t\in [0,1]\setminus \{\tfrac12\}\}\cup\{\varphi_-,\varphi_+\}.$$
Indeed, inclusion `$\supset$' follows from Lemma~\ref{L:ext a fibry}$(a)$ using the above analysis of the fibers of $\pi$. The converse inclusion follows by using moreover Lemma~\ref{L:ext a fibry}$(b)$.

Therefore, $\ext S(H_1)$ is a metrizable compact set homeomorphic to the union of two disjoint closed intervals. Since $H_1$ is canonically linearly isometric to the space $C(\ext S(H_1))$, we deduce that $S(H_1)$ is also a metrizable Bauer simplex.

It follows that
$$Z(A_c(S(H_1)))=M(A_c(S(H_1))=A_c(S(H_1))\quad\mbox{and}\quad m(S(H_1))=\ext S(H_1).$$
In particular,  $Z(H_1)=H_1$.

Further, $\frac12\in\Ch_{E_1}K_1$, but $\imath(\frac12)\notin \ext S(H_1)$. Hence, not only the extreme point $\phi(\frac12)$ is not preserved, but $\imath(\frac12)\notin m(S(H_1))$, so $M(H_1)\subsetneqq Z(H_1)$. It follows easily from Lemma~\ref{L:ob-soucin} and Lemma~\ref{L:platnost soucinu} that
$M(H_1)=A_c(X_1)$  (and hence it can be identified with $E_1$).

(2) Let $K_2=\TT=\{z\in\ce\setsep \abs{z}=1\}$, $E_2=C(\TT)$ and set $X_2=S(E_2)$. Then $X_2$ is a metrizable Bauer simplex and $\Ch_{E_2} \TT=\TT$.

Set
\[
\begin{aligned}
\widetilde {H_2}=\Bigl\{f\in\ell^\infty(\TT)\setsep
\forall z\in \TT \colon &
\lim_{w\to z-} f(w) \mbox{ and } \lim_{w\to z+} f(w)\mbox{ exist in }\er \\& \mbox{ and }
f(z)=\tfrac12(\lim_{w\to z-} f(w)+\lim_{w\to z+} f(w)) \Bigr\}.
\end{aligned}
\]
To unify the meaning of one-sided limits we assume that the circle is oriented counterclockwise.

It is clear that $E_2\subset \widetilde{H_2}$. Further, given any $f\in \widetilde{H_2}$ we define the gap at a point $z\in\TT$ as the difference of the one-sided limits at that point. It follows from the existence of one-sided limits at each point that, given $\varepsilon>0$, there are only finitely many points where the absolute value of the gap is above $\varepsilon$. Therefore, $f$ is continuous except at countably many points and, hence $f$ is a Baire-one function.

Thus $\widetilde{H_2}$ fits into the scheme from Lemma~\ref{L:function space}. Then $H_2=V(\widetilde{H_2})$  is an intermediate function space on $X_2$ contained in $A_1(X_2)$. Similarly as in (1) above we see that
$$\ext S(H_2)=\{\varphi_{z+},\varphi_{z-}\setsep z\in \TT\},$$
where 
$$\begin{aligned}
\varphi_{z+}(Vf)&=\lim_{w\to z+} f(w),\quad f\in \widetilde{H_2},\\
\varphi_{z-}(Vf)&=\lim_{w\to z-} f(w),\quad f\in \widetilde{H_2}.
\end{aligned}$$
The set $\ext S(H_2)$ is compact -- it is a circular variant of the double arrow space (cf. \cite[Theorem 2.3.1]{fabian-kniha} or \cite{kalenda-stenfra}). Indeed, the topology on $\ext S(H_2)$ is generated by `open arcs', i.e., by the sets of the form
$$\{\varphi_{e^{ia}-}\}\cup \{\varphi_{e^{it}+},\varphi_{e^{it}-}\setsep t\in (a,b)\} \cup \{\varphi_{e^{ib}+}\},\quad a,b\in\er,a<b.$$
We again get that $S(H_2)$ is a Bauer simplex (this time non-metrizable) and hence
$$Z(A_c(S(H_2)))=A_c(S(H_2)) \mbox{ and } m(S(H_2))=\ext S(H_2).$$ 
In particular, $Z(H_2)=H_2$.
 For each $z\in \TT$ we have
$$\imath(z)=\frac12(\varphi_{z-}+\varphi_{z+}),$$
hence $\imath(K_2)\cap \ext S(H_2)=\emptyset$, i.e., no extreme point is preserved.

Finally, using Lemma~\ref{L:ob-soucin} and Lemma~\ref{L:platnost soucinu} we easily get $M(H_2)=A_c(X_2)$.
\end{proof}

The previous examples show that, given $x\in\ext X$, $\iota(x)$ need not be an extreme point of $S(H)$, even if $X$ is a Bauer simplex. Moreover, in these cases we even have $M(H)\subsetneqq Z(H)$. Next we focus on some sufficient conditions. The first result is a refinement of \cite[Lemma 3]{cc}.

To formulate it we use the notion of a split face recalled in Section~\ref{ssc:ccs} above together with the respective notation.  
We further need to recall that, given a bounded function $f\colon X\to \er$ (where $X$ is a compact convex set), its \emph{upper envelope}\index{upper envelope} $f^*$ is defined as 
\[
\gls{f*}(x)=\inf\left\{a(x)\setsep a\in A_c(X),a\ge f\right\},\quad x\in X
\]
(see also \cite[p. 4]{alfsen} where the upper envelope is denoted by $\widehat{f}$). The relationship of upper envelopes and split faces is revealed by the following observation which follows from \cite[Proposition II.6.5]{alfsen}:
\begin{equation}\label{eq:lambda=1*}
 F\subset X\mbox{ a closed split face}\implies \lambda_F=1_F^*.   
\end{equation}

\begin{lemma}\label{l:lemma-split}
Let $X$ be a compact convex set and let $H$ be an intermediate function space. Let $x\in \ext X$. 
 If $\{x\}$ is a split face of $X$ and $1_{\{x\}}^*\in H$, then $\{\iota(x)\}$ is a split face of $S(H)$. In particular, $\iota(x)$ is an extreme point of $S(H)$.
\end{lemma}

\begin{proof}
Assume that $\{x\}$ is a split face. By \eqref{eq:lambda=1*} we know that $\lambda_{\{x\}}=1_{\{x\}}^*$ and hence
$F=(1_{\{x\}}^*)^{-1}(0)$ is the complementary face.
Set
$$\widetilde{F}=\{\varphi\in S(H)\setsep \varphi(1_{\{x\}}^*)=0\}.$$
Since $1_{\{x\}}^*\in H$, $\widetilde{F}$ is a closed face of $S(H)$ containing $\iota(F)$. We are going to show that $\{\iota(x)\}$ is a split face of $S(H)$ and its complementary face is $\widetilde{F}$.

To this end fix an arbitrary $\varphi\in S(H)$. By Lemma~\ref{L:reprezentace} there is a net $(y_\alpha)$ in $X$ such that $\iota(y_\alpha)\to\varphi$ in $S(H)$. We know that for each $\alpha$ we have
$$y_\alpha=\lambda_\alpha x+(1-\lambda_\alpha)z_\alpha, \mbox{ where }\lambda_\alpha=1_{\{x\}}^*(y_\alpha)\mbox{ and }z_\alpha\in F.$$
Then
$$\varphi(1_{\{x\}}^*)=\lim_\alpha \iota(y_\alpha)(1_{\{x\}}^*)=\lim_\alpha 1_{\{x\}}^*(y_\alpha)=\lim_\alpha \lambda_\alpha,$$
i.e., 
$$\lambda_\alpha\to \lambda:= \varphi(1_{\{x\}}^*).$$
If $\lambda=1$, then
$$\iota(y_\alpha)=\lambda_\alpha\iota(x)+(1-\lambda_\alpha)\iota(z_\alpha)\to \iota(x),$$
hence $\varphi=\iota(x)$.

If $\lambda<1$, then
$$\lim_\alpha \iota(z_\alpha)=\lim_\alpha\frac{\iota(y_\alpha)-\lambda_\alpha\iota(x)}{1-\lambda_\alpha}=\frac{\varphi-\lambda\iota(x)}{1-\lambda}.$$
So, $\iota(z_\alpha)\to\psi\in S(H)$ such that 
$$\varphi=\lambda\iota(x)+(1-\lambda)\psi.$$
Since $\varphi(1_{\{x\}}^*)=\lambda$, necessarily $\psi(1_{\{x\}}^*)=0$, so $\psi\in \tilde{F}$.

It follows that $\co(\{\iota(x)\}\cup \widetilde F)=S(H)$. Since $\widetilde F$ is a face, this completes the proof.
\end{proof}

\begin{cor}\label{cor:iotax}
Let $X$ be a compact convex set such that each $x\in\ext X$ forms a split face (in particular, this takes place if $X$ is a simplex). Let $H$ be an intermediate function space containing all affine semicontinuous functions (i.e., $A_s(X)\subset H$). Then $\iota(x)\in\ext S(H)$ for any $x\in\ext X$. 

This applies among others in the following cases:
\begin{enumerate}[$(a)$]
    \item $A_b(X)\cap\Bo_1(X)\subset H$;
    \item $X$ is metrizable and $A_1(X)\subset H$.
\end{enumerate}

If $H$ is moreover determined by extreme points, then $Z(H)=M(H)$.
\end{cor}

\begin{proof}
Assume each $x\in \ext X$ forms a split face. By the very definition of the upper envelope we see that functions $1_{\{x\}}^*$ are upper semicontinuous. By \eqref{eq:lambda=1*} they are also affine and hence they belong to $H$ by our assumptions.
We then conclude by Lemma~\ref{l:lemma-split}. 

The `in particular' part follows from \cite[Theorem II.6.22]{alfsen}.

The equality $Z(H)=M(H)$ then follows from Proposition~\ref{p:zh-mh}.
\end{proof}

Note that Corollary~\ref{cor:iotax} does not cover (among others) the case $H=A_1(X)$ if $X$ is a non-metrizable simplex.
However, in this case we can use a separable reduction method. This is the content of assertion $(a)$ of the  
 following proposition. Assertion $(b)$ shows the equality $Z(A_1(X))=M(A_1(X))$ in another special case. 

\begin{prop}\label{p:postacproa1}
Let $X$ be a a compact convex set and let $H$ be an intermediate function space on $X$. 
\begin{enumerate}[$(a)$]
    \item  Assume that $X$ is a simplex and $A_1(X)\subset H\subset (A_{c}(X))^\sigma$. Then $\iota(x)\in\ext S(H)$ whenever $x\in\ext X$. Hence $Z(H)=M(H)$.
    \item  If $H=A_1(X)$ and $\ext X$ is a \lin\ resolvable set, then $Z(H)=M(H)$.
    \end{enumerate}
\end{prop}

\begin{proof}
$(a)$: Let $x\in \ext X$ be given. We want to verify condition (3) in Lemma~\ref{L:zachovavani ext}. Let $(x_i)$, $(y_i)$ be nets converging to $x$ such that $a(\tfrac{x_i+y_i}{2})\to a(x)$ for each $a\in H$. Let $b\in H$ be given. We shall verify that $b(x_i)\to b(x)$ and $b(y_i)\to b(x)$. 

By \cite[Theorem 9.12]{lmns}, there exists a metrizable simplex $Y$, an affine continuous surjection $\varphi\colon X\to Y$ and a function $\tilde{b}\in (A_c(X))^\sigma$ such that $\varphi(x)\in\ext Y$ and $b=\tilde{b}\circ \varphi$.
Set
$$\widetilde{H}=\{\tilde{f}\in (A_c(Y))^\sigma \setsep \tilde{f}\circ\varphi\in H\}.$$
Then $\widetilde{H}$ is a closed subspace of $(A_c(Y))^\sigma$. Further, clearly $A_1(Y)\subset \widetilde{H}$.

We have $\varphi(x_i)\to\varphi(x)$, $\varphi(y_i)\to \varphi(x)$ and for each $\tilde{f}\in \widetilde{H}$
\[
\tilde{f}(\tfrac{\varphi(x_i)+\varphi(y_i)}{2})=(\tilde{f}\circ \varphi)(\tfrac{x_i+y_i}{2})\to
(\tilde{f}\circ \varphi)(x)=\tilde{f}(\varphi(x)).
\]
By Corollary~\ref{cor:iotax}$(b)$ and Lemma~\ref{L:zachovavani ext}, $\tilde{f}(\varphi(x_i))\to \tilde{f}(\varphi(x))$ as well as $\tilde{f}(\varphi(y_i))\to \tilde{f}(\varphi(x))$ for each $\tilde{f}\in \widetilde{H}$. In particular,
\[
b(x_i)=\tilde{b}(\varphi(x_i))\to \tilde{b}(\varphi(x))=b(x)
\]
for our function $b$. Similarly we have $b(y_i)\to b(x)$. Hence $\iota(x)\in \ext S(H)$ by Lemma~\ref{L:zachovavani ext}. The equality $Z(A_1(X))=M(A_1(X))$ then follows from Proposition~\ref{p:zh-mh}.

$(b)$: Since $H=A_1(X)$ is determined by $\ext X$, we have immediately $M(H)\subset Z(H)$ (see Proposition~\ref{P:mult}). To show the converse inclusion, let $T\in\frd(H)$ be given and $m=T(1)$. We want to show that $m\in M(A_1(X))$. Using Lemma~\ref{L:nasobeni} we infer that $T(a)=m\cdot a$ on $\ext X$ for each $a\in A_c(X)$. Let $a\in H=A_1(X)$ be given. We will find a function $b\in A_1(X)$ such that $b=m\cdot a$ on $\ext X$.

To this end, let $\{a_n\}$ be a bounded sequence in $A_c(X)$ converging to $a$ on $X$. Then the sequence of $\{T(a_n)\}$ in $A_1(X)$ satisfies $T(a_n)=m\cdot a_n$ on $\ext X$. Since $\{m\cdot a_n\}$ converges pointwise on $\ext X$, the bounded sequence $\{T(a_n)\}$  converges pointwise on $X$, say to a function $b$ on $X$. Then $b$ is a strongly affine Baire function on $X$, which equals $m\cdot a$ on $\ext X$. Since $m\cdot a$ is a Baire-one function on $\ext X$, from \cite[Theorem 6.4]{lusp} we obtain that $b\in A_1(X)$. Thus $b$ is the desired function and $m=T(1)\in M(A_1(X))$. Thus $Z(H)\subset M(H)$ and the proof is complete. 
\end{proof}

\begin{example}\label{ex:4prostory}
 There is a (non-metrizable) Bauer simplex $X$ and four intermediate function spaces on $X$ satisfying
 $$A_c(X)\subsetneqq H_1\subsetneqq H_2=A_1(X)\subsetneqq H_3\subsetneqq H_4=A_b(X)\cap\Bo_1(X)$$
 such that the following assertions hold (where $\iota_i:X\to S(H_i)$ is the mapping from Lemma~\ref{L:intermediate}):
 \begin{enumerate}[$(i)$]
     \item $M(H_1)=A_c(X)$,  $Z(H_1)=H_1$, $\iota_1(\ext X)\setminus\ext S(H_1)\ne\emptyset$;
     \item $M(H_2)=Z(H_2)=H_2$, $\iota_2(\ext X)\subset \ext S(H_2)$;
     \item $M(H_3)=H_2$,  $Z(H_3)=H_3$, $\iota_3(\ext X)\setminus\ext S(H_3)\ne\emptyset$;
     \item $M(H_4)=Z(H_4)=H_4$, $\iota_4(\ext X)\subset \ext S(H_4)$.
 \end{enumerate}
 \end{example}

\begin{proof}
We start by noticing that assertions $(ii)$ and $(iv)$ are valid for the given choices of intermediate function spaces as soon as $X$ is a simplex (by Proposition~\ref{p:postacproa1}$(a)$ and Corollary~\ref{cor:iotax}$(b)$). So, we need to construct $X$
and $H_1$ and $H_3$ such that $(i)$ and $(iii)$ are fulfilled.

We will again proceed using Lemma~\ref{L:function space} and the respective notation. 
Let $K=\omega_1+1+\omega_1^{-1}$, equipped with the order topology. By $\omega_1$ we mean the set of all countable ordinals with the standard well order, by $\omega_1^{-1}$ the same set with the inverse order. Then $K$ is a compact space (it coincides with the space from \cite[Example 1.10(iv)]{kalenda-survey}). Set $E=C(K)$ and $X=S(E)$. Then $X$ is a Bauer simplex which may be canonically identified with $M_1(K)$. The evaluation mapping $\phi:K\to X$ from Section~\ref{ssc:ch-fs} assigns to each $x\in K$ the Dirac measure $\ep_x$.

We fix two specific points in $K$ -- by $a$ we will denote the first accumulation point (from the left, it is usually denoted by $\omega$) and by $b$ the `middle point', i.e., the unique point of uncountable cofinality.
Let us define the following spaces:
$$\begin{alignedat}{4}
\widetilde{H_1}&=\Bigl\{f\in \ell^\infty(K)\setsep &&f\mbox{ is continuous at each }x\in K\setminus\{a\} \\ &&& \lim_{n<\omega} f(2n) \mbox{ and }\lim_{n<\omega} f(2n+1) \mbox{ exist in }\er \\ &&& \mbox{ and } f(a)=\tfrac12(\lim_{n<\omega} f(2n)+\lim_{n<\omega} f(2n+1))\Bigr\},
\\ \widetilde{H_2}&=\Ba_1^b(K), && 
\\ \widetilde{H_3}&=\Bigl\{f\in \ell^\infty(K)\setsep && \lim_{x\to b-} f(x) \mbox{ and }\lim_{x\to b+} f(x) \mbox{ exist in }\er \\ &&& \mbox{ and } f(b)=\tfrac12(\lim_{x\to b-} f(x)+\lim_{x\to b+} f(x))\Bigr\},
\\ \widetilde{H_4}&=\Bo_1^b(K).&&
\end{alignedat}$$
Then clearly
$$E\subsetneqq \widetilde{H_1}\subsetneqq \widetilde{H_2}\subsetneqq \widetilde{H_3}.$$
Moreover, any function on $K$ with finite one-sided limits at $b$ must be constant on the sets $(\alpha,b)$ and $(b,\beta)$ for some $\alpha\in\omega_1$ and $\beta\in\omega_1^{-1}$. So, any such function is of the first Borel class. Thus $\widetilde{H_3}\subsetneqq\widetilde{H_4}$.
It follows that $\widetilde{H_4}$ (and hence also the smaller spaces) fits into the scheme from Lemma~\ref{L:function space}. So, we set $H_j=V(\widetilde{H_j})$ for $j=1,2,3,4$.
Then the chain of inclusions and equalities is satisfied and, moreover, assertions $(ii)$ and $(iv)$ are fulfilled as explained above.

Similarly as in the proof of Example~\ref{ex:Bauer} (part (1)) we see that 
$$\ext S(H_1)=\{\imath_1(x)\setsep x\in K\setminus \{a\}\}\cup\{\varphi_o,\varphi_e\},$$
where
$$\varphi_o(Vf)=\lim_{n<\omega} f(2n+1)\quad\mbox{and}\quad\varphi_e(Vf)=\lim_{n<\omega} f(2n)$$
for $f\in \widetilde{H_1}$. And in the same way we deduce that
$\imath_1(a)\notin m(S(H_1))$ and hence assertion $(i)$ is fulfilled.

Finally, let us prove assertion $(iii)$. It follows from Lemma~\ref{l:lemma-split} that $\imath_3(x)\in \ext S(H_3)$ for $x\in K\setminus\{b\}$. On the other hand, $\imath_3(b)=\frac12(\varphi_-+\varphi_+)$, where
$$\varphi_-(Vf)=\lim_{x\to b-} f(x)\quad\mbox{and}\quad\varphi_+(Vf)=\lim_{x\to b+} f(x)$$
for $f\in \widetilde{H_3}$. Since $\varphi_-$ and $\varphi_+$ are distinct points of $S(H_3)$, we deduce that $\imath_3(b)\notin \ext S(H_3)$.

Further, it is easy to check that
$$M(H_3)=V(\{f\in\widetilde{H_3}\setsep f\mbox{ is continuous at }b\})=H_2.$$
It remains to check that $Z(H_3)=H_3$. To this end we observe that
$$\forall f,g\in \widetilde{H_3}\, \exists h\in\widetilde{H_3}\colon h=fg\mbox{ on }K\setminus\{b\}.$$
In other words,
$$\forall f,g\in H_3\, \exists h\in H_3\colon h=fg\mbox{ on }\ext X\setminus\{\ep_b\}.$$
Let $T:H_3\to A_c(S(H_3))$ be the mapping from Lemma~\ref{L:intermediate}. Then the above formula may be rephrased as
\begin{equation}\label{eq:H3}
    \forall u,v\in A_c(S(H_3))\, \exists w\in A_c(S(H_3))\colon w=uv\mbox{ on }\iota_3(\ext X\setminus\{\ep_b\}).\end{equation}
Since $H_3$ is determined by extreme points, Lemma~\ref{L:intermediate}$(d)$ yields that
$$\ext S(H_3)\subset\overline{\iota_3(\ext X)} =\overline{\iota_3(\ext X\setminus\{\ep_b\})}\cup\{\iota_3(\ep_b)\}.
$$
Since $\iota_3(\ep_b)\notin \ext S(H_3)$ (as we have shown above), we deduce that 
$$\ext S(H_3)\subset \overline{\iota_3(\ext X\setminus\{\ep_b\})}.$$
Using this inclusion and \eqref{eq:H3} we deduce that
$$ \forall u,v\in A_c(S(H_3))\, \exists w\in A_c(S(H_3))\colon w=uv\mbox{ on }\ext S(H_3),$$
so
$M(A_c(S(H_3)))=A_c(S(H_3))$, which means that $Z(H_3)=H_3$.
\end{proof}

The message of Example~\ref{ex:4prostory} is twofold. Firstly, it shows that the upper bound on $H$ in Proposition~\ref{p:postacproa1}$(a)$ cannot be dropped.
And secondly, it shows that the equality $Z(H)=M(H)$ and the preservation of extreme points depends not only on the size of $H$ -- when $H$ is enlarged, these properties may be fixed (like in passing from $H_1$ to $H_2$ or from $H_3$ to $H_4$) or spoiled (like in passing from $A_c(X)$ to $H_1$ or from $H_2$ to $H_3$).

The next example shows that the assignments $H\mapsto Z(H)$, $H\mapsto M(H)$ and $H\mapsto M^s(H)$ are not monotone.

\begin{example}\label{ex:inkluzeZ}
    There is a metrizable Bauer simplex $X$ and two intermediate function spaces $H_1, H_2$ on $X$ such that the following properties are fulfilled:
    \begin{enumerate}[$(i)$]
        \item $H_1\subset H_2\subset A_1(X)$;
        \item $S(H_2)$ is a metrizable Bauer simplex;
        \item $M(H_2)=Z(H_2)=H_2$;
        \item $M(H_1)=Z(H_1)\subsetneqq A_c(X)=Z(A_c(X))=M(A_c(X))$.
    \end{enumerate}
   Moreover, for these spaces all multipliers are strong. 
\end{example}

\begin{proof}
 Let $K=[0,2]\cup[3,5]$, $E=C(K)$ and $X=S(E)=M_1(K)$. Then $X$ is a metrizable Bauer simplex. In particular, $A_c(X)=Z(A_c(X))=M(A_c(X))$.
 
 We further set
 $$\begin{alignedat}{3}
 \widetilde{H_2}&=\{f\in\ell^\infty(K)\setsep &&\mbox{ the restrictions }f|_{[0,1]}, f|_{(1,2]}, f|_{[3,4]}, f|_{(4,5]} \mbox{ are continuous},\\
 &&& \lim_{t\to 1+} f(t), \lim_{t\to 4+} f(t) \mbox{ exist in }\er\} \\ 
 \widetilde{H_1}&=\{f\in \widetilde{H_2} \setsep && f(1)+f(4)= \lim_{t\to 1+} f(t) + \lim_{t\to 4+} f(t)\}
 \end{alignedat}$$

Then  $\widetilde{H_1}$ and $\widetilde{H_2}$ fit into the scheme of Lemma~\ref{L:function space} and $\widetilde{H_1}\subset \widetilde{H_2}\subset\Ba_1^b(K)$. Set $H_1=V(\widetilde{H_1})$ and $H_2=V(\widetilde{H_2})$. Thus $(i)$ is valid.

Similarly as in Example~\ref{ex:Bauer} (part (1)) we see that
the extreme points of $S(H_2)$ are
$$\ext S(H_2)=\{\imath_2(t)\setsep t\in K\}\cup \{\varphi_1,\varphi_2\},$$
where 
$$\varphi_1(f)=\lim_{t\to 1+} f(t),\quad \varphi_2(t)=\lim_{t\to 4+} f(t).$$
So, $\ext S(H_2)$ is homeomorphic to the union of four disjoint closed intervals. Since $H_2$ is canonically identified with $C(\ext S(H_2))$, we deduce that $S(H_2)$ is a metrizable Bauer simplex. So, $(ii)$ is proved and, moreover, $Z(H_2)=H_2$. By Proposition~\ref{p:zh-mh} we get $Z(H_2)=M(H_2)$, so $(iii)$ is proved as well.

The restriction map $\pi_{21}:S(H_2)\to S(H_1)$ is clearly one-to-one on $\ext S(H_2)$ and maps $\ext S(H_2)$ onto $\ext S(H_1)$. By definition of $H_1$ we see that
$$\pi_{21}(\varphi_1)+\pi_{21}(\varphi_2)=\imath_1(1)+\imath_1(4).$$
Using Lemma~\ref{L:mult-nutne} we see that
$$M(H_1)=V(\{f\in C(K)\setsep f(1)=f(4)\})\subsetneqq V(E)=A_c(X).$$
Finally, by Proposition~\ref{p:zh-mh} we get $M(H_1)=Z(H_1)$.
Taking into account that $X$ is a Bauer simplex, the proof of $(iv)$ is complete.

The `moreover statement' follows from Proposition~\ref{P:rovnostmulti}.
  \end{proof}

The next example shows that extreme points are not automatically preserved even if $H$ is one of the spaces $A_1(X)$, $A_b(X)\cap\Bo_1(X)$, $A_f(X)$.
 
\iffalse
\begin{example}\label{ex:slozity}
There is a compact convex set $X$ with the following properties:
\begin{enumerate}[$(a)$]
\item There is some $y\in\ext X$ such that $\iota_b(y)\notin \ext S(A_b(X)\cap\Bo_1(X))$ and  $\iota_f(y)\notin\ext S(A_f(X))$ (where $\iota_f$ and $\iota_b$ are the mappings provided by Lemma~\ref{L:intermediate}); 
    \item $Z(A_f(X))=M(A_f(X))$ and $Z(A_b(X)\cap\Bo_1(X))=M(A_b(X)\cap\Bo_1(X))$.
\end{enumerate}
Moreover, there are both metrizable and non-metrizable variants of such $X$. (Note that, in case $X$ is metrizable, $A_f(X)=A_1(X)$.)
\end{example}\fi

\begin{example}\label{ex:slozity}
There is a compact convex set $X$ such that for any intermediate function space $H$ on $X$ satisfying $A_b(X)\cap\Bo_1(X)\subset H\subset A_f(X)$ the following assertions hold:
\begin{enumerate}[$(a)$]
\item There is some $y\in\ext X$ such that $\iota(y)\notin \ext S(H)$; 
    \item $Z(H)=M(H)$.
\end{enumerate}
Moreover, there are both metrizable and non-metrizable variants of such $X$. (Note that, in case $X$ is metrizable, $H=A_1(X)$.)
\end{example}

\begin{proof} The proof is divided to several steps:

\smallskip

\noindent{\tt Step 1:} The construction of $X$.

\smallskip

Fix any compact space $K$ which is not scattered and which contains a one-to-one convergent sequence.
Let $(x_n)$ be a one-to-one sequence in $K$ converging to $x\in K$. Without loss of generality we may assume that $x\notin\{x_n\setsep n\in\en\}$. Further, fix $y\in K\setminus(\{x_n\setsep n\in\en\}\cup \{x\})$.

Since $K$ is not scattered, there is a continuous Radon probability $\mu$ on $K$.
Since $\mu$ is atomless, using \cite[215D]{fremlin2} it is easy to construct by induction Borel sets
$$B_s, s\in \bigcup_{n\in\en\cup\{0\}} \{0,1\}^n$$ 
satisfying the following properties for each $s$:
\begin{itemize}
    \item[(i)] $B_\emptyset=K$;
    \item[(ii)] $B_s=B_{s,0}\cup B_{s,1}$;
    \item[(iii)] $B_{s,0}\cap B_{s,1}=\emptyset$;
    \item[(iv)] $\mu(B_s)=2^{-n}$ where $n$ is the length of $s$.
\end{itemize}
For $n\in\en$ we define a function $f_n$ on $K$ by setting
$$f_n(t)=\begin{cases}
    1 & x\in B_s \mbox{ for some $s$ of length $n$ ending by }0,\\ 
    -1 & x\in B_s \mbox{ for some $s$ of length $n$ ending by }1.
\end{cases}$$
Then $(f_n)$ is an orthonormal sequence in $L^2(\mu)$, so $f_n\to0$ weakly in $L^2(\mu)$ and hence also in $L^1(\mu)$.
Note that $\norm{f_n}_1=1$ for each $n$. 
Let $\mu_n$ be the measure on $K$ with density $f_n$ with respect to $\mu$.
Then $(\mu_n)$ is a sequence in the unit sphere of $C(K)^*$ weakly converging to $0$. Note that $\mu_n$ are continuous measures and  $\mu_n(K)=0$ for each $n$.

For $n\in\en$ set 
$$u_n=\varepsilon_y+\mu_n+\varepsilon_{x_n}-\varepsilon_x\quad\mbox{and}\quad v_n=\varepsilon_y+\mu_n+\varepsilon_{x}-\varepsilon_{x_n}.$$
The sought $X$ is now defined by
$$X=\wscl{\co(M_1(K)\cup\{u_n,v_n\setsep n\in\en\})}.$$
It is obviously a compact convex set. 

\smallskip

\noindent{\tt Step 2:} A representation of elements of $X$.

\smallskip

Set $$L=\{\varepsilon_y\}\cup\{u_n,v_n\setsep n\in\en\}.$$
Since $u_n\to \varepsilon_y$ and $v_n\to\varepsilon_y$ in $X$, we deduce that $L$
is a countable weak$^*$-compact set. Hence
\[
    \wscl{\co L}=\left\{\lambda\varepsilon_y+(1-\lambda)\sum_{n=1}^\infty (s_n u_n+t_n v_n)\setsep \lambda\in[0,1], s_n,t_n\ge 0,
    \sum_{n=1}^\infty (s_n+t_n)=1\right\}.
\]
We note that 
    $$X=\co(M_1(K)\cup\wscl{\co L}),$$ hence any element of $X$ is of the form
$$\lambda_1\sigma + (1-\lambda_1)(\lambda_2\varepsilon_y+(1-\lambda_2)\sum_{n=1}^\infty (s_n u_n+t_n v_n)),$$
where $\sigma\in M_1(K)$, $\lambda_1,\lambda_2\in [0,1]$, $s_n,t_n\ge0$ and $\sum_{n=1}^\infty (s_n+t_n)=1$. Given such a representation, we set
$\lambda=\lambda_1+(1-\lambda_1)\lambda_2$. Then $\lambda\in[0,1]$ and $1-\lambda=(1-\lambda_1)(1-\lambda_2)$. Thus, any element of $X$ is represented as
\begin{equation}\label{eq:X-reprez}
    \lambda\nu+(1-\lambda)\sum_{n=1}^\infty (s_n u_n+t_n v_n),\end{equation}
where $\nu\in M_1(K), \lambda\in[0,1], s_n,t_n\ge 0, \sum_{n=1}^\infty (s_n+t_n)=1$. It is not hard to check that this representation is not unique.

\smallskip

\noindent{\tt Step 3:} A representation of $A_f(X)$.

\smallskip

Set $F=\{\varepsilon_t\setsep t\in K\}\cup\{u_n,v_n\setsep n\in\en\}$. Then $F$ is a closed subset of $X$ and $X$ is the closed convex hull of $F$. It thus follows from the Milman theorem that $\ext X\subset F$. In fact, the equality holds, as we will see later. Currently only the inclusion is essential.

Recall that affine fragmented functions are strongly affine (by \eqref{eq:prvniinkluze}) and determined by extreme points (by \cite{dostal-spurny}, cf. Theorem~\ref{extdeter} below). So, any $f\in A_f(X)$ may be reconstructed from its restriction to $F$ by integration with respect to suitable probability measures.

Hence, any function $f\in A_f(X)$ is represented by a bounded fragmented function $h$ on $K$ and a pair of bounded sequences $(a_n),(b_n)$ of real numbers, in such a way that
$$f(\lambda\nu+(1-\lambda)\sum_n(s_n u_n+t_n v_n))=\lambda\int h\di\nu+(1-\lambda)\sum_n(s_n a_n+t_n b_n).$$
Not all choices of $h$, $(a_n)$, $(b_n)$ are possible. 
In fact, $f$ is determined by its values on $\{\varepsilon_t\setsep t\in K\}$, i.e., by the function $h$. Indeed, if $\nu\in M_1(K)\subset X$, then
necessarily
$f(\nu)=\int h\di\nu$. 

Further, fix $n\in\en$. Then $\nu_n:=\frac23(\mu_n^-+\varepsilon_{x})\in M_1(K)$ and
$$\tfrac25\left(u_n+\tfrac32\nu_n\right)=\tfrac25(\varepsilon_y+\mu_n^++\varepsilon_{x_n})\in M_1(K).$$
Thus
$$\tfrac25(f(u_n)+\tfrac32f(\nu_n))=f\left(\tfrac25(\varepsilon_y+\mu_n^++\varepsilon_{x_n})\right),$$
so
$$\begin{aligned}f(u_n)&=\tfrac52 f\left(\tfrac25(\varepsilon_y+\mu_n^++\varepsilon_{x_n})\right)-\tfrac32f(\nu_n)\\&
=\tfrac52\left(\tfrac25\left(h(y)+\int h\di\mu_n^++h(x_n)\right)\right)-\tfrac32\left(\tfrac23\left(\int h\di\mu_n^-+h(x)\right)\right)
\\&=h(y)+\int h\di\mu_n^++h(x_n)-\int h\di\mu_n^--h(x) =\int h\di u_n
\end{aligned}$$
Similarly, 
$$f(v_n)=\int h\di v_n.$$

Hence, there is a one-to-one correspondence between $A_f(X)$ and $\Fr^b(K)$. If $h\in \Fr^b(K)$, then the corresponding function from $A_f(X)$ is given by
$$\mu\mapsto \int h\di\mu.$$
Note that $h\ge0$ does not imply that the corresponding function is positive.

\smallskip

\noindent{\tt Step 4:} Set 
$$\widetilde{H}=\left\{ h\in\Fr^b(K)\setsep \mbox{the function }\mu\mapsto \int h\di\mu\mbox{ belongs to }H\right\}.$$
Then  $\widetilde{H}$ is a closed subspace of $\Fr^b(K)$ containing $\Bo_1^b(K)$ and there is a canonical one-to-one correspondence between $\widetilde{H}$ and $H$.

\smallskip

\noindent{\tt Step 5:} If $t\in K\setminus\{y,x,x_n\setsep n\in\en\}$, then $\{\varepsilon_{t}\}$ is a split face of $X$ and, consequently, $\iota(\varepsilon_{t})\in\ext S(H)$.

\smallskip
\iffalse
Set
 $$M=\{\varepsilon_t\setsep t\in K\}\cup \{u_n,v_n\setsep n\in\en\}.$$
As remarked above, $\ext X\subset M$ by Milman's theorem as $M$ is weak$^*$-closed in $X$.

Fix $t\in K\setminus\{y,x,x_n\setsep n\in\en\}$. \fi 
The function 
$$f_t:\sigma\mapsto\sigma(\{t\})$$
is an affine function on $X$ with values in $[0,1]$ such that $f_t(\varepsilon_t)=1$.
%is an affine function of the first Borel class, hence fragmented, i.e., $f_t\in H$. We have $f_t(\varepsilon_t)=1$ and $f_t|_{M\setminus\{\varepsilon_t\}}=0$. 

Further, fix any $\sigma\in X\setminus\{\varepsilon_t\}$. Consider its representation \eqref{eq:X-reprez}. Then
$$f_t(\sigma)=\sigma(\{t\})=\lambda\nu(\{t\})<1.$$ 
Indeed, if $\lambda\nu(\{t\})=1$, then both $\lambda=1$ and $\nu(\{t\})=1$, i.e., $\sigma=\nu=\varepsilon_t$. 

So, $f_t(\varepsilon_{t})=1$ and $f_t(\sigma)<1$ for $\sigma\in X\setminus\{\varepsilon_{t}\}$, hence $\varepsilon_{t}\in \ext X$. 

Let us continue analyzing properties of $\sigma\in X\setminus\{\varepsilon_{t}\}$ represented by \eqref{eq:X-reprez}. 
Set $\alpha=\nu(\{t\})$. Then $\alpha\in[0,1]$ and there is $\widetilde{\nu}\in M_1(K)$ such that $\widetilde{\nu}(\{t\})=0$ and 
$$\nu=\alpha\varepsilon_{t}
+ (1-\alpha) \widetilde{\nu}.$$
It follows that
$$\begin{aligned}
\sigma&=  \lambda\nu+(1-\lambda)\sum_{n=1}^\infty (s_n u_n+t_n v_n)\\
&=\lambda(\alpha\varepsilon_{t}
+ (1-\alpha) \widetilde{\nu})+(1-\lambda)\sum_{n=1}^\infty (s_n u_n+t_n v_n)
\\&=\lambda\alpha \varepsilon_{t}
+(1-\lambda\alpha)\left(\frac{\lambda(1-\alpha)}{1-\lambda\alpha}\widetilde{\nu}+
\frac{1-\lambda}{1-\lambda\alpha}\sum_{n=1}^\infty (s_n u_n+t_n v_n)\right)
\\&=\lambda\alpha \varepsilon_{t}
+(1-\lambda\alpha)\widetilde{\sigma}=f_t(\sigma) \varepsilon_{t}
+(1-f_t(\sigma))\widetilde{\sigma} .
\end{aligned}$$
Note that $\widetilde{\sigma}\in X$ and $f_t(\widetilde{\sigma})=\widetilde{\sigma}(\{t\})=0$. It now easily follows that $\{\varepsilon_{t}\}$ is a split face and the zero level $[f_t=0]$ is the complementary face. 
Since $1_{\varepsilon_t}^*=f_t$ is upper semicontinuous and hence of the first Borel class, we deduce 
$1_{\varepsilon_t}^*\in H$, thus Lemma~\ref{l:lemma-split} yields $\iota(\varepsilon_{t})\in \ext S(H)$.

\smallskip

\noindent{\tt Step 6:} $\{\varepsilon_{x}\}\cup\{\varepsilon_{x_n},u_n,v_n\setsep n\in\en\}\subset\ext X.$

\smallskip

Let us consider functions $f_x$ and $f_{x_n}$ for $n\in\en$ defined in the same way as $f_t$ in Step 5 above. They are affine and, if $\sigma\in X$ is represented by \eqref{eq:X-reprez}, we get
$$\begin{aligned}
f_{x_n}(\sigma)&=\lambda\nu(\{x_n\}) + (1-\lambda)(s_n-t_n),\\
f_x(\sigma)&=\lambda\nu(\{x\}) + (1-\lambda)\left(\sum_n t_n-\sum_n s_n\right).
\end{aligned}$$

Observe that $f_{x_n}$ attains values from $[-1,1]$. 
We get:
\begin{itemize}
    \item $\{v_n\}=[f_{x_n}=-1]$,
    %$f_{x_n}(\sigma)=-1\Longleftrightarrow\sigma=v_n$, 
    hence $v_n\in\ext X$.
    \item  $[f_{x_n}=1]$ is the segment $[\varepsilon_{x_n},u_n]$. So, this segment is a face, thus $u_n,\varepsilon_{x_n}\in\ext X$.
\end{itemize}

Finally, $f_x$ also attains values in $[-1,1]$ and 
$[f_x=1]$ consists of infinite convex combinations of elements $\varepsilon_x,v_n\setsep n\in\en$. Note that $f_{x_n}$ attains values from $[-1,0]$ on  $[f_x=1]$ and 
$$\{\varepsilon_{x}\}=[f_x=1]\cap\bigcap_{n\in\en}[f_{x_n}=0],$$
thus  $\varepsilon_x\in\ext X$.

\smallskip

\noindent{\tt Step 7:} $\{\iota(\varepsilon_{x})\}\cup\{\iota(\varepsilon_{x_n}),\iota(u_n),\iota(v_n)\setsep n\in\en\}\subset\ext S(H).$

\smallskip

Elements $u_n,v_n$ are isolated extreme points, so for them we conclude using Corollary~\ref{c:spoj-ext}$(ii)$.

 Next we fix $n\in\en$ and we are going to show that $\iota(\varepsilon_{x_n})\in \ext S(H)$. So, let $\varphi_1,\varphi_2\in\pi^{-1}(\varepsilon_{x_n})$ be such that $\frac12(\varphi_1+\varphi_2)=\iota(\varepsilon_{x_n})$. If we plug there $f_{x_n}$ (note that $f_{x_n}\in H$, as by \cite[Lemma 3.1]{miryanalznam} it is of the first Borel class), we deduce that
 $$1=f_{x_n}(\varepsilon_{x_n})=\iota(\varepsilon_{x_n})(f_{x_n})=\tfrac12(\varphi_1(f_{x_n})+\varphi_2(f_{x_n})).$$
 Since $f_{x_n}$ has values in $[-1,1]$, we deduce that $\varphi_1(f_{x_n})=\varphi_2(f_{x_n})=1$.
 
 Let us analyze first $\varphi_1$. 
 By Lemma~\ref{L:reprezentace} we may find
a net $(\sigma_i)$ in $X$ converging to $\varepsilon_{x_n}$ such that 
$\iota(\sigma_i)\to \varphi_1$. 
 Assume that $\sigma_i$ is represented as in \eqref{eq:X-reprez} and the coefficients have additional upper index $i$. Then
$$\lambda^i\nu^i(\{x_n\})+(1-\lambda^i)(s_n^i-t_n^i)=\sigma_i(\{x_n\})=f_{x_n}(\sigma_i)=\iota(\sigma_i)(f_{x_n})\to \varphi_1(f_{x_n})=1.$$
Up to passing to a subnet we may assume that $\lambda^i\to\lambda\in[0,1]$. Let us distinguish some cases:

Case 1: $\lambda=1$. Then necessarily $\nu^i(\{x_n\})\to 1$. Thus for any $h\in \widetilde{H}$ and the corresponding $h'\in H$ we get
$$\begin{aligned}
\varphi_1(h')&=\lim_i \int h\di\sigma_i \\
&= \lim_i \left(
\lambda_i\int h\di\nu^i + (1-\lambda_i) \sum_n \left(s_n^i\int h\di u_n+ t_n^i\int h\di v_n\right)\right)\\
&=\lim_i \int h\di\nu^i= \lim_i \left(h(x_n)\nu^i(\{x_n\})+\int_{K\setminus\{x_n\}}h\di\nu^i\right)=h(x_n),
\end{aligned}$$
so $\varphi_1=\iota(x_n)$.

Case 2: $\lambda\in(0,1)$. Then necessarily $\nu^i(\{x_n\})\to1$, $s_n^i\to 1$ and $t_n^i\to0$. 
Thus for any $h\in \widetilde{H}$ and the corresponding $h'\in H$ we get
$$\begin{aligned}
\varphi_1(h')&=\lim_i \int h\di\sigma_i \\
&= \lim_i \left(
\lambda_i\int h\di\nu^i + (1-\lambda_i) \sum_n \left(s_n^i\int h\di u_n+ t_n^i\int h\di v_n\right)\right)\\
&= \lim_i \left(\lambda_i h(x_n)\nu^i(\{x_n\})
+ (1-\lambda_i)s_n^i \int h\di u_n\right)\\
&=
\lambda h(x_n)+(1-\lambda)\int h\di u_n, 
\end{aligned}$$
so $\varphi_1=\iota(\lambda \varepsilon_{x_n}+(1-\lambda)u_n)$. This contradicts the assumption that $\varphi_1\in \pi^{-1}(\varepsilon_{x_n})$.

Case 3: $\lambda=0$. Using a similar computation we get that $\varphi_1=\iota(u_n)$, which again contradicts the assumption.

Summarizing, Case 1 must take place, hence $\varphi_1=\iota(\varepsilon_{x_n})$.
The same works for $\varphi_2$, so $\iota(\varepsilon_{x_n})\in\ext S(H)$.

We continue by showing that also $\iota(\varepsilon_x)\in \ext S(H)$. We proceed similarly as above.
Let $\varphi_1,\varphi_2\in\pi^{-1}(\varepsilon_{x})$ be such that $\frac12(\varphi_1+\varphi_2)=\iota(\varepsilon_{x})$. If we plug there $f_{x}$, we deduce that
 $$1=f_{x}(\varepsilon_{x})=\iota(\varepsilon_{x})(f_{x})=\tfrac12(\varphi_1(f_{x})+\varphi_2(f_{x})).$$
 Since $f_{x}$ has values in $[-1,1]$, we deduce that $\varphi_1(f_{x})=\varphi_2(f_{x})=1$.
 
 Let us analyze first $\varphi_1$. 
 By Lemma~\ref{L:reprezentace} we may find
a net $(\sigma_i)$ in $X$ converging to $\varepsilon_{x}$ such that 
$\iota(\sigma_i)\to \varphi_1$. 
 Assume that $\sigma_i$ is represented as in \eqref{eq:X-reprez} and coefficients have additional upper index $i$. Then
$$\lambda^i\nu^i(\{x\})+(1-\lambda^i)(\sum_n t_n^i-\sum_n s_n^i)\to 1.$$
Up to passing to a subnet we may assume that $\lambda^i\to\lambda\in[0,1]$. Let us distinguish some cases:

Case 1: $\lambda=1$. Then we get, similarly as above, that $\varphi_1=\iota(\varepsilon_x)$.

Case 2: $\lambda\in(0,1)$. Then 
$$\nu^i(\{x\})\to 1, \sum_n t_n^i\to 1, \sum_n s_n^i\to 0.$$
Thus for any $h\in \widetilde{H}$ and the corresponding $h'\in H$ we get
$$\begin{aligned}
\varphi_1(h')&=\lim_i \int h\di\sigma_i \\
&= \lim_i \left(
\lambda_i\int h\di\nu^i + (1-\lambda_i) \sum_n \left(s_n^i\int h\di u_n+ t_n^i\int h\di v_n\right)\right)\\
&= \lim_i \left(\lambda_i h(x)\nu^i(\{x\})
+ (1-\lambda_i)\sum_n t_n^i \int h\di v_n\right)\\&=
\lambda h(x)+(1-\lambda)\lim_i \sum_n t_n^i \int h\di v_n. 
\end{aligned}$$
Set
$$\tau_n^i=\frac{t_n^i}{\sum_k t_k^i}.$$
Since $\sum_k t_k^i\to 1$, this definition has a sense for $i$ large enough. Then
$$\begin{aligned}
\varphi_1(h')&=\lambda h(x)+(1-\lambda)\lim_i \sum_n \tau_n^i \int h\di v_n \\
&= \lim_i \left(\lambda h(x)+(1-\lambda) \sum_n \tau_n^i \int h\di v_n\right),
\end{aligned}$$
hence
$$\iota(\lambda \varepsilon_x+(1-\lambda)\sum_n \tau_n^i v_n)\to \varphi_1 \mbox{ in }S(H).$$
It follows
$$
\lambda\varepsilon_x+(1-\lambda)\sum_n \tau_n^i v_n\to \varepsilon_x\mbox{ in }X.$$
But this is a contradiction as
any cluster point of $(\sum_n \tau_n^i v_n)_i$ must be in $\wscl{\co L}$. So Case 2 cannot take place.

Case 3: $\lambda=0$. We get a contradiction similarly as in Case 2.

Thus $\varphi_1=\iota(\varepsilon_x)$. The same applies to $\varphi_2$, hence $\iota(\varepsilon_x)\in\ext S(H)$.

\smallskip 

\noindent{\tt Step 8:} $\varepsilon_y\in\ext X$.

\smallskip

Consider the function $f_y(\sigma)=\sigma(\{y\})$, $\sigma\in X$. Then $f_y$ is affine and attains values in $[0,1]$. Further, $f_y(\sigma)=1$ if and only if $\sigma\in \wscl{\co L}$ (see the notation in Step 2). Hence $\wscl{\co L}$ is a face. A general element of this set is
$$\sigma=\lambda\varepsilon_y+(1-\lambda)\sum_{n=1}^\infty (s_n u_n+t_n v_n).$$
To prove that $\varepsilon_y$ is an extreme point, it is enough to show
$$\varepsilon_y=\lambda\varepsilon_y+(1-\lambda)\sum_{n=1}^\infty (s_n u_n+t_n v_n)\Rightarrow \lambda=1.$$
So, assume that $\varepsilon_y$ represented as above. Then
$$0=f_{x_n}(\varepsilon_y)=(1-\lambda)(s_n-t_n).$$
Hence, if $\lambda\ne1$, we get $s_n=t_n$.
Since $u_n+v_n=2(\varepsilon_y+\mu_n)$, we deduce that
$$\varepsilon_y=\lambda\varepsilon_y+2(1-\lambda)\sum_{n=1}^\infty s_n(\varepsilon_y+\mu_n)=\varepsilon_y+2(1-\lambda)\sum_n s_n\mu_n,$$
hence
$$\sum_n s_n\mu_n=0.$$
Since $\sum_n s_n=\frac12$ and $(\mu_n)$ is orthogonal in $L^2(\mu)$ (see the construction in Step 1), this is a contradiction.

This completes the proof that $\varepsilon_{y}\in \ext X$.

%So, we have completed the proof that $M=\ext X$.

\smallskip

\noindent{\tt Step 9:} $\iota(\varepsilon_y)\notin\ext S(H)$.

\smallskip

Recall that
$u_n\to \varepsilon_y$ and $v_n\to\varepsilon_y$ in $X$ and $\frac12(u_n+v_n)=\varepsilon_y+\mu_n\to \varepsilon_y$ weakly. In particular, $f(\tfrac12(u_n+v_n))\to f(\varepsilon_{y})$ for each $f\in H$. But $f_x\in H$, $f_x(u_n)=-1$, $f_x(v_n)=1$ and $f_x(\varepsilon_y)=0$. We conclude by Lemma~\ref{L:zachovavani ext}.

\smallskip

\noindent{\tt Step 10:} Summary and conclusion.

\smallskip

If we summarize the results on extreme points, we have proved that 
$$\ext X=\{\varepsilon_{t}\setsep t\in K\}\cup \{u_n,v_n\setsep n\in\en\},$$ 
$\iota(\varepsilon_y)\notin \ext S(H)$ and $\iota(\sigma)\in \ext S(H)$ for $\sigma\in \ext X\setminus\{\varepsilon_y\}$.

We further observe that $\iota(\varepsilon_y)\in m(S(H))$. Indeed, let $\psi$ be any cluster point of $(\iota(x_n))$ in $S(X)$. Then
$$\varphi_1=\iota(\varepsilon_y)+\psi-\iota(\varepsilon_x)\quad\mbox{and}\quad\varphi_2=\iota(\varepsilon_y)+\iota(\varepsilon_x)-\psi$$
belong to $\overline{\ext S(H)}$ and $\iota(\varepsilon_y)=\frac12(\varphi_1+\varphi_2)$. Since there are many different cluster points of $(\iota(x_n))$, we deduce using Lemma~\ref{L:mult-nutne}$(c)$ and Proposition~\ref{P:m(X)} that $\iota(\varepsilon_y)\in m(S(H))$.
In particular, in this case we get $Z(H)=M(H)$.

Finally note that $X$ may be metrizable -- it is enough to start with an uncountable metrizable compact $K$, for example $K=[0,1]$. A non-metrizable example $X$ is obtain if we start, for example, with $K=[0,1]^\Gamma$ for an uncountable $\Gamma$, or, $K=(B_H,w)$, the closed unit ball of a non-separable Hilbert space equipped with the weak topology.
\end{proof}

It seems that the following question remains open.

\begin{ques}
    Let $X$ be a compact convex set and $H$ be one of the spaces $A_1(X)$, $A_b(X)\cap\Bo_1(X)$, $A_f(X)$. Is $M(H)=Z(H)$?
\end{ques}

%%%%%%%%%%%%%%%%%%%%%%%%%%%%%%%
%%%%%%%%%%%%%%%%%%%%%

\section{Spaces determined by extreme points}
\label{s:determined}

In this section we prove two results on determination by extreme points. Let us start by explaining the context. The space $A_c(X)$ is determined by extreme points by the Krein-Milman theorem. More generally, the space $A_{sa}(X)\cap\Ba(X)$ is determined by extreme points because any Baire set containing $\ext X$ carries all maximal measures (see, e.g., the proof of \cite[Lemma 3.3]{smith-london}. Hence, a fortiori, spaces $(A_c(X))^\mu$ and $(A_c(X))^\sigma$ are determined by extreme points. Further, the same argument shows that $A_{sa}(X)$ is determined by extreme points whenever $X$ is a standard compact convex set. However, for a general compact convex set $X$ it is not true, a counterexample is given in \cite{talagrand}. 

Therefore it is natural to ask which subspaces of $A_{sa}(X)$ are determined by extreme points in general. In  \cite[Theorem 3.7]{smith-london}) this was proved for $(A_s(X))^\mu$ and in \cite{dostal-spurny} for $A_f(X)$ (hence for $A_b(X)\cap\Bo_1(X)$). We are going to extend these results to $(A_s(X))^\sigma$ and to $(A_f(X))^\mu$.

Before passing to the results we point out that even though 
determinacy by extreme points seems to be in some relation with strong affinity of functions, there is no implication between these two properties. It is witnessed on one hand by the above-mentioned counterexample from \cite{talagrand} and on the other hand by examples from Section~\ref{sec:strange}.

Now we pass to the result on $(A_s(X))^\sigma$.  It is based on the integral representation of semicontinuous affine functions proved in \cite{teleman}. Let us recall the related notions and constructions.
If $X$ is a compact convex set,  the \emph{Choquet topology}\index{Choquet topology} is the topology $\tau_{\Ch}$ on $\ext X$ such that \gls{tauCh}-closed sets are precisely the sets of the form $F\cap \ext X$, where $F\subset X$ is a closed extremal set.
As explained in \cite[Section 2]{batty-cambr} this topology coincides with the topology $\tau_{\ext}$ from \cite[Chapter 9]{lmns}. It is easy to check that $(\ext X, \tau_{\Ch})$ is a $T_1$ compact (in general non-Hausdorff -- see \cite[Theorem 9.10]{lmns}) topological space. 
The topology $\tau_{\Ch}$ is an important tool to define a canonical correspondence between maximal measures and certain measures on extreme points. It is proved in \cite[Theorem 9.19]{lmns}.

\begin{lemma}\label{L:miry na ext}
Let $X$ be a compact convex set.   Let $\Sigma$ denote the $\sigma$-algebra on $X$ generated by all Baire sets in $X$ and all closed extremal sets in $X$. Let 
$$\gls{Sigma'}=\{A\cap\ext X\setsep A\in\Sigma\}.$$
Then $\Sigma'$
is a $\sigma$-algebra on $\ext X$ and for any maximal measure $\mu\in M_1(X)$ there is a (unique) probability measure \gls{mu'} on the measurable space $(\ext X, \Sigma')$ such that, 
\begin{equation}
 \label{eq:mira1}
\begin{aligned}
\mu(A)&=\mu'(A\cap \ext X),\quad A\in\Sigma,\\
\mu'(A)&=\sup\{\mu'(F)\setsep F\subset A\text{ is }\tau_{\Ch}\text{-closed}\},\quad A\in \Sigma',\\
\mu'(F)&=\inf\{\mu'(B\cap \ext X)\setsep B\cap \ext X\supset F, B\text{ Baire in }X\},\\
&\qquad\qquad F\subset \ext X\text{ is }\tau_{\Ch}\text{-closed}.
 \end{aligned}
 \end{equation}  
\end{lemma}

We continue by a representation theorem for $(A_s(X))^\sigma$.
(We note that the measurability and integral is tacitly considered with respect to the completion of the respective measure.)

\begin{thm}\label{T:Assigma-reprez}
    Let $X$ be a compact convex set. Let $\mu\in M_1(X)$ be a maximal measure and $\mu'$ the measure provided by Lemma~\ref{L:miry na ext}. Then for each $f\in (A_s(X))^\sigma$ the restriction $f|_{\ext X}$ is $\mu'$-measurable
    $$\int_{\ext X} f\di\mu'=f(r(\mu)).$$ 
\end{thm}

\begin{proof}
 If $f$ is semicontinuous, the result follows from \cite[Theorem 1]{teleman}. The general case then follows easily by linearity of integral and Lebesgue dominated convergence theorem.
\end{proof}

The following result on determinacy by extreme points is an easy consequence of the previous theorem.

\begin{thm}
  \label{t:assigma-deter}  
Let $X$ be a compact convex set.
Then the intermediate function space  $(A_s(X))^\sigma$ is determined by extreme points.
\end{thm}

We continue by the promised result for the space $(A_f(X))^\mu$.

\begin{thm}
\label{extdeter}
Let $X$ be a compact convex set. Then the space
$(A_f(X))^\mu$
 is determined by extreme points.
\end{thm}

\begin{proof} 
%The statement for $A_f(X)$ is proved in \cite{dostal-spurny}, so we focus on  $A=(A_f(X))^\mu$. (We note that the case of $A_f(X)$ also follows from the case of $(A_f(X))^\mu$ which is a larger space and we will not use the result of \cite{dostal-spurny}.)  
We will adapt some ideas of the proof of \cite[Theorem 3.7]{smith-london}. To this end we define set $L$ by
$$\begin{aligned}
 L=\{ f:X\to\er\cup\{+\infty\}\setsep& l\mbox{ is concave and lower bounded}, \\ & l\mbox{ is universally measurable},
 \\ & l(r(\mu))\ge \int_X l\di\mu \mbox{ for }\mu\in M_1(X), \\
 & \{x\in F\setsep l|_F \text{ lower semicontinuous at $x$}\} 
 \\&\qquad\mbox{is  residual in $F$ whenever $F\subset X$ is closed}\}.
\end{aligned}$$
\iffalse
of all functions satisfying the following conditions:
\begin{itemize}
\item if $l\in L$, then $l$ has values in $\er\cup\{+\infty\}$,
\item if $l\in L$, then $l$ is concave and lower bounded,
\item if $l\in L$ and $\mu\in M_1(X)$, the integral $\int_X l\di\mu$ exists and $l(r(\mu))\ge \int_X l\di\mu$,
\item if $l\in L$ and $F\subset X$ nonempty closed, then the set
\[
\{x\in F\setsep l|_F \text{ lower semicontinuous at $x$}\}
\] is  residual in $F$.
\end{itemize} \fi
Further, let
\[
U=\{a\in A_b(X)\setsep \forall \ep>0\,\forall x\in X\,\exists\, l,-u\in L: u\le a\le l\ \&\  l(x)-u(x)<\ep \}.
\]
The formula for $U$ is taken from \cite[p. 102]{smith-london}, but in our case the set $L$ is much larger.
We proceed by a series of claims.

\begin{claim}
If $u\in -L$ and $x\in \ext X$  such that $u|_{\ov{\ext X}}$ is upper semicontinuous at $x$, then $u$ is upper semicontinuous at $x$.
\end{claim}

\begin{proof}[Proof of Claim 1]
 This claim is an ultimate generalization of \cite[Lemma 2.4]{rs-fragment} and the proof is quite similar:
The function $u$ is convex and upper bounded, say by a constant $M\ge 0$. Further, for each $\mu\in M_1(X)$ we have $u(r(\mu))\le \mu(u)$.
Let $h\in\er$ satisfying $h>u(x)$ be given. We aim to find a neighbourhood $V$ of $x$ such that  $u(y)< h$ for each $y\in V$. We select $h'\in (u(x),h)$. Next we find using the assumption a neighbourhood $U\subset X$ of $x$ such that $u(t)<h'$ for each $t\in U\cap \ov{\ext X}$. Let $\ep>0$ be such that $h'+\ep M<h$. By \cite[Lemma 2.2]{rs-fragment} there exists a neighbourhood $V$ of $x$ such that for each $t\in  V$ and $\mu\in M_t(X)$ it holds $\mu(U)>1-\ep$.
Then for arbitrary $y\in V$ we pick a measure $\mu_y\in M_y(X)$ with support in $\ov{\ext X}$. Then
\[
u(y)\le \mu_y(u)=\int_{\ov{\ext X}} u\di\mu_y=\int_{\ov{\ext X}\cap U}u\di\mu_y+\int_{\ov{\ext X}\setminus U}u\di\mu_y
\le h'+\ep M<h.
\]
Thus $u$ is upper semicontinuous at $x$.
\end{proof}

\begin{claim}
If $u\in -L$ and $c\in\er$ are such that $u\le c$ on $\ext X$, then $u\le c$ on $X$.
\end{claim}

\begin{proof}[Proof of Claim 2]
This claim is a generalization of \cite[Corollary 2.6]{rs-fragment} and the proof is similar to that of \cite[Theorem 2.5]{rs-fragment}.
Let $u\in -L$ satisfy $u\le c$ on $\ext X$. Assuming that $u(t)>c$ for some $t\in X$, let $\eta\in (c,u(t))$ and let $H=[u\ge \eta]$. Since $u$ is convex, $H$ is semi-extremal subset of $X$ (i.e., $X\setminus H$ is convex), which does not intersect $\ext X$. Let $E=\ov{\co} H$, then by the Milman theorem $\ext E\subset \ov{H}$. Let
\[
C=\{x\in \ov{\ext E}\setsep u|_{\ov{\ext E}}\text{ is upper semicontinuous at }x\}.
\]
Since $u\in -L$, the set $C$ is residual in $\ov{\ext E}$. Let $G\subset C$ be a dense $G_\delta$ set in $\ov{\ext E}$. If $U\subset \ov{\ext E}$ is open and dense in $\ov{\ext E}$, then $U\cap \ext E$ is open and dense in $\ext E$. Since $\ext E$ is a Baire space (see \cite[Theorem I.5.13]{alfsen}), the set $G\cap \ext E$ is nonempty. Thus there exists a point $x\in C\cap \ext E$, which means that $x$ is an extreme point of $E$ such that $u|_{\ov{\ext E}}$ is upper semicontinuous at $x$. By Claim 1, $u|_{E}$ is upper semicontinous at $x$. Since $\ext E\subset \ov{H}$, we obtain that $u(x)\ge \eta$, i.e., $x\in H$. Thus $x\in H\cap \ext E$. In particular, $x$ is not an interior point of any segment with endpoints in $\overline{H}$, i.e., $x\in H\cap \ext \ov{H}$ in the terminology of \cite{rs-fragment}. But this contradicts \cite[Lemma 2.1]{rs-fragment}.
\end{proof}

\begin{claim}
The set $L$ is a convex cone (i.e., $f+g, \alpha f\in L$ whenever $f,g\in L$ and $\alpha\ge 0$) containing $A_f(X)$ and closed with respect to taking pointwise limits of nondecreasing sequences.
\end{claim}

\begin{proof}[Proof of Claim 3]
It is clear that $L$ is a convex cone.
Further, given $f\in A_f(X)$, the function $f$ is strongly affine (cf.\ Section~\ref{ssc:meziprostory}), hence concave, bounded and $\mu(f)=f(r(\mu))$ for each $\mu\in M_1(X)$. If $F$ is nonempty closed set in $X$, $f|_F$ has the point of continuity property (by Theorem~\ref{T:a}), and thus the set $C_f=\{x\in F\setsep f|_F \text{ continuous at }x\}$ is a dense $G_\delta$ set in $F$ due to \cite[Theorem 2.3]{koumou}. Hence the set $L_f=\{x\in F\setsep f|_F \text{ lower semicontinuous at }x\}\supset C_f$ is residual in $F$. Thus, indeed, $f\in L$.

Let now $\{l_n\}$ be a nondecreasing sequence of functions in $L$. Then $l=\lim l_n$ is lower bounded, concave and has values in $\er\cup\{+\infty\}$. If $\mu\in M_1(X)$, $l$ is $\mu$-measurable (since the functions $l_n$, $n\in\en$, are $\mu$-measurable) and due to the Lebesgue monotone convergence theorem we have
\[
\mu(l)=\int_X (\lim_{n\to\infty} l_n)\di\mu=\lim_{n\to \infty} \int_X l_n\di\mu\le \lim_{n\to\infty} l_n(r(\mu))=l(r(\mu)).
\]
Finally, let $F\subset X$ be a nonempty closed set. Then the sets
\[
C_{l_n}=\{x\in F\setsep l_n|_F \text{ lower semicontinuous at }x\},\quad n\in\en
\]
 are residual in $F$. Hence $C=\bigcap_{n=1}^\infty C_{l_n}$ is residual in $F$ and for a point $x\in C$ we have $x\in C_l=\{x\in F\setsep l|_F \text{ lower semicontinuous at }x\}$. Indeed, if $c<l(x)$ is given, let $n\in \en$ satisfy $c<l_n(x)\le l(x)$. Let $U\subset X$ be a neighbourhood of $x$ such that $c<l_n(y)$ for each $y\in U\cap F$. Then for $y\in Y\cap F$ holds
$l(y)\ge l_n(y)>c$, which gives that $l$ is lower semicontinuous at $x$. Thus $C_l$ is residual as well and $l\in L$.
\end{proof}

\begin{claim}
The set $U$ is a vector space containing $A_f(X)$ and closed with respect to taking pointwise limits of bounded monotone sequences.
\end{claim}

\begin{proof}[Proof of Claim 4]
We proceed similarly as in the proof of \cite[Proposition 3.5]{smith-london}:
It is obvious from the definition and Claim 3 that $U$ is a vector space. Further, since $A_f(X)\subset L\cap -L$, we have $A_f(X)\subset U$.

We continue by showing that $U$ is closed with respect to taking pointwise limits of bounded monotone sequences. Since we already know that $U$ is a vector space, it is enough to consider nondecreasing sequences. 
So, let $\{a_n\}$ be a bounded nondecreasing sequence of functions from $U$ with pointwise limit $a\in A_b(X)$. Let $x\in X$ and $\ep>0$ be given. 
Fix $n\in\en$ such that $a(x)-a_n(x)<\frac{\ep}{4}$. Since $a_n\in U$, there exists $u\in -L$ with $u\le a_n$ and $a_n(x)-u(x)<\frac{\ep}{4}$. Then $u\le a_n\le a$ and $a(x)-u(x)=(a(x)-a_n(x))+(a_n(x)-u(x))<\frac{\ep}{2}$.

%For the construction of a suitable $l\in L$,
Further, given $n\in\en$, set $b_n=a_{n+1}-a_n$. Then $b_n\in U$ and thus there exists a function $l_n\in L$ such that
\[
b_n\le l_n\quad\text{and} \quad l_n(x)-b_n(x)<\frac{\ep}{2^{n+2}}.
\]
Then the partial sums $m_k=l_1+\cdots+l_k$, $k\in\en$, are contained in $L$ due to Claim 3 and the sequence $\{m_k\}$ is nondecreasing as $l_n\ge b_n=a_{n+1}-a_n\ge 0$ for each $n\in\en$. Hence $m=\lim_{k\to\infty} m_k$ belongs to $L$ by Claim 3. Moreover,
\[
m_k=l_1+\cdots+l_k\ge (a_2-a_1)+(a_3-a_2)+\cdots+(a_{k+1}-a_k)=a_{k+1}-a_1,\quad k\in\en
\]
and
\[
\begin{aligned}
&m_k(x)-\left(a_{k+1}(x)-a_1(x)\right)=\\
&=\left(l_1(x)+\cdots+l_k(x)\right)-\left((a_2(x)-a_1(x))+\cdots+(a_{k+1}(x)-a_k(x)\right)\\
&=(l_1(x)-b_1(x))+(l_2(x)-b_2(x))+\cdots +(l_k(x)-b_k(x))\\
&<\sum_{n=1}^k \frac{\ep}{2^{n+2}}<\frac{\ep}{4}.
\end{aligned}
\]
For the function $m$ we thus obtain inequalities
\[
m\ge a-a_1\quad\text{and}\quad m(x)-\left(a(x)-a_1(x)\right)\le\frac{\ep}{4}.
\]
Since $a_1\in U$, we can find $n\in L$ such that $a_1\le n$ and $n(x)-a_1(x)<\frac{\ep}{4}$.
Then $l=m+n\in L$ satisfies
\[
l=m+n\ge a-a_1+a_1=a
\]
and
\[
l(x)-a(x)=(m(x)-a(x)+a_1(x))+(n(x)-a_1(x))<\frac{\ep}{4}+\frac{\ep}{4}=\frac{\ep}{2}.
\]
Thus $l(x)-u(x)<\ep$ and $a\in U$.
\end{proof}

\begin{claim}
The space $(A_f(X))^\mu$ is determined by its values on $\ext X$.
\end{claim}

\begin{proof}[Proof of Claim 5]
Since $(A_f(X))^\mu\subset U$ by Claim 4, it is enough to show that $U$ is determined by its values on $\ext X$.
Let $a\in U$ satisfies $\alpha\le a(x)\le \beta$ for each $x\in \ext X$.

Fix any $x\in X$ and $\ep>0$. Let $-u,l\in L$ be such that $u\le a\le l$ and $l(x)-u(x)<\varepsilon$. Since $u\le \beta$ and $l\ge \alpha$ on $\ext X$, using Claim 2 we deduce $l(x)\ge \alpha$ and $u(x)\le \beta$, so
$$\alpha-\ep\le l(x)-\ep <u(x)\le a(x)\le l(x)< u(x)+\varepsilon\le \beta+\ep.$$
Since $\ep>0$ is arbitrary, we conclude $\alpha\le a(x)\le \beta$ which completes the proof.
%Assume that $a(x)<\alpha$ for some $x\in X$. Let $\ep<\alpha-a(x)$. Due to the definition of $U$, we can find $l\in L$ such that $a\le l$ and $l(x)-a(x)<\ep<\alpha$. Then $l\in L$ satisfies $l\ge \alpha$ on $\ext X$ and $l(x)<\alpha$, a contradiction with Claim 2.
\end{proof}
\end{proof}

We continue by an obvious consequence.

\begin{cor}
    Let $X$ be a compact convex set. Then each of the spaces
    $$A_b(X)\cap\Bo_1(X), (A_b(X)\cap\Bo_1(X))^\mu, A_f(X)$$
    is determined by extreme points.
\end{cor}

We note that the case of $A_f(X)$ (and hence also that of $A_b(X)\cap\Bo_1(X)$) is covered already by \cite{dostal-spurny}.

If we compare Theorem~\ref{t:assigma-deter} with Theorem~\ref{extdeter} (and its corollary), the following question seems to be natural.

\begin{ques}\label{q:sigma-deter}
    Let $X$ be a compact convex set. Are the spaces $(A_b(X)\cap \Bo_1(X))^\sigma$ and $(A_f(X))^\sigma$ determined by extreme points?
\end{ques}

In fact, even the following question seems to be open.

\begin{ques}
    Let $X$ be a compact convex set. Is the space $A_{sa}(X)\cap\Bo(X)$ of Borel strongly affine functions determined by extreme points?
\end{ques}

Note that the methods of proving Theorems~\ref{t:assigma-deter} and~\ref{extdeter} are very different. A natural attempt to answer Question~\ref{q:sigma-deter} in the positive would be to try to extend Theorem~\ref{T:Assigma-reprez} to larger classes of functions. There are two possible ways of such an extension. Either we take the statement as it is and we just consider more general functions. I.e., we require that $f$ is $\mu'$-measurable for each maximal probability measure $\mu$ and the formula holds.
It is not hard to find $X$ and $f\in A_b(X)\cap\Bo_1(X)$ such that this condition is not satisfied (see Example~\ref{ex:teleman-ne} below).

The second possibility is to weaken our requirements -- for any $x\in X$ there should be some probability measure $\mu$ defined on some $\sigma$-algebra on $\ext X$ such that each $f$ from the respective class is $\mu$-measurable and satisfies $\mu(f)=f(x)$. Even this weaker condition would be sufficient to derive determinacy by extreme points. However, under continuum hypothesis 
it is not satisfied for $A_b(X)\cap \Bo_1(X)$ in general (see Example~\ref{ex:AMC}$(1)$ below). But it is not clear whether we may
find a counterexample without additional set-theoretical assumptions (cf.\ Example~\ref{ex:AMC}$(2)$).

This is also related to \cite[Question 1' on p. 97]{teleman} asking essentially whether Theorem~\ref{T:Assigma-reprez} holds for strongly affine functions. The negative answer (even for the weaker variant) follows already from \cite{talagrand}, where a strongly affine function not determined by extreme points is constructed.

%%%%%%%%%%%%%%%%%%%%%%%%%
%%%%%%%%%%%%%%%%%%%%%%%%  KOnec sekce 7 %%%%%%%%%%%

%%%%%%%%%%%%%%%%%%%%%%%%%%%%%%%%%%%%%%%%%%%%%%%%%%%%%%%%%%%%%%%%%%%%%%%%%%%%%%%%%%%

\section{Boundary integral representation of spaces of multipliers}\label{s:reprez-abstraktni}

Main results of \cite{edwards,smith-london} include boundary integral representation of centers of spaces $(A_c(X))^\mu$, $(A_c(X))^\sigma$ and $(A_s(X))^\sigma$. If we use the approach of \cite[Theorem 1.2]{smith-pacjm}, this means that centers of these spaces are characterized by a measurability condition on $\ext X$ and, moreover, a mapping assigning to each $x\in X$ a representing measure is constructed. The proofs of these results follow the same pattern. This inspired us to provide a common roof of these results. This is the content of the present section, the promised common roof is Theorem~\ref{T:integral representation H}. We also investigate possible generalizations in Theorems~\ref{T: meritelnost H=H^uparrow cap H^downarrow} and~\ref{t:aha-regularita}.

There is, however, one difference in comparison with the results of \cite{edwards,smith-london}. One of the key ingredients of the quoted results was coincidence of the center and the spaces  of multipliers. In general these two spaces are different, so we focus on spaces of multipliers.

We start by an abstract lemma on function algebras. The lemma is known, but we give a proof for the sake of completeness.

\begin{lemma}\label{L:lattice}
    Let $\Gamma$ be a nonempty set and let $A\subset \ell^\infty(\Gamma)$ be a  linear subspace.
    \begin{enumerate}[$(a)$]
        \item Assume that $A$ is norm-closed subalgebra of $\ell^\infty(\Gamma)$. Then:
        \begin{itemize}
            \item[$\circ$] $A$ is also a sublattice of $\ell^\infty(\Gamma)$.
            \item[$\circ$] If $f\in A$ and $f\ge c$ for some strictly positive $c\in\er$, then $\frac1f\in A$. (In particular, in this case $A$ contains constant functions.)
        \end{itemize}
        \item If $A$ is a subalgebra closed with respect to pointwise limits of bounded monotone sequences, then $A$ is norm-closed and, moreover, closed 
        with respect to limits of bounded pointwise converging sequences.
        \item Assume that $A$ is a norm-closed sublattice of $\ell^\infty(\Gamma)$ containing constant functions. Then it is also a subalgebra of $\ell^\infty(\Gamma)$.
    \end{enumerate}
\end{lemma}

\begin{proof}
    $(a)$:  Assume $A$ is a norm-closed subalgebra. To show it is a sublattice, it is enough to prove that $\abs{f}\in A$ whenever $f\in A$. So, fix $f\in A$. Then $f$ is bounded, so there is some $r>0$ such that the range of $f$ is contained in $[-r,r]$. By the classical Weierstrass theorem there is a sequence $(p_n)$ of polynomials such that $p_n(t)\rightrightarrows \abs{t}$ on $[-r,r]$. Up to replacing $p_n$ by $p_n-p_n(0)$ we may assume that $p_n(0)=0$. Since $A$ is a subalgebra, we deduce $p_n\circ f\in A$ for each $n$. Since $p_n\circ f\rightrightarrows \abs{f}$ on $\Gamma$ and $A$ is closed, we conclude that $\abs{f}\in A$.

  The second assertion is similar. Since $f$ is bounded, there is some $r>c$ such that the range of $f$ is contained in $[c,r]$. Let
  $g:[0,r]\to \er$ be defined by
  $$g(t)=\begin{cases}
      \frac t{c^2} & t\in[0,c],\\ \frac1t & t\in [c,r].
  \end{cases}$$
  Then $g$ is continuous on $[0,r]$ and $g(0)=0$. By the classical Weierstrass theorem we find a sequence $(p_n)$ of polynomials converging to $g$ uniformly on $[0,r]$. 
  Up to replacing $p_n$ by $p_n-p_n(0)$ we may assume that $p_n(0)=0$. Since $A$ is an algebra, we deduce that $p_n\circ f\in A$ and this sequence uniformly converges to $\frac1f$. Thus $\frac1f\in A$. To see the `in particular part', observe that $1=f\cdot\frac1f\in A$ whenever both $f\in A$ and $\frac1f\in A$.

    $(b)$:  $A$ is norm-closed by Lemma~\ref{L:muclosed is closed}.  By $(a)$ we deduce it is a sublattice. It follows that $A$ is closed to pointwise suprema of bounded countable sets. Indeed, if $(a_n)$ is a bounded sequence, the lattice property yields that $b_n=\max\{a_1,\dots,a_n\}\in A$ for each $n$. Since $b_n\nearrow \sup_n a_n$, we deduce that the supremum belongs to $A$. Finally, it follows that $A$ is closed with respect to taking pointwise limsup of bounded sequences, which completes the proof.

    $(c)$: Assume that $A$ is a norm-closed sublattice containing constant functions. To prove it is a subalgebra, it is enough to show that $f^2\in A$ whenever $f\in A$. (Indeed, we have $fg=\frac12((f+g)^2-f^2-g^2)$.)  So, fix $f\in A$. Then $f$ is bounded, so there is some $r>0$ such that the range of $f$ is contained in $[-r,r]$. By a classical result there is a sequence of piecewise linear functions $l_n$ such that $l_n(t)\rightrightarrows t^2$ on $[-r,r]$. The assumptions yield that $l_n\circ f\in A$ for each $n$. Since $l_n\circ f\rightrightarrows f^2$ on $\Gamma$ and $A$ is norm-closed, we conclude that $f^2\in A$.
\end{proof}

\iffalse
We continue by a variant of \cite[Proposition II.7.9]{alfsen}:

\begin{prop}\label{P:algebra ZAc}
    Let $X$ be a compact convex set. The restriction operator
    $R_0:Z(A_c(X))\to C(m(X))$ defined by $R_0(a)=a|_{m(X)}$ for $a\in Z(A_c(X))$ is a linear order-isomorphic isometric inclusion and its range is a closed subalgebra and sublattice of $C(m(X))$.
\end{prop}

\begin{proof}
    It is clear that $R_0$ is a linear order-isomorphic isometric inclusion. So, by Lemma~\ref{L:uzavrenost}(i) its range is a closed subspace of $C(m(X))$. 
    To prove it is stable under pointwise multiplication fix $a_1,a_2\in Z(A_c(X))$. Let $T_1,T_2\in\frd(A_c(X))$ be such that $T_1(1)=a_1$ and $T_2(1)=a_2$. Then $T_1\circ T_2\in\frd(A_c(X))$ as well and for each $x\in m(X)$ we have
    $$(T_1\circ T_2)(1_X)(x)=T_1(T_2(1_X))(x)=T_1(a_2)(x)=T_1(1)(x)\cdot a_2(x)=a_1(x)\cdot a_2(x),$$
    thus $a_1|_{m(X)}\cdot a_2|_{m(X)}$ belongs to the range of $R_0$. So, the range is a closed subalgebra of $C(m(X))$. It remains to use Lemma~\ref{L:lattice}(a) to conclude it is also a sublattice.
    %closed subalgebra of a $C(K)$ space (where $K$ is a compact space) is a sublattice (see part 3 of the proof of \cite[Proposition II.7.9]{alfsen}).  
\end{proof}
\fi

We continue by the following result which is a generalization of \cite[Proposition II.7.9]{alfsen} and of some important steps in the proof of main results of \cite{edwards,smith-london}.

\begin{prop}\label{P:algebra ZH}
    Let $H$ be an intermediate function space determined by extreme points.  Let us define the restriction operator
    $\gls{R}:H\to \ell^\infty(\ext X)$  by $R(a)=a|_{\ext X}$ for $a\in H$. 
    Then the following assertions are valid:
    \begin{enumerate}[$(a)$]
        \item $R$ is a linear order-isomorphic isometric inclusion and its range is a closed linear subspace of $\ell^\infty(\ext X)$ containing constant functions.
        \item $R(H)$ is a sublattice $\Longleftrightarrow$ $R(H)$ is a subalgebra $\Longleftrightarrow$ $M(H)=H$.
       % \item If $R(H)$ is a lattice then so are $R(H^{\mu})$ and $R(H^{\sigma})$.
        \item $R(M(H))$ is a  closed subalgebra and sublattice of $\ell^\infty(\ext X)$. If, moreover, $H^\mu=H$, then  $R(H)$ is closed with respect to pointwise limits of bounded sequences.
        \item All the properties from $(c)$ are fulfilled also by $R(M^s(H))$.
    \end{enumerate}
  \end{prop}

\begin{proof}
$(a)$: It is clear that $R$ is a linear mapping. Since $H$ is determined by extreme points, it is both order-isomorphism and isometry, so its range is a closed linear subspace of $\ell^\infty(\ext X)$. Obviously it contains constant functions.

$(c)$: By Lemma~\ref{L:uzavrenost}$(ii)$ $M(H)$ is closed, so $R(M(H))$ is closed as well (by $(a)$).
Let us continue by showing that $R(M(H))$ is a subalgebra.
Let $a_1,a_2\in R(M(H))$. Let $m_1,m_2\in M(H)$ be such that $R(m_1)=a_1$ and $R(m_2)=a_2$. Since $m_1\in M(H)$, there is $m\in H$ such that $m=m_1m_2$ on $\ext X$. Then clearly $m|_{\ext X}=a_1a_2$. It remains to show that $m\in M(H)$. To this end fix $a\in H$. Since $m_1\in M(H)$, there is $b_1\in H$ such that $b_1=m_1a$ on $\ext X$. Since $m_2\in M(H)$, there is $b_2\in H$ such that $b_2=m_2b_1$ on $\ext X$. Then $b_2=m_2m_1a$ on $\ext X$, hence $b_2=ma$ on $\ext X$. This completes the proof that $m\in M(H)$. 
Thus $R(M(H))$ is a closed subalgebra of $\ell^\infty(\ext X)$,
so by Lemma~\ref{L:lattice}$(a)$ we deduce that it is also a sublattice.

Finally, assume that moreover $H=H^\mu$. Then by Proposition~\ref{p:multi-pro-mu}$(i)$ we get $M(H)=(M(H))^\mu$. It follows that $R(H)$ is closed with respect to pointwise limits of bounded monotone sequences. Indeed, let $(g_n)$ be a non-decreasing sequence in $R(M(H))$ pointwise converging to some $g\in\ell^\infty(\ext X)$. Find $m_n\in M(H)$ with $R(m_n)=g_n$. Since $H$ is determined by extreme points, the sequence $(m_n)$ is non-decreasing and bounded. So $m=\lim_n m_n\in M(H)$. Then $g=R(m)\in R(M(H))$. We conclude by Lemma~\ref{L:lattice}$(b)$.
%Taking into account that the range of $R$ is a sublattice of $\ell^\infty(\ext X)$, it is closed to taking suprema and infima of bounded countable sets. Now it easily follows it is closed to pointwise limits of bounded sequences.

$(d)$: This is completely analogous to $(c)$.

$(b)$: It follows from $(a)$ and Lemma~\ref{L:lattice} that $R(H)$ is a subalgebra if and only if it is a sublattice. If $M(H)=H$, then $R(H)$ is a subalgebra by $(c)$. Finally, if $R(H)$ is a subalgebra, the very definition of multipliers yields that $M(H)=H$. 
\end{proof}

The following lemma is proved in \cite[Lemma 3.5]{edwards}.

\begin{lemma}\label{L:sigmaalgebra}
Let $\Gamma$ be a set and let $A\subset\ell^\infty(\Gamma)$ be a subalgebra containing constant functions and closed with respect to taking limits of pointwise converging bounded sequences. Then
 $$\A=\{E\subset \Gamma\setsep 1_E\in A\}$$
 is a $\sigma$-algebra and
 $$A=\{f\in\ell^\infty(\Gamma)\setsep f^{-1}(U)\in\A\mbox{ for each }U\subset\er\mbox{ open}\},$$
 i.e.,  $A$ is formed exactly by bounded $\A$-measurable functions on $\Gamma$.
\end{lemma}

Now we are ready to formulate the promised common roof of the main results of \cite{edwards,smith-london}.

\begin{thm}\label{T:integral representation H}
    Let $H$ be an intermediate function space determined by extreme points satisfying $H^\mu=H$. Then the following assertions hold.
    \begin{enumerate}[$(i)$]
        \item The systems
        $$\gls{AH}=\{E\subset \ext X\setsep \exists m\in M(H)\colon m|_E=1 \ \&\ m|_{\ext X\setminus E}=0\}\index{sigma-algebra@$\sigma$-algebra!AH@$\A_H$}$$
        and 
         $$\gls{AHs}=\{E\subset \ext X\setsep \exists m\in M^s(H)\colon m|_E=1 \ \&\ m|_{\ext X\setminus E}=0\}\index{sigma-algebra@$\sigma$-algebra!AHs@$\A_H^s$}$$
        are $\sigma$-algebras.
        \item Let $u\in \ell^\infty(\ext X)$. Then $u$ may be extended to an element of $M(H)$ or $M^s(H)$ if and only if $u$ is $\A_H$ or $\A^s_H$-measurable, respectively.
        \item For any $x\in X$ there is a unique probability $\mu=\mu_{H,x}$ on $(\ext X,\A_H)$ such that
        $$\forall m\in M(H)\colon m(x)=\int m\di\mu.$$
    \end{enumerate}
\end{thm}

\begin{proof}
    Let $R$ be the restriction map from Proposition~\ref{P:algebra ZH}. It follow from the quoted proposition that $R(M(H))$ and $R(M^s(H))$ satisfy assumptions of Lemma~\ref{L:sigmaalgebra}. By applying this lemma we obtain assertions $(i)$ and $(ii)$.

    $(iii)$: Fix $x\in X$. Define a functional $\varphi:R(M(H))\to \er$ by
    $$\varphi(g)=R^{-1}(g)(x),\quad g\in R(M(H)).$$
    Since $H$ is determined by extreme points, it is a positive linear functional of norm one. By $(ii)$, $R(M(H))$ coincides with bounded $\A_H$-measurable functions on $\ext X$, hence by \cite[Theorem IV.5.1]{DS1} the functional $\varphi$ is represented by a finitely additive probability measure $\mu$ on $(\ext X,\A_H)$.
    It remains to observe that $\mu$ is $\sigma$-additive. To this end fix $(E_n)$, a disjoint sequence in $\A_H$. Let $m_n\in M(H)$ be such that $m_n|_{E_n}=1$ and $m_n|_{\ext X\setminus E_n}=0$ and set $f_n=m_1+\dots+m_n$. Then $(f_n)$ is a non-decresing sequence in $M(H)$ upper bounded by $1$, so $f=\lim_n f_n\in M(H)$. Observe that $f=1$ on $\bigcup_n E_n$ and $f=0$ on $\ext X\setminus \bigcup_n E_n$. Then
    $$\begin{aligned}
    \mu\left(\bigcup_n E_n\right)&=\varphi(1_{\bigcup_n E_n})=f(x)=\lim_n f_n(x)=\lim_n\varphi(1_{E_1\cup\dots E_n})
    \\&=\lim_n \mu(E_1\cup\dots\cup E_n)
    =\sum_{n=1}^\infty \mu(E_n).\end{aligned}$$ 
    This proves the existence of $\mu$. The uniqueness is obvious, for $E\in\A_H$ we necessarily have
    $$\mu(E)=\int 1_E\di\mu=R^{-1}(1_E)(x).$$
\end{proof} 

\begin{cor}
     Let $H$ be an intermediate function space determined by extreme points satisfying $H^\mu=H$.
     Then $(M(H))^\sigma=M(H)$ and $(M^s(H))^\sigma=M^s(H)$.
\end{cor}

\begin{proof}
    We show the proof for $M(H)$, the case of $M^s(H)$ is basically the same. Let $R$ be the restriction map from Proposition~\ref{P:algebra ZH}. Let $(f_n)$ be a bounded sequence in $M(H)$ pointwise converging to some $f\in A_b(X)$. By Proposition~\ref{P:algebra ZH}$(c)$ then $f|_{\ext X}$ belongs to $R(M(H))$, so there is some $g\in M(H)$ such that $f|_{\ext X}=g|_{\ext X}$. Fix $x\in X$ and let $\mu=\mu_{H,x}$ be the probability measure provided by Theorem~\ref{T:integral representation H}$(iii)$. Then
    $$f(x)=\lim_n f_n(x)=\lim_n \int_{\ext X} f_n\di\mu=\int_{\ext X}f\di\mu=\int_{\ext X}g\di\mu=g(x),$$
    where we used the Lebesgue dominated convergence theorem.
    So, $f=g$ and thus $f\in M(H)$.
\end{proof}

\begin{remark}\label{rem:intrepH} 
(1) Theorem~\ref{T:integral representation H} is an abstract common roof of several results. 
        The case $H=(A_c(X))^\mu$ is addressed in \cite{edwards}. The results are formulated in a different language using the spectrum $\widehat{A_c(X)}$ in place of $\ext X$, but this approach is equivalent due to \cite[Proposition 2.3]{edwards}. Assertions $(i)$ and $(ii)$ of Theorem~\ref{T:integral representation H} then correspond to \cite[Theorem 3.6]{edwards} and assertion $(iii)$ corresponds to \cite[Proposition 4.9]{edwards}.
   This research was continued in \cite{smith-london}. The case $H=(A_c(X))^\sigma$ corresponds to \cite[Theorems 5.2 and 5.3]{smith-london} and the case $H=(A_s(X))^\mu$ corresponds to \cite[Theorems 5.4 and 5.5]{smith-london}. In \cite{smith-london} the terminology from \cite{edwards} is used, in \cite[Theorem 1.2]{smith-pacjm} the case $H=(A_s(X))^\mu$ is reformulated using functions on $\ext X$, which corresponds to our setting.

   Our result is abstract, the only assumption is that $H$ is an intermediate function space determined by extreme points and satisfying $H^\mu=H$. By results of Section~\ref{s:determined} the natural spaces to which Theorem~\ref{T:integral representation H} applies include $(A_s(X))^\sigma$, $(A_f(X))^\mu$ and $(A_b(X)\cap\Bo_1(X))^\mu$. If $X$ is a standard compact convex set (for example if it is metrizable), the theorem applies also to $A_{sa}(X)$.

 (2) Let us compare Theorem~\ref{T:integral representation H} with Theorem~\ref{T:Assigma-reprez}. Although both theorems are devoted to boundary integral representation, they have different meaning and they apply in different situations.

 In Theorem~\ref{T:Assigma-reprez} one canonical $\sigma$-algebra is considered and the representing measures are derived from the respective maximal measures representing in the usual way. And the theorem says that this kind of a canonical representation holds for certain class of affine functions (more precisely, functions from $(A_s(X))^\sigma$). In some cases the same representation holds for a larger class of functions (for example for strongly affine functions if $X$ is metrizable).

 On the other hand, Theorem~\ref{T:integral representation H} provides a $\sigma$-algebra depending on $H$, the representing measures result from an abstract theorem on representation of dual spaces and are not directly related to the measures representing in the usual sense. Further, the representation works only for multipliers. However, there is an additional ingredient (not present in Theorem~\ref{T:Assigma-reprez}) -- a characterization of multipliers (and strong multipliers) using a kind of measurability of the restriction to $\ext X$. This becomes even more interesting if there is a more descriptive characterization of the $\sigma$-algebra. This task is addressed in Sections~\ref{s:meas sm} and~\ref{sec:baire} and also in concrete examples in Section~\ref{sec:stacey}.

    (3) We note that the results in \cite{edwards,smith-london,smith-pacjm} are formulated for $Z(H)$ in place of $M(H)$, but in all three cases equality $Z(H)=M(H)$ holds (see \cite[Propositions 4.4 and 4.9]{smith-london} for proofs for $(A_c(X))^s$ and $(A_s(X))^\mu$. The case of $(A_c(X))^\mu$ may be proved using the method of \cite[Proposition 4.9]{smith-london}.)
    Our results work for a general intermediate function space $H$ which is determined  by extreme points and satisfies $H^\mu=H$, but they are formulated for $M(H)$. If $Z(H)=M(H)$, the results may be obviously formulated for $Z(H)$ in place of $M(H)$. 
    
    If $M(H)\subsetneqq Z(H)$, then $Z(H)$ still admits a structure of an algebra and a lattice since it is isometric to $Z(A_c(S(H)))=M(A_c(S(H)))$. But the respective algebraic and lattice operations are not necessarily pointwise on $\ext X$ and, moreover, it seems not to be clear whether the assumption $H^\mu=H$ would guarantee some kind of monotone completeness of $Z(H)$. 

    (4) It is natural to ask whether in assertion $(iii)$ of Theorem~\ref{T:integral representation H} the converse holds as well. I.e., given a probability $\mu$ on $(\ext X,\A_H)$, is there some $x\in X$ such that $\mu=\mu_{H,x}$? We note that probability measures on $(\ext X,\A_H)$ represent exactly
    `$\sigma$-normal states' on $M(H)$ (cf. \cite[Proposition 4.7]{edwards} for a special case, the general case may be proved in the same way). So, the question is whether the only `$\sigma$-normal states' on $M(H)$ are the evaluation functionals in points of $X$.
    In \cite[paragraph before Proposition 4.11]{edwards} it is claimed that this is not clear for $H=(A_c(X))^\mu$.
\end{remark}

Another easy consequence of Theorem~\ref{T:integral representation H} is the following description of extreme points in spaces of multipliers.

%e further note that for an intermediate function space $H$ which satisfies $H=H^{\mu}$, the set of extreme points of its multipliers is easy to describe.

\begin{cor}
\label{c:ext body multiplikatoru}
Let $X$ be a compact convex set and $H=H^{\mu}$ be an intermediate function space determined by extreme points. Let $E=M(H)$ or $E=M^s(H)$ and let $B^+$ denote the positive part of the unit ball of $E$.
Then
$$\ext B^+=\{m\in E\setsep  m(\ext X) \subset \{0, 1\}\}
\mbox{ and }
B^+=\overline{\co\ext B^+}.$$
%Then the sets of extreme points of the positive part of unit balls of $M(H)$ and $M^s(H)$, are equal to $\{m \in M(H): m(\ext X) \subset \{0, 1\}\}$ and $\{m \in M^s(H): m(\ext X) \subset \{0, 1\}\}$, respectively.
\end{cor}

\begin{proof}
By Proposition \ref{P:algebra ZH} and Theorem \ref{T:integral representation H}, there exist $\sigma$-algebras $\A_H$ and $\A^s_H$ on $\ext X$ such that the spaces $M(H)$ and $M^s(H)$ are canonically isometric to subspaces of $\ell^{\infty}(\ext X)$ consisting of $\A_H$ or $\A^s_H$-measurable functions, respectively. From this  identification the statement easily follows.
\end{proof}

Assumption $H^\mu=H$ is one of the key ingredients of Theorem~\ref{T:integral representation H} -- it enables us to use Lemma~\ref{L:sigmaalgebra} to get assertions $(i)$ and $(ii)$. However, for $H=A_c(X)$ there is an analogue of assertion $(ii)$ in \cite[Theorem II.7.10]{alfsen}. The quoted theorem says (among others) that $u\in \ell^\infty(\ext X)$ may be extended to an element of $M(A_c(X))$ if and only if it is continuous with respect to the facial topology. It is natural to ask whether a similar statement is valid for more general intermediate function spaces. We continue by analyzing this question in the abstract setting. The first step is the following abstract lemma.

\begin{lemma}\label{L:jen algebra}
Let $\Gamma$ be a set and let $A\subset\ell^\infty(\Gamma)$ be a norm-closed subalgebra containing constant functions. 
 Set
    $$\A=\{ [f>0]\setsep f\in A\}.$$
\begin{enumerate}[$(a)$]
    \item
    Then $\A$ is a family of subsets of $\Gamma$ containing $\emptyset$ and $\Gamma$, and closed with respect to finite intersections and countable unions.
    \item Set
        $$B=\{f\in\ell^\infty(\Gamma)\setsep f^{-1}(U)\in\A\mbox{ for each }U\subset\er\mbox{ open}\}.$$
       Then $B$ is a norm-closed subalgebra of $\ell^\infty(\Gamma)$ containing $A$.
     \item The following assertions are equivalent:
     \begin{enumerate}[$(i)$]
         \item $A=B$;
         \item $\frac fg\in A$ whenever $f,g\in A$, $0\le f\le g$ and $g$ does not attain zero;
         \item given $E,F\subset\Gamma$ disjoint such that $\Gamma\setminus E,\Gamma\setminus F\in \A$, there is $f\in A$ with $0\le f\le 1$ such that $f=0$ on $E$ and $f=1$ on $F$.
     \end{enumerate}
\end{enumerate}
\end{lemma}

\begin{proof}
 First observe that by Lemma~\ref{L:lattice} $A$ is necessarily a sublattice of $\ell^\infty(\Gamma)$.
 %Further,
%for $f:\Gamma\to\er$ denote $E_f=\{\gamma\in\Gamma\setsep f(\gamma)>0\}$.

$(a)$: Since $A$ contains constant functions $0$ and $1$, we get $\emptyset\in\A$ and $\Gamma\in\A$.

 Given $f,g\in A$, we get $f^+\cdot g^+\in A$
and 
$$[f>0]\cap [g>0]=[f^+>0]\cap [g^+>0]=[f^+\cdot g^+>0]\in\A,$$
so $\A$ is closed with respect to finite intersections.
 Given a sequence $(f_n)$ in $A$, the functions $g_n=\min\{1,f_n^+\}$ belong to $A$ and hence also $g=\sum_{n=1}^\infty2^{-n}g_n$ belong to $A$. Since 
 $$\bigcup_{n\in\en}[f_n>0]=\bigcup_{n\in\en}[g_n>0]=[g>0]\in\A,$$ we conclude that $\A$ is closed with respect to countable unions. 

$(b)$: Using $(a)$ we deduce that
$$B=\{f\in\ell^\infty(\Gamma)\setsep [f>c]\in\A\mbox{ and }[f<c]\in\A\mbox{ for each }c\in\er\}.$$
Given $f\in A$ and $c\in\er$, then
$$[f>c]=[f-c>0]\in\A\quad\mbox{and}\quad[f<c]=[c-f>0]\in\A,$$
thus $A\subset B$. 

If $f,g\in B$ and $c>0$, then
$$[f+g>c]=\bigcup\{ [f>p]\cap [g>q]\setsep p,q\in\qe, p+q>c\},$$
which belongs to $\A$ by (a). Similarly we see that $[f+g<c]\in \A$. Hence $f+g\in B$. Similarly we may show that $B$ is stable under the multiplication, so it is an algebra. 

To see that $B$ is uniformly closed, fix a sequence $(f_n)$ in $B$ uniformly converging to $f\in\ell^\infty(\Gamma)$. Up to passing to a subsequence we may assume that $\norm{f_n-f}_\infty<\frac1n$. Observe that
$$[f>c]=\bigcup_{n\in\en}[f_n>c+\tfrac1n]\in \A,$$
and similarly $[f<c]\in\A$, thus $f\in B$.

$(c)$: Let us prove the individual implications:

$(i)\Rightarrow(ii)$: Let $f,g$ be as in $(ii)$. Then $\frac fg$ is well defined and bounded. Moreover, we easily see that $\frac fg\in B$, thus assuming $(i)$ we deduce $\frac fg\in A$.

$(ii)\Rightarrow(iii)$: Let $E,F$ be as in $(iii)$. Then there are $f,g\in A$ such that $E=[f\le 0]$ and $F=[g\le0]$. Since $A$ is a sublattice, up to replacing $f,g$ by $f^+$ and $g^+$, we may assume that $f,g$ are non-negative.
Then $E=[f=0]$ and $F=[g=0]$. Since $E,F$ are disjoint, $f+g$ is a strictly positive function from $A$. By $(ii)$ we get that $\frac{f}{f+g}\in A$ and this function has the required properties.

$(iii)\Rightarrow(i)$: Fix $f\in B$ and $\varepsilon>0$. Since $f$ is bounded, there are $t_1,\dots,t_n\in\er$ such that
$$f(\Gamma)\subset\bigcup_{j=1}^n (t_j-\tfrac\varepsilon2,t_j+\tfrac\ep2).$$
For each $j$ set
$$E_j=[t_j-\tfrac\ep2\le f\le t_j+\tfrac\ep2]\quad\mbox{and}\quad F_j=[f\le t_j-\ep]\cup[f\ge t_j+\ep].$$ 
By $(iii)$ there is $g_j\in A$ such that $0\le g_j\le 1$, $g_j|_{F_j}=0$ and $g_j|_{E_j}=1$. Since 
 $E_1\cup\dots\cup E_n=\Gamma$, we get $g_1+\dots+g_n\ge 1$. By Lemma~\ref{L:lattice}$(a)$ we deduce that $\frac{1}{g_1+\dots+g_n}\in A$. Set $h_j=\frac{g_j}{g_1+\dots +g_n}$. Then $h_j\in A$, thus
 $$g=f(t_1)h_1+\dots+f(t_n)h_n\in A$$
 and $\norm{g-f}\le\varepsilon$.
Since $\varepsilon>0$ is arbitrary and $A$ is closed, we deduce that $f\in A$.
\end{proof}

Assertion $(c)$ of the above lemma says that it is not automatic that a closed subalgebra of $\ell^\infty(\Gamma)$ containing constants may be described using a kind of measurability. In some cases it is true -- for functions continuous with respect to some topology or for Baire-one functions with respect to a Tychonoff topology. In some cases it fails -- for example for the algebra $c$ of converging sequences considered as a subalgebra of $\ell^\infty$.

Let us now discuss how Lemma~\ref{L:jen algebra} may be applied to intermediate functions spaces.
Such an analysis is contained in the following proposition.

\begin{prop} \label{p:system-aha}
Let $H$ be an intermediate function space determined by extreme points. Then the following assertions hold.
\begin{enumerate}[$(a)$]
    \item Set
    $$\gls{AH}=\{[m>0]\cap\ext X\setsep m\in M(H)\}. \index{system of sets!AH@$\A_H$}$$
    The family $\A_H$ contains the empty set, the whole set $\ext X$ and is closed with respect to finite intersections and countable unions.
    Moreover, $m|_{\ext X}$ is $\A_H$-measurable for each $m\in M(H)$.
    \item The following assertions are equivalent:
    \begin{enumerate}[$(i)$]
        \item A bounded function on $\ext X$ can be extended to an element of $M(H)$ if and only if it is $\A_H$-measurable.
        \item If $u,v\in M(H)$ satisfy $0\le u\le v$ and $v$ does not attain $0$ on $\ext X$, then there is $m\in M(H)$ such that $m=\frac uv$ on $\ext X$.
        \item If $u,v\in M(H)$ are such that $u\ge0$, $v\ge0$ and the sets $E=[u=0]\cap \ext X$ and $F=[v=0]\cap \ext X$ are disjoint, there is $m\in M(H)$ such that $0\le m\le 1$, $m=0$ on $E$ and $m=1$ on $F$.
    \end{enumerate}
    \item $\A_H=\{E \subset \ext X\setsep \exists m\in (M(H))^\uparrow \colon 0\le m\le 1\ \&\ m|_E=1 \ \&\ m|_{\ext X\setminus E}=0\}.$
    \item For each $E\in\A_H$ there are (not necessarily closed) faces $F_1,F_2$ of $X$ such that
    $E=F_1\cap \ext X$ and $\ext X\setminus E= F_2\cap \ext X$.
    \item If $H^\mu=H$, then $\A_H$ coincide with the $\sigma$-algebra from Theorem~\ref{T:integral representation H}.
\end{enumerate}
The same statements are valid when $M(H)$ is replaced by $M^s(H)$ and $\A_H$ is replaced by 
    $$\gls{AHs}=\{[m>0]\cap \ext X\setsep m\in M^s(H)\}. \index{system of sets!AHs@$\A_H^s$}$$
    \end{prop}

\begin{proof}
 Let $R$ be the restriction mapping from Proposition~\ref{P:algebra ZH}. Then $R(M(H))$ is (by the quoted proposition) a closed subalgebra and sublattice of $\ell^\infty(\ext X)$. It easily follows that 
$$\A_H= \{[m>0]\cap \ext X\setsep m\in M(H), m\ge0\}.$$
 
 Asertions $(a)$ and $(b)$ thus follow from Lemma~\ref{L:jen algebra}. 

$(c)$: Assume $E\in A_H$. Using the definitions and the lattice property of $R(M(H))$ we deduce that there is
  $f\in M(H)$ such that $0\le f\le 1$ and $E=[f<1]\cap \ext X$. 
 Given $n\in\en$, by Proposition~\ref{P:algebra ZH}$(c)$ there is $f_n\in M(H)$ be such that $f_n=1-f^n$ on $\ext X$. The sequence $(f_n)$ is bounded and non-decreasing, let $m$ denote its limit. Then $m\in(M(H))^\uparrow$, $0\le m\le1$, $m|_E=1$ and $m|_{\ext X\setminus E}=0$.
 This completes the proof of inclusion `$\subset$'.

 To prove the converse, assume that $E\subset \ext X$, $m\in M(H)^\uparrow$, $m|_E=1$ and $m|_{\ext X\setminus E}=0$. Let $(f_n)$
 in $M(H)$ with $f_n\nearrow m$. We observe that
 $$E=\bigcup_n [f_n>0]\cap\ext X,$$
 so $E\in\A_H$ by $(a)$.

 $(d)$: Let $E\in \A_H$. Let $m\in (M(H))^\uparrow$ be provided by $(c)$. Then clearly $[m=0]$ and $[m=1]$ are faces.

  $(e)$: If additionally $H^\mu=H$, then $(M(H))^\uparrow=M(H)$ (by Proposition~\ref{p:multi-pro-mu}$(i)$), so the assertion easily follows.

It is simple to check that the proof works also in the case when $M(H)$ is replaced by $M^s(H)$ and $\A_H$ is replaced by $\A^s_H$. 
\end{proof}

In case an intermediate function space $H$ satisfies that $H^\mu$ is determined  by $\ext X$,
by Theorem~\ref{T:integral representation H} we know that for each $x\in X$ there is a unique measure $\mu_{H^\mu,x}$ defined on $(\ext X,\A_{H^\mu})$ such that $m(x)=\int m\di\mu_{H^\mu,x}$ for each $m\in M(H^\mu)$. Since $M(H)\subset (M(H))^\mu\subset M(H^\mu)$ by Proposition~\ref{p:multi-pro-mu}$(i)$, we can use this integral representation for elements from $M(H)$. The next result shows that the measure $\mu_{H^\mu,x}$ is ``regular'' on the $\sigma$-algebra generated by $\A_H$.
%with respect to the family $\A_H$ made of ``cozero'' sets.

\begin{thm}
    \label{t:aha-regularita}
 Let $H$ be an intermediate function space on a compact convex set $X$ such that $H^\mu$ is determined by extreme points. For $x\in X$, let $\mu=\mu_{H^\mu,x}$ be the measure on $(\ext X,\A_{H^\mu})$ given by Theorem~\ref{T:integral representation H}$(iii)$.
 Then the following assertions hold.
 \begin{enumerate}[$(i)$]
\item  The $\sigma$-algebra $\sigma(\A_H)$ in $\ext X$ generated by $\A_H$ is contained in $\A_{H^\mu}$.
\item  For any $A\in \sigma(\A_H)$ and $\ep>0$ there exist sets $F\in (\A_H)^c$ and $G\in\A_H$ such that $F\subset A\subset G$ and $\mu(G\setminus F)<\ep$.
\end{enumerate}
\end{thm}

\begin{proof}
$(i)$: Since $M(H)\subset M(H^\mu)$ by Proposition~\ref{p:multi-pro-mu}$(i)$, we obtain inclusion $\A_H\subset \A_{H^\mu}$. This proves  inclusion $\sigma(\A_H)\subset \A_{H^\mu}$, because $\A_{H^\mu}$ is a $\sigma$-algebra by Theorem~\ref{T:integral representation H}$(i)$.

$(ii)$: We set
\[
\E=\{A\in \sigma(\A_H)\setsep \forall \ep>0\ \exists F\in (\A_{H})^c, G\in \A_H\colon F\subset A\subset G \And \mu(G\setminus F)<\ep\}.
\]
We claim that $\E$ is a Dynkin class of sets (see \cite[136A Lemma]{fremlin1}) containing $(\A_H)^c$. 

Indeed, $\E$ contains $\emptyset$, it is obviously closed with respect to complements (in $\ext X$). Finally, $\E$ is closed with respect to taking countable unions of disjoint families of sets. To see this, take a disjoint sequence $(A_n)$ in $\E$ and let $\ep>0$. We denote $A=\bigcup_{n} A_n$.  There exists $k\in\en$ such that $\mu(A)-\ep\le \mu(\bigcup_{n=1}^k A_n)$ We select $F_i\in (\A_H)^c, G_i\in \A_H$ such that $F_i\subset A_i\subset G_i$ and $\mu(G_i\setminus F_i)<2^{-i}\ep$, $i\in\en$. Then $F=F_1\cup \cdots \cup F_k\in (\A_H)^c$, $G=\bigcup_{i} G_i\in \A_H$, $F\subset A\subset G$ and 
\[
\begin{aligned}
\mu(G\setminus F)&\le \mu(G\setminus A)+\mu(A\setminus F)\le \sum_{i\in\en} \mu(G_i\setminus F_i)+\mu(A\setminus \bigcup_{i=1}^k A_i)+\sum_{i=1}^k\mu(A_i\setminus F_i)\\
&\le \ep+\ep+\ep=3\ep.
\end{aligned}
\]

To verify that $(\A_H)^c\subset \E$, let $E\in(\A_H)^c$ and $\ep>0$ be given. By Proposition~\ref{p:system-aha}$(c)$ there exists $m\in (M(H))^\downarrow$ such that $m|_E=1$ and $m|_{\ext X\setminus E}=0$. Let $(m_j)$ be a sequence in $M(H)$ such that $m_j\searrow m$. Then 
the sets $E_j=\ext X\cap [m_j> \frac12]$, $j\in\en$, belong to $\A_H$, they form a nonincreasing sequence and
$E=\bigcap_{j\in\en} E_j$.
Hence $\mu(E)=\lim_{j\to\infty} \mu(E_j)$. Thus we can select a set $E_j$ such that $\mu(E_j\setminus E)<\ep$. Hence $\E\supset (\A_H)^c$.

By \cite[136B Theorem]{fremlin1}, $\E$ contains the $\sigma$-algebra $\sigma(\A_H)$ generated by $\A_H$. Hence $\E=\sigma(\A_H)$, which concludes the proof.
\end{proof}

Next we provide a sufficient condition ensuring that measurability with respect to systems $\A_H$ and $\A^s_H$ characterizes multipliers on $H$.

%show that the assumption that an intermediate function space satisfies $H=H^{\uparrow}\cap H^{\downarrow}$ ensures that the systems $\A_H$ and $\A^s_H$ determine the measurability of multipliers on $H$.

\begin{thm}
 \label{T: meritelnost H=H^uparrow cap H^downarrow}   
 Let $H$ be an intermediate function space satisfying that $H=H^{\uparrow}\cap H^{\downarrow}$ and such that $H^{\uparrow}$ is determined by extreme points. Let $\A_H$, $\A^s_H$ be the families of sets from Proposition \ref{p:system-aha}. Then:
 \begin{enumerate}[$(i)$]
     \item $M(H)=M(H)^{\uparrow} \cap M(H)^{\downarrow}$ and $M^s(H)=M^s(H)^{\uparrow} \cap M^s(H)^{\downarrow}$.
     \item A bounded function on $\ext X$ can be extended to an element of $M(H)$ if and only if it is $\A_H$-measurable.
     \item A bounded function on $\ext X$ can be extended to an element of $M^s(H)$ if and only if it is $\A^s_H$-measurable.
 \end{enumerate}
\end{thm}

\begin{proof}
$(i)$: Clearly, $M(H) \subset M(H)^{\uparrow} \cap M(H)^{\downarrow}$. To prove the reverse inclusion, we assume that we are given a function $m$ such that there are sequences $(a_n), (b_n)$ from $M(H)$ with $a_n \nearrow m$ and $b_n \searrow m$. Then $m \in H^\uparrow\cap H^\downarrow=H$. To show that $m \in M(H)$, let $h \in H$ be given. We may assume that $h \geq 0$, otherwise we would add a suitable constant to it. Then there exist sequences of functions $(f_n), (g_n)$ from $H$ such that $f_n=a_n \cdot h$ and $g_n=b_n \cdot h$ on $\ext X$ (for $n\in\en$). Since $H$ is determined by extreme points, the sequence $(f_n)$ is non-decreasing and the sequence $(g_n)$ is non-increasing on $X$. Let $f$ and $g$ denote the pointwise limits of sequences $(f_n)$ and $(g_n)$, respectively. Then $f\in H^\uparrow$, $g\in H^\downarrow$  and $f=g=m \cdot h$ on $\ext X$. Since $f-g\in H^\uparrow$ and $f-g=0$ on $\ext X$, we deduce that $f=g$ on $X$. Thus $f=g \in H^{\uparrow}\cap H^{\downarrow}=H$, which proves that $m \in M(H)$. The proof that $M^s(H)=M^s(H)^{\uparrow} \cap M^s(H)^{\downarrow}$ is basically the same, except that we use the monotone convergence theorem.

For the proof of $(ii)$ we will verify condition $(ii)$ from Proposition \ref{p:system-aha}$(b)$. Assume $u,v\in M(H)$ such that $0\le u\le v$ and $v$ does not attain $0$ on $\ext X$. Let $R$ be the restriction map from Proposition~\ref{P:algebra ZH}. It follows from Proposition~\ref{P:algebra ZH}$(c)$ and Lemma~\ref{L:lattice}$(a)$ we deduce that $\frac{R(u)}{R(v)+\frac1n},\frac{R(u)+\frac1n}{R(v)+\frac1n}\in R(M(H))$. We observe that
$$\frac{R(u)}{R(v)+\frac1n}\nearrow \frac{R(u)}{R(v)}\quad\mbox{and}\quad\frac{R(u)+\frac1n}{R(v)+\frac1n}\searrow\frac{R(u)}{R(v)}.$$
Using $(i)$ we now see that $\frac{R(u)}{R(v)}\in R(M(H))$ which completes the argument.

Assertion $(iii)$ can be proven similarly as assertion $(ii)$. The proof is finished.
\end{proof}

We continue by describing a rather general situation and several concrete cases where the previous theorem may be applied.

\begin{lemma}\label{L:YupcapYdown=Y}
    Let $\Gamma$ be a set and let $\B$ be a family of subsets of $\Gamma$ containing the empty set and the whole set $\Gamma$ and closed with respect to taking countable unions and finite intersections. Let $Y$ denote the space of all bounded $\B$-measurable functions on $\Gamma$. Then
    $$Y=Y^\uparrow\cap Y^\downarrow=\overline{Y^\uparrow}\cap \overline{Y^\downarrow}.$$
\end{lemma}

\begin{proof}
    It is clear that $Y\subset Y^\uparrow\cap Y^\downarrow\subset\overline{Y^\uparrow}\cap \overline{Y^\downarrow}$.
        To prove the converse observe that 
    $[f>c]\in \B$ whenever $f\in Y^\uparrow$ and $c\in\er$. Moreover, this property is preserved by uniform limits (see the proof of Lemma~\ref{L:jen algebra}$(b)$). Hence 
     $[f>c]\in \B$ whenever $f\in \overline{Y^\uparrow}$ and $c\in\er$.
     Similarly we get that $[f<c]\in \B$ whenever $f\in \overline{Y^\downarrow}$ and $c\in\er$.

     Therefore, if $f\in \overline{Y^\uparrow}\cap \overline{Y^\downarrow}$ and $c\in\er$, we deduce $[f>c]\in\B$ and $[f<c]\in\B$. Using the properties of $\B$ we conclude that $f$ is $\B$-measurable, hence $f\in Y$.
\end{proof}

\begin{cor}\label{cor:hup+hdown pro A1aj}
    Let $X$ be a compact convex set and $H$ be one of the spaces
    $$A_1(X), A_b(X)\cap\Bo_1(X), A_f(X).$$
    Then $H=H^\uparrow\cap H^\downarrow=\overline{H^\uparrow}\cap\overline{H^\downarrow}$. 

    In particular, Theorem~\ref{T: meritelnost H=H^uparrow cap H^downarrow} applies to these spaces.
\end{cor}

\begin{proof}
 The named spaces are formed by affine functions which are $\Zer_\sigma$-measurable, $\FpG_\sigma$-measurable or $\H_\sigma$-measurable, respectively (see Section~\ref{ssc:csp}). Thus the statement follows immediately from Lemma~\ref{L:YupcapYdown=Y}.
\end{proof}

\begin{remarks}\label{rem:o meritelnosti}
(1) The key results of Theorem~\ref{T:integral representation H} and Theorem~\ref{T: meritelnost H=H^uparrow cap H^downarrow} include a characterization of multipliers (and strong multipliers) by a measurability condition on $\ext X$. If $H^\mu=H$, this is a measurability condtition with respect to a $\sigma$-algebra. This $\sigma$-algebra is canonical and unique (because $1_E$ is measurable if and only if $E$ belongs to the $\sigma$-algebra). However, the canonical description from Theorem~\ref{T:integral representation H}$(i)$ is not very descriptive. In the sequel we try to give a better description for some spaces $H$. A characterization for strongly affine Baire functions in given in Theorem~\ref{T:baire-multipliers}$(b)$ below, in some concrete examples in Section~\ref{sec:stacey}.

(2) For spaces not closed to monotone limits the situation is more complicated. Firstly, there are some cases when Theorem~\ref{T: meritelnost H=H^uparrow cap H^downarrow} may be applied (see Corollary~\ref{cor:hup+hdown pro A1aj})
and some cases when the characterization fails (see Proposition~\ref{P:dikous-lsc--new}$(f)$). Further, if multipliers are characterized by a measurability condition, the respective family of sets need not be uniquely determined. In such a case family $\A_H$ is the smallest one, but perhaps some bigger one works as well (cf.\ Proposition~\ref{P:dikous-spoj-new}, assertions $(b)$ and $(c)$). Therefore, in some cases we provide a more descriptive measurability condition characterizing multipliers but without claiming it is a description of family $\A_H$ (see Theorem~\ref{t:a1-lindelof-h-hranice}, Theorem~\ref{t:Bo1-fsigma-hranice}  and Theorem~\ref{t:af-fsigma-hranice} below).
\end{remarks}

%%%%%%%%%%%%%%%%%%%%%%%%%%%%%%%%%%%%%%%%%%%%%%%%%%%%%%%%%%%%%%%%%%%%%%
\section{Measurability of strong multipliers in terms of split faces}\label{s:meas sm}

From the previous section we already know that, given an intermediate function space $H$ with certain properties, multipliers and strong multipliers may be characterized by measurability with respect to the system $\A_H$ or $\A_H^s$, respectively, on $\ext X$ (see Theorem~\ref{T:integral representation H} and Theorem~\ref{T: meritelnost H=H^uparrow cap H^downarrow}). However, the systems $\A_H$ and $\A_H^s$ are described using $M(H)$ and $M^s(H)$, not just using $H$. This implies some limitations of possible applications of these characterizations.
In this section we try to partially overcome this issue for strong multipliers. We do not know how to do a similar thing for multipliers themselves. This was, in fact, one of our main motivations to introduce the concept of a strong multiplier.  The first step is the following improvement of Proposition~\ref{p:system-aha}$(c)$ for strong multipliers.

\begin{thm}\label{T:meritelnost-strongmulti}
Let $H \subset A_{sa}(X)$ be an intermediate function space such that $H^{\uparrow}$ is determined by extreme points. Then
\[\A^s_H=\{F \cap \ext X: F \text{ is a split face such that } \lambda_F \in M^s(H)^{\uparrow}\}.\]
\end{thm}

%We recall that in the case when $H=H^{\uparrow} \cap H^{\downarrow}$, we know by Theorem \ref{T: meritelnost H=H^uparrow cap H^downarrow} that strong multipliers are determined by their $\A^s_H$-measurability. Thus in this case, we obtain a characterization of strong multipliers by their measurability with respect to a certain family of split faces. When $H=H^{\mu}$, then $(\A^s_H)^c=\{ F \cap \ext X: F \in \F_H\}$, thus in the case we moreover have a better description of these split faces.  

Before we prove the result we need some definitions and a couple of lemmas. We start by the following stronger notions of convexity and extremality.

\begin{definition}
A universally measurable set $F\subset X$ is called \emph{measure convex}\index{set!measure convex} provided $r(\mu)\in F$ whenever $\mu\in M_1(X)$ is carried by $F$ (i.e., $\mu(F)=1$). The set $F$ is called \emph{measure extremal}\index{set!measure extremal} provided $\mu(F)=1$ whenever $\mu\in M_1(X)$ and $r(\mu)\in F$.
\end{definition}

\begin{remarks}
(1) It is clear that any measure convex set is convex and any measure extremal set is extremal.

(2) Any closed, open or resolvable convex set is also measure convex, see \cite[Proposition~2.80]{lmns}. Similarly, a
closed, open or resolvable extremal set is also measure extremal, see \cite[Proposition~2.92]{lmns}.

(3) There exist an $F_\sigma$ face $F$ and $G_\delta$ face $G$ in $X=M_1([0,1])$, which are not measure convex (see \cite[Propositions~2.95 and 2.96]{lmns}).
\end{remarks}

We continue by a lemma on a complementary pair of faces.

\begin{lemma}
\label{l:complementarni-facy}
Let $A, B$ be two disjoint faces of a compact convex set $X$ which are both measure convex and measure extremal, and such that $A \cup B$ carries each maximal measure on $X$. Then the following assertions hold.
\begin{enumerate}[$(i)$]
    \item $X=\co(A\cup B)$.
    \item $A^{\prime}=B$.
    \item If $f, h$ are strongly affine functions such that $f=h \cdot 1_A$ $\mu$-almost everywhere for each maximal measure $\mu \in M_1(X)$, then $f=h \cdot 1_A$ on $A \cup B$. 
\end{enumerate}
\end{lemma}

\begin{proof}
$(i)$: Let $x\in X$ be given. We pick a maximal measure $\mu\in M_x(X)$. Then $\mu(A\cup B)=1$ by the assumption.
 If $\mu(A)=1$, then $x \in A$ since $A$ is measure convex. Similarly $x\in B$ provided $\mu(B)=1$. Thus we may assume that $\lambda=\mu(A)\in (0,1)$. Then $\mu_A=\lambda^{-1}\mu|_A$ and $\mu_B=(1-\lambda)^{-1}\mu|_B$ are probability measures on $X$. By the measure convexity we see that $x_A=r(\mu_A)\in A$ and $x_B=r(\mu_B)\in B$. Since the barycentric mapping is affine, we have
\[
x=r(\mu)=r(\lambda\mu_A+(1-\lambda)\mu_B)=\lambda x_A+(1-\lambda)x_B\in\co(A\cup B).
\]

$(ii)$: Let $C\subset X$ be a face disjoint from $A$ and $x\in C$ be given. By $(i)$ we can write $x=\lambda x_A+(1-\lambda) x_B$ for some $\lambda\in[0,1]$ and $x_A\in A$, $x_B\in B$. If $\lambda=1$, then $x=x_A \in A$, which is impossible as $C\cap A=\emptyset$. If $\lambda \in (0,1)$, then $x_A, x_B\in C$ because $C$ is a face. Again we have a contradiction as $x_A\in C\cap A$.
Hence the only possibility is $\lambda=0$, which means $x=x_B\in B$. Thus $A'\subset B$.
Since obviously $B\subset A'$, the proof of $(ii)$ is finished.

$(iii)$: Let $f,h$ be as in the assumptions. Take an arbitrary $x \in A$. Fix a maximal measure $\mu \in M_1(X)$ with $x=r(\mu)$. Since $A$ is measure extremal, we get $\mu(A)=1$, and it follows that 
$$h(x)1_A(x)=h(x)=\int_{A} h(y) \di\mu(y)=\int_A f(y)\di\mu(y)=f(x),$$ hence $f(x)=h(x)1_A(x)$. The case $x \in B$ is treated similarly. The proof is finished.
\end{proof}

\iffalse
\begin{lemma}
\label{l:neseni max. mirami}
Let $H$ be an intermediate function space such that $H \subset A_{sa}(X)$ and $H$ is determined by extreme points. Let $m \in M^s(H)$ be such that $m(\ext X) \subset \{0, 1\}$. Then the set $[m=1] \cup [m=0]$ carry each maximal measure on $X$.
\end{lemma}

\begin{proof}
We first note that since $H$ consists of strongly affine functions, the functions from $H$ are in particular $\mu$-measurable for each maximal measure $\mu$ on $X$. Thus, 
since $m \in M^s(H)$, there exists a function $a \in H$ such that the set $[a=m^2]$ carries each maximal measure on $X$. Since $m=m^2=a$ on $\ext X$ and $H$ is determined by extreme points, $m=a$ on $X$. Thus the set $[m=1] \cup [m=0]$ carries each maximal measure on $X$, since it is equal to the set $[m=m^2]$.
\end{proof}
\fi

The next lemma shows that monotone limits of multipliers share some properties of multipliers.

\begin{lemma}\label{L:mult-uparrow}
    Let $H$ be an intermediate function space such that $H^\uparrow$ is determined by extreme points. Let $m\in (M(H))^\uparrow$ and $a\in H$ be a non-negative function. Then the following assertions hold.
    \begin{enumerate}[$(a)$]
        \item There is a unique $b\in H^\uparrow$ such that $b=ma$ on $\ext X$.
        \item If $m\in (M^s(H))^\uparrow$, then  $b=ma$ $\mu$-almost everywhere for every maximal $\mu\in M_1(X)$.
    \end{enumerate}
\end{lemma}

\begin{proof} Fix a sequence $(m_n)$ in $M(H)$ with $m_n\nearrow m$. 
Let $b_n\in H$ be such that $b_n=m_n a$ on $\ext X$. Then $b_n\nearrow ma$ on $\ext X$. Since $H$ is determined by extreme points, the sequence $(b_n)$ is non-decreasing on $X$, hence $b_n\nearrow b$ for some $b\in H^\uparrow$. Clearly $b=ma$ on $\ext X$. Since $H^\uparrow$ is determined by extreme points, such $b$ is unique. This completes the proof of $(a)$.

If $m\in (M^s(H))^\uparrow$, we may assume that $m_n\in M^s(H)$. Let $\mu$ be any maximal measure on $X$. Then $b_n=m_n a$ $\mu$-almost everywhere for each $n\in\en$, hence $b=ma$ $\mu$-almost everywhere.
This completes the proof of $(b)$.
\end{proof}

\begin{lemma}
\label{l:multi-face}
Let $H$ be an intermediate function space such that $H^\uparrow$ is determined by extreme points. Let $m \in (M^s(H))^\uparrow$ be such that $m(\ext X) \subset \{0, 1\}$. Then the following assertions hold:
\begin{enumerate}[$(a)$]
    \item The set $[m=0]\cup[m=1]$ carries all maximal measures.
    \item If $H\subset A_{sa}(X)$, then  $F=[m=1]$ is a split face with $\lambda_F=m$. Moreover,   $F^{\prime}=[m=0]$ and both $F, F^{\prime}$ are measure convex and measure extremal.
\end{enumerate}
\end{lemma}

\begin{proof}
$(a)$:   Let $(m_n)$ be a sequence in $M^s(H)$ such that $m_n\nearrow m$. By Proposition~\ref{P:algebra ZH}$(d)$ there is a sequence $(f_n)$ in $M^s(H)$ such that
    \begin{itemize}
        \item $f_1=\max\{0,m_1\}$ on $\ext X$;
         \item $f_{n+1}=\max\{m_{n+1},f_n\}$ on $\ext X$ for $n\in\en$.
    \end{itemize}
   Then $0\le f_n\le 1$, the sequence $(f_n)$ is non-decreasing and $f_n\nearrow m$ on $\ext X$. Since $H^\uparrow$ is determined by extreme points, we deduce that $f_n\nearrow m$ on $X$.

   Let $g_n\in H$ be such that $g_n=f_n^2$ on $\ext X$. Then $(g_n)$ is a non-increasing sequence in $H$, let $g$ denote its limit.
     
   Let $\mu\in M_1(X)$ be a maximal measure. Since $f_n\in M^s(H)$, we get
   $$\forall n\in\en\colon g_n=f_n^2\quad\mu\mbox{-almost everywhere}.$$
   Hence there is a set $N\subset X$ with $\mu(N)=0$ such that
   $$\forall x\in X\setminus N\,\forall n\in\en\colon g_n(x)=(f_n(x))^2.$$
   Passing to the limit we get
   $$\forall x\in X\setminus N\colon g(x)=(m(x))^2,$$
   hence $g=m^2$ $\mu$-almost everywhere. 

   Applying to $\mu=\ep_x$ for $x\in\ext X$, we deduce $g=m^2$ on $\ext X$. Since $m^2=m$ on $\ext X$ and $H^\uparrow$ is determined by extreme points, we deduce that $g=m$. 
   
   Hence, for any maximal measure $\mu$ we have $m=m^2$ $\mu$-almost everywhere, i.e., $m(x)\in\{0,1\}$ for $\mu$-almost all $x\in X$ which completes the proof.

$(b)$:  By $(a)$ we know that $[m=0]\cup[m=1]$ carries all maximal measures. Since $m$ is a strongly affine function and $m(X) \subset [0, 1]$, it follows using elementary methods that both sets $F=[m=1]$ and $[m=0]$ are measure convex and measure extremal faces. Therefore, by Lemma~\ref{l:complementarni-facy} we get $F'=[m=0]$ and  $X=\co(F\cup F')$. %It remains to show that the set $F=[m=1]$ is a split face. To this end it is enough to verify the uniqueness of convex combinations.

Further, let $x \in X\setminus (F\cup F')$ be such that 
\[
x=\lambda_1x_1+(1-\lambda_1)y_1=\lambda_2 x_2+(1-\lambda_2)y_2
\]
for some $x_1,x_2\in F$, $y_1,y_2\in F'$ and $\lambda_1,\lambda_2\in [0,1]$. An application of $m$ yields $\lambda_1=\lambda_2=m(x)\in(0,1)$.  
This already shows that $F$ is a parallel face and $\lambda_F=m$.
Let $\lambda=m(x)$ stand for the common value.

%We may assume that $\lambda\in (0,1)$.  
If $x_1\neq x_2$, let $h\in A_c(X)$ be a positive function satisfying $h(x_1)<h(x_2)$. Since $h \in H$ and $m \in (M^s(H))^\uparrow$, by Lemma~\ref{L:mult-uparrow} there exists a function $a \in H^\uparrow$ such that for each maximal measure $\mu \in M_1(X)$, $a=h \cdot m$ $\mu$-almost everywhere. Thus $a=h \cdot m$ on $F \cup F^{\prime}$ by Lemma~\ref{l:complementarni-facy}$(iii)$.
Consequently,
\[
\lambda h(x_1)=\lambda a(x_1)=\lambda a(x_1)+(1-\lambda) a(y_1)=a(x)=\lambda a(x_2)+(1-\lambda) a(y_2)=\lambda h(x_2).
\]
Hence $h(x_1)=h(x_2)$, a contradiction completing the proof.
\end{proof}

\begin{proof}[Proof of Theorem~\ref{T:meritelnost-strongmulti}]
Let $F$ be a split face with $\lambda_F\in (M^s(H))^\uparrow$. Then $\lambda_F|_F=1$ and $\lambda_F|_{F'}=0$. Since $F\cup F'\supset\ext X$, Proposition~\ref{p:system-aha}$(c)$ (applied to strong multipliers) shows that $F\cap\ext X\in\A_H^S$.

Conversely, assume $E\in\A_H^s$. By Proposition~\ref{p:system-aha}$(c)$ (applied to strong multipliers) we find $m\in (M^s(H))^\uparrow$ such that $m|_E=1$ and $m|_{\ext X\setminus E}=0$.
By Lemma~\ref{l:multi-face} we deduce that $F=[m=1]$ is a split face and $m=\lambda_F$. This completes the proof.
\end{proof}

\begin{remark} The proof of Theorem~\ref{T:meritelnost-strongmulti} illustrates
     the crucial difference between strong multipliers and ordinary multipliers. To be more precise, an~analogue of Lemma~\ref{l:multi-face} for multipliers does not hold. It may happen that $H$ is determined by extreme points, $H\subset A_{sa}(H)$ and $H=H^\mu$, but there is $m\in M(H)$ such that $m(\ext X)\subset\{0,1\}$, but $[m=1]$ is not a split face. An example illustrating it is described in Example~\ref{ex:dikous-mezi-new}.
\end{remark}

We further note that an important role is played by the condition that $F\cup F'$ carries all maximal measures. We do not know whether this condition is automatically satisfied for nice split faces.

\begin{ques}
    Let $X$ be a compact convex set and let $F\subset X$ be a split face such that both $F$ and $F'$ are measure convex and measure extremal. Does the set $F\cup F'$ carry all maximal measures?
\end{ques}

The answer is positive if $X$ is a simplex:

\begin{obs}
    Let $X$ be a simplex and let $A,B\subset X$ be two convex measure extremal sets such that $X=\co(A\cup B)$. Then $A\cup B$ carries all maximal measures.
\end{obs}

\begin{proof}
    Let $\mu\in M_1(X)$ be maximal. Let $x=r(\mu)$ be its barycenter. 
    Then $x=\lambda a+(1-\lambda)b$ for some $\lambda\in [0,1]$, $a\in A$ and $b\in B$. Let $\mu_a$ and $\mu_b$ be maximal measures representing $a$ and $b$, respectively. By measure extremality we deduce that $\mu_a$ is carried by $A$ and $\mu_b$ is carried by $B$. Thus $\lambda\mu_a+(1-\lambda)\mu_b$ is a maximal measure carried by $A\cup B$ with barycenter $x$. By simpliciality this measure coincide with $\mu$, thus $\mu$ is carried by $A\cup B$. 
\end{proof}

Theorem~\ref{T:meritelnost-strongmulti} provides a better characterization of the system $\A_H^s$ than Proposition~\ref{p:system-aha}. However, it still uses $M^s(H)$ and not just $H$. But it can be used, at least for some spaces $H$, to provide a characterization of strong multipliers just in terms of $H$. We recall that we know from Theorem \ref{T: meritelnost H=H^uparrow cap H^downarrow} that, under some assumptions on $H$, a bounded function on $\ext X$ can be extended to an element of $M^s(H)$ if and only if it is $\A^s_H$-measurable. It is easy to see that $\A^s_H$ is the smallest system with such property. If moreover $H=H^{\mu}$, then $\A^s_H$ is a $\sigma$-algebra and hence, it is the  unique $\sigma$-algebra with this property. However, when $H$ is not equal to $H^{\mu}$, there might be some systems larger than $\A^s_H$ which characterize strong multipliers. This motivates the following notion.

\begin{definition}
For a system $\A$ of subsets of $\ext X$ and an intermediate function space $H$, we say that $M^s(H)$ \emph{is determined by} $\A$\index{system of sets!determining strong multipliers} if a bounded function on $\ext X$ can be extended to an element of $M^s(H)$ if and only if it is $\A$-measurable. 
\end{definition}

Apart from the above-defined system $\A^s_H$, there are  another canonical systems of subsets of extreme points that we may consider. 

\begin{definition}
\label{d:es-ha}
For an intermediate function space $H$, let
\[\begin{aligned}  
\gls{SH}&=\{F\cap\ext X\setsep F\mbox{ is split face with }\lambda_F\in H^\uparrow\},
\\ \gls{ZH}&=\{[f=1]\cap\ext X\setsep f\in H^\uparrow, f(\ext X)\subset\{0,1\}\}.\end{aligned}
\]\index{system of sets!SH@$\ms_H$}\index{system of sets!ZH@$\Z_H$}
\end{definition}

If  $H \subset A_{sa}(X)$ is an intermediate function space such that $H^{\uparrow}$ is determined by extreme points,  by Theorem~\ref{T:meritelnost-strongmulti} we know that $\A^s_H \subset \ms_H$.
It appears that in some cases the equality holds (see Theorem~\ref{T:baire-multipliers}, Proposition~\ref{P:As pro simplex}, Theorem~\ref{t:metriz-sa-splitfacy} and Proposition~\ref{P:shrnutidikousu} below) or at least
$M^s(H)$ is determined by $\ms_H$ (see Theorem~\ref{t:a1-lindelof-h-hranice}, Theorem~\ref{t:Bo1-fsigma-hranice} and Theorem~\ref{t:af-fsigma-hranice} below). We note that unlike  the systems $\A^s_H$, the systems $\ms_H$ clearly reflect the inclusions between intermediate functions spaces. This is applied in the following proposition.

\begin{prop}
\label{P: silnemulti-inkluze}
Let $H_1, H_2$ be intermediate function spaces on a compact convex set $X$ that $H_1\subset H_2\subset A_{sa}(X)$. Assume that $H_2^\uparrow$ is determined by extreme points and $M^s(H_2)$ is determined by $\ms_{H_2}$. Then $M^s(H_1) \subset M^s(H_2)$.
\end{prop}

\begin{proof}
Let $m \in M^s(H_1)$. Then $m|_{\ext X}$ is $\A^s_{H_1}$-measurable by Proposition \ref{p:system-aha}. By Theorem~\ref{T:meritelnost-strongmulti} we get $\A^s_{H_1} \subset \ms_{H_1}$ and thus $m|_{\ext X}$ is $\ms_{H_1}$-measurable. Since clearly $\ms_{H_1} \subset \ms_{H_2}$, $m|_{\ext X}$ is $\ms_{H_2}$-measurable. By the assumption on $H_2$ there exists $f \in M^s(H_2)$ such that $f|_{\ext X}=m|_{\ext X}$. Since  $H_2$ is determined by extreme points, $m=f \in M^s(H_2)$.
\end{proof}

Next observe that $\ms_H \subset \Z_{H}$ for any intermediate function space $H$ on $X$. The inclusion may be proper in case $X$ is not a simplex (for example if $X$ is a square in the plane). However,
there are important cases when the equality holds:

\begin{lemma}\label{L:SH=ZH}
    Let $X$ be a simplex and let $H\subset A_{sa}(X)$ be an intermediate function space. Assume moreover that at least one of the following conditions is satisfied: 
    \begin{itemize}
        \item $H\subset \Ba(X)$;
        \item $X$ is a standard compact convex set.
    \end{itemize}
    Then $\ms_H=\Z_H$.
\end{lemma}

\begin{proof}
    Let $E\in\Z_H$. Fix $f\in H^\uparrow$ such that $f(\ext X)\subset \{0,1\}$ and $E=[f=1]\cap \ext X$. Assuming one of the conditions, we deduce that the set $[f=1]\cup[f=0]$ carries all maximal measures. Set $F=[f=1]$.
    By Lemma~\ref{l:complementarni-facy} we deduce that  $F'=[f=0]$ and $X=\co(F\cup F')$.

    Next we proceed similarly as in the proof of Lemma~\ref{l:multi-face}. Let $x \in X\setminus (F\cup F')$ be such that 
\[
x=\lambda_1x_1+(1-\lambda_1)y_1=\lambda_2 x_2+(1-\lambda_2)y_2
\]
for some $x_1,x_2\in F$, $y_1,y_2\in F'$ and $\lambda_1,\lambda_2\in [0,1]$. Let $\mu_1,\mu_2,\nu_1,\nu_2\in M_1(X)$ be maximal measures representing $x_1,x_2,y_1,y_2$, respectively. Since $F$ and $F'$ are clearly measure extremal, $\mu_1,\mu_2$ are carried by $F$ and $\nu_1,\nu_2$ are carried by $F'$. Then 
$$\lambda_1\mu_1+(1-\lambda_1)\nu_1,\lambda_2\mu_2+(1-\lambda_2)\nu_2$$ are two maximal measures representing $x$, therefore they are equal. Since $F$ and $F'$ are two disjoint universally measurable sets, we deduce that
$$\lambda_1\mu_1=\lambda_2\mu_2 \quad\mbox{and}\quad (1-\lambda_1)\nu_1=(1-\lambda_2)\nu_2.$$
Thus $\lambda_1=\lambda_2\in (0,1)$ and $\mu_1=\mu_2$ and $\nu_1=\nu_2$. Thus $x_1=x_2$. It follows that $F$ is a split face and $\lambda_F=f\in H^\uparrow$. Thus $E\in\ms_H$.
\end{proof}

We observe that the assumption that $X$ is a standard compact convex space is important as Proposition~\ref{P:shrnutidikousu}$(c)$ below shows.

We proceed to the final result of this section which provides a characterization of (strong) multipliers for a special type of intermediate function spaces defined via a topological property on $\ext X$. This will be applied later in Theorem~\ref{t:a1-lindelof-h-hranice} and in Section~\ref{ssce:fsigma-hranice}.

To formulate it we need to introduce some natural notation concerning split faces. We start by an easy lemma.

\begin{lemma}
\label{l:extense-splitface}
Let $X$ be a compact convex set and let $F\subset X$ be a split face with the complementary face $F'$. Then for each affine function $a$ on $F$ and $a'$ on $F'$ there exists a unique affine function $b:X\to\er$ satisfying $b=a$ on $F$ and $b=a'$ on $F'$.

Moreover, $b$ is bounded whenever $a$ and $a'$ are bounded.
\end{lemma}

\begin{proof}
The uniqueness is obvious as $X=\co(F\cup F')$. To prove the existence, given $x\in X$ we set
\[
b(x)=\lambda_F(x) a(y)+(1-\lambda_F(X)) a'(y'),
\]
where $y\in F$ and $y'\in F'$ satisfy $x=\lambda_F(x)y+(1-\lambda_F(x))y'$.
By a routine verification, the function $b$ is affine.
Further, if $a$ and $a'$ are bounded, $b$ is obviously bounded as well.
\end{proof}

The case $a'=0$ in the previous lemma is especially important. It inspires the following definition.

\begin{definition}\label{d:teckoacko}
If $F$ is a split face in $X$ and $a$ is an affine function defined at least on $F$, we denote by $\gls{Upsilon} a$\index{operator Upsilon@operator $\Upsilon$} the unique affine function $b$ satisfying $b=a$ on $F$ and $b=0$ on $F'$.
We write $\Upsilon_F$ instead of $\Upsilon$ in case we need to stress it is related to $F$.
\end{definition}

Next we collect basic properties of the extension operator. Their proofs are completely straightforward.

\begin{obs}\label{obs:operator rozsireni}
    Let $X$ be a compact convex set and let $F\subset X$ be a split face.
     \begin{enumerate}[$(i)$]
        \item $\Upsilon$ is a positive linear operator from the linear space of affine functions on $F$ to the linear space of affine functions on $X$.
        \item If $(a_n)$ is a sequence of affine functions on $F$ pointwise converging to an affine function $a$, then $\Upsilon a_n\to\Upsilon a$ pointwise on $X$.
        \item $\Upsilon$ maps $A_b(F)$, the space of bounded affine functions on $F$, isometrically into $A_b(X)$.
        \item $\Upsilon 1_F=\lambda_F$.
    \end{enumerate}
\end{obs}

The operator $\Upsilon$ is closely related to multipliers. More precisely, if $f\in A_b(X)$ is arbitrary, then $\Upsilon f\in A_b(X)$ and $\Upsilon (f|_F)=\lambda_F\cdot f$ on $\ext X$. This will be used in the following result. In the next section we will investigate in more detail when this feature produces a real multiplier in our sense.

\begin{prop}
\label{P:meritelnost multiplikatoru pomoci topologickych split facu}
Let $X$ be a standard compact convex set and let $T\subset\ell^\infty(\ext X)$ be a linear subspace with the following  properties:
\begin{enumerate}[$(i)$]
    \item $af \in T^{\uparrow}$ for each $a \in T$ positive and $f \in T^{\uparrow}$;
    \item $\overline{T^{\downarrow}} \cap \overline{T^{\uparrow}}=T$;
    \item $f|_{\ext X}\in T$ for each $f\in A_c(X)$.
\end{enumerate}
Set
$$H=\{f\in A_{sa}(X)\setsep f|_{\ext X}\in T\}.$$
Then $H$ is an intermediate function space and $M^s(H)=M(H)$ is determined by the system
$$\begin{aligned}
\B^s_{H}= \{ F\cap\ext X\setsep& F \text{ is a  split face such that } 1_{F \cap \ext X} \in T^{\uparrow} \\& \text{and }
\Upsilon_F(a|_F) \in A_{sa}(X)\text{ for every }a\in H \}.
\end{aligned}$$
%\begin{equation*}
%\begin{aligned}
%\B^s_{H}= \{ F\cap\ext X\setsep& F \text{ is a measure convex split face such that } F \cup F^{\prime} \\& \text{carries each maximal measure}, 1_{F \cap \ext X} \in T^{\uparrow} \\& \text{and }
%\Upsilon_F(a|_F) \in A_{sa}(X)\text{ for every }a\in H \}.
%\end{aligned}
%\end{equation*}
\end{prop}

\begin{proof} It is clear that $H$ is a linear subspace of $A_{sa}(X)$ containing $A_c(X)$. Moreover, by the assumptions we know that $A_{sa}(X)$ is determined by extreme points, so it follows from condition $(ii)$ that $H$ is closed. It is thus an intermediate function space. By Proposition~\ref{P:rovnostmulti} we get $M(H)=M^s(H)$.

By Proposition~\ref{p:system-aha}$(a)$ any $m\in M^s(H)$ is $\A_H^s$-measurable. To prove it is also $\B_H^s$-measurable it is enough to observe that $\A_H^s\subset\B_H^s$. So, assume $E\in \A_H^s$. By Theorem~\ref{T:meritelnost-strongmulti} we find a split face $F\subset X$ such that $\lambda_F\in (M^s(H))^\uparrow$ and $E=F\cap\ext X$.  Since $\lambda_F|_{\ext x}=1_{F\cap\ext X}$, this function belongs to $T^\uparrow$. Finally, let $a\in H$. There is $c\in\er$ such that $b=a+c\cdot1_X\ge 0$. By Lemma~\ref{L:mult-uparrow} there is $f\in H^\uparrow$ such that $f=\lambda_F\cdot b$ on $\ext X$.  Then $g=f-c\cdot 1_X\in A_{sa}(X)$ and $g=\lambda_F\cdot a$ on $\ext X$. Since $X$ is standard, the equality holds $\mu$-almost everywhere for each maximal $\mu\in M_1(X)$. Further, since $\lambda_F$ is strongly affine,  $F$ and $F'$ are measure convex and measure extremal. By Lemma~\ref{l:multi-face}$(a)$ we deduce that $F\cup F'$ carries all maximal measures.  By Lemma~\ref{l:complementarni-facy}$(iii)$ we conclude that $g=\lambda_F\cdot a$ on $F\cup F'$, i.e., $g=\Upsilon(a|_F)$.

\iffalse
By Proposition \ref{P: jiny popis A_H}, 
$$(\A^s_H)^c=\{E \subset \ext X\setsep \exists m\in (M^s(H))^\downarrow \colon m|_E=1 \ \&\ m|_{\ext X\setminus E}=0\}. $$
Thus for each $m \in (\A^s_H)^c$, the set $[m=1] \cap \ext X$ belongs to $1_{F \cap \ext X} \in T^{\downarrow}$. Further, since $H$ consists of strongly affine functions, so does the space $H^{\mu}$, thus from Theorem \ref{T:meritelnost-strongmulti} combined with the above we obtain that $\A^s_{H} \subset \B^s_{H}$. hence we get that each $M^s(H)$ is $\B^s_H$-measurable, since it is $\A^s_H$-measurable.
\fi

The converse implication will be proved in several steps:

{\tt Step 1:} Let $f\in\ell^\infty(\ext X)$ be  $\B^s_H$-measurable and $a\in H$, $a\ge 0$. Then there is $h\in A_{sa}(X)$ such that
\begin{itemize}
    \item $h=af$ on $\ext X$;
    \item $h|_{\ext X}\in\overline{T^\uparrow}$.
\end{itemize}

Let us first observe that we may assume without loss of generality that $f\ge0$. Indeed, if $f$ is general bounded function, then there is some $\lambda>0$ such that $f_1=f+\lambda \cdot 1_{\ext X}\ge 0$. If $h_1$ with the required properties corresponds to $f_1$, the function $h=h_1-\lambda a$ corresponds to $f$. (Note that $a|_{\ext X}\in T$.) 

So, assume $f\ge0$.
 Without loss of generality we assume that $0 \leq f \leq 1$.
For each $n\in \en$ we set
\[
C_{n,i}= [f> \tfrac{i-1}{2^n}],\quad i\in\{0,\dots, 2^n\}.
\]
By the assumption we find split faces $F_{n,i}$ from the system which defines $\B^s_{\H}$ with $C_{n,i}=F_{n,i} \cap \ext X$. Set
\[
a_n=2^{-n}\sum_{i=1}^{2^n} \Upsilon_{F_{n,i}}(a|_{F_{n,i}}).
\]
Then $a_n \in A_{sa}(X)$. Further, since $a|_{\ext X} \in T$ and $1_{F_{n,i} \cap \ext X} \in T^{\uparrow}$, we get $a_n|_{\ext X} \in T^{\uparrow}$ (by condition $(i)$). Moreover, it is simple to check that
\[
a(x) \cdot f(x)\le a_n(x)\le a(x) \cdot (f(x)+2^{-n})\mbox{ for }x \in \ext X
\]
 (cf.\ the proof of \cite[Theorem II.7.2]{alfsen}). 
Hence $(a_n)$ is a uniformly convergent sequence on $X$ (as $A_{sa}(X)$ is determined by extreme points). Denote the limit by $h$. Then clearly $h \in A_{sa}(X)$, $h=af$ on $\ext X$ and  $h|_{\ext X} \in \overline{T^{\uparrow}}$.

\smallskip

{\tt Step 2:} Let $f\in\ell^\infty(\ext X)$ be  $\B^s_H$-measurable and $a\in H$. Then there is  $h\in A_{sa}(X)$ such that
\begin{itemize}
    \item $h=af$ on $\ext X$;
    \item $h|_{\ext X}\in\overline{T^\uparrow}$.
\end{itemize}

Find $\lambda\ge0$ such that $b=a+\lambda\cdot 1_X\ge 0$. 
Observe that $-f$ is also $\B^s_H$-measurable.
We apply Step 1 to the pairs $f,b$ and $-f,\lambda\cdot 1_X$
and obtain strongly affine functions $h_1,h_2$ such that $h_1=f b$ and $h_2=-\lambda f$ on $\ext X$, and $h_i|_{\ext X}\in \overline{T^{\uparrow}}$ (for $i=1,2$). Then $h=h_1+h_2$ is strongly affine, $h|_{\ext X} \in \overline{T^{\uparrow}}$ and $h=fa$ for each $x \in \ext X$. 

\smallskip

{\tt Step 3:} Let $f\in\ell^\infty(\ext X)$ be  $\B^s_H$-measurable and $a\in H$. Then there is  $h\in H$ such that $h=af$ on $\ext X$.

Recall that  $-f$ is also $\B^s_H$-measurable. We apply Step 2 to the pairs $f,a$ and $-f,a$. We get strongly affine functions $h_1,h_2$ such that $h_1=fa$ and $h_2=-fa$ on $\ext X$ and $h_i|_{\ext X}\in \overline{T^{\uparrow}}$ (for $i=1,2$). Then $h_1=-h_2$ on $\ext X$, so $h_1=-h_2$ (as $A_{sa}(X)$ is determined by extreme points). Then  $h_1|_{\ext X}\in \overline{T^{\downarrow}} \cap \overline{T^{\uparrow}}=T$ (by condition $(ii)$). Thus $h \in H$.

\smallskip

{\tt Step 4:} Let  $f\in\ell^\infty(\ext X)$ be  $\B^s_H$-measurable. Apply Step 3 to the pair $f,1_X$. The resulting function $h$ is a multiplier and $h|_{\ext X}=f$. Since $M(H)=M^s(H)$, this finishes the proof.
\end{proof}

%%%%%%%%%%%%%%%%%%%%%%%ov{\ext X}%%%%%%%%%%%%%%%%%%%%%%%%%%%%%%%%%%%

%%%%%%%%%%%%%%%%%%%%%%%%%%%%%%%%%%%%%%%%%%%%%%%%%%%%%%%%%%
\section{Extending affine functions from split faces}
\label{sec:splifaces}

In this section we collect results on split faces needed in the sequel
with focus on extending affine mappings. More precisely, we investigate in more detail properties of the extension operator $\Upsilon$ from Definition~\ref{d:teckoacko}.
 Let us briefly recall some notation from Section~\ref{ssc:ccs}. Let $X$ be a compact convex set, $F\subset X$ a split face and $F'$ the complementary face. By $\lambda_F$ we denote the unique function $\lambda_F\colon X\to [0,1]$ such that for each $x\in X$ there are $y\in F$ and $y'\in F'$ such that
\begin{equation}\label{eq:split}
    x=\lambda_F(x)y+(1-\lambda_F(x))y'.\end{equation}
We recall that $\lambda_F$ is affine, $F=[\lambda_F=1]$, $F'=[\lambda_F=0]$ and, moreover, if $x\in X\setminus(F\cup F')$, the points $y$ and $y'$ are uniquely determined.

\subsection{More on measure convex split faces}

In this section we study in more detail measure convex split faces with measure convex complementary face.
We start by a lemma providing an automatic partial strong affinity.

\begin{lemma}\label{L:mc-partial-sa}
    Let $X$ be a compact convex set and let $A\subset X$ be a measure convex split face such that $A'$ is also measure convex.
    Let $a$ be a strongly affine function on $A$. Then $\Upsilon a$ is a bounded affine function on $X$ which satisfies the barycentric formula for any probability measure on $X$ carried by $A\cup A'$.
\end{lemma}

\begin{proof}
The strong affinity of $a$ on $A$ is defined in the natural sense (it is reasonable as $A$ is measure convex).  Further, a strongly affine function on a measure convex set is bounded as the proof of \cite[Satz 2.1.(c)]{krause} shows. 
Therefore $\Upsilon a$ is a bounded affine function on $X$ by Lemma~\ref{l:extense-splitface}.

Let $\mu\in M_1(X)$ be carried by $A\cup A'$. Then $\mu=t\mu_1+(1-t)\mu_2$ for some $t\in [0,1]$ and probability measures $\mu_1,\mu_2$ carried by $A,A'$, respectively. By measure convexity we get $r(\mu_1)\in A$ and $r(\mu_2)\in A'$. Thus
$$  \int \Upsilon(a)\di\mu=\int a\di (t\mu_1)=ta(r(\mu_1))$$
as $a$ is strongly affine on $A$. Further,
$$r(\mu)=tr(\mu_1)+(1-t)r(\mu_2).$$
As $A$ is a split face, $t=\lambda_A(r(\mu))$ and $r(\mu_1)=y$ (in the notation from \eqref{eq:split}). Hence,
$$\Upsilon(a)(r(\mu))=ta(r(\mu_1)).$$
We conclude by comparing the formulas.
\end{proof}

We continue by a result characterizes split faces among measure convex faces with measure convex complementary set. It is a generalization of \cite[Theorem II.6.12]{alfsen} where the case of closed faces is addressed. It has also similar flavour as condition $(ii')$ from Remark~\ref{rem:nutne}$(2)$.

We recall that a signed measure $\mu\in M(X)$ is called \emph{boundary}\index{measure!boundary} provided its variation $\abs{\mu}$ is maximal on $X$.

\begin{thm}
    \label{t:split-mc-me-char}
 Let $X$ be a compact convex set and let $A\subset X$ be a measure convex  face  such that the complementary set $A'$ is also measure convex. Assume that $A\cup A'$ carries every maximal measure. Then $A$ is a split face if and only if $\mu|_A\in (A_c(X))^\perp$   for each boundary measure $\mu\in (A_c(X))^\perp$.
\end{thm}

\begin{proof}
$\implies$: Let $A$ be a split face and  $\mu\in (A_c(X))^\perp$ be a boundary measure. Then $\mu(X)=0$, hence $\mu^+(X)=\mu^-(X)$. Hence, we may assume without loss of generality that $\mu^+$ and $\mu^-$ are maximal probability measures. The assumption $\mu\in (A_c(X))^\perp$ then means that $\mu^+$ and $\mu^-$ have the same barycenter, denote it by $x$. Further, $\mu^+$ and $\mu^-$ are carried by $A\cup A'$.
It follows that
$$\mu^+=a_1\nu_1+a_2\nu_2\quad\mbox{and}\quad\mu^-=a_3\nu_3+a_4\nu_4,$$ where $a_1,a_2,a_3,a_4\ge 0$, $a_1+a_2=a_3+a_4=1$, and $\nu_1,\nu_3\in M_1(A)$, $\nu_2,\nu_4\in M_1(A')$. Then 
$$a_1r(\nu_1)+a_2r(\nu_2)=r(\mu^+)=r(\mu^-)=a_3 r(\nu_3)+a_4r(\nu_4).$$
By measure convexity we deduce that $r(\nu_1),r(\nu_3)\in A$ and $r(\nu_2),r(\nu_4)\in A'$. Since $A$ is a split face, we deduce $a_1=a_3$ and, provided $a_1>0$, $r(\nu_1)=r(\nu_3)$. Thus $a_1\nu_1-a_3\nu_3=\mu|_A\in(A_c(X))^\perp$.

%By the assumption, $\abs{\mu}(X\setminus (A\cup A'))=0$. Since $A$ is a split face, $A$ and $a'$ are affinely independent (see \cite[p. 128]{alfsen}), and thus \cite[Proposition II.6.2]{alfsen} yields $\mu|_A\in (A_c(X))^\perp$.

$\Longleftarrow$: Since $A\cup A'$ carries every maximal measure, $X=\co (A\cup A')$. Indeed, let  $x\in X$ be given. We find a maximal measure $\mu\in M_1(X)$ with $r(\mu)=x$ and using the assumption we decompose it as $\mu=a_1\mu_1+a_2\mu_2$, where $a_1+a_2=1$, $a_1,a_2\ge 0$ and $\mu_1\in M_1(A), \mu_2\in M_1(A')$. Then their barycenters satisfy $x_1=r(\mu_1)\in A$, $x_2=r(\mu_2)\in A'$ and
$x=a_1x_1+a_2x_2$. Hence $x\in \co (A\cup A')$ and $X=\co (A\cup A')$.

To observe that $A$ is split, we first verify that both $A$ and $A'$ are measure extremal with respect to maximal measures. More precisely, we claim that given $x\in A$ and $\mu\in M_x(X)$ maximal, then $\mu(A)=1$. To see this, assume that $\mu=a_1\mu_1+a_2\mu_2$, where $a_1,a_2\ge 0$, $a_1+a_2=1$, and $\mu_1\in M_1(A)$, $\mu_2\in M_1(A')$. If $a_2>0$, we write 
\[
x=r(\mu)=a_1r(\mu_1)+a_2r(\mu_2)
\]
and observe, that since $A$ is a face, $r(\mu_2)\in A$, which contradicts the measure convexity of $A'$. Hence $a_2=0$ and $\mu\in M_1(A)$. Similarly we check that $A'$ is measure extremal with respect to maximal measures.

To check the uniqueness of convex combinations, take $x\in X\setminus (A\cup A')$ and assume that
\[
x=\lambda_1x_1+(1-\lambda_1)x_2=\lambda_2x_3+(1-\lambda_2)x_4
\]
are convex combinations with $x_1,x_3\in A$ and $x_2,x_4\in A'$. We find maximal measures $\mu_i\in M_{x_i}(X)$ for $i=1,\dots, 4$. Then $\mu_1,\mu_3\in M_1(A)$ and $\mu_2,\mu_4\in M_1(A')$. Since 
\[
\nu=(\lambda_1\mu_1+(1-\lambda_1)\mu_2)-(\lambda_2\mu_3+(1-\lambda_2)\mu_4)
\]
is boundary and $\nu\in (A_c(X))^\perp$, by the assumption we have 
\[
\nu|_{A}=\lambda_1\mu_1-\lambda_2\mu_3\in (A_c(X))^\perp.
\]
Hence $\lambda_1=\lambda_2$ and $x_1=r(\mu_1)=r(\mu_3)=x_3$. 
This proves that $A$ is a split face.
\end{proof}

If $X$ is a simplex, then $(A_c(X))^\perp$ contains no nonzero boundary measure, hence we get the following corollary.

\begin{cor}
\label{c:simplex-facejesplit}
Let $X$ be a simplex and let $A\subset X$ be a measure convex face such that the complementary set $A'$ is also measure convex and $A\cup A'$ carries every maximal measure. Then $A$ is a split face. 
\end{cor}

\subsection{Strongly affine functions on compact convex sets with $K$-analytic boundary}
\label{ss:measure-splitfaces}

This section is devoted to split faces $F$ with strongly affine $\lambda_F$ and to extending strongly affine functions to strongly affine ones. To this end we restrict ourselves to compact convex sets with $K$-analytic boundary.
We start by recalling one of the equivalent definitions of $K$-analytic spaces.

\begin{definition}
    A Tychonoff space $Y$ is called \emph{$K$-analytic} if it is a continuous image of an $F_{\sigma\delta}$ subset of a compact space.\index{K-analytic space@$K$-analytic space}
\end{definition}

This is the original definition due to Choquet \cite{choquet59}. Note that any $K$-analytic space is Lindel\"of and that a Polish space (i.e., a separable completely metrizable space) is $K$-analytic. 

The main result of this section is the following theorem.

\begin{thm}
  \label{t:srhnuti-splitfaceu-metriz} Let $X$ be a compact convex set  with $\ext X$ being $K$-analytic.
  (This is satisfied, in particular, if $X$ is metrizable or if $\ext X$ is $F_\sigma$.)
  Let $A\subset X$ be a split face. 
  Then the following assertions are equivalent.
  \begin{enumerate}[$(i)$]
      \item Both $A$ and $A'$ are measure convex.
      \item $A$ is measure convex and $\Upsilon a\in A_{sa}(X)$ whenever $a$ is a strongly affine function on $A$.
      \item $\Upsilon(a|_A)\in A_{sa}(X)$ whenever $a\in A_{sa}(X)$.
      \item  The function $\lambda_A$ is strongly affine.
      \item  Both $A$ and $A'$ are measure convex and measure extremal.
  \end{enumerate}
\end{thm}

\begin{proof}[Easy part of the proof.] We start by observing that some implications are easy and hold without any assumptions on $X$. More precisely:

$(ii)\implies(iii)$: Note that $a|_A$
 is strongly affine on $A$ whenever $A$ is measure convex and $a\in A_{sa}(X)$.

 $(iii)\implies(iv)$: Recall that $\lambda_A=\Upsilon(1_X|_A)$.

 $(iv)\implies(v)$: Observe that $A=[\lambda_A=1]$ and $A'=[\lambda_A=0]$.
 
 $(v)\implies(i)$: This is trivial.
\end{proof}

It remains to prove implication $(i)\implies(ii)$. In its proof we will use the assumption on $\ext X$. It will be done using some auxilliary results. 

\iffalse
The first one is the following easy observation.

\begin{lemma}
%\label{l:split-miry}
Let $A$ be a measure convex split face such that its complementary face $A'$ is also measure convex. Let $x\in X\setminus (A\cup A')$ be given and $y\in A$ and $y'\in A'$ are given such that \eqref{eq:split} is valid. Let $\mu\in M_x(X)$ be carried by $A\cup A'$. Then $\lambda_A(x)=\mu(A)$ and $y=r(\frac{\mu|_A}{\mu(A)})$.
\end{lemma}

\begin{proof}
If $\mu(A)=0$, then $\mu(A')=1$ and by measure convexity, $x=r(\mu)\in A'$. Hence $\mu(A)>0$. Analogously we obtain $\mu(A')>0$. We write
\[
\mu=\mu(A)\frac{\mu|_A}{\mu(A)}+\mu(A')\frac{\mu|_{A'}}{\mu(A')}.
\]
Then
\[
x=r(\mu)=\mu(A)r\left(\frac{\mu|_A}{\mu(A)}\right)+\mu(A')r\left(\frac{\mu|_{A'}}{\mu(A')}\right).
\]
By measure convexity, the barycenters are in $A$, $A'$, respectively. It follows from the uniqueness of the decomposition that $\mu(A)=\lambda_A(x)$ and $y=r(\frac{\mu|_A}{\mu(A)})$.
\end{proof}
\fi

The first one can be viewed as a variant of Lemma~\ref{L:kvocient} for continuous affine surjections and strongly affine functions.

\begin{lemma}
    \label{l:perfect-k-analytic}
    Let $X, Z$ be compact convex sets and let $Y\subset Z$ be a measure convex $K$-analytic subset. Let $\rho\colon Z\to X$ be a continuous affine mapping such that $\rho(Y)=X$. Let $f\colon X\to \er$ be a bounded function such that $f\circ\rho$ is strongly affine on $Y$. Then $f$ is strongly affine on $X$. 
 \end{lemma}

\begin{proof}
 Let $\mu\in M_1(X)$ be given.  We are going to verify that $f$ is $\mu$-measurable and $\mu(f)=f(r(\mu))$. Let us denote $s=\rho|_Y$.
 By \cite[Corollary~432G]{fremlin4} there is a Radon measure $\nu\in M_1(X)$ such that $\nu(Y)=1$ and $s(\nu)=\mu$. 

 We continue by proving $f$ is $\mu$-measurable. So, fix an open set 
$U\subset \er$. The function $\wt{f}=f\circ s=(f\circ\rho)|_Y$ is universally measurable, hence the set $\wt{f}^{-1}(U)$ is $\nu$-measurable. Since $\nu$ is Radon, we find a $\sigma$-compact set $K\subset \wt{f}^{-1}(U)$ with $\nu(K)=\nu(\wt{f}^{-1}(U))$.
Then $s(K)$ is a $\sigma$-compact set and
$$s(K)\subset s(\wt{f}^{-1}(U))=f^{-1}(U).$$
Moreover,
$$s^{-1}(f^{-1}(U)\setminus s(K))\subset s^{-1}(f^{-1}(U))\setminus K= 
 \wt{f}^{-1}(U)\setminus K,$$
 which is a $\nu$-null set. It follows that $f^{-1}(U)\setminus s(K)$ is $\mu$-null and thus $f^{-1}(U)$ is $\mu$-measurable.

Finally, we check the barycentric formula. 
Let  $x=r(\mu)$ be the barycenter of $\mu$ and $z=r(\nu)\in Z$ be the barycenter of $\nu$. Since $Y$ is measure convex, $z\in Y$. Further, $s(z)=x$. Indeed, given $a\in A_c(X)$, we have $a\circ \rho\in A_c(Z)$, and thus
 \[
 \begin{aligned}
 a(x)&=a(r(\mu))=\mu(a)=s(\nu)(a)=\nu(a\circ s)=\int_Y (a\circ s )\di\nu=\int_Z (a\circ\rho)\di\nu\\
 &=(a\circ\rho)(z)=a(s(z)).
 \end{aligned}
 \]
 Since continuous affine functions on $X$ separate points of $X$, we obtain $x=s(z)$. 
 Thus
 \[
 \begin{aligned}
\mu(f)&=s(\nu)(f)=\nu(f\circ s)=\int_Y (f\circ s)\di\nu=\int_Y (f\circ \rho)\di\nu=(f\circ \rho)(z)=f(s(z))\\
&=f(x).
 \end{aligned}
 \]
 
 Since $\mu$ is arbitrary, we deduce that $f$ is strongly affine.
\end{proof}

We continue by a sufficient condition for strong affinity on compact convex sets with $K$-analytic boundary.

 \begin{prop}
     \label{p:sa-k-analytic}
   Let $X$ be a compact convex set with $\ext X$ being $K$-analytic and let $f\colon X\to\er$ be a bounded function such that  for each $\mu\in M_1(X)$ maximal, $f$ is $\mu$-measurable and $\mu(f)=f(r(\mu))$. Then $f$ is strongly affine.
 \end{prop}

\begin{proof}
Since $\ext X$ is a $K$-analytic set, it is universally measurable by \cite[Corollary 2.9.3]{ROJA} and a measure $\mu\in M_1(X)$ is carried by $\ext X$ if and only if $\mu$ is maximal (see \cite[Theorem 3.79]{lmns}). Hence the function 
$$g=\begin{cases}
    f&\text{on }\ext X,\\
    0&\text{elsewhere},
\end{cases}$$
is universally measurable. Indeed, given a measure $\mu\in M_1(X)$, we can decompose $\mu=\mu|_{\ext X}+\mu|_{X\setminus \ext X}$ and $g$ is measurable with respect to both parts.

Let $Z=M_1(X)$ and $Y=\{\mu\in Z\setsep \mu(\ext X)=1\}$. Then $Y$ is a $K$-analytic subset of $Z$ by \cite[Theorem 3(a)]{Hol-Kal}, which is moreover measure convex. Indeed, the function $h=1_{\ext X}$ is universally measurable, and thus $\wt{h}(\mu)=\mu(h)$ is a strongly affine function on $Z$ by \cite[Proposition 5.30]{lmns}. Hence $Y=[\wt{h}=1]$ is a measure convex face in $Z$.

By \cite[Proposition 5.30]{lmns}, the function $\wt{g}\colon Z\to \er$ defined by $\wt{g}(\mu)=\mu(g)$, $\mu\in Z$, is strongly affine on $Z$. Thus its restriction $\wt{f}=\wt{g}|_Y$ to $Y$ is strongly affine on $Y$. To finish the proof it is enough to observe that, denoting $r\colon Z\to X$ the barycentric mapping, we have $\wt{f}=f\circ r$ on $Y$ by the assumption. (We observe that $r\colon Y\to X$ is surjective by \cite[Theorem 3.79 and Theorem 3.65]{lmns}.) Hence Lemma~\ref{l:perfect-k-analytic} finishes the reasoning.
 \end{proof}

\iffalse
\begin{lemma}
\label{l:extense-sa-kanalytic}
Let $X$ be a compact convex set with $\ext X$ being $K$-analytic and $A\subset X$ be a measure convex split face such that its complementary face $A'$ is also measure convex.  Then $\Upsilon a$ is strongly affine for any $a\in A_{sa}(A)$.
\end{lemma}
\fi

Now we are ready to prove the remaining implication.

\begin{proof}[Proof of implication $(i)\implies(ii)$ from Theorem~\ref{t:srhnuti-splitfaceu-metriz}.]
Assume $A$ and $A'$ are measure convex and fix a strongly affine function $a$ on $A$.
By Lemma~\ref{l:extense-splitface}, $b=\Upsilon a$ is a bounded affine function on $X$. To prove it is strongly affine we will use Proposition~\ref{p:sa-k-analytic}. 
So, let $\mu\in M_1(X)$ be maximal and $x=r(\mu)$. Then $\mu$ is carried by $\ext X$ (by \cite[Theorem 3.79]{lmns}) and hence by $A\cup A'$. Therefore we conclude by Lemma~\ref{L:mc-partial-sa}.
\iffalse

So, $b$ is clearly $\mu$-measurable. 

It remains to prove that $b(x)=\mu(b)$. If $\mu$ is carried by $A$, then $x\in A$ and we use the assumption that $a$ is strongly affine on $A$. If $\mu$ is carried by $A'$, then $x\in A'$ and we use the fact that $b=a=0$ on $A'$.  

Finally, assume $\mu(A)>0$ and $\mu(A')>0$. We write
\[
\mu=\mu(A)\tfrac{\mu|_A}{\mu(A)}+\mu(A')\tfrac{\mu|_{A'}}{\mu(A')}.
\]
Then
\[
x=r(\mu)=\mu(A)r\left(\tfrac{\mu|_A}{\mu(A)}\right)+\mu(A')r\left(\tfrac{\mu|_{A'}}{\mu(A')}\right).
\]
By measure convexity, the barycenters are in $A$, $A'$, respectively. It follows from the uniqueness of the decomposition that $\mu(A)=\lambda_A(x)$ and $y=r(\frac{\mu|_A}{\mu(A)})$ (using the notation from \eqref{eq:split}). Thus
$$b(x)=\lambda_A(x) a(x)=\mu(A)a\left(r\left(\tfrac{\mu|_A}{\mu(A)}\right)\right) = \mu(a) \int_A a\di \tfrac{\mu|A}{\mu(A)}=\int b\di\mu,
$$
which completes the argument. 

Finally, by Proposition~\ref{p:sa-k-analytic} we deduce that $b$ is strongly affine.  \fi
\end{proof}

The following example witnesses that in Theorem~\ref{t:srhnuti-splitfaceu-metriz}$(i)$ the assumption that both $A$ and $A'$ are measure convex is essential.

\begin{example}\label{ex:d+s}
    There exists a split face $A\subset X=M_1([0,1])$ with the following properties:
    \begin{enumerate}[$(i)$]
        \item Both $A$ and $A'$ are Baire sets;
        \item  $A$ is measure convex and $A'$ is measure extremal;
        \item $A\cup A'$  carries every maximal measure;
        \item $\lambda_A$ is an affine Baire function which is  not strongly affine.
    \end{enumerate}
\end{example}

\begin{proof}
Let 
\[
A=\{\mu\in X\setsep \mu(\{x\})=0, x\in [0,1]\}.
\] 
Then $A$ is a $G_\delta$ face of $X$ (see \cite[Proposition~2.58]{lmns}).
Let 
\[
B=\{\mu\in X\setsep \mu\text{ is discrete}\}.
\]
Then $B=\{\mu\in X\setsep \mu_d([0,1])=1\}$ (here $\mu_d$ denotes the discrete part of $\mu$) is a Borel face of $X$ by \cite[Proposition~2.63]{lmns}. We claim that $A$ is a split face and $B=A'$.

To this end we notice that each $\mu\in X$ can be uniquely decomposed as $\mu=\mu_c+\mu_d$, where $\mu_c$ is the continuous part of $\mu$ and $\mu_d$ is the discrete part of $\mu$.
Hence $X$ is the direct convex sum of $A$ and $B$ (see \cite[Chapter II, \S 6]{alfsen}).

To finish the proof that $A$ is a split face it is enough to show that $B=A'$. Obviously, $B\subset A'$. For the converse inclusion, let $\mu\in A'$ be given. Then $\mu=\alpha \mu_1+(1-\alpha)\mu_2$, where $\alpha\in [0,1]$, $\mu_1\in A$ and $\mu_2\in B$. Let $C\subset X$ be a face containing $\mu$ and disjoint from $A$. If $\alpha\in (0,1]$, then $\alpha\mu_1+(1-\alpha)\mu_2=\mu\in C$ implies $\mu_1\in C$, which is impossible as $C\cap A=\emptyset$. Hence $\alpha=0$, and consequently $\mu=\mu_2\in B$. Hence $A'\subset B$.

So far we have verified that $A$ is a Borel split face such that $A'=B$ is Borel as well and $A\cup A'\supset \ext X$ carries any maximal measure. Since $X$ is metrizable, Baire and Borel sets coincide. We have thus verified conditions $(i)$ and $(iii)$.

To check that $A$ is measure convex, we observe that
\[
A=\bigcap_{x\in [0,1]}\{\mu\in X\setsep \mu(\{x\})=0\}.
\]
Since each function $\mu\mapsto \mu(\{x\})$ is upper semicontinuous and affine on $X$ and has values in $[0,1]$, it is strongly affine and thus each set $\{\mu\in X\setsep \mu(\{x\})=0\}$ is measure convex. Since $A$ is Borel, it follows that $A$ is measure convex.

To check that $B=A'$ is measure extremal, fix $\mu\in B$ and $\Lambda\in M_1(X)$ with $r(\Lambda)=\mu$. Then $\mu=\sum_{n=1}^\infty t_n\ep_{x_n}$ for a sequence $(x_n)$ in $[0,1]$ and a sequence $(\lambda_n)$ of non-negative numbers with $\sum_{n=1}^\infty t_n=1$. The function $H(\nu)=\sum_{n=1}^\infty \nu(\{x_n\})$ is easily seen to be strongly affine and $H(\mu)=1$, Thus $H=1$ $\Lambda$-almost everywhere, hence $\Lambda$ is supported by $B$. 
Hence, we have verified $(ii)$.

Finally, the function $\lambda_A$ is an affine function of the second Baire class which
is not strongly affine, since $\lambda_A(\mu)=1-\mu_d([0,1])$ for $\mu\in X$, and \cite[Proposition~2.63]{lmns} applies. So, $(iv)$ is verified and the proof is complete.
\end{proof}

\iffalse
\begin{definition}
Let $X$ be a compact convex set and $A\subset X$ be measure convex. We call a function $a\colon A\to \er$ \emph{strongly affine} if $\mu(a)=a(r(\mu))$, whenever $\mu$ is a probability measure supported by $A$.    
\end{definition}

\begin{remark}
  \label{r:omezen-sa-naA}
 We remind that a strongly affine function on a measure convex set $A\subset X$ is bounded as the proof of \cite[Satz 2.1.(c)]{krause} shows. 
\end{remark}
\fi

We note that Theorem~\ref{t:srhnuti-splitfaceu-metriz} implies, in particular, that $\lambda_A$ is a (strong) multiplier for $A_{sa}(X)$ provided $X$ is a compact convex set with $K$-analytic boundary and $A$ is a split face with $\lambda_A$ strongly affine. Later we will discuss a converse and similar results. We note that Example~\ref{ex:dikous-divnyspitface} below shows that the assumption on $\ext X$ in Theorem~\ref{t:srhnuti-splitfaceu-metriz} cannot be dropped.
However, the following problems seem to be open.

\begin{ques}
    Let $X$ be a compact convex set and let $A\subset X$ be split face.
    \begin{enumerate}[$(1)$]
        \item Assume both $A$ and $A'$ are measure convex and measure extremal. Is $\lambda_A$ strongly affine?
        \item Assume $\lambda_A$ is strongly affine. Does $\Upsilon$ maps strongly affine functions to strongly affine functions?
    \end{enumerate}
\end{ques}

\subsection{Baire split faces}
\label{ss:extension-splitfaces}

The aim of this section is to characterize split faces $A\subset X$ with $\lambda_A$ Baire and strongly affine. This is the content of the following proposition, which is an analogue of Theorem~\ref{t:srhnuti-splitfaceu-metriz}. We point out that in this case $X$ is a general compact convex set (with no additional assumptions on $\ext X$), stronger assumptions are imposed on the face.

\begin{prop}
  \label{p:shrnuti-splitfaceu-baire}
  Let $X$ be a compact convex set and let $A\subset X$ be a split face. The the following assertions are equivalent.
  \begin{enumerate}[$(i)$]
      \item Both $A$ and $A'$ are Baire and measure convex.
      \item $A$ is measure convex and $\Upsilon a\in A_{sa}(X)\cap\Ba(X)$ whenever $a$ is a strongly affine Baire function on $A$.
      \item $\Upsilon(a|_A)\in A_{sa}(X)\cap\Ba(X)$ whenever $a\in A_{sa}(X)\cap\Ba(X)$.
      \item  The function $\lambda_A$ is strongly affine and Baire.
      \item  Both $A$ and $A'$ are Baire, measure convex and measure extremal.
  \end{enumerate}
\end{prop}

The proof follow the same scheme as that of Theorem~\ref{t:srhnuti-splitfaceu-metriz}. Again, implications $(ii)\implies(iii)\implies(iv)\implies(v)\implies(i)$ are easy and may be proved exactly in the same way. To prove the remaining implication $(i)\implies(ii)$ we will use two auxilliary results. The first one is a purely topological extension result.

\begin{lemma}
\label{l:extense-K-anal}
Let $X$ be a compact convex set and $A\subset X$ a $K$-analytic split face such that $A'$ is also $K$-analytic. Then
     $\Upsilon a$ is Baire for each Baire affine function on $A$.

In particular, both $A$ and $A'$ are Baire.
  \end{lemma}

%Before embarking on the proof of the lemma, we need to recall that a topological space is $K$-analytic if it is the image of a Polish space under a usco (an upper semicontinuous compact valued) map, see \cite[p. 11]{ROJA}.

\begin{proof}
 Let $a$ be a Baire affine function on $A$.
To verify that $\Upsilon a$ is Baire, we proceed as in \cite{stacey-maan}. Since $A$ and $A'$ are $K$-analytic,  the space $H=A\times A'\times [0,1]$ is $K$-analytic as well (see \cite[Theorem~2.5.5]{ROJA}). The mapping $\varphi\colon H\to X$ defined by $\varphi([y,y',\lambda])=\lambda y +(1-\lambda) y'$ is a continuous affine surjection. Further, the function $h\colon H\to \er$ defined as $h([y,y',\lambda])=\lambda a(y)$ is a Baire function on $H$ satisfying $b\circ \varphi=h$. Hence, for each $U\subset \er$, open the sets
\[
b^{-1}(U)=\varphi(h^{-1}(U))\quad\text{and}\quad b^{-1}(\er\setminus U)=\varphi(h^{-1}(\er\setminus U))
\]
are disjoint, $K$-analytic (as continuous images of $K$-analytic sets) and cover $X$. By \cite[Section~3.3]{ROJA}, they are Baire sets. Hence $b$ is Baire measurable and thus it is a Baire function.

In particular, if we apply it to $a=1$, we deduce that $A$ and $A'$ are Baire sets.
\end{proof}

We start by a sufficient condition for strong affinity of Baire functions which is a counterpart of Proposition~\ref{p:sa-k-analytic}.

\begin{lemma}
    \label{l:sa-baireovske}
 Let $a$ be a bounded Baire function on a compact convex set $X$ satisfying $a(r(\mu))=\mu(a)$ for each maximal $\mu\in M_1(X)$. Then $a$ is  strongly affine.
\end{lemma}
\begin{proof}
  By a separable reduction result (see \cite[Theorem 9.12]{lmns}), there exist a metrizable compact convex set $Y$, a continuous affine surjection $\varphi\colon X\to Y$  and a Baire function $b\colon Y\to \er$ such that $a=b\circ\varphi$. 
  Then $b$ is bounded and $b(r(\nu))=\nu(b)$ for each maximal $\nu\in M_1(Y)$. 

  Indeed, let $\nu\in M_1(Y)$ maximal be given. Using \cite[Proposition 7.49]{lmns} we find a maximal measure $\mu\in M_1(X)$ with $\varphi(\mu)=\nu$.
 Given $h\in A_c(Y)$, we have
   \[
  h(\varphi(r(\mu)))=(h\circ \varphi)(r(\mu))=\mu(h\circ\varphi)=(\varphi(\mu))(h)=\nu(h)=h(r(\nu)).
  \]
 Thus $\varphi(r(\mu))=r(\nu)$. Now we have
\[
\nu(b)=(\varphi(\mu))(b)=\mu(b\circ\varphi)=(b\circ\varphi)(r(\mu))=b(\varphi(r(\mu)))=b(r(\nu)).
\]

Thus we have proved that $b$ satisfies the barycentric formula with respect to each maximal measure on $Y$.  Since $Y$ is metrizable, $\ext Y$ is $K$-analytic, hence Proposition~\ref{p:sa-k-analytic} yields that $b$ is strongly affine on $Y$.
\iffalse
Let $\mu\in M_1(Y)$ be given. By \cite[Theorem 11.41]{lmns}, there exists a Borel measurable mapping $m\colon Y\to M_1(Y)$ such that $r(m(x))=x$, $x\in Y$, and $m(x)$ is maximal for each $x\in Y$. Let $\nu\in M_1(Y)$ be defined by the formula 
\[
\nu(g)=\int_Y \left(\int_Y g(t)\di m(x)(t)\right)\di \mu(x), \quad g\in C(Y).
\]
Note that $\nu$ is indeed a positive linear functional on $C(Y)$ satisfying $\nu(1_Y)=1$, so it can be viewed as an element of $M_1(Y)$. Further, 
\[
h(r(\nu))=\nu(h)=\int_Y m(x)(h)\di\mu(x)=\int_Y h(x)\di\mu(x)=h(r(\mu)),\quad h\in A_c(Y),
\]\
hence $r(\nu)=r(\mu)$. Further, for each $g\in C(Y)$ we have
\[
\nu(g^*)=\int_Y m(x)(g^*)\di\mu(x)=\int_Y m(x)(g)\di\mu(x)=\nu(g),
\]
and thus $\nu$ is maximal. Hence we obtain
\[
b(r(\mu))=b(r(\nu))=\nu(b)=\int_Y m(x)(b)\di\mu(x)=\int_Y b(x)\di\mu(x)=\mu(b).
\]
Hence $b\in A_{sa}(Y)$.\fi
%
By \cite[Proposition~5.29]{lmns}, the function $a=b\circ \varphi$ is strongly affine as well.
\end{proof}

\iffalse
\begin{lemma}
\label{l:urysohn-F}
Let $X$ be a compact convex set and $A\subset X$ a Baire split face such that $A'$ is also Baire.
\begin{itemize}
    \item[(a)] $\Upsilon a$ is Baire for each Baire affine function on $A$.
    \item[(b)] Assume moreover that both $A$ and $A'$ are measure convex. Then $\Upsilon a\in A_{sa}(X)\cap \Ba(X)$ for each $a\in A_{sa}(X)\cap \Ba(X)$.
\end{itemize}
\end{lemma}

Before embarking on the proof of the lemma, we need to recall that a topological space is $K$-analytic if it is the image of a Polish space under a usco (an upper semicontinuous compact valued) map, see \cite[p. 11]{ROJA}.
\fi

\begin{proof}[Proof of implication $(i)\implies(ii)$ from Proposition~\ref{p:shrnuti-splitfaceu-baire}.]
 Assume that both $A$ and $A'$ are Baire and measure convex and let $a$ be a strongly affine Baire function on $A$.
By Lemma~\ref{l:extense-K-anal} we know that $\Upsilon a\in\Ba(X)$. %By Remark~\ref{r:omezen-sa-naA} we know that $a$ is bounded, hence $\Upsilon a$ is bounded as well (by Observation~\ref{obs:operator rozsireni}(iii)). 
It remains to show that $b=\Upsilon a$ is strongly affine.
By Lemma~\ref{l:sa-baireovske} it is enough to prove that $b$  satisfies the barycentric formula with respect to every maximal measure. Since $A\cup A'$ is a Baire set containing $\ext X$, it carries all maximal measures. Hence, we conclude by Lemma~\ref{L:mc-partial-sa}.
\iffalse

So, let $\mu\in M_1(X)$ be maximal with $r(\mu)=x$. Then $\mu(A\cup A')=1$, as $A\cup A'$ is a Baire set containing $\ext X$. If $\mu$ is carried by $A$, then $x\in A$ by measure convexity of $A$ and $\mu(b)=\mu(a)=a(x)=b(x)$ as $b=a$ on $A$ and $a$ is strongly affine on $A$.
If $\mu$ is carried by $A'$, then $x\in A'$  by measure convexity of  $A'$ and $\mu(b)=0=b(x)$ as $b=0$ on $A'$.

Hence we may assume that $\mu(A)\in (0,1)$. %, then $\mu(A')=1-\mu(A)$. 
Then
\[
\mu=\mu(A)\frac{\mu|_A}{\mu(A)}+\mu(A')\frac{\mu|_{A'}}{\mu(A')},
\]
hence
\[
x=r(\mu)=\mu(A)r\left(\frac{\mu|_A}{\mu(A)}\right)+\mu(A')r\left(\frac{\mu|_{A'}}{\mu(A')}\right),
\]
where $r\left(\frac{\mu|_A}{\mu(A)}\right)\in A$ and $r\left(\frac{\mu|_{A'}}{\mu(A')}\right)\in A'$ by measure convexity.
Since $A$ is a split face, this convex combination is uniquely determined, and thus
\[
\begin{aligned}
b(x)&=(\Upsilon a)(x)=\mu(A)a\left(r\left(\frac{\mu|_A}{\mu(A)}\right)\right)=
\mu(A)\int_X a\di\frac{\mu|_A}{\mu(A)}=\int_X a1_A\di\mu\\
&=\int_{A\cup A'}a1_A\di\mu=\int_X b\di\mu.
\end{aligned}
\]
%By Lemma~\ref{l:sa-baireovske}, $a$ is strongly affine on $X$.
This finishes the proof.\fi
\end{proof}

\subsection{Extending strongly affine functions from closed split faces}

In this section we focus on extending affine functions of various descriptive classes from closed split faces. To this end, we need the following lemma.

\begin{lemma}
\label{l:frag-selekce} Let $r\colon M\to N$ be a continuous surjection of a compact space $M$ onto a compact space $N$. Let $g\colon M\to \er$ be in $\Bo_1(M)$ or a fragmented function. Then there exists a function $\phi\colon N\to M$ such that $r(\phi(y))=y$, $y\in N$, and $g\circ \phi$ is in $\Bo_1(N)$ or fragmented, respectively.
\end{lemma}

\begin{proof}
We first handle the case of fragmented functions.
Let $(U_n)$ be a countable base of open sets in $\er$. By Theorem~\ref{T:a}, each set $g^{-1}(U_n)$ can be written as a countable union of resolvable sets in $M$ called, say, $F_{n,k}$. Let $\A=\{F_{n,k}\setsep n,k\in\en\}$. By \cite[Lemma 1]{HoSp}, there exists a selection $\phi\colon N\to M$ from the multi-valued mapping $y\mapsto r^{-1}(y)$ such that $\phi^{-1}(A)$ is resolvable in $N$  for each $A\in \A$. Then for each $U_n$  the set
\[
(g\circ \phi)^{-1}(U_n)=\phi^{-1}(g^{-1}(U_n))=\phi^{-1}\left(\bigcup_{k=1}^\infty F_{n,k}\right)=\bigcup_{k=1}^\infty \phi^{-1}(F_{n,k})
\]
is a countable union of resolvable sets in $N$. By Theorem~\ref{T:a}, $g\circ\phi$ is fragmented.

If $g\in\Bo_1(M)$, we proceed as above  but we use the selection lemma \cite[Lemma 8]{HoSp} instead.
\end{proof}

\begin{lemma}
\label{l:rozsir-splitface}
Let $X$ be a compact convex set and let $F\subset X$ be a closed split face. Then the following assertions are valid.
\begin{enumerate}[$(a)$]
    \item Both $F$ and $F'$ are measure convex, measure extremal and $F\cup F'$ carries all maximal measures.
    \item $\Upsilon a\in A_s(X)$ whenever $a\in A_s(F)$.
    \item $\Upsilon a\in A_b(X)\cap \Bo_1(X)$ whenever $a\in A_b(A)\cap \Bo_1(A)$.
    \item $\Upsilon a\in A_f(X)$ whenever $a\in A_f(F)$.
\end{enumerate}
\end{lemma}

\begin{proof}
$(a)$: The face $F$, being closed, is measure convex. Further, by \cite[Proposition II.6.5 and Proposition II.6.9]{alfsen} we deduce that $\lambda_F=(1_F)^*$, so it is upper semicontinuous. It follows that $\lambda_F$ is strongly affine by \eqref{eq:prvniinkluze}, hence we easily get that both $F$ and $F'$ are measure convex and measure extremal. Further,
$$F\cup F'=[\lambda_F=1]\cup[\lambda_F=0]=[(1_F)^*=1_F],$$
so this set carries each maximal measure by \cite[Theorem 3.58]{lmns} (as $1_F$ is upper semicontinuous).

%Further, $F'$ is measure convex by \cite[]{alfsen}. Also, the set $F\cup F'=\{x\in X\setsep 1_F^*(x)=1_F(x)\}$ carries any maximal measure by \cite[]{alfsen}. Thus the function $Ta$ is a well defined strongly affine function by Lemma~\ref{l:extense-sa}.

\smallskip

Let us proceed to $(b)-(d)$: Since $\Upsilon 1_F=1_F^*\in A_s(X)\subset A_b(X)\cap \Bo_1(X)\subset A_f(X)$ and $\Upsilon$ is linear, it is enough to prove the statements assuming $0\le a\le1$.
%First observe that 
%Next we claim that we may assume that $a$ is positive. To see this, the function $T1=1_F^*$ is upper semicontinous affine, and thus it is contained in $A_s(X)$. Thus a suitable constant $c\ge 0$ ensures that $a+c\ge 0$ on $F$, and thus $Ta=T(a+c)-cT1$ belongs to $A_s(X)$ whenever $T(a+c)\in A_s(X)$ (and similarly for other function spaces).

Given $x\in X$, let $y(x)\in F$ and $y'(x)\in F'$ be such that \eqref{eq:split} is fulfilled. We recall that in this case $\Upsilon a(x)=\lambda_F(x) a(y(x))$ for any $a\in A_b(X)$. Further, it follows from Lemma~\ref{L:mc-partial-sa} that
$$\Upsilon a(x)=\int_F a\di\omega\mbox{\quad for any maximal }\omega\in M_1(X).$$
\iffalse

Further, let $\omega\in M_1(X)$ be any maximal representing $x$. Since $\lambda_F$ is strongly affine, we deduce $\lambda_F(x)=\omega(F)$. We further claim that
$$\Upsilon a(x)=\int_F a\di\omega.$$
%and $Ta(x)=\omega_F(a)$.
If $\omega(F)=0$, we have $x\in F'$ and thus $\Upsilon a(x)=0=\int_F a\di\omega|_F$. Analogously, if $\omega(F)=1$, $x\in F$ and $\Upsilon a(x)=a(x)=\int_F a\di \omega$ as $a$ is strongly affine on $F$.

If $\omega(F)\in (0,1)$, let $y_1=r\left(\frac{\omega|_F}{\omega(F)}\right)\in F$ and $y_2=r\left(\frac{\omega|_{F'}}{\omega(F')}\right)\in F'$. Then $x=\omega(F)y_1+\omega(F')y_2$, which by the uniqueness of the decomposition means $y_1=y(x)$ and $y'(x)=y_2$. Thus
\[
\Upsilon a(x)=\lambda_F(x)a(y(x))=\lambda_F(x)a(y_1)=\lambda_F(x)\int_F a\di\tfrac{\omega}{\omega(F)}=\int_F a\di\omega.
\]\fi
In other words, if we set
\[
d=\begin{cases} a,&\text{on } F,\\
0,&\text{on }X\setminus F,
\end{cases}
\]
then for any maximal $\omega\in M_1(X)$ we have
$$\Upsilon a(x)=\int_X d\di\omega.$$
%To see this, apply $\omega$ to the upper semicontinuous function $1_F^*$ to obtain
%\[
%\lambda(x)=1_F^*(x)=\omega(1_F^*)=\omega(1_F)=\omega(F).
%\]

Now we verify that $\Upsilon a$ is in the relevant function system.

$(b)$: It is enough to show that $\Upsilon a\in A_s(X)$ for any positive upper semicontinuous affine function $a$ on $F$.
In this case 
$d$ is an upper semicontinuous convex function and $\Upsilon a=d^*$. Indeed, let $x\in X$ be given. Let $\nu$ be a maximal measure representing $x$. Then
$$\Upsilon a(x)=\int_F a\di \nu=\nu(d)=\nu(d^*)\ge d^*(x),$$
where we used \cite[Theorem 3.58]{lmns} and the fact that $d^*$ is an upper semicontinuous concave function. On the other hand, $\nu(d)\le d^*(x)$ by \cite[Corollary I.3.6 and the following remark]{alfsen}.
Hence, we conclude that $\Upsilon a$ is upper semicontinuous.

$(c)$: Assume $a\in A_b(F)\cap\Bo_1(F)$. Since the space $\Bo_1(X)$ is uniformly closed, it is enough to show that $\Upsilon a$ may be uniformly approximated by functions from $\Bo_1(X)$.  To this end, fix $\ep>0$. Let  $0=\alpha_0<\alpha_1<\cdots<\alpha_n=1$ be a division of $[0,1]$ with $\alpha_{i}-\alpha_{i-1}<\ep$, $i\in \{1,\dots, n\}$.

For each $i\in\{0,\dots, n-1\}$ we consider the set
\[
M_i=\{\omega\in M_1(X)\setsep \omega(F)\ge \alpha_i\}.
\]
Then each $M_i$ is a closed set in $M_1(X)$ and $N_i=r(M_i)$ is a closed set containing $[\alpha_i\le \lambda_F<\alpha_{i+1}]$.
Let $\tilde{a}\colon M_1(X)\to\er$ be defined as 
$$\tilde{a}(\omega)=\omega|_F(a)=\omega(d),\quad\omega\in M_1(X).$$
Since $d\in\Bo_1(X)$, \cite[Proposition 5.30]{lmns} implies that $\tilde{a}\in\Bo_1(M_1(X))$. In particular, $\tilde{a}|_{M_i}\in \Bo_1(M_i)$ for $i\in\{0,\dots, n-1\}$.

We use Lemma~\ref{l:frag-selekce}  to find a selection function $\phi_i\colon N_i\to M_i$ such that $\tilde{a}\circ \phi_i\in \Bo_1(N_i)$. 
We define the sought function $b$ as 
\[
b(x)=
\begin{cases}
    a(x) & x\in F,\\
(\tilde{a}\circ \phi_i)(x)& x\in [\alpha_i\le \lambda_F<\alpha_{i+1}], i\in \{0,\dots,n-1\}.\end{cases}
\]
Clearly, the restriction of $b$ to each of the sets $[\alpha_i\le \lambda_F<\alpha_{i+1}]$ is of the first Borel class. By the very assumption $b|_F=a$ is also of the first Borel class. 
Since $\lambda_F$ is upper semicontinuous, the sets $[\alpha_i\le \lambda_F<\alpha_{i+1}]$ are of the form $U\cap H$, where $U$ open and $H$ closed. Thus $b\in \Bo_1(X)$. 

We are going to show that $\norm{b-\Upsilon a}_\infty<\ep$. To this end, fix $x\in X$. If $x\in F$, then $b(x)-\Upsilon a(x)=0$.

%Now we check that $\abs{\tilde{a}(\phi_i(x))-\Upsilon a(x)}\le \alpha_{i+1}-\alpha_i$, $x\in [\alpha_i\le \lambda_F<\alpha_{i+1}]$.

If $x\in X\setminus F$, we fix $i\in\{0,\dots,n-1\}$ with $x\in  [\alpha_i\le \lambda_F<\alpha_{i+1}]$.
We set $\mu=\phi_i(x)$ and decompose it as $\mu=\mu|_F+\mu|_{X\setminus F}$. Let $\nu_1$ and $\nu_2$ be maximal measures with $\mu|_F\prec \nu_1$ and $\mu|_{X\setminus F}\prec \nu_2$. Then $\nu=\nu_1+\nu_2$ is a maximal measure with $\mu\prec \nu$. Hence $r(\nu)=r(\mu)=x$.
If $\mu(F)=0$, then $\nu_1=0$. If $\mu(F)>0$, then -- using the fact that $F$ is measure convex and measure extremal -- we deduce that $\nu_1$ is carried by $F$. Thus we have
\[\begin{gathered}
    \alpha_i\le \mu(F)=\nu_1(F)\quad \text{and}\quad\\ \nu_1(F)+\nu_2(F)=\nu(F)=\int 1_F\di\nu\le\int 1_F\di\mu=\mu(F)<\alpha_{i+1},
\end{gathered}\]
where we used that $\mu\prec\nu$ and $1_F$ is an upper semicontinuous concave function.
We deduce that $\nu_2(F)<\alpha_{i+1}-\alpha_i$.
Since $a$ is strongly affine, we have $\mu|_F(a)=\nu_1(a)$.
Hence
\[
\begin{aligned}
\abs{b(x)-\Upsilon a(x)}&=\abs{(\tilde{a}\circ \phi_i)(x)-\Upsilon a(x)}=\abs{\mu|_F(a)-\nu|_F(a)}\\
&=\abs{\nu_1(a)-\nu_1(a)-\nu_2|_F(a)}\le \nu_2(F)<\alpha_{i+1}-\alpha_i.
\end{aligned}
\]
This completes the proof that $\norm{b-\Upsilon a}<\ep$.
Thus $\Upsilon a$ is the uniform limit of a sequence of functions in $\Bo_1(X)$, which implies that $\Upsilon a\in \Bo_1(X)$ as well.

%\emph{Case $a\in A_f(F)$.} 
$(d)$: Assume $a\in A_f(F)$. We proceed as in the proof of $(c),$ we only use \cite[Lemma 3.2]{lusp} to show that $\tilde{a}$ is fragmented and subsequently Lemma~\ref{l:frag-selekce} to find selections $\phi_i$ such that the functions $\tilde{a}\circ\phi_i$ are fragmented. The resulting function $b$ is then fragmented as well. So, $\Upsilon a$ can be uniformly approximated by fragmented functions, so $\Upsilon a$ is fragmented.
%This concludes the proof.
\end{proof}

The following statement follows from the previous lemma together with Observation~\ref{obs:operator rozsireni}.

\begin{cor}
   Let $X$ be a compact convex set and let $F\subset X$ be a closed split face. Then the following assertions are valid.
\begin{enumerate}[$(a)$]
     \item   $\Upsilon a\in \overline{(A_s(X)}$ whenever $a\in \overline{A_s(F)}$.
     \item $\Upsilon a\in (A_s(X))^\mu$ whenever $a\in (A_s(F))^\mu$.
     \item $\Upsilon a\in (A_s(X))^\sigma$ whenever $a\in (A_s(F))^\sigma$.
    \item $\Upsilon a\in (A_b(X)\cap \Bo_1(X))^\mu$ whenever $a\in (A_b(F)\cap \Bo_1(F))^\mu$.
    \item  $\Upsilon a\in (A_f(X))^\mu$ whenever $a\in (A_f(F))^\mu$.
\end{enumerate} 
\end{cor}

%%%%%%%%%%%%%%%%%%%%%%%%%%%%%%%%%%%%%%%%%%%%%%%%%%%%%%%%%%%%%%%%%
\section{Strongly affine Baire functions}\label{sec:baire}

In this section we focus on  canonical intermediate function spaces of Baire functions and mulitpliers on them. We start by comparing four of such spaces which are closed with respect to pointwise limits of monotone sequences.

\begin{prop}
\label{P:Baire-srovnani}
Let $X$ be a compact convex set. Then we have the following:
\begin{enumerate}[$(i)$]
    \item   $(A_c(X))^\mu\subset(A_1(X))^\mu\subset(A_c(X))^\sigma\subset A_{sa}(X)\cap \Ba^b(X)$.
    \item  Any of the inclusion from assertion (i) may be strict.
    \item If $X$ is a simplex, the four subspaces from assertion $(i)$ coincide.
\end{enumerate}
\end{prop}

To provide examples proving assertion $(ii)$ of the previous proposition we will use the following lemmma.

\begin{lemma}\label{L:symetricke X}
Let $X$ be a compact convex subset of a locally convex space which is symmetric (i.e., $X=-X$). Let $f\in A_b(X)$ and $x\in X$ be given. Then we have the following:
\begin{enumerate}[$(i)$]
    \item  If $f$ is lower semicontinuous at $x$, then $f$ is upper semicontinuous at $-x$.
    \item If $f$ is lower semicontinuous both at $x$ and at $-x$, then $f$ is continuous at $x$.
    \item Let $(f_\alpha)$ be a non-decreasing net in $A_b(X)$ with $f_\alpha\nearrow f$. If each $f_\alpha$ is continuous at $x$, then $f$ is continuous at $x$ as well.
\end{enumerate}
\end{lemma}

\begin{proof}
    $(i)$: Fix $\ep>0$. Then there is $U$, a neighborhood of $x$, such that $f(y)>f(x)-\ep$ for each $y\in U$. Then $-U$ is a neighborhood of $-x$ and for any $y\in -U$ we have
    $$f(y)=2f(0)-f(-y)<2f(0)-(f(x)-\ep)=f(-x)+\ep,$$
    where we used that $-y\in U$ and the assumption that $f$ is affine. This completes the proof.

    Assertion $(ii)$ follows immediately from $(i)$.

    $(iii)$ By applying $(i)$ both to $f_\alpha$ and $-f_\alpha$ we deduce that $f_\alpha$ is continuous also at $-x$. Now it follows that $f$ is lower semicontinuous both at $x$ and at $-x$, so it is continuous at $x$ by $(ii)$.
\end{proof}

Now we are ready to prove the above proposition:

\begin{proof}[Proof of Proposition~\ref{P:Baire-srovnani}]
Assertion $(i)$ is obvious. Let us prove assertion $(ii)$. 
Let $X_1=(B_{C([0,1])^*},w^*)$. Then $X_1$ is symmetric, so Lemma~\ref{L:symetricke X}$(iii)$ yields $(A_c(X_1))^\mu=A_c(X_1)$. Note that $X_1$ may be represented as the space of signed Radon measures
on $[0,1]$ with total variation at most $1$. The function $f_1(\mu)=\mu(\{0\})$ clearly belongs to $A_1(X_1)$ and is not continuous. Hence $(A_c(X_1))^\mu\subsetneqq (A_1(X_1))^\mu$.

Further, any $f\in A_1(X_1)$ is continuous at all points of a dense $G_\delta$-subset of $X_1$. It easily follows from Lemma~\ref{L:symetricke X}$(iii)$ (using the Baire category theorem) that the same is true for any $f\in (A_1(X_1))^\mu$. Define the function $f_2$ by $f_2(\mu)=\mu([0,1]\cap \qe)$. Then $f_2\in A_2(X_1)\subset (A_c(X))^\sigma$, but it has no point of continuity. Indeed, both sets
$$\begin{aligned} 
E_0&=\{\mu\in X_1\setsep \mu\mbox{ is discrete and }\abs{\mu}([0,1]\setminus \qe)=1\},\\
E_1&=\{\mu\in X_1\setsep\abs{\mu}([0,1]\cap\qe)=1\}\end{aligned}$$
are dense in $X_1$, $f_2|_{E_0}=0$ and $f_2|_{E_1}=1$. We deduce that $(A_1(X_1))^\mu\subsetneqq (A_c(X_1))^\sigma$.

Further, let $E$ be the Banach space provided by \cite{talagrand} and $X_2=(B_{E^*},w^*)$.
Then $(A_c(X_2))^\sigma\subsetneqq A_{sa}(X_2)\cap \Ba^b(X_2)$. Indeed, $E$ has the Schur property and hence $(A_c(X_2))^\sigma=A_c(X_2)$ and in \cite{talagrand} an element $\phi\in E^{**}\setminus E$ is constructed such that $\phi|_{X_2}$ is strongly affine and of the second Baire class.

Finally, if we set $X=X_1\times X_2$, it is easy to check that
 $$(A_c(X))^\mu\subsetneqq(A_1(X))^\mu\subsetneqq(A_c(X))^\sigma\subsetneqq A_{sa}(X)\cap \Ba^b(X).$$

$(iii)$: Let $a\in A_{sa}(X)\cap \Ba^b(X)$ be given.  Using \cite[Theorem 9.12]{lmns} we find a metrizable simplex $Y$ along with a continuous affine surjection $\varphi\colon X\to Y$ such that there exists a Baire function $\tilde{a}\colon Y\to \er$ with $a=\tilde{a}\circ \varphi$.
Then $\tilde{a}$ is a strongly affine Baire function on $Y$ by \cite[Proposition 5.29]{lmns}. By \cite[Corollary]{alfsen-note-scand}, $\tilde{a}\in (A_c(Y))^\mu$.
Hence $a=\tilde{a}\circ \varphi\in (A_c(X))^\mu$.
\end{proof}

\begin{remark}
 It is possible to consider a more detailed transfinite hierarchy of spaces of strongly affine Baire functions. In particular, we may consider spaces $(A_\alpha(X))^\mu$, $(A_{sa}(X)\cap \Ba_\alpha(X))^\mu$ or  $(A_{sa}(X)\cap \Ba_\alpha(X))^\sigma$, where $\alpha\in[2,\omega_1)$ is an ordinal. It is probably not known whether all these spaces may be different.
\end{remark}

Next we focus on describing the spaces of multipliers. Recall that 
 by Proposition~\ref{P:rovnostmulti} $M^s(H)=M(H)$ holds for any $H\subset \Ba(X)\cap A_{sa}(X)$, in particular for the four spaces addressed above. Further, by \cite[Proposition 4.4]{smith-london} we know that $M((A_c(X))^\sigma)=Z((A_c(X))^\sigma)$. We obtain from the proof of \cite[Proposition 4.9]{smith-london} that $M((A_c(X))^\mu)=Z((A_c(X))^\mu)$ as well. It seems not to be clear whether an analogous equality holds for the remaining two spaces:

\begin{ques} Let $X$ be a compact convex set
    Let $H=(A_1(X))^\mu$ or $H=\Ba(X)\cap A_{sa}(X)$. Is $M(H)=Z(H)$?
\end{ques}

 The four above mentioned spaces are intermediate function spaces closed under monotone limits, hence Theorem~\ref{T:integral representation H} applies to each of them. We shall look in 
 more detail at the respective $\sigma$-algebras.

\begin{thm}\label{T:baire-multipliers} 
Let $X$ be a compact convex set and let $H=A_{sa}(X)\cap\Ba(X)$. Then we have the following.
\begin{enumerate}[$(a)$]
     \item 
    $\begin{aligned}[t]
           \A_H&=\A^s_H=\ms_H\\&=\{ F\cap \ext X \setsep F\mbox{ is a split face, $F$ and $F'$ are Baire and measure convex}\}. \end{aligned}$
           
    \noindent{}If $X$ is a simplex, then $$\A_H=\Z_H=\{[f=1]\cap\ext X\setsep f\in A_{sa}(X)\cap\Ba(X), f(\ext X)\subset\{0,1\}\}.$$
    \item $M(H')\subset M(H)$ for each intermediate function space $H' \subset H $.
\end{enumerate}
    \end{thm}

\begin{proof}
 $(a)$: Equality $\A_H=\A_H^s$ follows from Theorem~\ref{T:integral representation H}$(ii)$ using the fact that $M(H)=M^s(H)$. Inclusion $\A_H^s\subset\ms_H$ follows from  Theorem~\ref{T:meritelnost-strongmulti}. Inclusion `$\subset$' from the last equality is obvious.

 Finally, assume that $F\subset X$ is a split face such that both $F$ and $F'$ are Baire and measure convex. It follows from Proposition~\ref{p:shrnuti-splitfaceu-baire} that $\lambda_F\in M(H)$. Thus $F\cap\ext X\in \A_H$ by the very definition of $\A_H$.

 If $X$ is a simplex, Lemma~\ref{L:SH=ZH} yields $\ms_H=\Z_H$ and so the proof is complete.

 $(b)$: This follows from $(a)$ and Proposition~\ref{P: silnemulti-inkluze}.
\end{proof}    

We continue by a comparison of several spaces of multipliers.
    
\begin{prop}\label{P:baire-multipliers-inclusions} 
Let $X$ be a compact convex set. Then we have the following inclusions:
 $$\begin{array}{ccccc}
             M(A_c(X))&\subset& M((A_c(X))^\mu)&\subset &M((A_c(X))^\sigma)\subset M(A_{sa}(X)\cap \Ba(X))\\
             \cap&&&\nesubset&\\
             M(A_1(X))&\subset&M((A_1(X))^\mu))&&\end{array}$$
    \end{prop}

\begin{proof}
   First recall that by Proposition~\ref{P:rovnostmulti} for all spaces in question multipliers and strong multipliers coincide. Thus inclusions  $M(A_c(X))\subset M((A_c(X))^\mu)\subset M((A_c(X))^\sigma)$ follow from Proposition~\ref{p:multi-pro-mu} (one can use also the old results of \cite{edwards,smith-london}).
  
  Further, inclusion $M(A_c(X))\subset M(A_1(X))$ follows from Remark~\ref{rem:Ms(H1)}. 
  Inclusions $M(A_1(X))\subset M((A_1(X))^\mu)\subset M((A_c(X))^\sigma)$ follow from Proposition~\ref{p:multi-pro-mu} applied to $H=A_1(X)$.

  Finally, inclusion $M((A_c(X))^\sigma)\subset M(A_{sa}(X)\cap\Ba(X))$ follows from Theorem~\ref{T:baire-multipliers}$(b)$.
\end{proof}

We continue by two natural open problems.

\begin{ques}
    Let $X$ be a compact convex set. Is $M((A_c(X))^\mu)\subset M((A_1(X))^\mu)$?
\end{ques}

Note that this inclusion is missing in Proposition~\ref{P:baire-multipliers-inclusions}. We have no idea how to attack it.

\begin{ques}
    Let $X$ be a compact convex set. Is $M((A_c(X))^\mu)= M(A_{sa}(X)\cap\Ba(X))$? 
\end{ques}

By Proposition~\ref{P:baire-multipliers-inclusions} we know that inclusion `$\subset$' holds, but we do not know any counterexample to the converse inclusion. Note that in two extreme cases the equality holds: If $X$ is a simplex, the equality follows from Proposition~\ref{P:Baire-srovnani}$(iii)$.
On the other hand, if $X$ is symmetric, inclusions from Proposition~\ref{P:Baire-srovnani}$(i)$ may be strict, but in this case the spaces of multipliers are the same and trivial by Example~\ref{ex:symetricka}.

Let us now pass to the space $A_1(X)$. We start by the following easy consequence of Corollary~\ref{cor:hup+hdown pro A1aj}.

\begin{prop}\label{P:A1 determined}
    Let $X$ be a compact convex set. Then a bounded function on $\ext X$ may be extended to an element of $M(A_1(X))=M^s(A_1(X))$ if and only if it is $\A_{A_1(X)}$-measurable. I.e.,  the space $M(A_1(X))=M^s(A_1(X))$ is determined by $\A_{A_1(X)}=\A^s_{A_1(X)}$.
\end{prop}

It is not clear whether this characterization can be improved:

\begin{ques}\label{q:A1-split}
    Let $X$ be a compact convex set and let $H=A_1(X)$. Is $\A_H=\ms_H$? Or, at least, is $A_1(X)$ determined by $\ms_H$?
\end{ques}

The answer is positive for the class of simplices addressed in Section~\ref{sec:stacey} (see Proposition~\ref{P:dikous-A1-new}).
We further get a better characterization for compact convex sets $X$  with $\ext X$ being \lin\ resolvable: 

\begin{thm}
    \label{t:a1-lindelof-h-hranice}
  Let $X$ be a compact convex such that $\ext X$ is a \lin\   resolvable set and let $H=A_1(X)$. Then $M(H)=M^s(H)$ is determined by
  \[
  \begin{aligned}
  \B=\{\ext X \setminus F\setsep &F\text{ is a }\Coz_\delta\text{ measure convex split face}\\
  &\text{such that }F'\text{ is Baire and measure convex}\}.
  \end{aligned}
  \]
  
 Consequently, $M(H)=M^s(H)$ is determined by $\ms_H$, and hence for each intermediate function space $H' \subset H$, $M(H') \subset M(H)$. 
\end{thm}

\begin{proof}
We first prove the `consequently' part. It is easy to see that $\ms_H \subset \B$. By Theorem~\ref{T:meritelnost-strongmulti} we know that $\A^s_H\subset\ms_H$. By the first part of the theorem, $M^s(H)$ is determined by $\A$. By Proposition~\ref{P:A1 determined} we deduce that $M^s(H)$ is determined by $\A^s_H$. Thus $M^s(H)$ is determined by $\ms_H$. The rest follows from Proposition \ref{P: silnemulti-inkluze}.

To prove the first part of the statement, we use Proposition~\ref{P:meritelnost multiplikatoru pomoci topologickych split facu} for $H=A_1(X)$ and $T=\Ba_1^b(\ext X)$.  First we need to verify that $T$ satisfies conditions $(i)-(iii)$ of the quoted proposition. Condition $(i)$  follows easily from the fact that Baire-one functions are stable with respect to the pointwise multiplication.
Condition $(ii)$ follows from Lemma~\ref{L:YupcapYdown=Y} and condition $(iii)$ is obvious. Further, the formula for $H$
follows from \cite[Theorem 6.4]{lusp} due to the assumption on $\ext X$. 

Therefore, the assumptions of Proposition~\ref{P:meritelnost multiplikatoru pomoci topologickych split facu} are satisfied, hence $H$ is determined by $\B^s_H$. To complete the proof it is enough to show that $\B\subset\B^s_H$. So, fix $E\in \B$. Let $F$ be a $\Coz_\delta$ measure convex split face with $F'$ Baire and measure convex such that $E=\ext X\setminus F$. Let $(U_n)$ be a decreasing sequence of cozero subsets of $X$ with $\bigcap_n U_n=F$. Then $1_{U_n\cap\ext X}\in T$ for each $n\in\en$ and hence $1_{F'\cap\ext X}\in T^\uparrow$. 
Further, by Proposition~\ref{p:shrnuti-splitfaceu-baire}, $\Upsilon_{F'}(a|_{F'})\in A_{sa}(X)\cap\Ba(X)$ for each $a\in H$. Thus $E=F'\cap\ext X\in\B^s_H$.
This concludes the proof.
\end{proof}

\begin{remarks}\label{r:fsigma-resolvable}
(1) By Proposition~\ref{p:postacproa1}, $Z(A_1(X))=M(A_1(X))$ in case $\ext X$ is \lin\ resolvable set. Hence we have obtained a measurable characterization of $Z(A_1(X))$ in Theorem~\ref{t:a1-lindelof-h-hranice} for this class of compact convex sets.

(2) In case $X$ is metrizable the assumptions on $X$ in Theorem~\ref{t:a1-lindelof-h-hranice} are equivalent with the fact that $\ext X$ is an $F_\sigma$ set. Indeed, if $\ext X$ is an $F_\sigma$ set in $X$, its characteristic function is Baire-one (as $\ext X$ is a $G_\delta$ set) and thus $1_{\ext X}$ is a fragmented function. Conversely, if $1_{\ext X}$ is a fragmented function, by Theorem B the see that $\ext X$ is both $F_\sigma$ and $G_\delta$.

(3) We also point out that the assumption in Theorem~\ref{t:a1-lindelof-h-hranice} is
satisfied provided $\ext X$ is an $F_\sigma$ set. To see this it is enough to show that $\ext X$ is
a resolvable set in $X$. To this end, assume that $F\subset  X$ is a nonempty
closed set such that both $F \cap \ext X$ and $F\setminus \ext X$ are dense in $F$. By \cite[Th\'eor\`eme 2]{tal-kanal}, we can write
$\ext X=\bigcap_{n=1}^\infty (H_n \cup V_n)$,
where $H_n\subset X$ is closed and $V_n\subset X$ is open, $n \in\en $. Thus both $F\setminus \ext X$ and $F\cap \ext X$ are comeager disjoint sets in $F$, in contradiction with the Baire category theorem. Hence $\ext X$ is a resolvable set.
\end{remarks}

%%%%%%%%%%%%%%%%%%%%%%%%
%%%%%% ZDE KONCI SEKCE O SILNE AFINNICH BAIREOVSKYCG
%%%%%%%%%%%%%%%%%%%%%%%
%%% TO NIZE JSOU ZBYTKY K MOZNEMU POUZITI JINDE
%%%%%%%%%%%%%%%%%

\section{Beyond Baire functions}\label{sec:beyond}

In this section we investigate intermediate function spaces which are not necessarily contained in the space of Baire functions. More precisely, we consider functions derived from semicontinuous affine functions, from affine functions of the first Borel class or from fragmented affine functions.

\subsection{Comparison of the spaces}
We start by collecting basic properties and mutual relationship of these spaces. Before formulating the first proposition, we introduce another piece of notation:

If $K$ is a compact space, we denote by \gls{Lb(K)} the space of differences of bounded lower semicontinuous functions on $K$.

\begin{prop}\label{P:vetsi prostory srov}
Let $X$ be a compact convex set. Then the following assertions are valid.
    \begin{enumerate}[$(a)$]
        \item $\begin{array}[t]{ccccccc}
            A_s(X) &\subset &\ov{A_s(X)}&\subset& A_b(X)\cap\Bo_1(X) & \subset & A_f(X)  \\
             & & \cap & & \cap & & \cap  \\
             && (A_s(X))^\mu & \subset &(A_b(X)\cap \Bo_1(X))^\mu & \subset & (A_f(X))^\mu \\
             && \cap &&&& \cap       \\
             && (A_s(X))^\sigma&&\subset && A_{sa}(X).
        \end{array}$
       \item All the inclusions in assertion $(a)$ except for $(A_s(X))^\mu\subset (A_b(X)\cap \Bo_1(X))^\mu$ and $(A_s(X))^\mu\subset (A_s(X))^\sigma$ may be strict even if $X$ is a Bauer simplex.
       \item Inclusion  $(A_s(X))^\mu\subset (A_b(X)\cap \Bo_1(X))^\mu$ may be strict even if $X$ is a simplex, but in case $X$ is a Bauer simplex, the equality holds.
       \item Inclusion $(A_s(X))^\mu\subset (A_s(X))^\sigma$ may be strict, but if $X$ is a simplex, the equality holds.
      
    \end{enumerate}
\end{prop}

Before passing to the proof itself, we give the following easy lemma.

\begin{lemma}\label{L:borel=Lbmu}
    Let $L$ be a compact space.
    \begin{enumerate}[$(a)$]
        \item $\Lb(L)\subset\overline{\Lb(L)}\subset\Bo^b_1(L)$;
        \item $(\Lb(L))^\mu=(\Lb(L))^\sigma=(\Bo_1^b(L))^\mu=(\Bo_1^b(L))^\sigma=\Bo^b(L)$.
    \end{enumerate}
\end{lemma}

\begin{proof}
    Assertion $(a)$ is well known. Let us prove $(b)$. Since both $\Lb(L)$ and $\Bo^b_1(L)$ are lattices, we deduce that  $(\Lb(L))^\mu=(\Lb(L))^\sigma$ and $(\Bo_1^b(L))^\mu=(\Bo_1^b(L))^\sigma$. Further, we clearly have 
    $(\Lb(L))^\sigma\subset (\Bo_1^b(L))^\sigma\subset \Bo^b(L)$.
    
Next we observe that if $G\subset L$ is open, then $1_G\in \Lb(L)$, because this function is lower semicontinuous. Thus the set
\[
\{B\subset L\setsep 1_B\in (\Lb(L))^\sigma\}
\]
contains open sets and it is easily seen to be a $\sigma$-algebra, Hence $1_B\in (\Lb(L))^\sigma$ for each Borel set in $L$, which implies $f\in (\Lb(L))^\sigma$ for each bounded Borel function on $L$ (as simple functions are dense in $\Bo^b(L)$).
\end{proof}

Now we proceed to the proof of the above proposition:

\begin{proof}[Proof of Proposition~\ref{P:vetsi prostory srov}.]
 Assertion $(a)$ is clear.

$(c)$: If $X$ is a Bauer simplex, the equality follows from Lemma~\ref{L:borel=Lbmu}$(b)$ using Lemma~\ref{L:function space}.
An example showing that the inclusion may be strict for a simplex $X$ is provided by Proposition~\ref{p:vztahy-multi-ifs}$(a)$ below.

 $(d)$: If $X$ is a simplex, the equality is proved in \cite[p. 104]{smith-london} (using results of \cite{krause}).
  An example when the inclusion is proper is provided by Lemma~\ref{L:symetricke X}.   Indeed, let $X$ be a symmetric compact convex set. By assertions $(ii)$ and $(iii)$ of the quoted lemma we have $(A_s(X))^\mu=(A_c(X))^\mu=A_c(X)$ and $(A_s(X))^\sigma=(A_c(X))^\sigma$.  Thus if we take $X$ to be (for example) the dual unit ball of the Banach space $E=c_0$, then
\[
(A_s(X))^\sigma=(A_c(X))^\sigma\neq A_c(X)=(A_c(X))^\mu=(A_s(X))^\mu.
\]

 $(b)$: We will use Lemma~\ref{L:function space}. Therefore we first look at  inclusions of the respective spaces of bounded functions on a compact space.
 
 Let $K_1=[0,1]$. It follows from \cite[Proposition 5.1]{odell-rosen} that there are two functions $f_1,f_2:K_1\to \er$ such that
 $f_1\in\overline{\Lb(K_1)}\setminus\Lb(K_1)$ and $f_2\in\Ba_1^b(K_1)\setminus\overline{\Lb(K_1)}$.  Further, let $A\subset K_1$ be an analytic set which is not Borel. Then $f_3=1_A$ is  universally measurable but not Borel, hence neither in $(\Fr^b(K_1))^\mu$ nor in $(\Lb(K_1))^\sigma$ (as $K_1$ is metrizable).  The function $f_4=1_{K_1\cap\qe}$ is Borel but not fragmented.

Further, let $K_2$ be the long line, i.e., $K_2=[0,\omega_1)\times[0,1)\cup\{(\omega_1,0)\}$ equipped with the order topology induced by the lexicographic order. Let $S\subset [0,\omega_1)$ be a stationary set whose complement is also stationary (i.e., both $S$ and it complement intersect each closed unbounded subset of $[0,\omega_1)$).
We define the following functions on $K_2$:
\begin{itemize}
   \item Let 
   $$g_1(\alpha,t)=\begin{cases}
       1 & \alpha\in S,\\ 0 & \mbox{ otherwise}.
   \end{cases}$$
   Then $g_1$ is a fragmented non-Borel function on $K_2$.
   \item  Let $$g_2(\alpha,t)=\begin{cases}
       1 & \alpha\in S, t\in\qe,\\ 0 &\mbox{otherwise}.
   \end{cases}$$
    Then $g_2$ belongs to $(\Fr^b(K_2))^\mu$ but it is neither fragmented not Borel.  
\end{itemize}

Now we conclude using Lemma~\ref{L:function space}. Indeed, for $i=1,2$ we set $E_i=C(K_i)$, $X_i=S(E_i)=M_1(K_i)$ and let $V_i$ be the operator from Lemma~\ref{L:function space}. Then we have:
$$\begin{gathered}
    V_1(f_1)\in \overline{A_s(X_1)}\setminus A_s(X_1), V_1(f_2)\in A_b(X_1)\cap\Bo_1(X_1)\setminus \overline{A_s(X_1)},\\
    V_2(g_1)\in A_f(X_2)\setminus (A_b(X_2)\cap\Bo_1(X_2))^\mu,
    V_1(f_4)\in (A_b(X_1)\cap\Bo_1(X_1))^\mu\setminus A_f(X_1),\\
    V_2(g_2)\in (A_f(X_2))^\mu\setminus (A_f(X_2)\cup (A_b(X_2)\cap\Bo_1(X_2))^\mu),\\
    V_1(f_3)\in A_{sa}(X_1)\setminus (A_f(X_1))^\mu\cup(A_s(X_1))^\sigma.
\end{gathered}$$

Since $(A_s(X_i))^\mu=(A_s(X_i))^\sigma=(A_b(X_i)\cap\Bo_1(X_i))^\mu$ (by $(c)$ and $(d)$), we deduce that no more inclusions hold. 
\end{proof}

\begin{remarks}
\label{r:as-sigma-mu}
    (1) Assertion $(b)$ of  Proposition~\ref{P:vetsi prostory srov} in particular shows that $A_s(X)$ need not be a Banach space, even if $X$ is a Bauer simplex.     This answers a question in \cite[p. 100]{smith-london}.

    (2) Assertion $(d)$ of Proposition~\ref{P:vetsi prostory srov} answers in the negative a question asked in \cite[p. 104]{smith-london}. 
\end{remarks}

In view of the validity of $(A_s(X))^\sigma=(A_s(X))^\mu$ for a simplex, the following question is natural:

\begin{ques}
   Assume $X$ is a simplex. Do the equalities  $(A_b(X)\cap\Bo_1(X))^\sigma=(A_b(X)\cap\Bo_1(X))^\mu$ and $(A_f(X))^\sigma=(A_f(X))^\mu$ hold?
\end{ques}

We note that these equalities clearly hold if $X$ is a Bauer simplex and also for simplices addressed in Section~\ref{sec:stacey} below (see Proposition~\ref{P:dikous-Afmu-new}$(a)$ and Proposition~\ref{P:dikous-bo1-mu-new}$(a)$). But the method used for $A_s(X)$ in \cite{krause} fails.

We continue by a comparison of the spaces addressed in this section with spaces of Baire functions from the previous section.

\begin{prop}\label{P:srovnani baire a vetsich}
    Let $X$ be a compact convex set. Then the following assertions are valid.
    \begin{enumerate}[$(a)$]
        \item $\begin{array}[t]{ccccccc}
             A_s(X)&\subset&\overline{A_s(X)}&\subset& (A_s(X))^\mu & \subset & (A_s(X))^\sigma  \\
            \cup &&&&\cup&&\cup \\
            A_c(X)&&\subset&&(A_c(X))^\mu&\subset&(A_c(X))^\sigma \\
            &\sesubset&&&\cap&\nesubset&\\
            && A_1(X)&\subset&(A_1(X))^\mu&& \\
            &&\cap&&\cap&& \\
            &&A_b(X)\cap\Bo_1(X)&\subset&(A_b(X)\cap\Bo_1(X))^\mu.&&
        \end{array}$
        
       No more inclusions between spaces of Baire functions and the other ones are valid in general.
       \item Assume that $X$ is metrizable. Then:
       $$\begin{array}{ccccccc}
            A_s(X) &\subset &\overline{A_s(X)}&\subset&(A_c(X))^\mu & = & (A_s(X))^\mu  \\
            \cup&&\cap&&\cap &&\cap\\
           A_c(X)&&A_1(X)&\subset&(A_1(X))^\mu &\subset& (A_c(X))^\sigma  
            \\
             &&\parallel&&\parallel &&\parallel\\
               &&A_b(X)\cap\Bo_1(X)&\subset&(A_b(x)\cap\Bo_1(X))^\mu &\subset& (A_s(X))^\sigma\\
               &&\parallel&&\parallel &&\parallel\\
                 &&A_f(X)&\subset&(A_f(X))^\mu &\subset& (A_b(x)\cap\Bo_1(X))^\sigma  
                  \\
             &&&&&&\parallel\\
              &&&&&& (A_f(X))^\sigma.
             \end{array}$$
           No more inclusion are valid in general.  
    \end{enumerate}
\end{prop}

\begin{proof}
    $(a)$: The validity of inclusions is clear. To illustrate that the relevant inclusions may be proper, even for a Bauer simplex, we use the long line, i.e., the compact space $K_2$ from the proof of Proposition~\ref{P:vetsi prostory srov}$(b)$. The function $f=1_{\{(\omega_1,0)\}}$ is upper semicontinuous and hence of the first Borel class, but it is not a Baire function.
    Hence $V_2(f)\in A_s(X_2)\setminus (A_c(X_2))^\sigma$ and also $V_2(f)\in A_b(X_2)\cap\Bo_1(X_2)\setminus (A_c(X_2))^\sigma$.

    $(b)$:  Assume $X$ is metrizable. Then clearly $A_l(X)\subset (A_c(X))^\mu$, so $(A_s(X))^\mu=(A_c(X))^\mu$. The remaining equalities follow from Theorem~\ref{T:b}$(b)$. No more inclusions hold by Proposition~\ref{P:Baire-srovnani} and the proof of Proposition~\ref{P:vetsi prostory srov}$(b)$ (note that the space $X_1$ is metrizable).
\end{proof}

\subsection{On spaces derived from semicontinuous affine functions}
Now we pass to the investigation of multipliers of the spaces of our interest. We start by looking at the spaces derived from semicontinuous functions.

\begin{prop}\label{P:multi strong pro As}
    Let $X$ be a compact convex set and let $H\subset (A_s(X))^\sigma$ be an intermediate function space. Then $M(H)=M^s(H)$.
\end{prop}

\begin{proof} The proof will be done in three steps:

{\tt Step 1:} Assume that $m\in M(H)$ satisfies $m(\ext X)\subset\{0,1\}$. We will prove that $m\in M^s(H)$.

Set $F=[m=1]$ and $E=[m=0]$. Then both $F$ and $E$ are measure convex and measure extremal disjoint faces. First we shall verify that $F\cup E$ carries every maximal measure.
To this end, let $\mu\in M_1(X)$ maximal be given. Let $(\ext X, \Sigma_\mu, \wh{\mu})$ be the completion of the measure space provided by Lemma~\ref{L:miry na ext}. Denote $c=\wh{\mu}(F\cap \ext X)$ (note that then $1-c=\wh{\mu}(E\cap \ext X)$). By equation~\eqref{eq:mira1}, there exist closed extremal sets $F_n$, $n\in\en$, such that $F_n\cap \ext X\subset F\cap \ext X$ and
$\mu(F_n)=\wh{\mu}(F_n\cap \ext X)\nearrow c$.

We claim that $F_n\subset F$. To see this, consider $x\in F_n$ and its maximal representing measure $\nu$. Let $(\ext X,\Sigma_\nu, \wh{\nu})$ be the completion of the measure space provided by Lemma~\ref{L:miry na ext} for $\nu$. Since $F_n$ is measure extremal, 
\[
1=\nu(F_n)=\wh{\nu}(F_n\cap \ext X)\le \wh{\nu}(F\cap \ext X)\le 1.
\]
Hence $m(x)=\int_{\ext X} m\di\wh{\nu}=1$, i.e., $x\in F$.

From this observation now follows that $c=\lim \mu(F_n)\le \mu(F)$. 
Similarly we obtain $\mu(E)\ge 1-c$. Hence $\mu(F)=c$ and $\mu(E)=1-c$, which gives that $F\cup E$ carries $\mu$. 

We will show that $m$ is a strong multiplier. To this end, let $f\in H$ be given and let $g\in H$ be such that $g=mf$ on $\ext X$. Let $\mu\in M_1(X)$ be any maximal measure. We need to check that $\mu([g=mf])=1$. This will be achieved by showing that $F\cup E\subset [g=mf]$. So let $x\in F$ be given. We select a maximal measure $\nu\in M_x(X)$ and consider the induced measure $\wh{\nu}$. Then
\[
1=m(x)=\int_{\ext X} m\di\wh{\nu}=\wh{\nu}(F\cap \ext X).
\]
Thus
\[
\begin{aligned}
g(x)&=\int_{\ext X}g\di\wh{\nu}=\int_{\ext X} mf\di\wh{\nu}=\int_{F\cap \ext X} f\di\wh{\nu}=\int_{\ext X} f\di\wh{\nu}\\
&=f(x)=m(x)f(x).
\end{aligned}
\]
Similarly we obtain $E\subset [g=mf]$. Hence $m$ is a strong multiplier of $H$.

\smallskip

{\tt Step 2:} Assume $H^\mu=H$. Then $M(H)=M^s(H)$.

By Step 1 we deduce that $\A_H\subset\A^s_H$. Theorem~\ref{T:integral representation H}$(ii)$ then yields that $M(H)=M^s(H)$.

\smallskip

{\tt Step 3:}  $M(H)=M^s(H)$ for general $H$.

Assume $H\subset (A_s(X))^\sigma$ and $m\in M(H)$. Then $H^\mu\subset (A_s(X))^\sigma$. By Proposition~\ref{p:multi-pro-mu}$(i)$ we get $m\in M(H^\mu)$, so by Step 2 we deduce $m\in M^s(H^\mu)$. Let us now check that $m\in M^s(H)$: 

Fix $f\in H$. Then there is $g\in H$ such that $g=mf$ on $\ext X$. Simultaneously, there is $h\in H^\mu$ such that $h=mf$ $\mu$-almost everywhere for each maximal $\mu\in M_1(X)$. Finally, since $H^\mu$ is determined by extreme points, we get $h=g$. This completes the proof.
\end{proof}

We continue by looking at the relationship of the spaces of multipliers and centers.

\begin{prop}\label{P:centrum pro As}
    Let $X$ be a compact space and let $H$ be any of the spaces
    $$\overline{A_s(X)}, (A_s(X))^\mu, (A_s(X))^\sigma.$$
    Then
    $$Z(H)=M(H)=M^s(H).$$
\end{prop}

\begin{proof}
    Equality $M(H)=M^s(H)$ follows from Proposition~\ref{P:multi strong pro As}. Inclusion $M(H)\subset Z(H)$ follows from Proposition~\ref{P:mult} and Theorem~\ref{t:assigma-deter}. It remains to prove $Z(H)\subset M(H)$.

    The case $H=(A_s(X))^\mu$ is proved in \cite[Proposition 4.9]{smith-london}. We will adapt the proof for the other cases:

    Let $T\in\frd(H)$ be arbitrary. Without loss of generality we assume $0\le T\le I$. Set $m=T(1)$. By Corollary~\ref{cor:rovnost na ext} we get
    $$\forall f\in A_c(X)\colon T(f)=mf\mbox{ on }\ext X.$$
    By the argument of \cite[Proposition 4.9]{smith-london} we deduce
    $$\forall f\in A_l(X)\colon T(f)=mf\mbox{ on }\ext X.$$
    Since $T$ is a bounded linear operator, we get
    $$\forall f\in \overline{A_s(X)}\colon T(f)=mf\mbox{ on }\ext X.$$
    This completes the proof if $H=\overline{A_s(X)}$.

    As in  \cite[Proposition 4.9]{smith-london} we infer
     $$\forall f\in (A_s(X))^\mu\colon T(f)=mf\mbox{ on }\ext X,$$
     which proves the case $H=(A_s(X))^\mu$.

    Finally, assume $H=(A_s(X))^\sigma$ and consider
    $$\F=\{ f\in H\setsep \exists g\in H\colon g=mf\mbox{ on }\ext X\}.$$
    It is clearly a linear space. By the above it contains $A_s(X)$. It remains to prove it is closed with respect to pointwise limits of bounded convergent sequences. So, assume $(f_n)$ is a bounded sequence in $\F$ pointwise converging to some $f\in H$. For each $n\in\en$ let $g_n\in H$ be such that $g_n=mf_n$ on $\ext X$
    Let $x\in X$ be arbitrary. Let $\mu$ be a maximal measure representing $x$ and let $\wh{\mu}$ be the respective measure provided by Lemma~\ref{L:miry na ext}. Then
    $$g_n(x)= \int_{\ext X} g_n\di\wh{\mu}= \int_{\ext X} mf_n\di\wh{\mu} \overset{n}{\longrightarrow} \int_{\ext X} mf\di\wh{\mu}.$$
    Thus the sequence $(g_n)$ pointwise converges to some $g\in H$. Applying the above computation to $x\in\ext X$ we deduce $g=mf$ on $\ext X$. Thus $f\in\F$.
    \end{proof}

For a simplex, the situation is easier:

\begin{prop}\label{P:As pro simplex}
    Let $X$ be a simplex and let $H$ be any of the spaces
    $$\overline{A_s(X)}, (A_s(X))^\mu, (A_s(X))^\sigma.$$
    Then $M(H)=H$. In particular, in this case
    $$\A_H=\ms_H=\Z_H.$$
\end{prop}

\begin{proof}
   Assume $X$ is a simplex. By Proposition~\ref{P:vetsi prostory srov}$(d)$ we have $(A_s(X))^\mu=(A_s(X))^\sigma$. Further, by \cite[Lemma 2.3]{krause} the spaces $A_s(X)$ and $(A_s(X))^\mu$ are lattices, if the lattice operations are defined pointwise on $\ext X$ (i.e., the images under the restriction operator $R$ from Proposition~\ref{P:algebra ZH} are sublattices of $\ell^\infty(\ext X)$). This property clearly passes to $\overline{A_s(X)}$.  So, we conclude by Proposition~\ref{P:algebra ZH}$(b)$. 

   The `in particular' part follows from Proposition~\ref{p:system-aha}$(c)$.
\end{proof}

The previous proposition inspires the following question.

\begin{ques}
    Let $X$ be a compact convex space. Is $M((A_s(X))^\mu)=M((A_s(X))^\sigma)$?
\end{ques}

Note that the answer is trivially positive if $X$ is a simplex by Proposition~\ref{P:vetsi prostory srov}$(d)$.

It is further not clear whether some of the properties of the spaces generated by semicontinuous affine functions pass to larger spaces.
Let us formulate the relevant questions:

\begin{ques}\label{q:m=ms}
    Let $X$ be a compact convex set and let $H$ be any of the spaces
 $$A_b(X)\cap \Bo_1(X),A_f(X),(A_b(X)\cap \Bo_1(X))^\mu,(A_f(X))^\mu.$$
 Is it true that $M(H)=M^s(H)$?
\end{ques}

Note that Example~\ref{ex:dikous-mezi-new} below provides an example of an intermediate function space $H$ on a simplex $X$ such that $M^s(H)\subsetneqq M(H)$. Moreover, $H^\mu=H$ and $H\subset (A_b(X)\cap\Bo_1(X))^\mu$. However, we know no counterexample among the above-named natural intermediate function spaces (in particular, 
for the special simplices from Section~\ref{sec:stacey} the answer is positive by Proposition~\ref{P:shrnutidikousu}).

\begin{ques}
    Let $X$ be a compact convex set and let $H$ be any of the spaces
 $$A_b(X)\cap \Bo_1(X),A_f(X),(A_b(X)\cap \Bo_1(X))^\mu,(A_f(X))^\mu.$$
 Is it true that $Z(H)=M(H)$?
\end{ques}

By Corollary~\ref{cor:iotax} the answer is positive if $X$ is a simplex (or, more generally, if any $x\in \ext X$ forms a split face of $X$). But it is not clear, whether this assumption is needed.

\subsection{More on measurability of strong multipliers}
Now we look in more detail at the strong multipliers and the systems $\A^s_H$ for natural choices of $H$. We start by clarifying which of the spaces that are not closed to monotone limits are determined by measurability on $\ext X$.

\begin{prop}\label{P:determ-Afaj}
    Let $X$ be a compact convex set. Then the following assertions hold.
    \begin{enumerate}[$(a)$]
        \item Let $H$ be one of the spaces $A_b(X)\cap\Bo_1(X)$, $A_f(X)$. Then $M(H)$ is determined by $\A_H$ and $M^s(H)$ is determined by $\A^s_H$.
        \item It may happen that $H=\overline{A_s(X)}$ is not determined by $\A_H$ (even if $X$ is a simplex).
    \end{enumerate}
\end{prop}

\begin{proof}
    Assertion $(a)$ follows from Corollary~\ref{cor:hup+hdown pro A1aj}. Assertion $(b)$ follows from Proposition~\ref{P:dikous-lsc--new}$(f)$ below.
\end{proof}

We continue by two theorems providing a lower bound for systems $\A^s_H$. 

\begin{thm}
\label{t:mh-asfrag}
Let $H$ be one of the spaces
$$\overline{A_s(X)},A_b(X)\cap \Bo_1(X), A_f(X).$$
\begin{enumerate}[$(i)$]
    \item Any bounded  facially upper semicontinuous function on $\ext X$ may be (uniquely) extended to an element of $M^s(H)$.
    \item $\A^s_H$ contains the algebra generated by the facial topology.
    \item $M(A_c(X))\subset M^s(H)$.
\end{enumerate}
%Let  $f\colon \ext X\to [0,1]$ be a facially upper semicontinuous function. Then for each $a\in A_f(X)$ positive there exists $b\in A_f(X)$ such that $b=af$ on $\ext X$. If moreover $a\in \Bo_1(X)$, we obtain that $b\in\Bo_1(X)$. If $a\in A_s(X)$, then $b$ is in the uniform closure of functions in  $A_s(X)$.
\end{thm}

\begin{proof}
$(i)$: Let $f\colon \ext X\to\er$ be bounded and facially upper semicontinuous. Since $M^s(H)$ is a linear space containing constant functions, we may assume without loss of generality that $0\le f\le 1$.

%Fix $a\in H$. We will show that there is $b\in H$ such that $b=fa$ on $\ext X$. Since any $a\in H$ is the difference of two positive functions, we may assume that $a\ge0$.
%
%Let $H$ stand for the respective function space (i.e., $H=A_f(X)$, $\Bo_1(X)\cap A_b(X)$ or $A_s(X)$). We want to prove that for any $a\in H$ positive there exists $b\in \overline{H}$ with $b=af$ on $\ext X$. 
%
%Let $a\in H$ positive be given. 
For each $n\in \en$ we set
\[
C_{n,i}=\{x\in\ext X\setsep f(x)\ge \tfrac{i}{2^n}\},\quad i\in\{0,\dots, 2^n\}.
\]
Each of these sets is facially closed, so set we find a closed split face $F_{n,i}$ with $C_{n,i}=F_{n,i}\cap \ext X$. It follows from Lemma~\ref{l:rozsir-splitface} that $\lambda_{F_{n,i}}\in M^s(H)$. Set
\[
a_n=2^{-n}\sum_{i=1}^{2^n}\lambda_{F_{n,i}}.
\]
Then $a_n\in M^s(H)$ and
\[
f(x)\le a_n(x)\le f(x)+2^{-n},\quad x\in \ext X
\]
(cf. the proof of \cite[Theorem II.7.2]{alfsen}).
Hence $(a_n)$ is uniformly convergent sequence on $X$ whose limit $a\in M^s(H)$ satisfies 
$f(x)=a(x)$ for each $x\in \ext X$. This proves the existence. Uniqueness follows from the fact that $H$ is determined by extreme points.

%Set 
%$$b_{n,i}=2^{-n}\Upsilon_{F_{n,i}}(a|_{F_{n,i}})$$
%(we use the notation from Section~\ref{sec:splifaces}).
%By Lemma~\ref{l:rozsir-splitface} we deduce that $b_{n,i}\in H$.
%can find a function $b_{n,i}\in H$ such that
%\[
%h_{n,i}(x)=\begin{cases}
%                 2^{-n}a(x),& x\in F_{n,i},\\
%                 0,& x\in F_{n,i}'.
%\end{cases}
%\]
%Then 
%\[
%b_n=\sum_{i=1}^{2^n}b_{n,i}
%\]
%are elements of $H$ such that 
%\[
%f(x)a(x)\le b_n(x)\le a(x)f(x)+a(x)2^{-n},\quad x\in \ext X
%\]
%(cf. the proof of \cite[Theorem II.7.2]{alfsen}).
%Hence $(b_n)$ is uniformly convergent sequence on $X$ whose limit $b\in H$ satisfies 
%$f(x)a(x)=b(x)$ for each $x\in \ext X$.

$(ii)$: Let $F$ be a closed split face. It follows from Lemma~\ref{l:rozsir-splitface} that $\lambda_F\in M^s(H)$.
%
%Then $F\cap\ext X$ is a facially closed set, so $1_{F\cap\ext X}$ is a facially upper semicontinuous function. By assertion (i) this function may be extended to some $b\in M^s(H)$. Then clearly
Then $F=[\lambda_F>0]\cap \ext X$ and $\ext X\setminus F=[\lambda_F<1]\cap\ext X$ belong to $\A^s_H$. So, we conclude by applying properties of $\A^s_H$ from Proposition~\ref{p:system-aha}.

$(iii)$: Let $u\in M(A_c(X))$. By \cite[Theorem II.7.10]{alfsen} the restriction $u|_{\ext X}$ is facially continuous, hence, a fortiori, facially upper semicontinuous. By $(i)$ we deduce that $u\in M^s(H)$.
\end{proof}

\begin{thm}
\label{t:mh-asfragmu}
Let $H$ be one of the spaces
$$(A_s(X))^\mu,(A_s(X))^\sigma,(A_b(X)\cap \Bo_1(X))^\mu, (A_f(X))^\mu.$$
\begin{enumerate}[$(i)$]
    \item Any bounded  facially upper semicontinuous function on $\ext X$ may be (uniquely) extended to an element of $M^s(H)$.
    \item $\A^s_{H}$ contains the $\sigma$-algebra of facially Borel sets.
    \item $M(A_c(X))\subset M^s(H)$.
\end{enumerate}
%Let  $f\colon \ext X\to [0,1]$ be a facially upper semicontinuous function. Then for each $a\in A_f(X)$ positive there exists $b\in A_f(X)$ such that $b=af$ on $\ext X$. If moreover $a\in \Bo_1(X)$, we obtain that $b\in\Bo_1(X)$. If $a\in A_s(X)$, then $b$ is in the uniform closure of functions in  $A_s(X)$.
\end{thm}

\begin{proof}
  Assertions $(i)$ and $(iii)$ follow from Theorem~\ref{t:mh-asfrag} together with Proposition~\ref{p:multi-pro-mu}. To prove $(ii)$ we moreover use that $\A^s_{H}$ is a $\sigma$-algebra.  
\end{proof}

We continue by some natural open problems.

\begin{ques}\label{q:fr-split}
    Let $X$ be a compact convex set and let $H$ be one of the spaces
    $$\overline{A_s(X)},(A_s(X))^\mu,(A_s(X))^\sigma,A_f(X),(A_f(X))^\mu.$$
    Is $\A_H=\ms_H$?
\end{ques}

\begin{remarks}
   (1) If $X$ is a simplex, a positive answer for the first three spaces is provided by Proposition~\ref{P:As pro simplex}. But we have no idea how to attack the general case.

   (2) For the last two spaces the answer is positive for a class of simplices addressed in Section~\ref{sec:stacey} (see Proposition~\ref{P:dikous-af-new} and Proposition~\ref{P:dikous-Afmu-new} below).

   (3) The analogous question for spaces $A_b(X)\cap\Bo_1(X)$ and $(A_b(X)\cap\Bo_1(X))^\mu$ has negative answer, even if $X$ is a simplex (see Proposition~\ref{P:dikous-Bo1-new} and Proposition~\ref{P:dikous-bo1-mu-new} below).
\end{remarks}

The next problem is related to inclusion of spaces of multipliers.

\begin{ques}
    Let $X$ be a compact convex set. Is it true that
    $$\begin{gathered}
        M^s(\overline{A_s(X)})\subset M^s(A_b(X)\cap \Bo_1(X))\subset M^s(A_f(X)),\\
        M((A_c(X))^\mu)\subset M^s((A_s(X))^\mu)\subset M^s((A_b(X)\cap \Bo_1(X))^\mu)\subset M^s((A_f(X))^\mu)?
    \end{gathered}$$
    Do the analogous inclusions hold for the spaces of multipliers?
\end{ques}

We note that these inclusions hold within the class of simplices addressed in Section~\ref{sec:stacey} (see Proposition~\ref{p:vztahy-multi-ifs} below).

%%%%%%%%%%%%%%%%%%%%%%%%%%%%%%%%%%%%%%%%%%%%%%%%%%
\subsection{Affine functions on compact convex sets with $F_\sigma$ boundary}
\label{ssce:fsigma-hranice}

In this section we prove a result on transferring some topological properties of strongly affine functions from an $F_\sigma$ set containing $\ext X$ to the whole set $X$.
This result will be applied 
 to get an analogue of  Theorem~\ref{t:a1-lindelof-h-hranice} for affine functions of the first Borel class and for affine fragmented functions on compact convex sets with $F_\sigma$ boundary.
 %The promised characterization is contained in the following theorem, which may be seen as an analogue of \cite[Theorem 6.4]{lusp}.

The promised transfer result is the following theorem which may be seen as a generalization of the respective cases of \cite[Theorem 3.5]{lusp} (where a transfer of properties from $\overline{\ext X}$ to $X$ is addressed). We think that a full generalization may be proved, but we restrict ourselves to first class functions which is the case important in the context of the present paper.

\begin{thm}
  \label{t:prenos-fsigma}
 Let $X$ be a compact convex set,  $F\supset \ext X$ be an $F_\sigma$ set and $f\in A_{sa}(X)$. Then the following assertions are valid.
 \begin{enumerate}[$(i)$] 
 \item If $f|_{F}\in \Bo_1(F)$, then $f\in \Bo_1(X)$.
 \item  If $f|_{F}\in \Fr(F)$, then $f\in \Fr(X)$.
 \end{enumerate}
\end{thm}

\begin{proof}
Without loss of generality we may assume that $\norm{f}\le 1$.

$(i)$: We will show that $f$ can be approximated uniformly by a function in $\Bo_1(X)$. To this end, let $\eta>0$ be given. Let $r\colon  M_1(X)\to X$ denote the barycenter mapping.  We write $F=\bigcup_{n=1}^\infty F_n$, where the sets $F_n$ are closed and satisfy $F_1\subset F_2\subset F_3\subset F_4\subset \cdots$. 

 Set $M_n=\{\omega\in M_1(X)\setsep \omega(F_n)\ge 1-\eta\}$. Then $M_n$ are closed sets in $M_1(X)$. Further, $N_n=r(M_n)$, $n\in\en$, are closed sets in $X$ satisfying $\bigcup_{n=1}^\infty N_n=X$. (Indeed, for $x\in X$ pick a maximal measure $\nu\in M_x(X)$. Then $\nu(F)=1$, and thus there exists  $n\in\en$ satisfying $\nu(F_n)\ge 1-\eta$.)

Fix $n\in\en$. Let $g_n\colon M_n\to \er$ be defined as $g_n(\mu)=\mu(f|_{F_n})$. Then $g_n\in \Bo_1(M_n)$ by Lemma~\ref{L:function space}$(b)$. Using Lemma~\ref{l:frag-selekce} we find a mapping $s_n\colon N_n\to M_n$ such that $g_n\circ s_n\in \Bo_1(N_n)$ and $r(s_n(x))=x$, $x\in N_n$. Then  we have for $x\in N_n$ an estimate
\[
\begin{aligned}
\abs{g_n(s_n(x))-f(x)}&=\abs{(s_n(x))(f|_{F_n})-f(r(s_n(x))}=\abs{(s_n(x))(f|_{F_n})-s_n(x)(f)}\\
&=\abs{\int_{X\setminus F_n} f\di (s_n(x))}\le \eta.
\end{aligned}
\]
Hence $\abs{g_n\circ s_n-f}\le \eta$ on $N_n$.

Now we denote $N_0=\emptyset$ and define a function $g\colon X\to \er$ as
\[
g(x)=g_n(s_n(x)),\quad x\in N_{n}\setminus N_{n-1}, n\in\en.
\]
Then $g$ is $(F\wedge G)_\sigma$-measurable on $X$, because 
\[
g^{-1}(U)=\bigcup_{n=1}^\infty \left((g_n\circ s_n)^{-1}(U)\cap (N_n\setminus N_{n-1})\right)
\]
is of type $(F\wedge G)_\sigma$ in $X$ for every open $U\subset \er$.
Thus $g\in \Bo_1(X)$ and satisfies $\norm{f-g}\le \eta$. Since $\eta>0$ is arbitrary, we deduce $f\in \Bo_1(X)$.

$(ii)$: We proceed as in $(i)$, so let $M_n$ and $N_n$ be as above. Then $g_n$ defined as in $(i)$ is a fragmented function on $M_n$ by Lemma~\ref{L:function space}$(b)$. We use again Lemma~\ref{l:frag-selekce} to obtain mappings $s_n\colon N_n\to M_n$ such that $r(s_n(x))=x$, $x\in N_n$, and $g_n\circ s_n\in \Fr(N_n)$. Then $g_n\circ s_n$ is $\H_\sigma$-measurable on $N_n$ by Theorem~\ref{T:a}$(a)$. As above we see that $\abs{g_n\circ s_n-f}\le \eta$ on $N_n$. The function
\[
g(x)=g_n(s_n(x)),\quad x\in N_{n}\setminus N_{n-1}, n\in\en,
\]
then satisfies $\norm{g-f}\le \eta$, and it is $\H_\sigma$-measurable. Indeed, for any $U\subset \er$ open we have
\[
g^{-1}(U)=\bigcup_{n=1}^\infty \left((g_n\circ s_n)^{-1}(U)\cap (N_n\setminus N_{n-1})\right),
\]
which is a set of type $\H_\sigma$. Hence $g\in \Fr(X)$ by Theorem~\ref{T:a}. Since $\eta>0$ is arbitrary, we deduce $f\in\Fr(X)$.
\end{proof}

Now we can formulate the main results of this section. The first one applies to the functions of the first Borel class.

\begin{thm}
\label{t:Bo1-fsigma-hranice}
 Let $X$ be a compact convex set with $\ext X$ being an $F_\sigma$ set and $H=A_b(X)\cap\Bo_1(X)$. Then $M(H)=M^s(H)$ is determined by
  \[
  \begin{aligned}
  \B=\{\ext X\setminus F\setsep &F\text{ is an }\FsG_\delta\text{ measure convex split face}\\
  &\text{such that }F'\text{ is measure convex}\}.
  \end{aligned}
  \]
 Consequently, $M^s(H)$ is determined by $\ms_H$, and thus for each intermediate function space $H' \subset H$, $M(H')\subset M(H)$.    
\end{thm}

\begin{proof}
   The `consequently' part follows as in Theorem~\ref{t:a1-lindelof-h-hranice}: Indeed, it is easy to check that $\ms_H \subset \B$. By Theorem~\ref{T:meritelnost-strongmulti} we know that $\A^s_H\subset\ms_H$. By the first part of the theorem, $M^s(H)$ is determined by $\B$. Further,  by Proposition~\ref{P:determ-Afaj} $M^s(H)$ is determined also by $\A_H^s$. It follows that $M^s(H)$ is determined by $\ms_H$. Now it is enough to use Proposition \ref{P: silnemulti-inkluze}. 

   For the first part of the theorem, we use  Proposition~\ref{P:meritelnost multiplikatoru pomoci topologickych split facu} for $H=A_b(X)\cap\Bo_1(X)$ and $T=\Bo_1^b(\ext X)$.  First we need to verify that $T$ satisfies conditions $(i)-(iii)$ of the quoted proposition. Condition $(i)$  follows easily from the fact that functions of the first Borel clas are stable with respect to the pointwise multiplication. Further, condition $(ii)$ follows from Lemma~\ref{L:YupcapYdown=Y} and condition $(iii)$ is obvious.
The formula for $H$
follows from Theorem~\ref{t:prenos-fsigma} due to the assumption on $\ext X$. 

We deduce that the assumptions of Proposition~\ref{P:meritelnost multiplikatoru pomoci topologickych split facu} are fulfilled and hence $M^s(H)$ is determined by the system $\B^s_H$. Hence, to complete the proof it is enough to show that $\B\subset\B^s_H$. 
So, let $F$ be an $\FsG_\delta$ face such that both $F$ and $F'$ are measure convex. Let $(R_n)$ be a sequence of $\FsG$ sets in $X$ such that $X\setminus F=\bigcup_n R_n$. Then $1_{(R_1\cup\dots\cup R_n)\cap\ext X}\nearrow 1_{F'\cap\ext X}=\lambda_{F'}|_{\ext X}$, thus $\lambda_{F'}|_{\ext X}\in T^\uparrow$. Finally, using implication $(i)\implies(iii)$ of Theorem~\ref{t:srhnuti-splitfaceu-metriz} we complete the proof that
$\ext X\setminus F=F'\cap \ext X\in\B^s_H$.
\end{proof}

\begin{thm}
\label{t:af-fsigma-hranice}
 Let $X$ be a compact convex set with $\ext X$ being an $F_\sigma$ set and $H=A_f(X)$. Then $M(H)=M^s(H)$ is determined by
  \[
  \begin{aligned}
  \B=\{\ext X\setminus F\setsep &F\text{ is an }\H_\delta\text{ measure convex split face}\\
  &\text{such that }F'\text{ is measure convex}\}.
  \end{aligned}
  \]
 Consequently, $M^s(H)$ is determined by $\ms_H$, and thus for each intermediate function space $H' \subset H$, $M(H')\subset M(H)$.    
\end{thm}

\begin{proof} The proof is completely analogous to that of Theorem~\ref{t:Bo1-fsigma-hranice}, we just replace functions of the first Borel class by fragmented functions and $\FsG$ sets by resolvable sets. In the proof we additionally use Theorem~\ref{T:a} to identify fragmented functions with $\H_\sigma$-measurable functions.

In this way we may simply copy the above proof except for checking the validity of condition $(ii)$ of Proposition~\ref{P:meritelnost multiplikatoru pomoci topologickych split facu}:
 We additionally observe that by \cite[Th\'eor\`eme 2]{tal-kanal} $\ext X$, being $F_\sigma$, is necessarily an $\FsG_\delta$ set, hence an $\H_\delta$ set. Since it is simultaneously an $\H_\sigma$ set, it easily follows from \cite[Proposition 2.1(iii)]{koumou} that $\ext X$ is in fact a resolvable set, hence hereditarily Baire. Hence, by Theorem~\ref{T:a}$(b)$ the space $T=\Fr^b(\ext X)$ coincides with bounded $\H_\sigma$-measurable functions on $\ext X$. Therefore, condition $(ii)$ follows from Lemma~\ref{L:YupcapYdown=Y}. 
\end{proof}

We continue by a natural open problem.

\begin{ques}
Do Theorem~\ref{t:Bo1-fsigma-hranice} and Theorem~\ref{t:af-fsigma-hranice} hold under the assumption that $\ext X$ is a \lin\ resolvable set?
\end{ques}

If $\ext X$ is $F_\sigma$, then $\ext X$ is Lindel\"of (being $\sigma$-compact) and resolvable (see the proof of Theorem~\ref{t:af-fsigma-hranice}). Hence the positive answer would give a natural generalization of the two above theorems and would be 
a more precise analogue of Theorem~\ref{t:a1-lindelof-h-hranice}.
The key problem is the possibility to generalize Theorem~\ref{t:prenos-fsigma}.

\subsection{Multipliers of strongly affine functions}
\label{ssce:x-kanalytic}

 $A_{sa}(X)$, the space of strongly affine functions should serve as a natural roof for intermediate function spaces determined by extreme points. However, as we have already noticed above, the situation is more difficult. Firstly, by \cite{talagrand} strongly affine functions need not be determined by extreme points and, secondly, due to examples from Section~\ref{sec:strange} below there are intermediate function spaces which are determined by extreme points but not contained in $A_{sa}(X)$. Anyway, the space $A_{sa}(X)$ remains to be a natural object of interest.
The following natural question seems to be open.

\begin{ques}
    Let $X$ be a compact convex set such that $A_{sa}(X)$ is determined by extreme points. Is $M(A_{sa}(X))=M^s(A_{sa}(X))$?
\end{ques}

The answer is positive if $X$ is a standard compact convex set
(see Proposition~\ref{P:rovnostmulti}). But there are compact convex sets which are not standard but still strongly affine functions are determined by extreme points, see, e.g., Proposition~\ref{P:dikous-sa-new} and Lemma~\ref{L:maxmiry-dikous}$(2)$. For these examples, the answer remains to be positive (see Proposition~\ref{p:vztahy-multi-ifs}$(d)$). 

Another intriguing question is the following:

\begin{ques}
     Let $X$ be a compact convex set such that $H=A_{sa}(X)$ is determined by extreme points. Is $\A^s_H=\ms_H$?
     Is it true at least assuming $X$ is standard?
\end{ques}

The answer is positive for the class of simplices addressed in Section~\ref{sec:stacey}, see Proposition~\ref{P:shrnutidikousu}.
The general case remains to be open, but there is one more special case when the answer is positive -- if $\ext X$ is $K$-analytic.
This is the content of the following theorem.
%, the strong multipliers of the intermediate function space $A_{sa}(X)$ can be characterized in the spirit of Theorem~\ref{T:meritelnost-strongmulti}. 

\begin{thm}
\label{t:metriz-sa-splitfacy}
Let $X$ be a  compact convex set with $\ext X$ being $K$-analytic and $H=A_{sa}(X)$. Then 
\[\begin{aligned}
\A_H&=\A^s_H=\ms_H\\&=\{F\cap \ext X\setsep F\text{ is a split face such that both }F,F'\text{ are measure convex}\}.  
\end{aligned}\]
%Then $\A^s_H=\A$, and hence $M(H)=M^s(H)$ is determined by the system $\A$. 
Consequently, $M(H')\subset M(H)$  for each intermediate function space $H' \subset H$.    
\end{thm}

\begin{proof} Since $X$ is standard, $\A_H=\A^s_H$ by Proposition~\ref{P:rovnostmulti}. Inclusion $\A^s_H\subset\ms_H$ follows from Theorem~\ref{T:meritelnost-strongmulti}. The last equality follows from Theorem~\ref{t:srhnuti-splitfaceu-metriz} (equivalence $(iv)\Longleftrightarrow(i)$). Finally, if $F$ is a split face such that both $F$ and $F'$ are measure convex, Theorem~\ref{t:srhnuti-splitfaceu-metriz} (implication $(i)\implies (iii)$) shows that $\lambda_F\in M(H)$. Therefore $F\cap\ext X=[\lambda_F=1]\cap\ext X\in\A_H$. This completes the proof.

 Since $M^s(H)$ is determined by $\A^s_H$ due to Theorem~\ref{T:integral representation H}$(ii)$ and $\A^s_H=\ms_H$, 
 the `consequently part' follows  from Proposition \ref{P: silnemulti-inkluze} (using again Proposition~\ref{P:rovnostmulti}).
\end{proof}

%%%%%%%%%%%%%%%%%%%%%%%%%%%%%%%%%%%%%%%%%%%%%%%%%%%%
\section{Examples of strange intermediate function spaces}\label{sec:strange}

In this section we collect three examples. The first two show, in particular, that there is no relationship between strongly affine functions and determinacy by extreme points. The third one shows that split faces need not induce multipliers (even on a metrizable Bauer simplex).

Let us now pass to the first example. We note that most of the intermediate function spaces we address in this paper are formed by strongly affine functions
and are determined by extreme points. But in general, these two properties are incomparable. On one hand, by \cite{talagrand} there is a strongly affine function not determined by extreme points. It is not hard to find one affine function which is not strongly affine but is determined by extreme points. We present something more -- intermediate function spaces generated by such functions. The first example is the following:

\begin{example}
 \label{ex:deter-ext-body}
 Let $X=M_1([0,1])$ and let $\II$ stand for the set of all irrational points in $[0,1]$. We define a function $G\colon X\to [0,1]$ by $$G(\mu)=\mu_d(\II), \quad\mu\in X,$$
 where $\mu_d$ denotes the discrete part of $\mu$. Let  $H=\span (A_c(X)\cup\{G\})$. Then $X$ is a metrizable Bauer simplex and $H$
 is an intermediate function space on $X$ with the
  following properties.
 \begin{enumerate}[$(a)$]
\item  $H\subset\Ba_2(X)$;
\item  $G$ is not strongly affine and hence $H$ is not contained in $A_{sa}(X)$;
\item  $H$ is determined by $\ext X$ but $H^\mu$ is not determined by $\ext X$; 
\item  $M(H)$ contains only constant functions.
 \end{enumerate}
\end{example}

\begin{proof}
It is clear that $X$ is a metrizable Bauer simplex and $H$ is an intermediate function space on $X$.

To prove properties $(a)$ and $(b)$ it is enough to show that  $G$ is a Baire-two function on $X$ that is not strongly affine.

To see this, we notice that the function $\mu\mapsto \mu(\{x\})$, $\mu\in X$, is upper semicontinuous on $X$ for each $x\in [0,1]$. Hence the function
\[
G_1(\mu)=\mu_d(\qe)=\sum_{q\in\qe} \mu(\{q\}),\quad \mu\in X,
\]
is Baire-two. Since $G(\mu)=\mu_d([0,1])-G_1(\mu)$ and $\mu\mapsto \mu_d([0,1])$ is Baire-two (see \cite[Corollary I.2.9]{alfsen}), the function $G$ is Baire-two as well.

To verify that $G$ is not strongly affine, let $\lambda$ denote the Lebesgue measure on $[0,1]$. If $\phi\colon [0,1]\to X$ is the evaluation mapping from Section~\ref{ssc:ch-fs}, then $\phi(\lambda)$ is a probability measure on $X$ whose barycenter is $\lambda$ (see \cite[Proposition 2.54]{lmns}).  Then $G(\lambda)=\lambda_d(\II)=0$. On the other hand, for $t\in\II$ we have $G(\ep_t)=1$. Since $\phi(\lambda)(\{\ep_t\setsep t\in\II\})=1$, we get $\int_X G(\mu)\di\phi(\lambda)(\mu)=1$. Hence $G$ is not strongly affine.

$(c)$: Note that $H=\{F+cG\setsep F\in A_c(X), c\in\er\}$. 
We will show that $H$ is determined by extreme points. To this end, let $F\in A_c(X)$ and $c\in \er$ be given. By Section~\ref{ssc:ch-fs}, there exists a function $f\in C([0,1])$ such that $F(\mu)=\mu(f)$, $\mu\in X$. Then the function $K=F+cG$ satisfies
\[
K(\ep_t)=\begin{cases} f(t)+c,& t\in \II,\\ 
              f(t),& t\in \qe.
\end{cases}
\]
Since $\sup f(\qe)=\sup f(\II)=\sup f([0,1])$, we obtain
\[
\sup K(\ext X)=\max\left\{c+\sup f(\qe), \sup f(\II)\right\}=\max\left\{c+\sup f([0,1]), \sup f([0,1])\right\}
\]
Analogously we obtain that
\[
\inf K(\ext X)=\min\left\{c+\inf f(\qe), \inf f(\II)\right\}=\min\left\{c+\inf f([0,1]), \inf f([0,1])\right\}.
\]
For any $\mu\in X$ we have
\[
\inf f([0,1])+cG(\mu)\le K(\mu)=\mu(f)+cG(\mu)\le \sup f([0,1])+cG(\mu).
\]
If $c\ge 0$, we have 
\[
\inf K(\ext X)=\inf f([0,1])\le K(\mu)\le \sup f([0,1])+c=\sup K(\ext X).
\]
If $c<0$, we obtain
\[
\inf K(\ext X)=\inf f([0,1])+c\le  K(\mu)\le \sup f([0,1])=\sup K(\ext X).
\]
Hence $H$ is determined be extreme points.

 We continue by showing that $H^\mu$ is not determined by extreme points. 
Let $f=-1_{\qe}$ and $c=-1$. Then the function $K(\mu)=\mu(f)-G(\mu)$, $\mu\in X$, is in $H^\mu$ (as $1_{\qe}\in (C([0,1])^\mu$). Further,
\[
K(\ep_t)=-1,\quad t\in [0,1],
\]
but for the Lebesgue measure $\lambda$ we obtain
$K(\lambda)=0.$
Hence $H^\mu$ is not determined by extreme points.

$(d)$: Given $K\in H$ we set $\widehat K(t)=K(\ep_t)$ for $t\in[0,1]$. Any such function is of the form $\widehat K=f+c 1_{\II}$ for some $f\in C([0,1])$ and $c\in\er$. Then
\begin{equation}\label{eq:rozdil}
    \limsup_{s\to t} \widehat K(s)-\liminf_{s\to t}\widehat K(s)=\abs{c}
\mbox{ for each }t\in[0,1].\end{equation}

 Assume now $K\in H$ is a non-constant function. Then $\widehat{K}$ is not constant either. Let $\widehat K=f+c 1_{\II}$ as above. There are two possibilities:

 Case 1: $c=0$ and $f$ is not constant.  Then there is no $L\in H$ with $L=KG$ on $\ext X$. Indeed, we would have $\widehat{L}=f\cdot 1_{\II}$ and the difference from \eqref{eq:rozdil} would not be constant on $[0,1]$.

 Case 2: $c\ne 0$. Let $F(\mu)=\int_{[0,1]} t\di\mu(t)$ for $\mu\in X$. Then there is no $L\in H$ with $L=KF$ on $\ext X$. Indeed, we would have
 $$\widehat L(t)=tf(t)+t 1_{\II}(t), \qquad t\in[0,1],$$
 and, again, the difference from \eqref{eq:rozdil} would not be constant on $[0,1]$.

 Thus $K\notin M(H)$ and the proof is complete.
\end{proof}

Before presenting another variant of the preceding example  we need the following lemma.

\begin{lemma}
\label{l:baire-one-cont}    
Let $f\in \Ba_1([0,1])$ be such that the restrictions  $f|_{\qe}$ and $f|_{\II}$ are continuous. Then $f\in C([0,1])$.
\end{lemma}

\begin{proof}
Assume that $x\in \qe$ is such that $f$ is not continuous at $x$. Then there exist $\eta>0$ and a sequence $(x_n)$ of irrational numbers converging to $x$ such that $\abs{f(x)-l}>\eta$, where $l=\lim f(x_n)$. Let $U$ be a neighborhood of $x$ such that $\abs{f(x)-f(y)}<\frac\eta4$ for each $y\in U\cap \qe$. Pick $n\in\en$ such that $x_n\in U$ and $\abs{f(x_n)-l}<\frac\eta4$.  Let $V\subset U$ be a neighborhood of $x_n$ such that $\abs{f(x_n)-f(y)}<\frac\eta4$ for each $y\in V\cap \II$. 
Then $f$ has no point of continuity on $V$ because 
\[
f(y)\in (f(x)-\tfrac\eta4,f(x)+\tfrac\eta4),\quad y\in V\cap \qe,
\]
and
\[
f(y)\in (l-\tfrac\eta4,l+\tfrac\eta4), \quad y\in V\cap \II.
\]
But this is a contradiction with the fact that $f\in \Ba_1([0,1])$.

The case $x\in \II$ can be treated similarly, thus the proof is done.
\end{proof}

\begin{example}
\label{ex:hacko-mu-deter} Let $Y=B_{M([0,1])}$.  We define a function $G\colon Y\to [0,1]$ by $$G(\mu)=\mu_d(\II), \quad\mu\in Y,$$
where we use notation from Example~\ref{ex:deter-ext-body}. Set $H=\span\{A_c(Y)\cup\{G\})$. Then $Y$ is a metrizable compact convex set set and $H$ is an  intermediate function space on $Y$ with the following properties:
\begin{enumerate}[$(a)$]
\item  $H\subset \Ba_2(Y)$;
\item  $H$ is not contained in $A_{sa}(Y)$;
\item  $H$ is determined by $\ext X$;
\item  $H=H^\mu$;
\item  $M(H)$ contains only constant functions.
 \end{enumerate}
%$H=H^\mu$ , $H$ is determined by extreme points and $H$ is not contained in $A_{sa}(Y)$.
\end{example}

\begin{proof} 
It is clear that $Y$ is a metrizable compact convex space and $H$ is an intermediate function space on $Y$.

Let $X$ be the compact convex set from Example~\ref{ex:deter-ext-body}. We observe that $X\subset Y$ and $Y=\co (X\cup -X)$. Then $\ext Y=(\ext X)\cup (-\ext X)$. Further, $G|_X$ coincide with the function $G$ from Example~\ref{ex:deter-ext-body}.

$(a)$: We need to show that $G$ is a Baire-two function on $Y$. We know that $G|_X$ and $G|_{-X}$ are Baire-two functions (by Example~\ref{ex:deter-ext-body}(a)) and $G$ is affine. Consider the mapping 
\[
q\colon X\times(-X)\times[0,1]\to Y,\quad q(x_1,-x_2,\lambda)=\lambda x_1-(1-\lambda)x_2.
\]
Then $q$ is a continuous surjection and $\tilde{G}(x_1,-x_2,\lambda)=\lambda G(x_1)-(1-\lambda)G(x_2)$ is Baire-two. Since $\tilde{G}=G\circ q$, $G$ is Baire-two by \cite[Theorem 5.26]{lmns}.

$(b)$: The function $G$ is not strongly affine as Example~\ref{ex:deter-ext-body}(b) says that $G|_X$ is not strongly affine.

$(c)$: The space $H$ is determined by extreme points.
Let $F\in A_c(Y)$ and $c\in\er$ be given. Then there exist a function $f\in C([0,1])$ and $d\in\er$ such that $(F+cG)(\mu)=\mu(f)+d+cG(\mu)$ for $\mu\in Y$. From Example~\ref{ex:deter-ext-body}(c) we know that
$$\sup (F+cG)(X)\le \sup (F+cG)(\ext X)\le \sup (F+cG)(\ext Y)$$
and
$$\begin{aligned}
\sup(F+cG)(-X)&=-\inf(F+cG)(X)\le-\inf(F+cG)(\ext X)\\&=\sup(F+cG)(-\ext X)\le \sup (F+cG)(\ext Y).\end{aligned}$$
Since $F+cG$ is affine and $Y=\co(X\cup(-X))$, we deduce 
$\sup(F+cG)(Y)\le\sup (F+cG)(\ext Y)$. Similarly we infer the inequality for infimum.

$(d)$: Let a bounded non-decreasing sequence $(K_n)$ in $H$ pointwise converge to $K$ on $Y$. Then $K$ is affine on $Y$, and thus we can assume by adding a suitable constant that $K(0)=0$. We find continuous functions $k_n\in C([0,1])$, $d_n, c_n\in\er$ such that $K_n(\mu)=\mu(k_n)+d_n+c_nG(\mu)$.
Then $d_n=K_n(0)\nearrow K(0)=0$, and thus $(d_n)$ converges to $0$.
In particular, the sequence $(d_n)$ is bounded. By the very assumption the sequence $(K_n)$ is bounded, so $(K_n-d_n)$ is bounded as well. For $t\in\qe$ we have $K_n(\ep_t)-d_n=k_n(t)$, so the sequence $(k_n|_{\qe})$ is bounded. Since each $k_n$ is continuous, $\norm{k_n}_\infty=\norm{k_n|_{\qe}}_\infty$, so the sequence $(k_n)$ is bounded as well. Now we deduce that also the sequence $(c_n)$ is bounded.  By choosing a suitable subsequence we may assume that $c_n\to c\in\er$. 

\iffalse
For $t\in \qe$ we have 
\[
K_n(\ep_t)=k_n(t)+d_n\nearrow K(\ep_t).
\]
Since $K$ is bounded, $(\abs{k_n(t)})$ is bounded by $M=\norm{K}+\norm{(d_n)}_\infty$. Hence $\sup_{t\in [0,1]} \abs{k_n(t)}=\sup_{t\in\qe} \abs{k_n(t)}\le M$. If we fix $t\in\II$, then 
\[
K_n(\ep_t)=k_n(t)+d_n+c_n\nearrow K(\ep_t)
\]
implies that $(c_n)$ is bounded. By choosing a suitable subsequence we may assume that $c_n\to c$. 
\fi

Now we have for $t\in\qe$
\[
\begin{aligned}
K_n(\ep_t)&=k_n(t)+d_n\nearrow K(\ep_t)\quad\text{and}\\
 K_n(-\ep_t)&=-k_n(t)+d_n\nearrow K(-\ep_t)=-K(\ep_t).
\end{aligned}
\]
Hence both the functions $t\mapsto K(\ep_t)$ and $t\mapsto -K(\ep_t)$ are limits of non-decreasing sequences of continuous functions on $\qe$, and hence the function $t\mapsto K(\ep_t)$ is continuous on $\qe$.

Similarly we have for $t\in\II$
\[
\begin{aligned}
K_n(\ep_t)&=k_n(t)+d_n+c_n\nearrow K(\ep_t)\quad\text{and}\\
K_n(-\ep_t)&=-k_n(t)+d_n-c_n\nearrow K(-\ep_t)=-K(\ep_t).
\end{aligned}
\]
As above, $t\mapsto K(\ep_t)$ is continuous on $\II$. Thus the function $k(t)=K(\ep_t)$, $t\in [0,1]$, is continuous on $\qe$ and on $\II$. Further, the functions $k_n$ satisfy
\[
k_n(t)= \begin{cases} k_n(t)+d_n-d_n=K_n(\ep_t)-d_n\to k(t),& t\in \qe,\\
                      k_n(t)+d_n+c_n-d_n-c_n=K_n(\ep_t)-d_n-c_n\to k(t)-c,& t\in\II.
                      \end{cases}
\]
Hence the function $l(t)=k(t)-c1_{\II}$ is Baire-one on $[0,1]$. Also, $l|_{\qe}$ and $l|_{\II}$ are continuous. It follows from Lemma~\ref{l:baire-one-cont} that $l\in C([0,1])$. 

Hence the functions $L_n(\mu)=\mu(k_n)$ and $L(\mu)=\mu(l)$, $\mu\in Y$, satisfy fro $\mu\in Y$
\[
L(\mu)=\lim_{n\to\infty}L_n(\mu)=\lim_{n\to\infty}\left(\mu(k_n)+d_n+c_nG(\mu)-d_n-c_nG(\mu)\right)
=K(\mu)-cG(\mu).
\]
Since $L$ is strongly affine, $K-cG$ is strongly affine as well. Since 
\[
\begin{aligned}
L(\ep_t)&=\begin{cases} k(t), &t\in\qe,\\
                            k(t)-c, &t\in\II,
              \end{cases},\\
L(-\ep_t)&=\begin{cases} -k(t), &t\in\qe,\\
                            -k(t)+c, &t\in\II,            
              \end{cases}                                           
\end{aligned}              
\]
we have $L(\mu)=\mu(l)$, $\mu\in\ext Y$. But the function on the right hand side is continuous on $\ext Y$, and thus $L$ is continuous on $\ext Y$. By \cite[Corollary 5.32]{lmns}, $K-cG$ is continuous on $Y$. Hence $K=(K-cG)+cG\in H$.

$(e)$: $M(H)$ contains only constant functions by Example~\ref{ex:symetricka}.
\end{proof}

We continue by the third example showing that split faces
need not generate multipliers even on a metrizable Bauer simplex.

\begin{example}
\label{ex:multi-analytic}
There exist an intermediate function space $H$ on the metrizable Bauer simplex $X=M_1([0,1])$  and a subset $F\subset X$ such that the following properties hold.
\begin{enumerate}[$(a)$]
\item   $H^\sigma=H$, $H\subset A_{sa}(X)$ and $H$ is determined by extreme points.
\item  $F$ is a  measure convex split face, its complementary face $F'$ is also measure convex and $\lambda_F\in H$.
\item  $\lambda_F\notin M(H)$.
\end{enumerate}
\end{example}

\begin{proof}
$(a)$: Let $\tilde{H}=\span(\Ba^b([0,1])\cup\{\tilde{f}\})$, where $\tilde{f}=1_A$ for $A=A_1\cup A_2$ with $A_1\subset [0,\frac13]$ analytic non-Borel and $A_2\subset [\frac23,1]$ analytic non-Borel.
Using Lemma~\ref{L:function space} we obtain an intermediate function space $H=V(\tilde{H})$ on the Bauer simplex $X=M_1([0,1])$ with $H\subset A_{sa}(X)$. Since $X$ is a standard compact convex set, $H$ is determined by extreme points. 

Now we show that $H^\sigma=H$. To this end, let  a bounded sequence $h_n=b_n+c_n f$, where $b_n=V(\tilde{b}_n)$ for some $\tilde{b}_n\in \Ba^b([0,1])$, $c_n\in\er$ and $f=V(\tilde{f})$, be such that $h_n\to h\in A_{sa}(X)$. If we identify points of $[0,1]$ with $\ext X$, we have $h_n\to h$ on $[0,1]$. 

We claim that $(c_n)$ is bounded. If this is not the case, we may assume that $\abs{c_n}\to \infty$. Then the set $B=\{x\in [0,1]\setsep \abs{b_n(x)}\to\infty\}$ is Borel and contains $A$. Indeed, for $x\in A$ we have
\[
\abs{b_n(x)}=\abs{b_n(x)+c_n-c_n}=\abs{h_n(x)-c_n}\ge \abs{c_n}-\norm{h_n}\to\infty.
\]
Since $A$ is not Borel, there exists $x\in B\setminus A$. Then $h_n(x)=b_n(x)\to \infty$, a contradiction. Hence $(c_n)$ is bounded.
Thus we may assume that $c_n\to c$ for some $c\in\er$. Then 
\[
b_n=h_n-c_nf\to h-cf\text{ on }\ext X.
\]
Hence $b=h-cf$ is a Borel function on $\ext X$, and thus have a Borel strongly affine extension $b$ on $X$. Further, $h=b+cf$ on $\ext X$. Since $h, b,c$ are strongly affine, $h=b+cf$ on $X$. Hence $H=H^\sigma$.

$(b)$: Let $F=[f=1]$. Then $F$ is a measure convex face with $F'=[f=0]$ (this follows from Lemma~\ref{l:complementarni-facy}(i)). We aim to check that $F$ is split face. To this end, let $s\in X$ be given. Then $s$ is a probability measure on $[0,1]$. If $s(A)=1$, then $s\in F$. Similarly, if $s([0,1]\setminus A)=1$, we have $s\in F'$. So let $s(A)\in (0,1)$. Then 
\[
s=s(A)\frac{s|_{A}}{s(A)}+(1-s(A))\frac{s|_{[0,1]\setminus A}}{s([0,1]\setminus A)},
\]
where $\frac{s|_{A}}{s(A)}\in F$ and 
$\frac{s|_{[0,1]\setminus A}}{s([0,1]\setminus A)}\in F'$. Hence this provides a decomposition of $s$ into a convex combination  of elements of $F$ and $F'$. To check its uniqueness, let $s=\lambda t+(1-\lambda)t'$, where $\lambda\in (0,1)$ and $t\in F$, $t'\in F'$. The application of the function $f$ yields $\lambda=s(A)$. Since $t\in A$, we have $1=f(t)=t(1_{{A}})=t({A})$, and thus $t$  is carried by ${A}$. Similarly we obtain that $t'$ is carried by $[0,1]\setminus {A}$.
Given a universally measurable $D\subset {A}$, we have
\[
s(D)=s({A})t(D).
\]
Hence $t=\frac{s|_{{A}}}{s({A})}$. 
Then $t'=\frac{s|_{[0,1]\setminus {A}}}{s([0,1]\setminus{A})}$ and $F$ is a split face.

Since $\lambda_F=f$, we have $\lambda_F\in H$.

$(c)$: Nevertheless, $\lambda_F=f$ is not a multiplier for $H$. Indeed, let $g\colon [0,1]\to[0,1]$ be a continuous function with $g=1$ on $[0,\frac13]$ and $g=0$ on $[\frac23,1]$. Then $\lambda_F\cdot V(g)=fg$ on $\ext X$. Thus it is enough to show that $fg\notin \{h|_{\ext X}\setsep h\in H\}$. But this follows from the fact that any function $h$ from the right hand side is of the form $h=b+c\tilde{f}$ for some bounded Borel $b$ and $c\in\er$. Then $\tilde{f}g=b+c\tilde{f}$ implies $b+c\tilde{f}=0$ on $[\frac23,1]$. This implies $c=0$ as $A_2$ is non-Borel. But then $\tilde{f}=b$ on $[0,\frac13]$, which is impossible. Hence the conclusion follows.
\end{proof}

%%%%%%%%%%%%%%%%%%%%%%%%%%%%%%%%%%%%%%%%%%%%%%%%%%%%%%%%%%%%%%%%%
\section{Examples of Stacey's compact convex sets}
\label{sec:stacey}

The aim of this section is to illustrate the abstract results obtained in the previous sections by means of the concrete examples of Stacey's simplices. Their construction provides (in particular) examples distinguishing intermediate function spaces and their multipliers.

\subsection{Construction}
\label{subs:construction}

Let $L$ be a (Hausdorff) compact topological space and let $A\subset L$ be an arbitrary subset. Let
the set
$$\gls{KLA}=(L\times\{0\}) \cup (A\times \{-1,1\})$$
be equipped with the porcupine topology. I.e., points of $A\times\{-1,1\}$ are isolated and a neighborhood basis of $(t,0)$ for $t\in L$ is formed by
\[
(U\times \{-1,0,1\})\cap K_{L,A} \setminus\{(t,-1),(t,1)\},\quad U\mbox{ is a neighborhood of $t$ in }L.
\]
Then $K_{L,A}$ is a compact Hausdorff space (this was observed in \cite[Section VII]{bi-de-le} and it is easy to check). We further set
\[
E=\gls{ELA}=\{f\in C(K_{L,A})\setsep f(t,0)=\tfrac{1}{2}(f(t,-1)+f(t,1)) \mbox{ for }t\in A\}.
\]
Then $E_{L,A}$ is a function space on $K_{L,A}$ and its Choquet boundary is
\[
\Ch_{E_{L,A}}K_{L,A}=(A\times\{-1,1\}) \cup ((L\setminus A)\times\{0\})
\]
(see \cite[formula (7.1) on p. 328]{bi-de-le}). Denote $X=\gls{XLA}=S(E_{L,A})$ the respective state space. By \cite[Theorem 3]{stacey} each $X_{L,A}$ is a simplex (see the explanation in \cite[Section 2]{kalenda-bpms}).

We will describe several intermediate function spaces, their centers and multipliers on $X$. We adopt the notation from Section~\ref{ssc:ch-fs} and from Lemma~\ref{L:function space}. In particular, $\Phi:E_{L,A}\to A_c(X_{L,A})$ will denote the canonical isometry, $\phi:K_{L,A}\to X_{L,A}$ the evaluation mapping and $V:\ell^\infty(K_{L,A})\cap (E_{L,A})^{\perp\perp} \to A_{sa}(X_{L,A})$ the isometry from Lemma~\ref{L:function space}. We further define mappings $\jmath\colon L\to K_{L,A}$ and $\psi\colon K_{L,A}\to L$ by
\[
\gls{jmath}(t)=(t,0) \mbox{ for }t\in L 
\quad\mbox{ and }\quad \gls{psi}(t,i)=t\mbox{ for }(t,i)\in K_{L,A}.
\]
Then $\jmath$ is a homeomorphic injection, $\psi$ is a continuous surjection and $\psi\circ\jmath$ is the identity on $L$.

We conclude this section by few easy formulas which we will use several times. Assume that $K=K_{L,A}$, $E=E_{L,A}$ and $B\subset \Ch_E K$ is arbitrary. Then we have the following equalities:
\begin{equation}\label{eq:podmny ChEK}
 \begin{aligned}
 \psi(B)\cap\psi(\Ch_E K\setminus B)&=\{t\in A\setsep B\cap\{(t,-1),(t,1)\}\\ &\qquad\qquad\mbox{ contains exactly one point} \},\\
 \psi(B)\setminus\psi(\Ch_E K\setminus B)&=\{t\in L\setminus A\setsep (t,0)\in B\} \\&\qquad\cup\{t\in A\setsep \{(t,-1),(t,1)\}\subset B\}
 \\&=\{t\in L\setsep\psi^{-1}(t)\cap \Ch_E K\subset B\}.\end{aligned}
\end{equation}

\subsection{Description of topological properties of functions on $K_{L,A}$}
\label{ssec:topology}

We start with a characterization of topological properties of bounded functions on $K_{L,A}$.

\begin{prop}\label{p:topol-stacey}
Let $f\in \ell^\infty(K_{L,A})$ and let $K=K_{L,A}$. Then the following assertions hold.
\begin{enumerate}[$(a)$]
    \item $f\in C(K)$ if and only if $f\circ\jmath$ is continuous on $L$ and for all $\ep>0$ the set
    \[
    \{t\in A\setsep \abs{f(t,0)- f(t,-1)}\ge\ep\mbox{ or }\abs{f(t,0)-f(t,1)}\ge\ep\}
    \]
    is  finite.        
    \item  $f\in \Ba_1^b(K)$ if and only if $f\circ\jmath\in \Ba_1^b(L)$ and 
    \[
    \{t\in A\setsep f(t,-1)\ne f(t,0)\mbox{ or }f(t,1)\ne f(t,0)\} \mbox{ is countable}\}.
    \]
     \item  $f\in \Ba^b(K)$ if and only if $f\circ\jmath\in \Ba^b(L)$ and
     \[
     \{t\in A\setsep f(t,-1)\ne f(t,0)\mbox{ or }f(t,1)\ne f(t,0)\} \mbox{ is countable}\}.
     \]
 \item  $f$ is lower semicontinuous if and only if the following two conditions are satisfied:
        \begin{enumerate}[$(i)$]
            \item $f\circ\jmath$ is lower semicontinuous on $L$;
            \item for each $\ep>0$ and each accumulation point $t_0$ of the set $$A_\ep=\{t\in A\setsep f(t,0)\ge\min\{f(t,-1),f(t,1)\}+\ep\}$$ we have $f(t_0,0)\le\liminf\limits_{t\to t_0, t\in A_\ep}f(t,0)-\ep$.
        \end{enumerate}
    \item   $f\in \Bo_1^b(K)$ if and only if $f\circ\jmath\in \Bo_1^b(L).$
 \item  $f\in\Bo^b(K)$ if and only if $f\circ\jmath\in \Bo^b(L).$
   \item  $f\in \Fr^b(K)$ if and only if $f\circ\jmath\in \Fr^b(L).$ 
   \item $f\in (\Fr^b(K))^\mu$ if and only if $f\circ\jmath\in (\Fr^b(L))^\mu.$ 
   \item $f$ is universally measurable if and only if $f\circ\jmath$ is universally measurable.
\end{enumerate}    
\end{prop}

\begin{proof}
$(a)$: This easily follows from the definitions and is proved in \cite{stacey} within the proof of Theorem 3. 

$(b)$: Inclusion `$\subset$' follows from $(a)$. To prove the converse take any $f$ satisfying the condition. Let $\{t_n\setsep n\in\en\}$ be the respective countable set. Let $(g_n)$ be a bounded sequence in $C(L)$ pointwise converging to $f\circ\jmath$. Define $f_n$ by
  $$f_n(t,i) = \begin{cases}
      f(t,i)&\mbox{if }t=t_k\mbox{ for some }k\le n\mbox{ and }i\in\{-1,1\},\\
      g_n(t) &\mbox{otherwise}.
  \end{cases}$$
  Then $(f_n)$ is a bounded sequence of continuous functions on $K$ pointwise converging to $f$.
  
$(c)$: This follows easily from the proof of $(b)$.

$(d)$: $\implies$: Assume $f$ is lower semicontinuous. Clearly, $f\circ \jmath$ is also lower semicontinuous, hence $(i)$ is satisfied. To prove $(ii)$ fix $\ep>0$ and let $t_0$ be an accumulation point of $A_\ep$. Let $(t_\alpha)$ be a net in $A_\ep$ converging to $t_0$ such that $\lim_\alpha f(t_\alpha,0)=\liminf\limits_{t\to t_0, t\in A_\ep} f(t,0)$.
 Since $t_\alpha\in A_\ep$, we may find $i_\alpha\in\{-1,1\}$ such that $f(t_\alpha,i_\alpha)\le f(t_\alpha,0)-\ep$. In the topology of $K$ we have $(t_\alpha,i_\alpha)\to(t_0,0)$, hence
 \[
 f(t_0,0)\le \liminf_{\alpha} f(t_\alpha,i_\alpha)\le  \liminf_{\alpha} f(t_\alpha,0)-\ep=\liminf\limits_{t\to t_0, t\in A_\ep}f(t,0)-\ep
 \]
    and the proof is complete.

    $\impliedby$: Assume $f$ satisfies conditions $(i)$ and $(ii)$. Let us show $f$ is lower semicontinuous. Since the points of $A\times\{-1,1\}$ are isolated, it is enough to prove the lower semicontinuity at $(t,0)$ for each $t\in L$.  So, fix any $t_0\in L$.
    It follows from $(i)$ that $f(t_0,0)\le\liminf_{s\to t_0}f(s,0)$.
    If $t_0$ is not an accumulation point of $A$, the proof is complete.

    Assume now that $t_0$ is an accumulation point of $A$. We need to prove that $f(t_0,0)\le\liminf\limits_{\alpha}f(t_\alpha, i_\alpha)$ whenever $(t_\alpha)_{\alpha\in I}$ is a net in $A\setminus \{t_0\}$ converging to $t_0$ and $i_\alpha\in\{-1,1\}$. Assume the contrary, i.e., there is a net $\left((t_\alpha, i_\alpha)\right)_{\alpha\in I}$ such that $\lim\limits_{\alpha} f(t_\alpha, i_\alpha)<f(t_0,0)$. Up to passing to a subnet we may assume that $f(t_\alpha,0)\to c\in\er$. Then $c\ge f(t_0,0)$. Let $\ep>0$ be such that $\ep<c-\lim\limits_{\alpha} f(t_\alpha, i_\alpha)$. Then $t_0$ is an accumulation point of $A_\ep$ and  hence by $(ii)$ we deduce 
    that 
    $f(t_0,0)\le \liminf\limits_{\alpha} f(t_\alpha,0)-\ep=c-\ep$. Since $\ep$ may be arbitrarily close to $c-\lim\limits_\alpha f(t_\alpha, i_\alpha)$, we infer $f(t_0,0)\le\lim\limits_\alpha f(t_\alpha,i_\alpha)$, which is a contradiction with our assumption. Hence $f$ is lower semicontinuous at $t_0$.
    
$(e)$: Since $K\setminus \jmath(L)$ is an open discrete set in $K$, any function on $K$ with $f\circ \jmath\in \Bo_1(L)$ satisfies the condition that for each $U\subset \er$ open the set $f^{-1}(U)$ is 
expressible as a countable union of sets contained in the algebra generated by open sets.

On the other hand, if $f\in \Bo_1^b(K)$, then $f\circ \jmath\in \Bo_1^b(L)$ as $\jmath$ is a homeomorphic embedding. 

$(f), (g), (h), (i)$: This also follows from the fact that any function on $K\setminus \jmath(L)$ is Borel and fragmented.
\end{proof}

\subsection{Strongly affine functions on $X_{L,A}$}
\label{ssec:sa-functions-onX}

This section is devoted to describing strongly affine functions and their multipliers on $X_{L,A}$. We start by describing maximal measures and characterizing standard compact convex set among simplices $X_{L,A}$.

\begin{lemma}\label{L:maxmiry-dikous}
    Let $K=K_{L,A}$, $E=E_{L,A}$ and $X=X_{L,A}$.
   \begin{enumerate}[$(1)$] 
   \item A measure $\mu\in M_1(X)$ is maximal if and only if $\mu=\phi(\nu)$ for some $\nu\in M_1(K)$ such that the discrete part of $\nu$ is carried by $\Ch_E K$.
       \item $X$ is a standard compact convex set if and only if $A$ contains no compact perfect subset.
   \end{enumerate} 
\end{lemma}

\begin{proof}
    $(1)$ Assume that $\mu\in M_1(X)$ is maximal. Then $\mu$ is carried by $\overline{\ext X}$ (see \cite[Proposition I.4.6]{alfsen}), hence $\mu=\phi(\nu)$ for some $\nu\in M_1(K)$. Moreover, if $\nu(\{x\})>0$ for some $x\in K\setminus\Ch_E K$, it is easy to check that $\mu$ is not maximal (cf. the characterization of simple boundary measures in \cite[p. 35]{alfsen}). This proves the `only if' part.

    To prove the converse we observe that discrete measures carried by $\ext X$ are maximal. Further, assume that $\nu$ is a continuous measure on $K$. It is obviously carried by $\jmath(L)$. Then $\nu$ is $E$-maximal by the Mokobodzki maximal test \cite[Theorem 3.58]{lmns}. Indeed, let $f\in C(K)$ and $\ep>0$. Let
    $$C=\{t\in A\setsep \abs{f(t,0)-f(t,1)}\ge\ep\mbox{ or } \abs{f(t,0)-f(t,-1)}\ge\ep\}.$$
    By Proposition~\ref{p:topol-stacey} the set $C$ is finite. 
   Fix any $t_0\in L\setminus C$ and find a continuous function $h:L\to[0,\infty)$ such that $h(t_0)=0$ and
   $$h(t)=\max\{ \abs{f(t,0)-f(t,1)},\abs{f(t,0)-f(t,-1)} \}\mbox{ for }t\in C.$$
   Then 
   $$g=(f\circ\jmath+\ep+h)\circ\psi \in E\mbox{ and }g\ge f.$$
   In particular, 
   $$f^*(t_0,0)\le g(t_0,0)\le f(t_0,0)+\ep.$$
   Hence $\{t\in L\setsep f^*(t,0)>f(t,0)+\ep\}\subset C$, so it is a finite set. We conclude that $[f^*\ne f]\cap \jmath(L)$ is countable and hence $\nu$-null. So, we have verified condition $(iv)$ from \cite[Theorem 3.58]{lmns} and hence we conclude that $\nu$ is $E$-maximal. Thus $\phi(\nu)$ is a maximal measure on $X$ by \cite[Proposition 4.28(d)]{lmns}. Since maximal measures form a convex set, the proof of the `if part' is complete.

   $(2)$ Assume that $A$ contains a compact perfect set $P$. Then $P$ carries a continuous Radon probability $\nu$. By $(1)$ the measure $\mu=\phi(\jmath(\nu))$ is maximal. Since $\mu$ is carried by $\phi(\jmath(P))$ which is a compact set disjoint with $\ext X$, $X$ is not standard.

   To prove the converse, assume that $A$ contains no compact perfect set. Let $N\supset \ext X$ be a universally measurable set. Then $B=\phi^{-1}(N\cap\overline{\ext X})$ is a universally measurable subset of $K$ containing $\Ch_E K$. In particular, $\jmath(L)\setminus B$ is a universally measurable subset of $\jmath(A)$, so it is universal null (as it contains no compact perfect set). Now it easily follows from $(1)$ that $N$ carries all maximal measures.
\end{proof}

We continue by introducing a piece of notation which we will repeatedly use in this section.

\begin{notation}
If $f$ is any function on $K$, we define
\[
\gls{ftilde}(t,i)=\begin{cases}
         \frac12(f(t,-1)+f(t,1)),& t\in A, i=0,\\
         f(t,i)&\mbox{otherwise}.
     \end{cases}
\]
\end{notation}

In the next result we describe the structure of multipliers for the intermediate function system of strongly affine functions on $X_{L,A}$.

\begin{prop} \label{P:dikous-sa-new}
Let $K=K_{L,A}$, $E=E_{L,A}$, $X=X_{L,A}$ and $H=A_{sa}(X)$. Then the following assertions are valid.
\begin{enumerate}[$(a)$]
    \item The space $H$ is determined by extreme points.
    \item $H=V(\ell^\infty(K)\cap E^{\perp\perp})$ and
 $$\begin{aligned}
        \ell^\infty(K)\cap E^{\perp\perp}=\{f\in \ell^\infty(K)\setsep & f\circ\jmath\mbox{ is universally measurable and }
        \\&
    f(t,0)=\tfrac{1}{2}(f(t,-1)+f(t,1)) \mbox{ for }t\in A\}.\end{aligned}$$
   \item 
   $\begin{aligned}[t]
        M^s(H)=M(H)=V(\{f\in  \ell^\infty(K)\cap E^{\perp\perp} \setsep & \{t\in A\setsep f(t,1)\ne f(t,-1)\}\\&\mbox{ is a universally null set}\}).\end{aligned}$
    \item %The $\sigma$-algebra $\A_{H}$ from Theorem~\ref{T:integral representation H} is described by 
    $\begin{alignedat}[t]{4}
      \A_H=\A^s_H&=  \Big\{ &&\phi((B\setminus A)\times\{0\} \cup (B\cap A)\times \{-1,1\}\cup N_1\times\{-1\} \cup N_2\times\{1\})\setsep \\&&&
    B\subset L\mbox{ universally measurable}, N_1,N_2\subset A \mbox{ universally null} \Big\}
    \\&=\Big\{&&\phi(B)\setsep B\subset\Ch_E K, \psi(B)\mbox{ is universally measurable},
    \\&&&\qquad\psi(B)\cap\psi(\Ch_E K\setminus B)\mbox{ is universally null}\Big\}
        \end{alignedat}
    $
    \end{enumerate}
\end{prop}

\begin{proof}
$(a)$: This follows from an easy observation that $\overline{\ext X}\subset\co(\ext X)$.

$(b)$: The first equality follows from Lemma~\ref{L:function space}$(a)$. Let us prove the second one.
Inclusion `$\subset$' is clear as $\frac12(\ep_{(t,-1)}+\ep_{(t,1)})-\ep_{(t,0)}\in E^\perp$ for each $t\in A$. 

Let us continue by showing the converse. 
     Given $f\in C(K)$, Proposition~\ref{p:topol-stacey}$(a)$  yields that $f$ and $\widetilde{f}$ differ only at a countable set, in particular $\widetilde{f}$ is a Borel function.
    It is proved in \cite[Section 2]{kalenda-bpms} that 
    $$E^\perp=\{\nu\in M(K)\setsep \int_{K}\widetilde f \di\nu=0\mbox{ for each }f\in C(K)\}.$$
    Fix any $\nu\in E^\perp$. By the above formula we deduce that for any $g\in C(L)$ we have $\int_{K} g\circ\psi\di\nu=0$, thus
    $\psi(\nu)=0$. Since $\nu$ is a Radon measure, it may be expressed
    in the form
    $$\nu=\jmath(\nu_0) + \sum_{j\in\en} (a_j \ep_{(t_j,-1)}+b_j\ep_{(s_j,1)}),$$
    where $\nu_0\in M(L)$, $s_j,t_j\in A$ and $\sum_{j\in\en}(\abs{a_j}+\abs{b_j})<\infty$. Equality $\psi(\nu)=0$ then means that $\nu_0+ \sum_{j\in\en} (a_j \ep_{t_j}+b_j\ep_{s_j})=0$. 
    We deduce that $\nu$ may be expressed as
    $$\nu=\sum_{j\in\en} (a_j\ep_{(t_j,-1)}+b_j \ep_{(t_j,1)}-(a_j+b_j)\ep_{(t_j,0)})$$
   for some $t_j\in A$ and real numbers $a_j,b_j$ satisfying   $\sum_{j\in\en}(\abs{a_j}+\abs{b_j})<\infty$.   
   For each $j\in\en$ let $f_j=1_{\{(t_j,1)\}}-1_{\{(t_j,-1)\}}$. Then $f_j\in C(K)$ and $\widetilde{f_j}=f_j$, so $\int f_j\di\nu=0$. It follows that $a_j=b_j$, i.e., 
    $$\nu=\sum_{j\in\en} a_j(\ep_{(t_j,-1)}+ \ep_{(t_j,1)}-2\ep_{(t_j,0)}),$$
    thus clearly $\int f\di\nu=0$ for each $f$ from the set on the right-hand side. This completes the proof of assertion $(b)$.

   $(c)$: 
   It follows from Lemma~\ref{L:function space}, Lemma~\ref{L:maxmiry-dikous}, assertion $(b)$ and the definitions that, given $f\in \ell^\infty(K)\cap E^{\perp\perp}$ we have
   $$\begin{gathered}
     V(f)\in M(A_{sa}(X))\Longleftrightarrow \forall g\in \ell^\infty(K)\cap E^{\perp\perp}\colon \widetilde{fg}\circ\jmath\mbox{ is universally measurable}, \\  
     V(f)\in M^s(A_{sa}(X))\Longleftrightarrow \forall g\in \ell^\infty(K)\cap E^{\perp\perp}\colon [\widetilde{fg}\ne (fg)] \mbox{ is universally null}.\end{gathered}$$
   Therefore, any function from the last set is a strong multiplier. Since $M^s(H)\subset M(H)$ holds always, it remains to prove that any multplier belongs to the last set.
   %Therefore, inclusion `$\supset$' easily follows. (Note that the condition on $f$ on the right-hand side implies that $fg$ and $\widetilde{fg}$ differ only on a universally null set for any $g\in \ell^\infty(K)\cap E^{\perp\perp}$.)
   To this end assume that $B=\{t\in A\setsep f(t,1)\ne f(t,-1)\}$ is not universally null. It follows that there is a continuous measure $\nu\in M_1(L)$ with $\nu^*(B)>0$. Then there is a set $C\subset B$ which is not $\nu$-measurable. Let $g=\chi_{C\times \{1\}}-\chi_{C\times\{-1\}}$. Then $g\in \ell^\infty(K)\cap E^{\perp\perp}$, but $\widetilde{fg}\circ\jmath$ is not $\nu$-measurable. Thus $V(f)\notin M(A_{sa}(X_A))$. 

   $(d)$: This follows from $(c)$ and Theorem~\ref{T:integral representation H} (using moreover \eqref{eq:podmny ChEK}).
\end{proof}

To compare systems $\A^s_H$ and $\ms_H$ for the individual intermediate function spaces $H$ we need a description of nice split faces. Such a description is contained in the following lemma.
%If $X$ is a compact convex set, $F$ a split face and $F'$ its complementary face, we recall that $\lambda_F$ is the witnessing function (see the beginning of Section~\ref{sec:splifaces}).
%denote by $\lambda_F$ the affine function on $X$ provided by Lemma~\ref{l:extense-splitface} to $a=1$ and $a'=0$. I.e., this is the unique affine function such that for each $x\in X$ we have
%$$x=\lambda(x) y+(1-\lambda(x))y'\mbox{ for some }y\in F, y'\in F'.$$
%
%We continue by characterizing split faces with strongly affine witnessing function in Stacey simplices.

\begin{lemma}\label{L:dikous-split-new} Let $K=K_{L,A}$, $E=E_{L,A}$ and $X=X_{L,A}$.
Let $B\subset\Ch_{E}K$. Then the following assertions are equivalent:
\begin{enumerate}[$(1)$]
    \item There is a split face $F$ of $X$ such that $F\cap\ext X=\phi(B)$ and $\lambda_F$ is strongly affine.
    \item There is a split face $F$ of $X$ such that $F\cap\ext X=\phi(B)$ and both $F$ and $F'$ are  measure extremal.
    \item $\psi(B)$ is universally measurable and 
    \[
    \{t\in A\setsep B\cap\{(t,-1),(t,1)\} \mbox{ contains exactly one point}\}
    \]
     is a universally null set.
\end{enumerate}
\end{lemma}

 \begin{proof}
 $(1)\implies(2)$: This is obvious.

  $(2)\implies(3)$: Let $f= \lambda_F\circ \phi$  and $$C=\{t\in A\setsep B\cap\{(t,-1),(t,1)\} \mbox{ contains exactly one point}\}.$$ 
  Since $\lambda_F$ is affine and $[\lambda_F=1]\cup[\lambda_F=0]=F\cup F'\supset \ext X$, we deduce that $f$ attains only values $0,\frac12,1$ and
   $$C=\{t\in L\setsep f(t,0)=\tfrac12\}.$$ Since $[f=0]$ and $[f=1]$ are universally measurable, we deduce that $C$ is also a universally measurable set. 
     Assume it is not universally null. It follows there is a continuous probability $\mu\in M_1(L)$ supported by a compact subset of $C$. Let $\nu=\jmath(\mu)$. By Lemma~\ref{L:maxmiry-dikous} we deduce that $\phi(\nu)$ is a maximal measure on $X$.  Let $z$ be the barycenter of $\phi(\nu)$ in $X$. Then $z=\lambda_F(z)y+(1-\lambda_F(z))y'$ for some $y\in F$ and $y'\in F'$. Let $\mu_y$ and $\mu_{y'}$ be maximal measures representing $y$ and $y'$. Since $F$ and $F'$ are measure extremal, we deduce that $\mu_y$ is supported by $F$ and $\mu_{y'}$ is supported by $F'$. Since $\lambda_F(z)\mu_y+(1-\lambda_F(z))\mu_{y'}$ is a maximal measure representing $z$, the simpliciality of $X$ implies
     $\phi(\nu)=\lambda_F(z)\mu_y+(1-\lambda_F(z))\mu_{y'}$, so $\phi(\nu)$ is supported by $F\cup F'=[\lambda_F=1]\cup[\lambda_F=0]$, i.e., $\nu$ is supported by $[f=0]\cup[f=1]$. But by the construction $\nu$ is supported by $C=[\lambda_F=\frac12]$. This contradiction completes the argument.
     
     \iffalse $(1)\Rightarrow(2)$: Denote by $f$ the function on $K$ such that $V(f)=\lambda_F$. Then $f=1$ on $B$ and $f=0$ on $\Ch_{E}K\setminus B$. Moreover,
     $f(t,0)=0$ if and only if $t\in L\setminus\psi(B)$. It follows that $\psi(B)$ is universally measurable.

     Set $$C=\{t\in A\setsep B\cap\{(t,-1),(t,1)\} \mbox{ contains exactly one point}\}.$$ Then $$C=\{t\in L\setsep f(t,0)=\tfrac12\},$$ so it is a universally measurable set. 
     Assume it is not universally null. It follows there is a continuous probability $\mu\in M_1(L)$ supported by a compact subset of $C$. Let $\nu=\jmath(\mu)$. By Lemma~\ref{L:maxmiry-dikous} we deduce that $\phi(\nu)$ is a maximal measure on $X$.  Let $z$ be the barycenter of $\phi(\nu)$ in $X$. Then $z=\lambda_F(z)y+(1-\lambda_F(z))y'$ for some $y\in F$ and $y'\in F'$. Let $\mu_y$ and $\mu_{y'}$ be maximal measures representing $y$ and $y'$. Since $F=[\lambda_F=1]$ is measure extremal, we deduce that $\mu_y$ is supported by $F$. Similarly $\mu_{y'}$ is supported by $F'$. Since $\lambda(z)\mu_y+(1-\lambda(z))\mu_{y'}$ is a maximal measure representing $z$, the simpliciality of $X$ implies
     $\phi(\nu)=\lambda(z)\mu_y+(1-\lambda(z))\mu_{y'}$, so $\phi(\nu)$ is supported by $F\cup F'=[\lambda_F=1]\cup[\lambda_F=0]$. But the construction provides that $\phi(\nu)$ is supported by $[\lambda_F=\frac12]$. This contradiction completes the argument.\fi

     $(3)\implies(1)$: Define $$C=\{t\in A\setsep B\cap\{(t,-1),(t,1)\} \mbox{ contains exactly one point}\}.$$ By the assumption $C$ is universally null. Set
     $$f(t,i)=\begin{cases}
     1, & (t,i)\in B,\\
     0, & (t,i)\in \Ch_{E_A}K_A\setminus B,\\
     1, & t\in A, i=0, \{(t,-1),(t,1)\}\subset B,\\
     0, & t\in A, i=0, \{(t,-1),(t,1)\}\cap B=\emptyset,\\
     \frac12, & t\in C, i=0.
     \end{cases}$$
     Then $V(f)$ is strongly affine by Proposition~\ref{P:dikous-sa-new}. Since $0\le V(f)\le 1$,
     clearly $F_1=[V(f)=1]$ and $F_2=[V(f)=0]$ are faces of $X$. Since $V(f)$ is strongly affine, these faces are measure convex and measure extremal. By Lemma~\ref{L:maxmiry-dikous} we deduce that $F_1\cup F_2$ carries all maximal measures. Lemma~\ref{l:complementarni-facy} then yields that $F_2=F_1'$. By Corollary~\ref{c:simplex-facejesplit} we know that $F_1$ is a split face. Since obviously $V(f)=\lambda_{F_1}$, the proof is complete.
 \end{proof}   

\begin{cor}\label{cor:dikous-split-me}  Let $K,E,X$ be as in Proposition~\ref{P:dikous-sa-new}.
  Let $F\subset X$
    be a split face  such that both $F$ and $F'$ are measure extremal. Then  the function $\lambda_F$ is strongly affine.
\end{cor}

\begin{proof}
   Assume that $F$ is such a split face. Define $f$ and $C$ as in the proof of implication $(2)\implies(3)$ of the previous proposition. Then $f$ is universally measurable and $C$ is universally null (by the quoted implication). By Proposition~\ref{P:dikous-sa-new} we deduce that $f\in \ell^\infty(K)\cap E^{\perp\perp}$ and hence $V(f)$ is strongly affine. It is thus enough to prove that $\lambda_F=V(f)$. Let $x\in X$ be arbitrary. Then $x=\lambda_F(x)y+(1-\lambda_F(x))y'$ for some $y\in F$ and $y'\in F'$.  Let $\mu_y$ and $\mu_{y'}$ be maximal measures representing $y$ and $y'$. Since $F$ and $F'$ are measure extremal, we deduce that $\mu_y$ is supported by $F$ and $\mu_{y'}$ is supported by $F'$. Let $\nu_y$ and $\nu_{y'}$ be the unique measures on $K$ with $\phi(\nu_y)=\mu_y$ and  $\phi(\nu_{y'})=\mu_{y'}$.  Let
     $\nu=\lambda_F(x)\nu_y+(1-\lambda_F(x))\nu_{y'}$. Then $\phi(\nu)$ is a maximal measure representing $x$. Therefore
     $$V(f)(x)=\int f\di\nu = \lambda_F(x)\int f\di\nu_y=\lambda_F(x)$$
     as $\nu_y$ is supported by $[f=1]$. This completes the proof.
\end{proof}

\begin{cor}\label{cor:dikous-sa-split}
    Let $K,E,X,H$ be as in Proposition~\ref{P:dikous-sa-new}. Then 
    $$\begin{aligned}
           \A^s_H&= \A_H=\ms_H\\&=\{\ext X\cap F\setsep F\subset X\mbox{ is a split face with $F$ and $F'$ measure extremal}\}
          \subset \Z_H. \end{aligned}  $$
   The converse to the last inclusion holds if and only if $A$ does not contain any compact perfect subset.
\end{cor}

\begin{proof}
    The equalities follow by combining Proposition~\ref{P:dikous-sa-new}$(d)$ with Lemma~\ref{L:dikous-split-new}. The inclusion is obvious.

    Assume $A$ contains no compact perfect subset. By Lemma~\ref{L:maxmiry-dikous} we get that $X$ is a standard compact convex set. So, Lemma~\ref{L:SH=ZH} shows that the equality holds.

    \iffalse
    Let $f\in A_{sa}(X)$ be such that $\ext X\subset [f=0]\cup[f=1]$. Let $g$ be the function satisfying $V(g)=f$. Then $g$ attains only values $0,\frac12,1$. Let
    $$B=\{t\in L\setsep g(t,0)=0\}\mbox{ and }N=\{t\in L\setsep g(t,0)=\frac12\}.$$
    Then $B$ and $N$ are universally measurable and $N\subset A$. If $N$ is not universal null, there is a continuous Radon probability $\mu$ on $L$ such that $\mu(N)>0$. Using the Radon property, we find a compact set $F\subset L$ with $\mu(F)>0$. Then $F$ cannot be scattered and hence $B$ (and thus $A$) contains a compact perfect set, a contradiction. Thus $N$ is universal null. It follows from Lemma~\ref{L:dikous-split-new} that $[f=1]\cap \ext X$ is a split face.\fi
  
Conversely, assume that $A$ contains a compact perfect subset $D$. 
Let $$f=1_{D\times\{1\}}+\frac12 1_{D\times\{0\}}.$$ Then $V(f)\in A_{sa}(X)$ and $\ext X\subset [V(f)=1]\cup[V(f)=0]$. Since $D$ is not universal null, Lemma~\ref{L:dikous-split-new} shows that $[V(f)=1]\cap \ext X\in\Z_H\setminus\ms_H$.
\end{proof}

The following example show that in condition $(2)$ of Lemma~\ref{L:dikous-split-new} the assumption of measure extremality cannot be replaced by measure convexity. It also witnesses that in Theorem~\ref{t:srhnuti-splitfaceu-metriz} some assumption on $X$ is needed.

\begin{example}\label{ex:dikous-divnyspitface}
    Let $L=A=[0,1]$, $K=K_{L,A}$, $E=E_{L,A}$, $X=X_{L,A}$ and $H=A_{sa}(X)$. Then there is as split face $F\subset X$ with the following properties:
    \begin{enumerate}[$(i)$]
        \item Both $F$ and $F'$ are measure convex.
        \item $\lambda_F$ is not strongly affine.
        \item $F\cap\ext X\notin \ms_H$.
        \item There are maximal measures not carried by $F\cup F'$.
    \end{enumerate}
\end{example}

\begin{proof} The proof will be done in several steps.

{\tt Step 1:} Set
$$\begin{aligned}
   N_1&=\{\mu\in M_1(K)\setsep \mu(L\times\{-1\})=\mu([\tfrac12,1]\times\{0\})=\mu_d(L\times\{0\})=0\},\\ 
    N_2&=\{\mu\in M_1(K)\setsep \mu(L\times\{1\})=\mu([0,\tfrac{1}{2}]\times\{0\})=\mu_d(L\times\{0\})=0\},
\end{aligned}$$
where $\mu_d$ denotes the discrete part of $\mu$.
Then $N_1$ and $N_2$ are Borel measure convex subsets of $M_1(K)$.

\smallskip

It is enough to prove the statement for $N_1$. This set is the intersection of the following three sets:
\begin{itemize}
    \item $\{\mu\in M_1(K)\setsep \mu (L\times\{-1\})=0\}$: This is a closed convex set (note that $L\times\{-1\}$ is an open subset of $K$ and hence the function $\mu\mapsto \mu (L\times\{-1\})$ is lower semicontinuous), so it is a Borel measure convex set.
    \item  $\{\mu\in M_1(K)\setsep\mu([\tfrac12,1]\times\{0\})=0\}$: It follows for example from Lemma~\ref{L:function space}$(b)$ that $\mu\mapsto \mu([\tfrac12,1]\times\{0\})$ is a strongly affine nonnegative Borel function. Hence the set is Borel and measure convex.
    \item $\{\mu\in M_1(K)\setsep \mu_d (L\times\{0\})=0\}$: It follows from \cite[Proposition 2.58]{lmns} that it is a $G_\delta$ set. Further, it is measure convex by the proof of Example~\ref{ex:d+s}. 
\end{itemize}
So, $N_1$ is indeed a Borel measure convex set, being the intersection of three Borel measure convex sets.

\smallskip

{\tt Step 2:} Let $\theta:M_1(K)\to X$ be defined by $\theta(\mu)=r(\phi(\mu))$ for $\mu\in M_1(K)$. Then $\theta$ is an affine continuous surjection of $M_1(K)$ onto $X$. Moreover, $N_1=\theta^{-1}(\theta(N_1))$ and $N_2=\theta^{-1}(\theta(N_2))$.

\smallskip

The mapping $\theta$ is the composition of two mappings -- the mapping $\mu\mapsto\phi(\mu)$ which is an affine homeomorphism of $M_1(K)$ onto $M_1(\overline{\ext X})$ and the barycenter mapping which is affine and continuous and, due to the Krein-Milman theorem maps $M_1(\overline{\ext X})$ onto $X$. Therefore $\theta$ is an affine continuous surjection.

It is enough to prove the equality for $N_1$. Assume that $\mu_1\in N_1$ and $\mu_2\in M_1(K)$ such that $\theta(\mu_1)=\theta(\mu_2)$. Then $\mu_2-\mu_1\in E^\perp$. It follows from the proof of Proposition~\ref{P:dikous-sa-new}$(b)$ that there is a sequence $(t_j)$ in $L$ and a summable sequence $(a_j)$ of real numbers such that
$$\mu_2=\mu_1+\sum_j a_j(\ep_{(t_j,1)}+\ep_{(t_j,-1)}-2\ep_{(t_j,0)}).$$
Since $\mu_2\ge0$ and $\mu_1(\{(t_j,-1)\})=\mu_1(\{(t_j,0)\})=0$, necessarily $a_j=0$ for each $j\in\en$. Thus $\mu_2=\mu_1$ and the argument is complete.

\smallskip

{\tt Step 3:} Let $F_1=\theta(N_1)$ and $F_2=\theta(N_2)$. Then $F_1$ and $F_2$ are disjoint Borel measure convex subsets of $X$.

\smallskip 

Since $N_1$ and $N_2$ are clearly disjoint, by Step 2 we get that also $F_1$ and $F_2$ are disjoint. Further, by Step 1 we know that $N_1$ and $N_2$ are Borel sets, hence $F_1$ and $F_2$ are Borel sets by Lemma~\ref{L:kvocient}$(f)$ and Step 2. Finally, let us prove that $F_1$ is measure convex (the case of $F_2$ is completely analogous): Let $\mu\in M_1(X)$ be carried by $F_1$. Since $\theta$ is a continuous surjection (by Step 2), the mapping from $M_1(M_1(K))$ to $M_1(X)$ assigning to a measure its image measure is also surjective, there is $\nu\in M_1(M_1(K))$ such that $\theta(\nu)=\mu$. Then $\nu(N_1)=\nu(\theta^{-1}(F_1))=\mu(F_1)=1$, thus $\nu$ is carried by $N_1$. Let $\sigma=r(\nu)$ be the barycenter of $\nu$. By Step 1 we know that $\sigma\in N_1$. Since clearly $r(\mu)=r(\theta(\nu))=\theta(r(\nu))=\theta(\sigma)\in F_1$, the argument is complete.

\smallskip

{\tt Step 4:} Given $x\in X$, there are $\lambda\in[0,1]$, $y_1\in F_1$ and $y_2\in F_2$ such that $x=\lambda y_1+(1-\lambda)y_2$. Moreover, $\lambda$ is uniquely determined. If $x\in X\setminus(F_1\cup F_2)$, then $y_1$ and $y_2$ are uniquely determined as well.

\smallskip

Let $x\in X$. Let $\mu$ be a maximal measure representing $x$. Then $\mu$ is carried by $\overline{\ext X}$, so there is $\nu\in M_1(K)$ with $\phi(\nu)=\mu$. It follows from Lemma~\ref{L:maxmiry-dikous} that $\nu_d(L\times\{0\})=0$. Hence there is $\lambda\in[0,1]$, $\nu_1\in N_1$ and $\nu_2\in N_2$ such that $\nu=\lambda\nu_1+(1-\lambda)\nu_2$. Then
$$x=\theta(\nu)=\lambda \theta(\nu_1)+(1-\lambda)\theta(\nu_2),$$
which proves the existence.

To prove the uniqueness, assume that
$$x=ty_1+(1-t)y_2$$
for some $t\in[0,1]$, $y_1\in F_1$ and $y_2\in F_2$. Let
$\sigma_1\in N_1$ and $\sigma_2\in N_2$ be such that $y_1=\theta(\sigma_1)$ and $y_2=\theta(\sigma_2)$. Let
$\sigma=t\sigma_1+(1-t)\sigma_2$. Then $x=\theta(\sigma)$, i.e., $x$ is the barycenter of $\phi(\sigma)$. Since $\sigma_d(L\times\{0\})=0$, Lemma~\ref{L:maxmiry-dikous} says that $\phi(\sigma)$ is maximal. Since $X$ is a simplex, we deduce that $\phi(\sigma)=\mu$, hence $\sigma=\nu$. It follows that
$$t=\sigma(L\times\{0\}\cup [0,\tfrac12]\times\{1\})=\nu(L\times\{0\}\cup [0,\tfrac12]\times\{1\})=\lambda.$$
Moreover, if $x\in X\setminus(F_1\cup F_2)$, then $\lambda\in (0,1)$, so $\sigma_1=\nu_1$ and $\sigma_2=\nu_2$. Hence $y_1=\theta(\nu_1)$ and $y_2=\theta(\nu_2)$ which completes the argument.

\smallskip

{\tt Step 5:} $F_1$ is a split face of $X$ and $F_1'=F_2$.

\smallskip

Let $x\in F_1$ and $x=\frac12(y+z)$ for some $y,z\in X$. Let 
$y=t y_1+(1-t) y_2$ and $z= s z_1+(1-s)z_2$ be the decompositions provided by Step 4. If $s=t=0$, then $x\in F_2$, which is impossible. If $s+t\in(0,2)$, then
$$x=\tfrac12(ty_1+(1-t)y_2+s z_1+(1-s) z_2)= \tfrac{t+s}{2}\cdot\tfrac{t y_1+s z_1}{t+s} +\tfrac{2-t-s}{2}\cdot\tfrac{(1-t)y_2+(1-s)z_2}{2-t-s}.$$
Step 4 then shows that $s+t=2$, a contradiction. Finally, if $s=t=1$, then $y,z\in F_1$. This completes the proof that $F_1$ is a face. In the same way we see that $F_2$ is also a face. Hence clearly $F_2\subset F_1'$. The converse inclusion can be proved by repeating the proof of Lemma~\ref{l:complementarni-facy}$(ii)$.

\smallskip

{\tt Step 6:} $\lambda_{F_1}$ is not strongly affine.

\smallskip

Let $\mu$ be the normalized Lebesgue measure on $[0,\frac12]\times\{0\}$ and $x=\theta(\mu)$. Since $\mu\in N_1$, we get $x\in F_1$ and hence $\lambda_{F_1}(x)=1$. On the other hand, $x$ is the barycenter of $\phi(\mu)$, the measure $\phi(\mu)$ is carried by $\phi([0,\frac12]\times\{0\})$ and hence $\lambda_{F_1}=\frac12$ on the support of $\phi(\mu)$.

\smallskip

{\tt Step 7:} $F_1\cap\ext X\notin \ms_H$.

\smallskip

It is enough to observe that $F_1\cap\ext X=\phi(L\times\{1\})$ and this set is not in $\ms_H$ by Lemma~\ref{L:dikous-split-new}.

\smallskip

{\tt Step 8:} Let $\mu$ be the Lebesgue measure on $L\times\{0\}$. Then $\phi(\mu)$ is a maximal measure (by Lemma~\ref{L:maxmiry-dikous}) not carried by $F_1\cup F_2$.

\smallskip

\end{proof}

\subsection{Continuous affine functions on  $X_{L,A}$}
\label{ssec:cont-af-X}

Now we focus on the space of continuous affine functions on $X_{L,A}$.

\begin{prop}\label{P:dikous-spoj-new}
Let $K=K_{L,A}$, $E=E_{L,A}$, $X=X_{L,A}$ and $H=A_c(X)$. Then the following assertions are valid.
    \begin{enumerate}[$(a)$]
        \item $M^s(H)=M(H)=\Phi(\{f\circ\psi\setsep f\in C(L)\})$.
        \item The facial topology on $\ext X$ is formed by sets
        \[\begin{aligned}
                   \phi((\psi^{-1}(U)\cap \Ch_{E}K) \setminus F),\quad &
        U\subset L\mbox{ open}, F\subset A\times\{-1,1\}\\&\mbox{ such that $\psi(F)$ has no accumulation point in }U.\end{aligned}
        \]
        \item $\A^s_H=\A_H=\{\phi(\psi^{-1}(U)\cap\Ch_E K)\setsep  U\subset L\mbox{ open}\}.$
    \end{enumerate}
\end{prop}

\begin{proof}
 $(a)$: The first equality follows from Proposition~\ref{P:rovnostmulti}, let us look at the second one. Inclusion `$\supset$' is clear. The converse follows from Lemma~\ref{L:mult-nutne}(i) since $\phi(t,0)$ is the barycenter of $\frac12(\ep_{\phi(t,-1)}+\ep_{\phi(t,1)})$ for each $t\in A$.

 $(b)$: Since $X$ is a simplex, \cite[Theorem II.6.2]{alfsen} implies that the facial topology is
 $$\{\ext X\setminus G\setsep G\subset X\mbox{ a closed  face}\}.$$
 So, assume $G$ is a closed face of $X$. Then $B=\phi^{-1}(G)$ is a closed subset of $K$. Let 
 $$F=\{(t,i)\in B\cap (A\times\{(-1,1)\})\setsep (t,-i)\notin B\}.$$
 If $t_0$ is an accumulation point of $\psi(F)$ in $L$, then $(t_0,0)\in\overline{F}\subset B$. Hence, if $t_0\in A$, we deduce that $\{t_0\}\times\{-1,0,1\}\subset B$ (as $G$ is a face).
 Then $B'=B\setminus F$ is a closed subset of $K$ with $\psi^{-1}(\psi(B'))=B'$, hence $U=I\setminus\psi(B')$ is an open subset of $L$ and no accumulation point of $\psi(F)$ belongs to $U$. Clearly
 $$\Ch_{E}K\setminus B= \psi^{-1}(U)\cap\Ch_{E}K\setminus F,$$
 so any facially open set is of the given form.

 Conversely, assume $U$ and $F$ satisfy the above conditions. Let
 $$B=\psi^{-1}(L\setminus U)\cup F.$$
 Then $B$ is a closed subset of $K$. Further, it is easy to check that it is a Choquet set in the sense of \cite[Definition 8.27]{lmns}. Since $X$ is a simplex, it follows from \cite[Theorem 8.60 and Proposition 8.42]{lmns} that $\overline{\co\phi(B)}$ is a split face. Since clearly
 $$\ext X\setminus \overline{\co\phi(B)} = \phi(\Ch_{E}K\setminus B)= \phi((\psi^{-1}(U)\cap \Ch_{E}K) \setminus F),$$
 the proof is complete.

 $(c)$: This follows from $(a)$ and Proposition~\ref{p:system-aha}.
\end{proof}

\begin{remark}
    In the situation from the previous proposition, $M(H)$ is determined by $\A_H$ (this follows easily from Theorem~\ref{T: meritelnost H=H^uparrow cap H^downarrow}) and also by the facial topology (see \cite[Theorem II.7.10]{alfsen}). But $\A_H$ is a proper subfamily of the facial topology whenever $A\ne\emptyset$.
    This illustrates the feature addressed in Remark~\ref{rem:o meritelnosti}$(2)$.
\end{remark}

\subsection{Baire strongly affine functions on  $X_{L,A}$}
\label{ssec:baire-af-X}

Next we look at spaces of Baire strongly affine functions. The relevant results are contained in two propositions
-- the first one is devoted to Baire-one functions and the second one to general Baire functions. Before coming to these results we give a topological lemma.

\begin{lemma}\label{L:countable Baire}
    Let $L$ be a compact space and let $B\subset L$ be a countable set. Then $B$ is a Baire set if and only if it consists of $G_\delta$ points of $L$.
\end{lemma}

\begin{proof}
    Any closed $G_\delta$ subset of $L$ is $\Coz_\delta$, hence Baire. Since Baire sets form a $\sigma$-algebra, the proof of the `if part' is complete.

    To prove the converse, assume $B$ is a Baire set and fix any $x\in B$. Enumerate $B\setminus\{x\}=\{b_n\setsep n\in\en\}$. By the Urysohn lemma there is a continuous function $f_n:L\to [0,1]$ such that $f_n(x)=1$ and $f_n(b_j)=0$ for $j\le n$. Then $B_n=B\cap[f_n>0]$ is a Baire set.
    We conclude that $\{x\}=\bigcap_n B_n$ is a Baire set as well. As the complement $L\setminus\{x\}$ is also Baire, hence Lindel\"of, it is $F_\sigma$. It follows that $x$ is a $G_\delta$ point.
\end{proof}

\begin{prop}\label{P:dikous-A1-new}
   Let $K=K_{L,A}$, $E=E_{L,A}$,  $X=X_{L,A}$ and $H=A_1(X)$. Then the following assertions are valid:
   \begin{enumerate}[$(a)$]
         \item $\begin{aligned}[t]
             H=V\Big(\Big\{f\in\ell^\infty(K)\setsep& f\circ\jmath\in\Ba_1(L), f(t,0)=\tfrac12(f(t,-1)+f(t,1))\mbox{ for }t\in A, \\& \{t\in A\setsep f(t,1)\neq f(t,-1)\}\mbox{ is countable}\Big\}\Big).
         \end{aligned}$
       
        \item $\begin{aligned}[t]
            M&^s(H)=M(H)\\&=V\Big(\Big\{f\in\ell^\infty(K)\setsep f\circ\jmath\in\Ba_1(L), f(t,0)=\tfrac12(f(t,-1)+f(t,1))\mbox{ for }t\in A, \\& \quad\forall\ep>0\colon\{t\in A\setsep \abs{f(t,1)- f(t,-1)}\ge\ep\}\mbox{ is a countable $G_\delta$ subset of }L\Big\}\Big)
        \end{aligned}$
          
        \item $\begin{aligned}[t]
            \A_H&=\A^s_H=\{\phi(\psi^{-1}(F)\cap \Ch_{E}(K)\cup  C_1\times\{1\}\cup 
        C_2\times\{-1\})\setsep F\subset L\ \Zer_\sigma\mbox{set},\\
        &\qquad\qquad C_1,C_2\subset A\mbox{ countable  sets  consisting of $G_\delta$ points in }L\} 
        \\ &=\{ \phi(B)\setsep B\subset\Ch_E K, \psi(B)\mbox{ is a $\Zer_\sigma$ subset of }L, \\ & \quad \psi(B)\cap\psi(\Ch_E K\setminus B) 
        \mbox{ is a countable set consisting of $G_\delta$ points in }L\}.
            \end{aligned}$

     \item $\begin{aligned}[t]
             \A_H&=\A^s_H=\ms_H\\&=\{\ext X\setminus F\setsep F\mbox{ is a $\Coz_\delta$ split face such that $F'$ is a Baire set}
             \\&\qquad\qquad \mbox{ and
             $F,F'$ are measure convex}\}=\Z_H. \end{aligned} $
           \end{enumerate}
\end{prop}

\begin{proof}  
  $(a)$: Inclusion `$\subset$' follows from Proposition~\ref{p:topol-stacey}$(b)$ and Proposition~\ref{P:dikous-sa-new}$(b)$. Conversely, assume that $f$ satisfies the conditions on the right-hand side. By Proposition~\ref{p:topol-stacey}$(b)$ and Proposition~\ref{P:dikous-sa-new}$(b)$ we deduce that $f\in \Ba_1^b(K)\cap E^{\perp\perp}$, and thus $V(f)\in A_1(X)$ by Lemma~\ref{L:function space}$(b)$. 

  $(b)$:  The first equality follows from Proposition~\ref{P:rovnostmulti}, let us look at the second one. Fix a function $f$ on $K$ such that $V(f)\in A_1(X)$.   Using $(a)$ and Proposition~\ref{p:topol-stacey}$(b)$  we see that
  $V(f)\in M(H)$ if and only if
  $$V(g)\in H\Longrightarrow \widetilde{fg}\circ\jmath\in\Ba_1(L).$$
   Let us now prove the two inclusions:

  `$\supset$': Assume $f$ satisfies the condition on the right-hand side and fix any $g$ with $V(g)\in A_1(X)$. Then $fg\in \Ba_1(K)$,  $fg=\widetilde{fg}$ on $A\times \{-1,1\}$ and on $(L\setminus A)\times \{0\}$. For $t\in A$ we have
   \begin{equation*}
  \begin{aligned}
 f(t,0)g(t,0)-\widetilde{fg}(t,0)&=
  \tfrac14(f(t,-1)+f(t,1))(g(t,-1)+g(t,1))\\&\qquad-\tfrac12(f(t,-1)g(t,-1)+f(t,1)g(t,1))
  \\&=\tfrac14(f(t,-1)-f(t,1))(g(t,1)-g(t,-1)),
  \end{aligned}
  \end{equation*}
  hence
  $$\abs{f(t,0)g(t,0)-\widetilde{fg}(t,0)}\le\tfrac12\norm{g}_\infty\abs{f(t,1)-f(t,-1)},$$
  so, for each $\ep>0$,
  \[
  \{t\in L\setsep \abs{f(t,0)g(t,0)-\widetilde{fg}(t,0)}\ge\ep\}
 \subset
  \{t\in A\setsep \tfrac12\norm{g}_\infty\abs{f(t,1)-f(t,-1)}\ge \ep\}.
  \]
The latter set is a countable $G_\delta$ subset of $L$, and hence the former set is also a countable $G_\delta$ set in $L$. 
  
So, $(fg-\widetilde{fg})\circ\jmath\in \Ba_1(L)$, because  Baire-one functions on normal spaces are characterized via $G_\delta$ measurability of their level sets (see \cite[Exercise 3.A.1]{lmz}). Thus $\widetilde{fg}\circ \jmath\in\Ba_1(L)$ and therefore $V(f)\in M(A_1(X))$.

  `$\subset$': Assume that $V(f)\in M(H)$ and $\ep>0$ is given. Let
  \[
B=\{t\in A\setsep \abs{f(t,1)-f(t,-1)}\ge\ep\}.
  \]
By Proposition~\ref{p:topol-stacey}$(b)$, $B$ is countable. We use the condition from the beginning of the proof for $g=f$, i.e., we use the fact that $\widetilde{f^2}\circ\jmath\in \Ba_1(L)$. Since $f^2\circ\jmath\in \Ba_1(L)$, the function
\[
t\in L\mapsto \abs{f^2(t,0)-\widetilde{f^2}(t,0)}=\frac14\abs{(f(t,1)-f(t,-1))^2}
\]
is Baire-one on $L$. Hence the level sets of this function are $G_\delta$ set, from which it follows that $B$ is a $G_\delta$ set.

$(c)$: The first equality follows from $(b)$, let us prove the second one:

`$\supset$': Let $F\subset L$ be a $\Zer_\sigma$ set. Then there is a Baire-one function $g:L\to[0,1]$ such that $F=[g>0]$ (see, e.g., \cite[Proposition 2]{kalenda-spurny}). Let $f=g\circ\psi$. Then $V(f)\in M(H)$ by $(b)$. Then
  $[f>0] = \psi^{-1}(F)$  and hence
  $$\phi(\psi^{-1}(F)\cap\Ch_{E}K)\in\A_{H}$$
  by Proposition~\ref{p:system-aha}. 
  
  Further, fix any $t\in A$ such that $\{t\}$ is a $G_\delta$ set in $L$ and $i\in\{-1,1\}$. Let $h=1_{(t,i)}-1_{(t,-i)}$. By $(b)$ we see that $V(h)\in M(H)$ and hence
  \[
  \{\phi(t,i)\}=[V(h)>0]\cap\ext X\in \A_{H}.
  \]
    Since $\A_{H}$ is closed with respect to countable unions by Proposition~\ref{p:system-aha}$(a)$, the proof of inclusion `$\supset$' is complete.
  
  `$\subset$': Let $f$ be such that $V(f)\in M(H)$. For each $n\in\en$ set
  \[
  F_n=\{t\in L\setsep f(t,0)>\tfrac1n\}\quad \mbox{and}\quad C_n=\{t\in A\setsep \abs{f(t,1)-f(t,-1)}>\tfrac1n\}.
  \]
  Then $F_n$ is a $\Zer_\sigma$ set and $C_n$ is a countable $G_\delta$ set in $L$ (by $(b)$). Hence it consists of $G_\delta$ points in $L$. The  equality
  \[
  \begin{aligned}\relax
  [f>0]\cap\Ch_{E}K&= \psi^{-1}\left(\bigcup_{n\in\en} (F_n\setminus C_n)\right)\cap \Ch_{E}K \cup\\
 &\quad \bigcup_{n\in\en}\{(t,i)\in C_n\times\{-1,1\}\setsep f(t,i)>0\}.
  \end{aligned}
  \]
now yields the assertion. Indeed, each set $C_n$, as a countable $G_\delta$ set, is a $\Coz_\delta$ set, and thus $F_n\setminus C_n$ is a $\Zer_\sigma$ set. Hence their union is a $\Zer_\sigma$ set.

Finally, the last equality follows from \eqref{eq:podmny ChEK}.

$(d)$: The first equality is repeated from $(c)$. Inclusion $\A^s_H\subset\ms_H$ follows from Theorem~\ref{T:meritelnost-strongmulti}. 
Inclusion `$\subset$' from the third equality is obvious. 

Let us prove the converse inclusion. Let $F\subset X$ be a $\Coz_\delta$ split face such that $F'$ is a Baire set and both $F$ and $F'$ are measure convex. By Proposition~\ref{p:shrnuti-splitfaceu-baire} we get that $\lambda_F$ is Baire and strongly affine. Let $f$ be the function on $K$ such that $V(f)=\lambda_{F'}$ and let $B=\phi^{-1}(F)$. Let 
$$
\begin{aligned}
    C&=\{t\in L\setsep f(t,0)=\tfrac{1}{2}\}\\
    &=\{t\in A\setsep B\cap\{(t,1),(t,-1)\}\mbox{ contains exactly one point}\}.
\end{aligned}  $$
It follows from Proposition~\ref{p:topol-stacey}$(c)$ that $C$ is a countable Baire set, hence by Lemma~\ref{L:countable Baire} it consists of $G_\delta$ points.

Let us enumerate $C=\{t_n\setsep n\in\en\}$ and for each $n\in\en$ let $i_n\in \{-1,1\}$ be such that $(t_n,i_n)\notin B$. 
 Let $g\in \Ba_1(L)$ be such that $0\le g\le 1$ and $[g=1]=B$. Such a function exists for example by \cite[Proposition 2]{kalenda-spurny}. We define a sequence $(g_n)$ of functions on $K$ by
 $$g_n(t,i)=\begin{cases}
    1- g(t)^n, & t\in L\setminus\{t_1,\dots,t_n\},\\
     f(t,i), & t\in\{t_1,\dots,t_n\}.
 \end{cases}$$
Then $g_n\in \Ba_1(K)\cap E^{\perp\perp}$ and $g_n\nearrow f$. 
It follows that $V(g_n)\in A_1(X)$ and $V(g_n)\nearrow V(f)=\lambda_{F'}$. Thus $\lambda_{F'}\in H^\uparrow$ and the argument is complete.

Finally, $\ms_H=\Z_H$ by Lemma~\ref{L:SH=ZH}.
  \end{proof}

We pass to the case of general strongly affine Baire functions.

\begin{prop}\label{P:dikous-Bair-new}
 Let $K=K_{L,A}$, $E=E_{L,A}$, $X=X_{L,A}$ and $H=(A_c(X))^\mu$. Then the following assertions are valid.
    \begin{enumerate}[$(a)$]
        \item  $H=(A_c(X))^\sigma=\Ba^b(X)\cap A_{sa}(X)$.
        \item $\begin{aligned}[t]
             H=V\Big(\Big\{f\in\ell^\infty(K)\setsep& f\circ\jmath\in\Ba(L), f(t,0)=\tfrac12(f(t,-1)+f(t,1))\mbox{ for }t\in A, \\& \{t\in A\setsep f(t,1)\neq f(t,-1)\}\mbox{ is countable}\Big\}\Big)\end{aligned}$
        \item $\begin{aligned}[t]
             M&^s(H)=M(H)\\&=V\Big(\Big\{f\in\ell^\infty(K)\setsep f\circ\jmath\in\Ba(L), f(t,0)=\tfrac12(f(t,-1)+f(t,1))\mbox{ for }t\in A, \\& \{t\in A\setsep f(t,1)\neq f(t,-1)\}\mbox{ is a countable Baire subset of }L\Big\}\Big)\end{aligned}$
        \item $\begin{aligned}[t]
            \A_H&=\A^s_H=\{\phi(\psi^{-1}(F)\cap \Ch_{E}(K)\cup  C_1\times\{1\}\cup 
        C_2\times\{-1\})\setsep F\subset L\ \mbox{Baire},\\
        &\qquad\qquad C_1,C_2\subset A\mbox{ countable  sets  consisting of $G_\delta$ points in }L\} 
        \\ &=\{ \phi(B)\setsep B\subset\Ch_E K, \psi(B), \psi(\Ch_E K\setminus B)\mbox{ are Baire subsets of }L, \\ & \qquad \qquad\psi(B)\cap\psi(\Ch_E K\setminus B) 
        \mbox{ is  countable}\}.
            \end{aligned}$

       \item  $\begin{aligned}[t]
             \A_H&=\A^s_H=\ms_H\\&=\{\ext X\cap F\setsep F\mbox{ is a split face such that}\\&\qquad\qquad F, F'\mbox{ are Baire and measure convex}\}= \Z_H. \end{aligned} $
               
       %For a set $F\subset X$ the following assertions are equivalent:
       %\begin{itemize}
       %    \item [(e1)] $F\cap \ext X\in\A_H$.
        %   \item [(e2)] $F$ is a Baire measure convex split face  whose complementary face $F'$ is also Baire and measure convex.
        %  \end{itemize}
    \end{enumerate}
\end{prop}

\begin{proof}
$(a)$:  Since $X$ is a simplex, the equality $H=(A_c(X))^\sigma=\Ba^b(X)\cap A_{sa}(X)$ follows from Proposition~\ref{P:Baire-srovnani}$(iii)$.

$(b)$: This follows by combining Proposition~\ref{p:topol-stacey}$(c)$, Proposition~\ref{P:dikous-sa-new}$(b)$ and Lemma~\ref{L:function space}$(b)$.

$(c)$:  The first equality follows from Proposition~\ref{P:rovnostmulti}, let us look at the second one.
We first observe that $V(g\circ\psi)\in M(H)$ whenever $g\in \Ba_1^b(L)$. So, given $f$ such that $V(f)\in H$, we get $V(f)\in M(H)\Leftrightarrow V(f-f\circ\jmath\circ\psi)\in M(H)$. Moreover, 
$$f(t,1)-(f\circ\jmath\circ\psi)(t,1)
-(f(t,-1)-(f\circ\jmath\circ\psi)(t,-1))=f(t,1)-f(t,-1),$$
so it is enough to prove the equality for the functions satisfying $f\circ\jmath=0$ on $L$.

`$\subset$': Assume that $f\circ\jmath=0$, $V(f)\in M(H)$ and $B=\{t\in A\setsep f(t,1)\neq f(t,-1)\}$. Then $V(f)\in H$ and hence $f$ satisfies the conditions in $(b)$. 
 Further, the function 
\[
g(t,i)=\begin{cases}
                 1,& t\in B, f(t,i)>0,\\
                 -1,& t\in B, f(t,i)<0,\\
                 0&\text{ otherwise}
\end{cases}
\]
satisfies $V(g)\in H$ by $(b)$. Since $V(f)\in M(H)$, $\wt{fg}\circ\jmath$ is a Baire function on $L$. Since
\[
\wt{fg}(t,0)\begin{cases}
           =0,& t\in L\setminus B,\\
           >0,& t\in B,
\end{cases}
\]
the set $B$ is a Baire subset of $L$. (It is countable by $(b)$.)

`$\supset$': Let $f\circ\jmath=0$ and $f$ satisfy the conditions on the right-hand side. Then $V(f)\in H$ by $(b)$. Further, let $g$  with $V(g)\in H$ be given. If $\wt{fg}(t,0)$ is nonzero for some $t\in A$, then $f(t,1)\neq f(t,-1)$. But the set $$B=\{t\in A\setsep \abs{f(t,1)-f(t,-1)}>0\}$$ is a countable Baire set in $L$, and thus each of its subsets is also Baire (this follows from Lemma~\ref{L:countable Baire}). 
From this we obtain that $\wt{fg}$ is Baire on $K$, and consequently $V(f)\in M(H)$.

$(d)$: The first equality follows from $(c)$, let us continue by the second one. 

`$\supset$': Let $F\subset L$ be a Baire set. Then $1_F$ is a Baire function on $L$. Let $f=1_F\circ\psi$. Then $V(f)\in M(H)$ by $(c)$. Then
  $[f=1] = \psi^{-1}(F)$  and hence
  $$\phi(\psi^{-1}(F)\cap\Ch_{E}K)\in\A_{H}$$
  by Theorem~\ref{T:integral representation H}(i). 
    Further, fix any countable $C\subset A$  such that $C$ is a Baire set in $L$. Let $h=1_{C\times \{1\}}-1_{C\times\{-1\}}$. By $(c)$ we see that $V(h)\in M(H)$ and hence
  \[
  \phi(C\times\{1\})=[V(h)>0]\cap\ext X\in \A_{H}\mbox{ and } \phi(C\times\{-1\})=[V(h)<0]\cap\ext X\in \A_{H}
  \]
  by Theorem~\ref{T:integral representation H}. Since $\A_H$ is a $\sigma$-algebra, the proof of inclusion `$\supset$' is complete.
  
  `$\subset$': Let $f$ be such that $V(f)\in M(H)$ and $\ext X\subset[V(f)=1]\cup [V(f)=0]$. Set
  \[
  F=\{t\in L\setsep f(t,0)=1\}\mbox{ and }C=\{t\in A\setsep f(t,1)\ne f(t,-1)\}.
  \]
  Then $F$ is a Baire set and $C$ is a countable Baire set in $L$ by $(c)$. 
  The equality
  \[
  \begin{aligned}\relax
  [f=1]\cap\Ch_{E}K&= \psi^{-1}\left(F\right)\cap \Ch_{E}K \cup\{(t,i)\in C\times\{-1,1\}\setsep f(t,i)=1\}
  \end{aligned}
  \]
now yields the assertion. 
%Indeed, each set $C_n$ is Baire, and thus $F_n\setminus C_n$ is a Baire set in $L$. Also, the set $\{(t,i)\in C_n\times\{-1,1\}\setsep f(t,i)>0\}$ is of the prescribed type, since it is countable and Baire.

$(e)$: This is just a special case of Theorem~\ref{T:baire-multipliers}$(a)$.
\end{proof}

\subsection{Fragmented affine functions on  $X_{L,A}$}
\label{ssec:fragmented-dikous}

We continue by analyzing spaces $A_f(X_{L,A})$ and $(A_f(X_{L,A}))^\mu$. Their subclasses consisting of Borel functions will be analyzed in the subsequent section. We start by providing a topological lemma on fragmented functions and resolvable sets.

\begin{lemma}\label{L:dedicneres}   Let $L$ be a compact space.
    \begin{enumerate}[$(a)$]
        \item Let $B\subset L$. Then $B$ is scattered if and only if each subset of $B$ is resolvable.
        \item Let $B\subset L$. Then $B$ is $\sigma$-scattered if and only if each subset of $B$ belongs to $\sigma(\H)$ (the $\sigma$-algebra generated by resolvable sets).
        \item $(\Fr^b(L))^\mu=(\Fr^b(L))^\sigma$ and this system coincide with the family of all bounded $\sigma(\H)$-measurable functions on $L$.
    \end{enumerate}
\end{lemma}

\begin{proof}
    $(a)$: Assume $B$ is scattered. Then any subset of $B$ is also scattered, hence resolvable. Conversely, assume $B$ is not scattered, i.e., there exists a relatively closed set $H\subset B$ without isolated points. Then $F=\overline{H}$ is a compact set in $L$ without isolated points.  We assume first that $F\setminus H$ is not nowhere dense in $F$. Then there exists a closed set $C\subset F$ such that $\overline{C\cap B}=\overline{C\setminus B}=C$. Thus $B$ is not resolvable.
    In case $F\setminus H$ is nowhere dense in $F$, we use \cite[Theorem 3.7]{comfort} to get a pair $D_1,D_2\subset F$ of disjoint dense sets in $F$. Then both the sets $D_1\cap H$ and $D_2\cap H$ are dense in $H$. It follows that $D_1\cap H$ is a subset of $B$ which is not resolvable.

   $(b)$: The `only if' part follows from $(a)$.
   To prove the `if' part assume that $B$ is not $\sigma$-scattered. If $B\notin\sigma(\H)$, the proof is complete. Thus we assume that $B\in\sigma(\H)$. It follows from \cite[Theorem 6.13]{hansell-DT} that $B$ is `scattered $K$-analytic'. By combining \cite[Theorem 1]{holicky-cmj01} with \cite[Lemma 2.3]{namioka-pol} we deduce that $B$ is `almost \v{C}ech-analytic' in the sense of \cite{namioka-pol}. Hence, we deduce from \cite[Corollary~5.4]{namioka-pol} that $B$ contains a nonempty compact set $K$ without isolated points. Then $K$ can be mapped onto $[0,1]$ by a continuous function $\varphi\colon K\to [0,1]$. We select a set $C\subset [0,1]$ that is not analytic (by \cite[Theorem 29.7]{kechris} it is enough to take a non-measurable set.). Then $D=\varphi^{-1}(C)\subset K$ is not scattered-$K$-analytic. Indeed,  if this were the case, by \cite[Corollary 7]{HoSp} $C$ would be scattered-$K$-analytic and hence analytic by \cite[Proposition 1]{holicky-cmj93} and \cite[Theorem 25.7]{kechris}.) Thus $D\notin\sigma(\H)$ (by  \cite[Theorem 6.13]{hansell-DT}).

   $(c)$: Inclusion $(\Fr^b(L))^\mu\subset (\Fr^b(L))^\sigma$ is obvious. The converse follows easily from the fact that $\Fr^b(L)$ is a lattice. Further, the family of $\sigma(\H)$-measurable functions is closed to limits of pointwise converging sequences and  by Theorem~\ref{T:a} it contains $\Fr^b(L)$. So, it contains $(\Fr^b(L))^\mu$. Conversely, it is easy to see that $1_A\in (\Fr^b(L))^\mu$ for any $A\in\sigma(\H)$ and that simple functions are uniformly dense in bounded $\sigma(\H)$-measurable functions, which proves the converse inclusion.
\end{proof}

\begin{prop}\label{P:dikous-af-new}
Let $K=K_{L,A}$, $E=E_{L,A}$  $X=X_{L,A}$ and $H=A_f(X)$. Then the following assertions are valid.
\begin{enumerate}[$(a)$]
    \item $\begin{aligned}[t]
        H=V\big(\{f\in\ell^\infty(K)\setsep &f\circ\jmath\in \Fr^b(L) \\&\mbox{ and }f(t,0)=\tfrac12(f(t,1)+f(t,-1))\mbox{ for }t\in A\}\big).\end{aligned}$
\item  $\begin{aligned}[t]
        M^s&(H)=M(H)\\&=V\big(\{f\in\ell^\infty(K)\setsep f\circ\jmath\in \Fr^b(L), f(t,0)=\tfrac12(f(t,1)+f(t,-1))\mbox{ for }t\in A
        \\&\qquad\qquad
        \mbox{and }\forall\ep>0\colon \{t\in A\setsep \abs{f(t,1)-f(t,-1)}\ge\ep\}\mbox{ is scattered}\}\big).\end{aligned}$
\item $\begin{aligned}[t]
    \A_H=\A_H^s&= \{\phi(\psi^{-1}(F)\cap \Ch_{E} K\cup  C_1\times\{1\}\cup 
        C_2\times\{-1\})\setsep \\&\qquad\qquad F\subset L\text{ is a } \H_\sigma\mbox{ set},
        C_1,C_2\subset A\text{ are $\sigma$-scattered}\}
        \\&=\{\phi(B)\setsep B\subset \Ch_E K, \psi(B)\in\H_\sigma, \\&\qquad\qquad\psi(B)\cap\psi(\Ch_E K\setminus B)\mbox{ is $\sigma$-scattered}\}.
\end{aligned}$

 \item  $\begin{aligned}[t]
      \A_H&=\A^s_H=\ms_H\\&=\{\ext X\setminus F\setsep F\subset X \mbox{ is a split face}, F\in\H_\delta, F'\in\sigma(\H),\\&\qquad\qquad F, F'\mbox{ are measure extremal}\}\subset\Z_H\end{aligned}$      

             Moreover, the converse to the last inclusion holds if and only if $A$ contains no compact perfect subset.
       \end{enumerate}
       
\end{prop}

\begin{proof}
$(a)$: The assertion follows from Proposition~\ref{p:topol-stacey}$(g)$, Proposition~\ref{P:dikous-sa-new}$(b)$ and Lemma~\ref{L:function space}$(b)$.

$(b)$: Inclusion $M^s(H)\subset M(H)$ holds always. To prove the remaining ones we proceed similarly as for Baire functions. First observe that $V(g\circ\psi)\in M^s(H)$ for each $g\in \Fr^b(L)$. Hence $V(f)\in M^s(H)$ if and only if $V(f-f\circ\jmath\circ\psi)\in M^s(H)$ and also  $V(f)\in M(H)$ if and only if $V(f-f\circ\jmath\circ\psi)\in M(H)$. So, as above, it is enough to prove the remaining inclusions within the functions satisfying $f\circ\jmath=0$.

Let us continue by inclusion `$\subset$' from the second equality.
Assume that $f\circ\jmath=0$ and $V(f)\in M(H)$. Fix $\ep>0$ and let $$B=\{t\in A\setsep \abs{f(t,1)-f(t,-1)}\ge \ep\}.$$ Assume $B$ is not scattered. By Lemma~\ref{L:dedicneres}$(a)$ we find a subset $C\subset B$ which is not resolvable.
Let $g$ be the function on $K$ defined by
\[
g(t,i)=\begin{cases}
              1, & t\in C,  f(t,i)>0,\\
              -1, &t\in C, f(t,i)<0,\\
              0, &\text{otherwise}.
\end{cases}
\]
Then $V(g)\in H$ (by $(a)$). Since $V(f)\in M(H)$, we deduce $V(\widetilde{fg})\in H$. But for $t\in C$ we have
\[
\wt{fg}(t,0)=\begin{cases}
                0,& t\in L\setminus C,\\
                \frac12\abs{f(t,1)-f(t,-1)}\ge\frac{\ep}{2}, & t\in C.
\end{cases}
\]
Since $C$ is not resolvable, we easily deduce that  $\wt{fg}\circ \jmath$ is not fragmented, a contradiction.

To prove the remaining inclusion assume that $f$ satisfies the conditions in the right-hand side and $f\circ\jmath=0$. We aim to prove that $V(f)\in M^s(H)$. 
Given $\ep>0$ let 
$$B_\ep=\{t\in A\setsep \abs{f(t,1)-f(t,-1)}\ge \ep\} \quad\mbox{and}\quad f_\ep(t,i)=\begin{cases}
   f(t,i), & t\in B_\ep,\\ 0, & \mbox{otherwise}.
\end{cases}$$
Then $\|f-f_\ep\|\le\frac\ep2$. Since $M^s(H)$ is closed, it is enough to show that $V(f_\ep)\in M^s(H)$. 
So, let a function $g$ with $V(g)\in H$ be given. 
Then $\widetilde{f_\ep g}(t,0)=0$ for $t\in L\setminus B_\ep$. Since $B_\ep$ is scattered, it follows that $\widetilde{f_\ep g}\circ \jmath$ is fragmented and so $\widetilde{f_\ep g}\in H$. Moreover,
$$[\widetilde{f_\ep g}\ne f_\ep g]\subset B_\ep\times\{0\}.$$
Since $B_\ep$ is scattered and hence universally null, it follows from Lemma~\ref{L:maxmiry-dikous} that $V(\widetilde{f_\ep g})=V(f_\ep)V(g)$ $\mu$-almost everywhere for each maximal measure $\mu\in M^1(X)$. Thus $V(f_\ep)\in M^s(H)$ and the proof is complete.

$(c)$: The first equality follows from $(b)$, let us look at the second one:

`$\supset$': Let $F\subset L$ be an $\H_\sigma$ set. Then there exists a function $g\in \Fr^b(L)$ such that $[g>0]=F$. Then $f=g\circ\psi$ satisfies $V(f)\in M(H)$, and thus $\phi(\psi^{-1}(F)\cap \Ch_E K)\in \A_H$. If $C\subset A$ is scattered, the function $f=1_{C\times \{1\}}-1_{C\times \{-1\}}$ satisfies $V(f)\in M(H)$ by $(b)$. Hence $\phi(C\times\{1\})\in \A_H$, because $C\times\{1\}=[f>0]$. Similarly $\phi(C\times\{-1\})\in\A_H$. Now we conclude by noticing that $\A_H$ is closed with respect to finite intersections and countable unions, which finishes the proof of the inclusion `$\supset$'.

 `$\subset$': Let $f$ with $V(f)\in M(H)$ be given.  For each $n\in\en$ set
  \[
  F_n=\{t\in L\setsep f(t,0)>\tfrac1n\}\quad \mbox{and}\quad C_n=\{t\in A\setsep \abs{f(t,1)-f(t,-1)}\ge\tfrac1n\}.
  \]
  Then $F_n$ is an $\H_\sigma$ set and $C_n$ is a scattered set in $L$ (by $(b)$). The equality
  \[
  \begin{aligned}\relax
  [f>0]\cap\Ch_{E}K&= \psi^{-1}\left(\bigcup_{n\in\en} (F_n\setminus C_n)\right)\cap \Ch_{E}K \cup\\
 &\quad \bigcup_{n\in\en}\{(t,i)\in C_n\times\{-1,1\}\setsep f(t,i)>0\}.
  \end{aligned}
  \]
now yields the assertion. Indeed, each $C_n$ is a resolvable set, and thus $F_n\setminus C_n$ is an $\H_\sigma$ set. Hence their union is an $\H_\sigma$ set.

     $(d)$:  The first equality is repeated from $(c)$. Inclusion $\A^s_H\subset\ms_H$ follows from Theorem~\ref{T:meritelnost-strongmulti}. Inclusion `$\subset$' from the third equality is obvious. 

     Next assume that $F$ is a split face such that $F\in\H_\delta$, $F'\in\sigma(H)$ and $F,F'$ are measure extremal. By Corollary~\ref{cor:dikous-split-me} we know that $\lambda_F$ is strongly affine. Let $f$ be the function on $K$ such that $\lambda_F=V(f)$. Then $f$ attains only values $0,\frac12,1$. Since $[f=1]\in\H_\delta$ and $[f=0]\in\sigma(\H)$, we deduce that $[f=\frac12]\in\sigma(\H)$ as well. Further, by Lemma~\ref{L:dikous-split-new} the set $[f=\frac12]$ is universally null and hence it contains no compact perfect subset.
     It follows from the proof of Lemma~\ref{L:dedicneres}$(b)$ that $[f=\frac12]$ is $\sigma$-scattered. Let
     $$B=\phi^{-1}(\ext X\setminus F)=[f=0]\cap\Ch_E K.$$
     Then 
     $$L\setminus\psi(B)=\jmath^{-1}([f=1]\cap\jmath(L))\in\H_\delta,$$
     thus $\psi(B)\in\H_\sigma$. Further, $\psi(B)\cap\psi(\Ch_E K\setminus B)=\jmath^{-1}([f=\frac12])$ is $\sigma$-scattered.
     Thus $\phi(B)\in\A_H$ by $(c)$. This completes the proof of the equalities.    
   
The last inclusion is obvious. 

Assume $A$ contains no compact perfect subset. By Lemma~\ref{L:maxmiry-dikous} we know that $X$ is standard, so the equality follows from Lemma~\ref{L:SH=ZH}.

Conversely, assume that $A$ contains a compact perfect subset $D$. Let $f=1_{D\times\{1\}}+\frac12 1_{D\times\{0\}}$. Then $V(f)\in A_f(X)$, $\ext X\subset [V(f)=1]\cup[V(f)=0]$, but
$[V(f)>0]\cap \ext X\notin \A_H$ by $(c)$.
\end{proof}

\begin{prop}\label{P:dikous-Afmu-new}
Let $K=K_{L,A}$, $E=E_{L,A}$, $X=X_{L,A}$ and $H=(A_f(X))^\mu$.
Then the following assertions are valid.
    \begin{enumerate}[$(a)$] 
    \item  $\begin{aligned}[t]
            H=(A_f(X))^\sigma=V\big(\{f\in&\ell^\infty(K)\setsep f\circ\jmath\in (\Fr^b(L))^\sigma\mbox{ and }\\&f(t,0)=\tfrac12(f(t,1)+f(t,-1))\mbox{ for }t\in A\}\big).\end{aligned}$
     \item  $\begin{aligned}[t]
            M^s(H)=M(H)=V\big(\{f\in&\ell^\infty(K)\setsep f\circ\jmath\in (\Fr^b(L))^\sigma,\\&f(t,0)=\tfrac12(f(t,1)+f(t,-1))\mbox{ for }t\in A
            \\ & \{t\in A\setsep f(t,1)\neq f(t,-1)\} \mbox{ is $\sigma$-scattered}
            \}\big).\end{aligned}$
    
     \item  $\begin{aligned}[t]
     \A_H=\A^s_H=&
        \{\phi(\psi^{-1}(F)\cap \Ch_{E} K\cup  C_1\times\{1\}\cup 
        C_2\times\{-1\})\setsep\\&\qquad\qquad F\subset L, F\in\sigma(\H),
        C_1,C_2\subset A\mbox{ are $\sigma$-scattered}\}
        \\=&\{\phi(B)\setsep B\subset\Ch_E K, \psi(B)\in\sigma(\H), \\&\qquad\qquad\psi(B)\cap\psi(\Ch_E K\setminus B)\mbox{ is $\sigma$-scattered}\}.
        \end{aligned}
        $
 \item  $\begin{aligned}[t]
             \A_H&=\A_H^s=\ms_H\\&=\{F\cap\ext X\setsep F\mbox{ is a split face}, F,F' \mbox{belong to }\sigma(\H)\\&\qquad\mbox{ and are measure extremal}\}\subset\Z_H. \end{aligned} $      

             Moreover, the converse to the last inclusion holds if and only if $A$ contains no compact perfect subset.

      \end{enumerate}      
\end{prop}

\begin{proof}
    $(a)$: Inclusion `$\subset$' in the first equality is obvious, in the second one it follows from Proposition~\ref{P:dikous-af-new}$(b)$. %Also, $f$ satisfies the condition $f(t,0)=\frac12(f(t,1)+f(t,-1))$, $t\in A$, by Proposition~\ref{P:dikous-sa-new}(b).

For the proof of the converse inclusions we set
\[
\begin{aligned}
\A&=\{f\in\ell^\infty(K)\setsep f\circ \jmath=0 \text{ on } L\mbox{ and }f(t,0)=\tfrac12(f(t,1)+f(t,-1))\mbox{ for }t\in A\},\\
  \B&=\{f\in\ell^\infty(K)\setsep f\circ\jmath\in (\Fr^b(L))^\mu\mbox{ and }f(t,0)=\tfrac12(f(t,1)+f(t,-1))\mbox{ for }t\in A\},\\
\C&=\{f\in\ell^\infty(K)\setsep f\circ\jmath\in (\Fr^b(L))^\mu\mbox{ and }f=f\circ j\circ\psi\},\\
\D&=\{f\in\ell^\infty(K)\setsep f\circ\jmath\in \Fr^b(L)\mbox{ and }f=f\circ j\circ\psi\}.\\
\end{aligned}
\]
Note that due to Lemma~\ref{L:dedicneres}$(c)$ it is enough to
verify that $V(\B)\subset (A_f(X))^\mu$. To this end, let $b\in \B$ be given. Then we set 
$c=b\circ\jmath\circ\psi$ and $a=b-c$.
%\[
%a(t,i)=b(t,i)-b(t,0)\quad \text{and}\quad c(t,i)=b(t,0),\qquad (t,i)\in K.
%\]
Then $a\in \A$, $c\in \C$ and $b=a+c$. It follows that $\B\subset\A+\C$. Since 
 $V(\A)\subset A_f(X)\subset (A_f(X))^\mu$ by Proposition~\ref{P:dikous-af-new}$(a)$,  it is enough to check that $V(\C)\subset (A_f(X))^\mu$. But $V(\D)\subset A_f(X)$ by  Proposition~\ref{P:dikous-af-new}$(a)$ and clearly $\C=\D^\mu$, so $V(\C)=V(\D^\mu)=V(D)^\mu\subset (A_f(X))^\mu$, which finishes the proof.

$(b)$: We proceed similarly as in the proof of Proposition~\ref{P:dikous-af-new}$(b)$: Inclusion $M^s(H)\subset M(H)$ holds always. Again, $V(f\circ\psi)\in M^s(H)$ whenever $f\in(\Fr^b(L))^\sigma$, hence it is enough to prove the remaining inclusions within functions satisfying 
 $f\circ\jmath=0$ (i.e., $f\in\A$ using the notation from the proof of $(a)$).

Let us prove inclusion `$\subset$' from the second equality. Assume that $f\circ\jmath=0$ and $V(f)\in M(H)$.  Let 
\[
B=\{t\in A\setsep f(t,1)\neq f(t,-1)\}.
\]
We want to prove that $B$ is $\sigma$-scattered. Assume not. Then Lemma~\ref{L:dedicneres}$(b)$ provides a subset $C\subset B$ such that $C\notin\sigma(\H)$.
Set
\[
g(t,i)=\begin{cases}
         1,&t\in C, f(t,i)>0,\\
         -1,&t\in C, f(t,i)<0,\\
         0&\text{ otherwise}.
\end{cases}
\]
Then $g$ satisfies $V(g)\in A_f(X)\subset H$. It follows from $(a)$ that $\wt{fg}\circ \jmath\in (\Fr^b(L))^\sigma$.  But
\[
\{t\in L\setsep \wt{fg}(t,0)>0\}=C\notin\sigma(\H).
\]
 Hence $\wt{fg}\circ j\notin (\Fr^b(L))^\sigma$ (by Lemma~\ref{L:dedicneres}$(c)$), a contradiction.
 Hence $B$ is $\sigma$-scattered.

To prove the remaining inclusion assume that $f\circ\jmath=0$ and that the set
\[
B=\{t\in A\setsep f(t,1)\neq f(t,-1)\}
\]
is $\sigma$-scattered. Let now $g$ with $V(g)\in H$ be given. Then $\wt{fg}$ is nonzero only on some subset $C$ of $B$, which is $\sigma$-scattered. Hence it follows by Lemma~\ref{L:dedicneres}$(c)$ that $\wt{fg}\circ j\in (\Fr^b(L))^\sigma$, which proves that $V(\wt{fg})\in H$ by $(a)$. Moreover, $[\wt{fg}\ne fg]\subset B\times\{0\}$. Since $B$ is $\sigma$-scattered and hence universally null, it follows from Lemma~\ref{L:maxmiry-dikous} that $V(\wt{fg})=V(f)V(g)$ $\mu$-almost everywhere for each maximal measure $\mu\in M_1(X)$. 
Hence  $V(f)\in M^s(H)$.

$(c)$: The first equality follows from $(b)$. Let us prove the second one.

`$\subset$': Let $f$ with $V(f)\in M(H)$ and $\ext X\subset [V(f)=0]\cup[V(f)=1]$ be given. 
Set
  \[
  F=\{t\in L\setsep f(t,0)=1\}\quad \mbox{and}\quad C=\{t\in A\setsep f(t,1)\ne f(t,-1)\}.
  \]
Then $F\in\sigma(\H)$ and $C$ is a $\sigma$-scattered set in $L$ (by $(b)$).
The equality
  \[
  \begin{aligned}\relax
  [f=1]\cap\Ch_{E}K&= \psi^{-1}\left(F\right)\cap \Ch_{E}K \cup
 \{(t,i)\in C\times\{-1,1\}\setsep f(t,i)=1\}
  \end{aligned}
  \]
now yields the assertion.

`$\supset$': Let $F\subset L$ be a $\sigma(\H)$ set. Then $1_F\in (\Fr^b(L))^\sigma$ (by Lemma~\ref{L:dedicneres}$(c)$), hence $f=1_F\circ\psi$ satisfies $V(f)\in M(H)$, and thus $\phi(\psi^{-1}(F)\cap \Ch_E K)\in \A_H$. If $C\subset A$ is $\sigma$-scattered, the function $f=1_{C\times \{1\}}+\frac12\cdot 1_{C\times \{0\}}$ satisfies $V(f)\in M(H)$ by $(b)$. Hence $\phi(C\times\{1\})\in \A_H$. Similarly $\phi(C\times\{-1\})\in\A_H$. Now we conclude by noticing that $\A_H$ is closed with respect to finite intersections and countable unions, which finishes the proof of the inclusion ``$\supset$''.

The last equality follows from \eqref{eq:podmny ChEK}.

$(d)$: The first equality is repeated from $(c)$. Inclusion $\A_H^s\subset \ms_H$ follows from Theorem~\ref{T:meritelnost-strongmulti}. Inclusion `$\subset$' from the third equality is obvious.

 Next assume that $F$ is a split face such that $F,F'\in\sigma(\H)$ and $F,F'$ are measure extremal. By Corollary~\ref{cor:dikous-split-me} we know that $\lambda_F$ is strongly affine. Let $f$ be the function on $K$ such that $\lambda_F=V(f)$. Then $f$ attains only values $0,\frac12,1$. Since $[f=1]\in\sigma(\H)$ and $[f=0]\in\sigma(\H)$, we deduce that $[f=\frac12]\in\sigma(\H)$ as well. Further, by Lemma~\ref{L:dikous-split-new} the set $[f=\frac12]$ is universally null and hence it contains no compact perfect subset.
     It follows from the proof of Lemma~\ref{L:dedicneres}$(b)$ that $[f=\frac12]$ is $\sigma$-scattered.        
     Let
     $$B=\phi^{-1}(F\cap\ext X)=[f=1]\cap\Ch_E K.$$
     Then 
     $$\psi(B)=\jmath^{-1}(([f=1]\cup[f=\tfrac{1}{2}]\cap\jmath(L))\in\sigma(\H)$$
    and $\psi(B)\cap\psi(\Ch_E K\setminus B)=\jmath^{-1}([f=\frac12])$ is $\sigma$-scattered.
     Thus $\phi(B)\in\A_H$ by $(c)$. This completes the proof of the equalities.    
   
The last inclusion is obvious. 

Assume $A$ contains no compact perfect subset. By Lemma~\ref{L:maxmiry-dikous} we know that $X$ is standard, so the equality follows from Lemma~\ref{L:SH=ZH}.

Finally, assume that $A$ contains a compact perfect subset $D$. Let $f=1_{D\times\{1\}}+\frac12 1_{D\times\{0\}}$. Then $V(f)\in A_f(X)\subset H$, $\ext X\subset [V(f)=1]\cup[V(f)=0]$, but
$[V(f)=1]\cap \ext X\notin \A_H$ by $(c)$.
\end{proof}

\subsection{Borel strongly affine functions on $X_{L,A}$}

In this section we investigate spaces $\overline{A_s(X_{L,A})}$, $(A_s(X_{L,A}))^\mu$, $A_b(X_{L,A})\cap \Bo_1(X_{L,A})$ and $(A_b(X_{L,A})\cap \Bo_1(X_{L,A}))^\mu$. So, these spaces consist of Borel strongly affine functions which are not necessarily Baire.
We start by describing the space $\overline{A_s(X)}$. To this end, we need a basic information on derivation of a topological space. If $B$ is a topological space, let $B^{(0)}=B$, $B^{(1)}$ denote the set of all accumulation points of $B$ (in $B$). Inductively we define $B^{(n+1)}=(B^{(n)})^{(1)}$ for $n\in\en$. We say that $B$ is  \emph{of finite height} provided $B^{(n)}=\emptyset$ for some $n\in\en\cup\{0\}$. We further recall that a set $B$ is \emph{isolated} if all its points are isolated (i.e., $B^{(1)}=\emptyset$).

\begin{lemma}\label{L:finite rank}
    Let $B$ be a topological space. Then the following two assertions are equivalent.
    \begin{enumerate}[$(i)$]
        \item $B$ is of finite height.
        \item $B$ is a finite union of isolated subsets of $B$.
    \end{enumerate}
\end{lemma}

\begin{proof}
    $(i)\implies(ii)$: Assume $B^{(n)}=\emptyset$. Then sets 
    $$B^{(0)}\setminus B^{(1)},B^{(1)}\setminus B^{(2)},\dots B^{(n-1)}\setminus B^{(n)}$$
    are isolated and their union is $B$.

    $(ii)\implies(i)$: If $B$ is isolated, then $B^{(1)}=\emptyset$ and hence $B$ has finite height. To complete the proof it remains to show that $B$ has finite height provided $B=B_1\cup B_2$, where $B_1$ has finite height and $B_2$ is isolated. 
    This may be proved by induction on the height of $B_1$. If $(B_1)^{(0)}=\emptyset$, then $B=B_2$ is isolated, hence of finite height. Assume that $n\in\en\cup\{0\}$ and the statement holds if $(B_1)^{(n)}=\emptyset$.    
    Next, assume $(B_1)^{(n+1)}=\emptyset$. Let $x\in B_1$ be any isolated point of $B_1$. Then there is an open set $U\subset B$ with $U\cap B_1=\{x\}$, so $U\subset B_2\cup\{x\}$. It follows that $U^{(2)}=\emptyset$. Since $x$ is arbitrary, we deduce that $B^{(2)}\subset (B_1)^{(1)}\cup B_2$. Since $((B_1)^{(1)})^{(n)}=(B_1)^{(n+1)}=\emptyset$, the induction hypothesis says that $B^{(2)}$ is of finite height. Thus $B$ is of finite height as well.
\end{proof}

\begin{lemma}\label{L:sigma isolated} Let $L$ be a compact space and $B\subset L$.
\begin{enumerate}[$(a)$]
    \item If $B$ is isolated, then $B$ is an $(F\wedge G)$ set.
    \item If $B$ is of finite height, then $1_B\in\Bo_1(L)$.
    \item If $B$ is $\sigma$-isolated, then each subset of $B$ is an $(F\wedge G)_\sigma$ set.
\end{enumerate} 
\end{lemma}

\begin{proof} These results are known and easy to check. For the sake of completeness we provide the short proofs.

$(a)$: Assume $B$ is isolated. For each $t\in B$ let $U_t\subset L$ be open such that $U_t\cap B=\{t\}$. Then $G=\bigcup_{t\in B}U_t$ is an open set and $B=\overline{B}\cap G$.

$(b)$: If $B$ is isolated, the assertion follows from $(a)$. The general case then follows by Lemma~\ref{L:finite rank}.

Assertion $(c)$ follows from $(a)$.
\end{proof}

This easy lemma will be used in the following propositions. There are some interesting related open problems discussed in the following section.

\begin{prop}\label{P:dikous-lsc--new}
 Let $K=K_{L,A}$, $E=E_{L,A}$, $X=X_{L,A}$ and $H=\overline{A_s(X)}$. Then the following assertions are valid.
    \begin{enumerate}[$(a)$]
       \item Let $f\in\ell^\infty(K)\cap E^{\perp\perp}$ be such that  $V(f)\in A_l(X)$. Then $f\circ\jmath$ is lower semicontinuous on $L$ and for each $\ep>0$ the set $B_\ep=\{t\in A\setsep \abs{f(t,1)-f(t,-1)}\ge \ep\}$ is of finite height.
       %\item[(a)] For a function $f\in \ell^\infty(K)$ the following assertions are equivalent.
      %\begin{itemize}
       %   \item [(i)] It holds $f\in E^{\perp\perp}$ and $V(f)$ is lower semicontinuous on $X$.
        %  \item [(ii)] The function $f$ is lower semicontinuous on $K$ and $f(t,0)=\frac12(f(t,-1)+f(t,1))$ for $t\in A$.
    %\end{itemize}
    \item  $\begin{aligned}[t]
        H=V\big(\{f\in& \ell^\infty(K)\setsep f\circ \jmath\in \overline{\Lb(L)}, f(t,0)=\tfrac{1}{2}(f(t,1)+f(t,-1))\mbox{ for }t\in A,\\ &
        \forall\ep>0\colon \{t\in A\setsep \abs{f(t,1)-f(t,-1)}\ge \ep\}\mbox{ is of finite height}\}\big).
    \end{aligned}$
    
    \item  $M^s(H)=M(H)=H$.
    \item $\begin{aligned}[t]
        \A_H=\A^s_H&= \{\phi(\psi^{-1}(F)\cap \Ch_{E} K\cup  C_1\times\{1\}\cup 
        C_2\times\{-1\})\setsep\\
        &\qquad\qquad F\subset L\text{ is an } (F\wedge G)_\sigma\mbox{ set},C_1,C_2\subset A\text{ are $\sigma$-isolated}\}
        \\&=\{\phi(B)\setsep B\subset\Ch_E K, \psi(B)\mbox{ is an $\FpG_\sigma$ set},\\&\qquad\qquad
        \psi(B)\cap\psi(\Ch_E K\setminus B)\mbox{ is $\sigma$-isolated}\}.
    \end{aligned}$
    
     \item  $\A_H=\A^s_H=\ms_H=\Z_H$.
     
      \item In case $L=[0,1]$ the space $M(H)=M^s(H)$ is not determined by the family $\A_H=\A^s_H$.
              \end{enumerate}
\end{prop}

\begin{proof}
$(a)$: Assume that $V(f)\in A_l(X)$. Then $f$ is lower semicontinuous on $K$ and hence $f\circ \jmath$ is lower semicontinuous on $L$. Fix $\ep>0$ and let $B=B_\ep$.  
Let $M>\norm{f}=\norm{V(f)}$ be chosen.  We claim that $f(t,0)\le M-k\frac{\ep}{2}$ for each $t\in B^{(k)}$ (where $k\in\en\cup\{0\}$).

To this end we proceed inductively. For $k=0$ the assertion obviously holds true.
Assume that for some $k\in\en\cup\{0\}$ the claim holds. Let $t\in B^{(k+1)}$ be given. Then there exists a net $(t_\alpha)\subset B^{(k)}$ such that $t_\alpha\to t$. Since $\abs{f(t_\alpha, 1)-f(t_\alpha,-1)}\ge \ep$, we can select $i_\alpha\in \{-1,1\}$ such that $f(t_\alpha, i_\alpha)\le f(t_\alpha,0)-\frac{\ep}{2}$. By Proposition~\ref{p:topol-stacey}$(d)$ and the induction hypothesis  we get
\[
f(t,0)\le  \liminf_\alpha f(t_\alpha,0)-\tfrac{\ep}{2}\le M-k\tfrac{\ep}{2}-\tfrac{\ep}{2}= M-(k+1)\tfrac{\ep}{2}.
\]
This proves the claim.

The claim now yields that $B^{(n)}=\emptyset$ whenever $n\in\en$ satisfies $M-n\frac{\ep}2<-M$.
Hence $B$ is of finite height.

$(b)$: '$\subset$': If $V(f)\in \overline{A_s(X)}$, then clearly $f\circ\jmath\in\overline{\Lb(L)}$. Further, let
$$\F=\{f\in\ell^\infty(K)\cap E^{\perp\perp}\setsep \forall\ep>0\colon B_\ep(f)\mbox{ has finite height}\},$$
where $B_\ep(f)$ is given by the formula from $(a)$.
By assertion $(a)$ we know that $f\in\F$ whenever $V(f)$ is lower semicontinuous. It is easy to see that $\F$ is a closed linear subspace (note that Lemma~\ref{L:finite rank} implies that
the union of two sets of finite height has finite height). It follows that  $f\in\F$ whenever $V(f)\in H=\overline{A_s(X)}$.

\smallskip

Inclusion `$\supset$' will be proved in several steps:

\smallskip

{\tt Step 1:} Let $B\subset A$ be isolated and $f\in\ell^\infty(K)\cap E^{\perp\perp}$ satisfy $f=0$ on $\jmath(L)\cup ((A\setminus B)\times\{1,-1\})$. Then there exists a bounded lower semicontinuous function $g$ on $L$ such that $g\circ\psi+f$ is lower semicontinuous function on $K$. In particular, $V(f)\in A_s(X)$.

\smallskip

To prove the claim, fix $M>\norm{f}$. Since $B$ is isolated, for each $t\in B$ we find and open neighborhood $U_t$ of $t$ in $L$ such that $U_t\cap B=\{t\}$. Let $G=\bigcup\{U_t\setsep t\in  B\}$. Then  $G$ is an open set and we set $g=M\cdot 1_G$.
Clearly, $g$ is bounded and lower semicontinuous. 
We will check that $h=g\circ\psi +f$ is lower semicontinuous on $K$.

To this end fix $(t,i)\in K$. If $i\in\{1,-1\}$, then $(t,i)$ is an isolated point of $K$ and hence the function $h$ even continuous at $(t,i)$. So, assume $i=0$. If $t\in L\setminus G$, then $h(t)=0$. Since clearly $h\ge0$ on $K$, $h$ is lower semicontinuous at $t$. Finally, assume $t\in G$. It means that there is some $s\in B$ with $t\in V_s$. Then
$U=\psi^{-1}(V_s)\setminus\{(s,1),(s,-1)\}$ is an open neighborhood of $(t,0)$ and $h=M$ on $U$. Thus $h$ is continuous at $(t,0)$.

\smallskip

{\tt Step 2:} Let $B\subset A$ be of finite height and $f\in\ell^\infty(K)\cap E^{\perp\perp}$ satisfy $f=0$ on $\jmath(L)\cup ((A\setminus B)\times\{1,-1\})$. Then $V(f)\in A_s(X)$.

\smallskip

Since $B$ is a finite union of isolated sets (by Lemma~\ref{L:finite rank}), $f$ is the sum of a finite number of functions satisfying assumptions of Step 1. It is now enough to use that $A_s(X)$ is a linear space.

\smallskip

{\tt Step 3:} We complete the proof. Let $f$ satisfy the conditions on the right-hand side. Let $\ep>0$ be given. Then there is $g\in \Lb(L)$ such that $\norm{f\circ\jmath-g}<\ep$. Further, the set $B=\{t\in A\setsep \abs{f(t,1)-f(t,-1)}\ge \ep\}$  is of finite height. Let  
\[
h(t,i)=\begin{cases} 
               f(t,i)-f(t,0),& t\in B, i\in\{-1,1\},\\
               0,&\text{otherwise}.
         \end{cases}      
\]
By Step 2 we deduce that $V(h)\in A_s(X)$. By Proposition~\ref{P:dikous-sa-new}$(v)$ and Lemma~\ref{L:function space}$(b)$  we deduce that $V(g\circ \psi)\in A_s(X)$. So, $V(h+g\circ\psi)\in A_s(X)$ as well. Let us estimate $\|f-h-g\circ\psi\|$:

If $t\in L$, then
\[
\begin{aligned}
\abs{f(t,0)-h(t,0)-g(t))}&=\abs{(f\circ \jmath)(t)-g(t)}<\ep.
\end{aligned}
\]
If $t\in A\setminus B$ and $i\in\{1,-1\}$, then 
\[
\begin{aligned}
\abs{f(t,i)-h(t,i)-g(t)}&\le\abs{f(t,i)-f(t,0)}+\abs{f(t,0)-g(t)}\\&=\tfrac12\abs{f(t,1)-f(t,-1)}+\abs{(f\circ \jmath)(t)-g(t)} <\tfrac\ep2+\ep<2\ep
\end{aligned}
\]
Finally, if $t\in B$ and $i\in\{-1,1\}$, then
\[
\begin{aligned}
\abs{f(t,i)-h(t,i)-g(t)}&=\abs{f(t,0)-g(t))}=\abs{(f\circ \jmath)(t)-g(t)}<\ep.
\end{aligned}
\]
Thus $\norm{f-h-g\circ\psi}< 2\ep$. Since $\ep>0$ is arbitrary, we deduce $V(f)\in\overline{A_s(X)}=H$.

$(c)$:  The first equality follows from Proposition~\ref{P:multi strong pro As}. Since $X$ is a simplex, the second equality follows from Proposition~\ref{P:As pro simplex}.

$(d)$: The first equality follows from $(c)$. Let us prove the second one.

`$\subset$': Let $f$ with $V(f)\in M(H)=H$ be given.  For each $n\in\en$ set
  \[
  F_n=\{t\in L\setsep f(t,0)>\tfrac1n\}\quad \mbox{and}\quad C_n=\{t\in A\setsep \abs{f(t,1)-f(t,-1)}\ge\tfrac1n\}.
  \]
  Since $f\circ\jmath\in\overline{\Lb(L)}\subset\Bo_1(L)$, we deduce that $F_n$ is an $(F\wedge G)_\sigma$ set. By $(b)$ each $C_n$ is of finite height. The equality
  \[
  \begin{aligned}\relax
  [f>0]\cap\Ch_{E}K&= \psi^{-1}\left(\bigcup_{n\in\en} (F_n\setminus C_n)\right)\cap \Ch_{E}K \cup\\
 &\quad \bigcup_{n\in\en}\{(t,i)\in C_n\times\{-1,1\}\setsep f(t,i)>0\}.
  \end{aligned}
  \]
now yields the assertion. Indeed, each $C_n$ is a finite union of isolated sets (by Lemma~\ref{L:finite rank}) and thus $F_n\setminus C_n$ is an $(F\wedge G)_\sigma$ set (by Lemma~\ref{L:sigma isolated}(b)). Hence their union is an $(F\wedge G)_\sigma$ set.

`$\supset$': Let $G\subset L$ be an open set. Then $1_G\circ \psi\in A_s(X)$, hence $\psi^{-1}(G)\cap\Ch_E K\in\A_H$. The same holds for $F\subset L$ closed. Hence, by Proposition~\ref{p:system-aha} we deduce that $\psi^{-1}(F)\cap\Ch_E K\in\A_H$ for any $(F\wedge G)_\sigma$ set $F\subset L$.
 If $C\subset A$ is isolated, the function $f=1_{C\times \{1\}}-1_{C\times \{-1\}}$ satisfies $V(f)\in H$ by $(b)$. Hence $\phi(C\times\{1\})\in \A_H$, because $C\times\{1\}=[f>0]$. Similarly $\phi(C\times\{-1\})\in\A_H$. Now we conclude by noticing that $\A_H$ is closed with respect to finite intersections and countable unions, which finishes the proof of the inclusion `$\supset$'.

Finally, the last equality follows from \eqref{eq:podmny ChEK}.

 $(e)$:  The first equality is repeated from $(d)$. The remaining equalities follow from Proposition~\ref{P:As pro simplex}.

$(f)$: Assume $L=[0,1]$. Within the proof of Proposition~\ref{P:vetsi prostory srov}$(b)$
we found a function $f\in\Bo_1^b(L)\setminus\overline{\Lb(L)}$. Then $V(f\circ\psi)\notin H$ (by $(b)$), but $V(f\circ\psi)|_{\ext X}$ is $\A_H$-measurable by $(d)$.
\end{proof}

\begin{prop}\label{P:dikous-Asmu-new}
    Let $K=K_{L,A}$, $E=E_{L,A}$  $X=X_{L,A}$ and $H=(A_s(X))^\mu$. Then the following assertions are valid.
    \begin{enumerate}[$(a)$]
        \item $\begin{aligned}[t]
                    H=(A_s(X))^\sigma=V\big(\{f\in \ell^\infty(K)&\setsep
                      f\circ\jmath\in\Bo^b(L),\\&f(t,0)=\tfrac{1}{2}(f(t,1)+f(t,-1))\mbox{ for }t\in A, \\&  \{t\in A\setsep f(t,1)\ne f(t,-1)\}\mbox{ is $\sigma$-isolated}\}\big)\end{aligned}$ 
        \item $M^s(H)=M(H)=H$.
        \item $\begin{aligned}[t]
            \A_H=\A_H^s&= \{\phi(\psi^{-1}(F)\cap \Ch_{E} K\cup  C_1\times\{1\}\cup 
        C_2\times\{-1\})\setsep\\ &\qquad\qquad F\subset L\mbox{ is a Borel set, }
        C_1,C_2\subset A\mbox{ are $\sigma$-isolated}\}
        \\&=\{\phi(B)\setsep B\subset\Ch_E K, \psi(B)\mbox{ is a Borel set},\\&\qquad\qquad \psi(B)\cap\psi(\Ch_E K\setminus B)\mbox{ is $\sigma$-isolated}\}.
        \end{aligned}$
    
     \item $\A_H=\A^s_H=\ms_H=\Z_H=\Sigma'$, where $\Sigma'$ is the $\sigma$-algebra from Lemma~\ref{L:miry na ext}.  
    \end{enumerate}
\end{prop}

\begin{proof}
    $(a)$: The first equality follows from Proposition~\ref{P:vetsi prostory srov}$(d)$ as $X$ is a simplex. Inclusion `$\subset$' of the second inclusion follows from Proposition~\ref{P:dikous-lsc--new}$(b)$ using Lemma~\ref{L:finite rank}. To prove the converse we proceed similarly as in the proof of Proposition~\ref{P:dikous-Afmu-new}$(a)$: We set
    \[
\begin{aligned}
\A&=\{f\in\ell^\infty(K)\setsep f\circ \jmath=0 \text{ on } L\mbox{ and }f(t,0)=\tfrac12(f(t,1)+f(t,-1))\mbox{ for }t\in A\},\\
  \B&=\{f\in\ell^\infty(K)\setsep f\circ\jmath\in (\Lb(L))^\mu\mbox{ and }f(t,0)=\tfrac12(f(t,1)+f(t,-1))\mbox{ for }t\in A\},\\
\C&=\{f\in\ell^\infty(K)\setsep f\circ\jmath\in (\Lb(L))^\mu\mbox{ and }f=f\circ j\circ\psi\},\\
\D&=\{f\in\ell^\infty(K)\setsep f\circ\jmath\in \overline{\Lb(L)}\mbox{ and }f=f\circ j\circ\psi\},\\
\end{aligned}
\]

Note that due to Lemma~\ref{L:borel=Lbmu}$(b)$ it is enough to
verify that $V(\B)\subset (A_s(X))^\mu$. To this end, let $b\in \B$ be given. We set 
$c=b\circ\jmath\circ\psi$ and $a=b-c$.
%\[
%a(t,i)=b(t,i)-b(t,0)\quad \text{and}\quad c(t,i)=b(t,0),\qquad (t,i)\in K.
%\]
Then $a\in \A$, $c\in \C$ and $b=a+c$. It follows that $\B\subset\A+\C$. Since 
 $V(\A)\subset \overline{A_s(X)}\subset (A_s(X))^\mu$ by Proposition~\ref{P:dikous-lsc--new}$(b)$,  it is enough to check that $V(\C)\subset (A_s(X))^\mu$. But $V(\D)\subset \overline{A_s(X)}$ by  Proposition~\ref{P:dikous-lsc--new}$(b)$ and clearly $\C=\D^\mu$, so $V(\C)=V(\D^\mu)=V(D)^\mu\subset (A_s(X))^\mu$, which finishes the proof.

$(b)$: The first equality follows from Proposition~\ref{P:multi strong pro As}. Since $X$ is a simplex, the second equality follows from Proposition~\ref{P:As pro simplex}.

$(c)$: The first equality follows from $(b)$, let us prove the second one.

`$\subset$': Let $f$ with $V(f)\in M(H)=H$ and $\ext X\subset [V(f)=0]\cup[V(f)=1]$ be given. % such that $Y=[f>0]\cap \ext X$. 
Set
  \[
  F=\{t\in L\setsep f(t,0)=1\}\mbox{ and }C=\{t\in A\setsep f(t,1)\ne f(t,-1)\}.
  \]
Then $F$ is Borel and $C$ is a $\sigma$-isolated set in $L$ (by $(a)$).
The equality
  \[
  \begin{aligned}\relax
  [f=1]\cap\Ch_{E}K&= \psi^{-1}\left(F\right)\cap \Ch_{E}K \cup
 \{(t,i)\in C\times\{-1,1\}\setsep f(t,i)=1\}.
  \end{aligned}
  \]
now yields the assertion.

`$\supset$': Let $F\subset L$ be a Borel set. Then $1_F$ is a Borel function, hence $f=1_F\circ\psi$ satisfies $V(f)\in H= M(H)$ by $(a)$, and thus $\phi(\psi^{-1}(F)\cap \Ch_E K)\in \A_H$. If $C\subset A$ is $\sigma$-isolated, the function $f=1_{C\times \{1\}}+\frac12\cdot 1_{C\times \{0\}}$ satisfies $V(f)\in H= M(H)$ by $(a)$ (and Lemma~\ref{L:sigma isolated}$(c)$). Hence $\phi(C\times\{1\})\in \A_H$. Similarly $\phi(C\times\{-1\})\in\A_H$. Now we conclude by noticing that $\A_H$ is closed with respect to finite intersections and countable unions, which finishes the proof of the inclusion `$\supset$'.

Finally, the last equality follows from \eqref{eq:podmny ChEK}.

$(d)$: The first equality is repeated from $(c)$, the second and the third one follow from Proposition~\ref{P:As pro simplex}. Let us prove the last one. We will prove two inclusions:

\smallskip

$\Sigma'\subset \A_H$: Let $B\subset X$ be a Baire set. Then $B'=\phi^{-1}(B)$ is a Baire subset of $K$. Let 
$$\begin{gathered}
    C=\{t\in A\setsep B'\cap \{(t,-1),(t,0),(t,1)\} \mbox{ contains one or two points}\},\\ D=\{t\in L\setsep (t,0)\in B'\}.\end{gathered}$$
By Proposition~\ref{p:topol-stacey}$(c)$  we deduce that $D$ is a Baire set and $C$ is countable. Thus $D\setminus C$ is Borel and $C$ is $\sigma$-isolated and
$$B\cap\ext X=\phi(B'\cap\Ch_E K)=\phi( \psi^{-1}(D\setminus C)\cap\Ch_E K\cup (C\times\{-1,1\}\cap B'))\in\A_H$$
by $(c)$.

Further, let $F\subset X$ be a closed extremal set. Let $F'=\phi^{-1}(F)$. Then $F'$ is a closed set.
Let 
$$\begin{gathered}
    C=\{t\in A\setsep F'\cap \{(t,-1),(t,0),(t,1)\} \mbox{ contains one or two points}\},\\ D=\{t\in L\setsep (t,0)\in F'\}.\end{gathered}$$
Then $D$ is closed and, moreover, $\psi^{-1}(D)\subset F'$ as $F$ is extremal. Therefore $(t,0)\notin F'$ whenever $t\in C$. It follows that any accumulation point of $C$ belongs to $D$, thus $C$ is isolated. We get that
$$F\cap\ext X=\phi(F'\cap\Ch_E K)=\phi( \psi^{-1}(D)\cap\Ch_E K\cup (C\times\{-1,1\}\cap F'))\in\A_H$$
by $(c)$. 

This completes the proof of inclusion $\Sigma'\subset\A_H$.

\smallskip

$\A_H\subset\Sigma'$: Assume $F\subset L$ is closed. Then $F'=\phi(\psi^{-1}(F))$ is a closed extremal set, thus
$$\phi(\psi^{-1}(F)\cap\Ch_E L)=F'\cap\ext X\in \Sigma'.$$
Since $\Sigma'$ is a $\sigma$-algebra,  we deduce that
$\phi(\psi^{-1}(B)\cap\Ch_E L)\in\Sigma'$ for any Borel set $B\subset L$.

Further, let $C\subset A$ be an isolated set. Then $D=\overline{C}\setminus C$ is a closed subset of $L$ and 
$\phi(\psi^{-1}(D)\cup C\times\{1\})$ is a closed extremal set in $X$. Thus 
$$\phi(\psi^{-1}(D)\cap\Ch_E K\cup C\times\{1\})=\phi(\psi^{-1}(D)\cup C\times\{1\})\cap\ext X\in\Sigma'.$$
Since $\Sigma'$ is a $\sigma$-algebra, using the previous paragraph we deduce that
\[
\phi(C\times\{1\})=\phi(\psi^{-1}(D)\cap\Ch_E K\cup C\times\{1\})\setminus \phi(\psi^{-1}(D)\cap\Ch_E K)\in\Sigma'.
\]
Similarly $\phi(C\times\{-1\})\in\Sigma'$. Using $(c)$ we now easily conclude that $\A_H\subset\Sigma'$ and the proof is complete.
\end{proof}

Now we come to the case of affine functions of the first Borel class on $X$. To simplify its formulation we introduce the following piece of notation. If $L$ is a compact space, we set
$$\begin{aligned} 
\H_B&=\{D\subset L\setsep \forall C\subset D\colon 1_C\in\Bo_1(L)\} \\&= \{D\subset L\setsep \forall C\subset D\colon C\in \FpG_\sigma\ \&\ C\in\FsG_\delta\}.\end{aligned}$$
It follows from Lemma~\ref{L:sigma isolated}$(b)$ that any set of finite height belongs to $\H_B$. Further, by Lemma~\ref{L:dedicneres}$(a)$ we get that any element of $\H_B$ is scattered.

\begin{prop}\label{P:dikous-Bo1-new}
Let $K=K_{L,A}$, $E=E_{L,A}$  $X=X_{L,A}$ and $H=\Bo_1(X)\cap A_b(X)$. Then the following assertions are valid.
\begin{enumerate}[$(a)$]
    \item $\begin{aligned}[t]
            H=V(\{f\in\ell^\infty(K)\setsep f\circ\jmath&\in \Bo_1(L), \mbox{ and }\\&f(t,0)=\tfrac12(f(t,1)+f(t,-1))\mbox{ for }t\in A\}.\end{aligned}$
\item $\begin{aligned}[t]
            M^s(H)=M(H)=V\big(\{f\in&\ell^\infty(K)\setsep f\circ\jmath\in \Bo_1(L),\\&f(t,0)=\tfrac12(f(t,1)+f(t,-1))\mbox{ for }t\in A,\\
            & \forall\ep>0\colon \{t\in A\setsep \abs{f(t,1)-f(t,-1)}\ge\ep\}\in\H_B \}.\end{aligned}$
\item $\begin{aligned}[t]
    \A_H=\A^s_H&=\{\phi(\psi^{-1}(F)\cap \Ch_{E} K\cup  C_1\times\{1\}\cup 
        C_2\times\{-1\})\setsep \\
        &\qquad\qquad F\subset L\text{ is an } \FpG_\sigma\mbox{ set},\ C_1,C_2\subset A,\ C_1,C_2\in(\H_B)_\sigma\}
        \\&=\{\phi(B)\setsep B\subset \Ch_E K, \psi(B)\in\FpG_\sigma, \\ &\qquad\qquad
        \psi(B)\cap\psi(\Ch_E K\setminus B)\in (\H_B)_\sigma\}
\end{aligned}$

   \item  $\A_H=\A^s_H\subset\ms_H\subset\Z_H$.  
            The first inclusion is proper if, for example, $A$ contains a homeomorphic copy of the ordinal interval $[0,\omega_1)$.
            The converse to the second inclusion holds if and only if $A$ contains no compact perfect subset.
       \end{enumerate}
\end{prop}

\begin{proof}
 $(a)$: This follows by combining Proposition~\ref{p:topol-stacey}$(e)$, Proposition~\ref{P:dikous-sa-new}$(b)$ and Lemma~\ref{L:function space}$(b)$.

 $(b)$: Inclusion $M^s(H)\subset M(H)$ holds always. To prove the remaining ones we proceed similarly as in the previous proofs. First observe that $V(g\circ\psi)\in M^s(H)$ for each $g\in \Bo_1^b(L)$.  So, as above, it is enough to prove the remaining inclusions within functions satisfying $f\circ\jmath=0$.

We continue by inclusion `$\subset$' from the second equality. Assume that $f\circ\jmath=0$ and $V(f)\in M(H)$. Fix $\ep>0$ and let $B=\{t\in A\setsep \abs{f(t,1)-f(t,-1)}\ge \ep\}$. Assume there is $C\subset B$ such that $1_C\notin \Bo_1(L)$.
Let $g$ be the function on $K$ defined by
\[
g(t,i)=\begin{cases}
              1, & t\in C,  f(t,i)>0,\\
              -1, &t\in C, f(t,i)<0,\\
              0, &\text{otherwise}.
\end{cases}
\]
Then $V(g)\in H$ (by $(a)$). Since $V(f)\in M(H)$, we deduce $V(\widetilde{fg})\in H$. But  for $t\in C$ we have
\[
\wt{fg}(t,0)=\begin{cases}
                0,& t\in L\setminus C,\\
                \frac12\abs{f(t,1)-f(t,-1)}\ge\frac{\ep}{2} & t\in C.
\end{cases}
\]
Since $1_C\notin\Bo_1(L)$, we easily deduce that  $\wt{fg}\circ\jmath\notin\Bo_1(L)$, a contradiction.

To prove the remaining inclusion assume that $f$ satisfies the conditions on the right-hand side and $f\circ\jmath=0$. 
Given $\ep>0$ let 
$$B_\ep=\{t\in A\setsep \abs{f(t,1)-f(t,-1)}\ge \ep\} \quad\mbox{and}\quad f_\ep(t,i)=\begin{cases}
   f(t,i) & t\in B_\ep,\\ 0 & \mbox{otherwise}.
\end{cases}$$
Then $\|f-f_\ep\|\le\frac\ep2$. Since $M^s(H)$ is closed, it is enough to show that $V(f_\ep)\in M^s(H)$. 
So, let $g$ with $V(g)\in H$ be given. Then $\widetilde{f_\ep g}(t,0)=0$ for $t\in L\setminus B_\ep$. Our assumption implies that
any function on $L$ which is zero outside $B_\ep$ belongs to $\Bo_1(L)$. Hence $V(\widetilde{f_\ep g})\in H$. Moreover, 
$[\widetilde{f_\ep g}\ne f_\ep g]\subset B_\ep\times\{0\}$. Since $B_\ep\in\H_B$, it is scattered (see the remarks before this proposition) and hence universally null. It follows from Lemma~\ref{L:maxmiry-dikous} that $V(\widetilde{f_\ep g})=V(f_\ep)V(g)$ $\mu$-almost everywhere for each maximal $\mu\in M_1(X)$. Hence $V(f_\ep)\in M^s(H)$ and this completes the proof.

$(c)$: The first equality follows from $(b)$. Let us prove the second one.

`$\supset$': Let $F\subset L$ be a $\FpG_\sigma$ set. Then there exists a bounded function $g\in \Bo_1(L)$ such that $[g>0]=F$. Then $f=g\circ\psi$ satisfies $V(f)\in M(H)$, and thus $\phi(\psi^{-1}(F)\cap \Ch_E K)\in \A_H$. If $C\subset A$ is  such that $C\in\H_B$, the function $f=1_{C\times \{1\}}-1_{C\times \{-1\}}$ satisfies $V(f)\in M(H)$ by $(b)$. Hence $\phi(C\times\{1\})\in \A_H$, because $C\times\{1\}=[f>0]$. Similarly  $\phi(C\times\{-1\})\in \A_H$. Now we conclude by noticing that $\A_H$ is closed with respect to finite intersections and countable unions, which finishes the proof of the inclusion `$\supset$'.

 `$\subset$': Let $f$ with $V(f)\in M(H)$ be given.  For each $n\in\en$ set
  \[
  F_n=\{t\in L\setsep f(t,0)>\tfrac1n\}\mbox{ and }C_n=\{t\in A\setsep \abs{f(t,1)-f(t,-1)}\ge\tfrac1n\}.
  \]
  Then $F_n$ is a $\FpG_\sigma$ set and $C_n\in\H_B$ (by $(b)$). The  equality
  \[
  \begin{aligned}\relax
  [f>0]\cap\Ch_{E}K&= \psi^{-1}\left(\bigcup_{n\in\en} (F_n\setminus C_n)\right)\cap \Ch_{E}K \cup\\
 &\quad \bigcup_{n\in\en}\{(t,i)\in C_n\times\{-1,1\}\setsep f(t,i)>0\}.
  \end{aligned}
  \]
now yields the assertion. Indeed, each $C_n$ is also a $\FsG_\delta$ set, and thus $F_n\setminus C_n$ is a $\FpG_\sigma$ set. Hence their union is a $\FpG_\sigma$ set.

Finally, the last equality follows from \eqref{eq:podmny ChEK}.

$(d)$: The first equality is repeated from $(c)$. Inclusion $\A_H^s\subset\ms_H$ follows from Theorem~\ref{T:meritelnost-strongmulti}.
The second inclusion is obvious.

Assume that $B\subset A$ is homeomorphic to $[0,\omega_1)$. Let $f=1_{B\times\{1\}}+\frac12 1_{B\times\{0\}}$. Then $V(f)\in H$ (as $f\circ\jmath=\frac12\cdot 1_B\in\Bo_1(L)$). Further, $[V(f)=1]$ is a split face by Lemma~\ref{L:dikous-split-new} as $B$ is scattered and hence universal null. Hence $\ext X\cap [V(f)=1]\in\ms_H$.
However, $\ext X \cap [V(f)=1] \notin\A_H$ by $(c)$ as $B$ contains a non-Borel subset. (It is enough to take a stationary set whose complement is also stationary provided for example by \cite[Lemma 8.8]{jech-book} and use \cite[Lemma 1]{raorao71}.)

Assume $A$ contains no compact perfect subset. By Lemma~\ref{L:maxmiry-dikous} we know that $X$ is standard and hence $\ms_H=\Z_H$ by Lemma~\ref{L:SH=ZH}.

Finally, assume that $A$ contains a compact perfect subset $D$. Let $f=1_{D\times\{1\}}+\frac12 1_{D\times\{0\}}$. Then $V(f)\in H$, $\ext X\subset [V(f)=1]\cup[V(f)=0]$, but
$[V(f)=1]\cap\ext X\notin\ms_H$  by Lemma~\ref{L:dikous-split-new}.
\end{proof}

\begin{prop}\label{P:dikous-bo1-mu-new}
Let $K=K_{L,A}$, $E=E_{L,A}$, $X=X_{L,A}$ and $H=(\Bo_1(X)\cap A_b(X))^\mu$. Then the following assertions are valid.
\begin{enumerate}[$(a)$]
         
    \item  $\begin{aligned}[t]
           H=(\Bo_1(X)\cap A_b(X))^\sigma=V&(\{f\in\ell^\infty(K)\setsep f\circ\jmath\in \Bo^b(L)\mbox{ and }\\&f(t,0)=\tfrac12(f(t,1)+f(t,-1))\mbox{ for }t\in A\}). \end{aligned}$
     \item $\begin{aligned}[t]
           M^s&(H)=M(H)\\&=V\big(\{f\in\ell^\infty(K)\setsep f\circ\jmath\in \Bo^b(L),f(t,0)=\tfrac12(f(t,1)+f(t,-1))\mbox{ for }t\in A,\\& \qquad\qquad
           \{t\in A\setsep f(t,1)\neq f(t,-1)\} \mbox{ is a hereditarily Borel subset of }L
           \}\big). \end{aligned}$

      \item  $\begin{aligned}[t]
          \A_H=\A^s_H&=\{\phi(\psi^{-1}(F)\cap \Ch_{E} K\cup  C_1\times\{1\}\cup 
        C_2\times\{-1\})\setsep \\& \qquad\qquad F\subset L\text{ Borel},      C_1,C_2\subset A\mbox{ hereditarily Borel subsets of }L\}
        \\&=\{\phi(B)\setsep B\subset\Ch_E K, \psi(B)\mbox{ Borel},\\&
        \qquad\qquad\psi(B)\cap\psi(\Ch_E K\setminus B)\mbox{ hereditarily Borel}\}.
      \end{aligned}$

      \item  $\A_H=\A^s_H\subset\ms_H\subset\Z_H$. 
            The first inclusion is proper if, for example, $A$ contains a homeomorphic copy of the ordinal interval $[0,\omega_1)$.
            The converse to the second inclusion holds if and only if $A$ contains no compact perfect subset.
         \end{enumerate}
\end{prop}

\begin{proof}
(a) Inclusions `$\subset$' in both equalities are obvious. 
 For the proof of the converse inclusion we proceed as in the preceding proofs. We set
\[
\begin{aligned}
\A&=\{f\in\ell^\infty(K)\setsep f\circ \jmath=0 \text{ on } L\mbox{ and }f(t,0)=\tfrac12(f(t,1)+f(t,-1))\mbox{ for }t\in A\},\\
  \B&=\{f\in\ell^\infty(K)\setsep f\circ\jmath\in \Bo^b(L)\mbox{ and }f(t,0)=\tfrac12(f(t,1)+f(t,-1))\mbox{ for }t\in A\},\\
\C&=\{f\in\ell^\infty(K)\setsep f\circ\jmath\in \Bo^b(L)\mbox{ and }f=f\circ \jmath\circ\psi\},\\
\D&=\{f\in\ell^\infty(K)\setsep f\circ\jmath\in \Bo_1^b(L)\mbox{ and }f=f\circ \jmath\circ\psi\}
\end{aligned}
\]
We want to verify that $V(\B)\subset H$. To this end, let $b\in \B$ be given. Then we set 
$c=b\circ\jmath\circ\psi$ and $a=b-c$.
Then $a\in \A$, $c\in \C$ and $b=a+c$. Thus $\B\subset\A+\C$. Since $V(\A)\subset A_b(X)\cap\Bo_1(X)\subset H$ by Proposition~\ref{P:dikous-Bo1-new}$(a)$, it is enough to show that $V(\C)\subset H$. But $\C=\D^\mu$ by Lemma~\ref{L:borel=Lbmu}, so 
$$V(\C)=V(\D^\mu)\subset V(\D)^\mu\subset(A_b(X)\cap\Bo_1(X))^\mu=H,$$ where we used Proposition~\ref{P:dikous-Bo1-new}$(a)$.

$(b)$: Inclusion $M^s(H)\subset M(H)$ holds always. Similarly as above, it is enough to prove the remaining inclusions within functions satisfying $f\circ\jmath=0$.

Let us prove inclusion `$\subset$' from the second equality. Let $f$ with $f\circ\jmath=0$ and $V(f)\in M(H)$ be given. Let $C\subset B=\{t\in A\setsep f(t,1)\neq f(t,-1)\}$ be arbitrary. Then the function
\[
g(t,i)=\begin{cases}
         1,&t\in C, h(t,i)>0,\\
         -1,&t\in C, h(t,i)<0,\\
         0,&\text{ otherwise}
\end{cases}
\]
satisfies $V(g)\in H$, and thus $\wt{fg}\circ j\in \Bo^b(L)$. Since $C=\{t\in L\setsep \wt{fg}(t,0)>0\}$, the set $C$ is a Borel subset of $L$.

To prove the remaining inclusion assume that $f$ satisfies the condition on the right-hand side and $f\circ\jmath=0$. Set 
$$B=\{t\in A\setsep f(t,1)\ne f(t,-1)\}.$$
We shall show that $V(f)\in M^s(H)$. Let $g$ be given such that $V(g)\in H$. 
Then $\{t\in L\setsep \wt{fg}(t,0)\ne 0\}\subset B$ and thus $\wt{fg}\circ j\in\Bo^b(L)$. Moreover, 
$[\wt{fg}\ne fg]\subset B\times\{0\}$. Since $B$ is hereditarily Borel, it is $\sigma$-scattered by Lemma~\ref{L:dedicneres}$(b)$, in particular it is universally null. It follows from Lemma~\ref{L:maxmiry-dikous} that $V(\wt{fg})=V(f)V(g)$ $\mu$-almost everywhere for each maximal $\mu\in M_1(X)$. This completes the proof that $V(f)\in M^s(H)$.

$(c)$: The first equality follows from $(b)$, let us prove the second one.

`$\subset$': Let $f$ with $V(f)\in M(H)$ and $\ext X\subset [V(f)=0]\cup[V(f)=1]$ be given.  
Set
  \[
  F=\{t\in L\setsep f(t,0)=1\}\quad \mbox{and}\quad C=\{t\in A\setsep f(t,1)\ne f(t,-1)\}.
  \]
Then $F$ is Borel and $C$ is hereditarily Borel in $L$ (by $(b)$).
The  equality
  \[
  \begin{aligned}\relax
  [f=1]\cap\Ch_{E}K&= \psi^{-1}\left(F\right)\cap \Ch_{E}K \cup
 \{(t,i)\in C\times\{-1,1\}\setsep f(t,i)=1\}.
  \end{aligned}
  \]
now yields the assertion.

`$\supset$': Let $F\subset L$ be a Borel set. Then $1_F$ is a Borel function, hence $f=1_F\circ\psi$ satisfies $V(f)\in M(H)$ by $(b)$, and thus $\phi(\psi^{-1}(F)\cap \Ch_E K)\in \A_H$. If $C\subset A$ is hereditarily Borel in $L$, the function $f=1_{C\times \{1\}}+\frac12\cdot 1_{C\times \{0\}}$ satisfies $V(f)\in M(H)$ by $(b)$. Hence $\phi(C\times\{1\})\in \A_H$. Similarly $\phi(C\times\{-1\})\in\A_H$. Now we conclude by noticing that $\A_H$ is closed with respect to finite intersections and countable unions, which finishes the proof of the inclusion `$\supset$'.

$(d)$: The proof can be done by copying the proof of Proposition~\ref{P:dikous-Bo1-new}$(d)$. 
\end{proof}

\subsection{Overview and counterexamples}
\label{ssec:relations-stacey-ifs}

The aim of this part is to summarize relations between the considered intermediate function spaces on Stacey's simplices. We also relate these examples to the general theory, give some counterexamples and formulate some questions.
First we provide a basic overview of the considered spaces in the following table

\begin{prop}
\label{p:vztahy-multi-ifs}
Let $K=K_{L,A}$, $E=E_{L,A}$ and $X=X_{L,A}$. Then the following assertions hold
(we omit the letter $X$).
\begin{enumerate}[$(a)$]
\item We have the following inclusions between intermediate function spaces:
\begin{equation}
    \label{eq:table-ifs}
\begin{array}{ccccccccc}
          &        &A_c           &\subset &A_1                &\subset  &(A_{sa}\cap \Ba^b)& =& (A_{c})^\mu \\
          &        &\cap       &        &\cap            &         &                &  &  \\
          &        &\overline{A_s}&\subset &A_b\cap \Bo_1      & \subset & A_f            &  & \\  
          &        & \cap      &        &\cap            &         &\cap         &  &   \\
 (A_c)^\mu&\subset &(A_s)^\mu     &\subset & (A_b\cap\Bo_1)^\mu&\subset  &(A_f)^\mu      & \subset & A_{sa}.
\end{array}
\end{equation}
Moreover, it may happen that all the inclusions are proper and no more inclusions hold.
 \item  Whenever we have an inclusion in table~\eqref{eq:table-ifs}, the same inclusion holds for the spaces of multipliers. Moreover, it may happen that all the inclusions are strict except possibly for $M((A_s)^\mu)\subset M((A_b\cap\Bo_1)^\mu)$.
 \item  If $H$ is any of the spaces $A_c, A_1,\overline{A_s}, A_b\cap \Bo_1, A_f$, then $H^\mu=H^\sigma$.
    \item If $H$ is any of the spaces
    listed in table~\eqref{eq:table-ifs},     then $Z(H)=M(H)=M^s(H)$.
   \item If $H$ is any of the spaces $A_1$, $\overline{A_s}$, $A_f$, then $M(H^\mu)=(M(H))^\mu$.  
      Further, $(M(A_c))^\mu = M((A_c)^\mu)$ if and only if  $A$ contains no $G_\delta$-point of $L$.   
\end{enumerate}
\end{prop}

\begin{proof}
$(a)$ and $(b)$:
The relations in table \eqref{eq:table-ifs} are valid for any simplex $X$. The validity in our case is also witnessed by the above-given characterizations.  The validity of the same inclusions of the spaces of multipliers follows from the above description of these spaces (together with Lemmata~\ref{L:dedicneres}, \ref{L:finite rank} and~\ref{L:sigma isolated}). 

\smallskip

To distinguish these spaces we choose
$L=A=[0,1]^{\er}$:

Denote by $p:L\to[0,1]$ the projection onto the first coordinate. Then the function $1_{\{0\}}\circ p\circ\psi$ witnesses that $A_c\subsetneqq A_1$ and $1_{\qe}\circ p\circ \psi$ witnesses that $A_1\subsetneqq (A_c)^\mu$. The same functions witness that $M(A_c)\subsetneqq M(A_1)\subsetneqq M((A_c)^\mu)$.

Let $f\in \Ba^b_1([0,1])\setminus\overline{\Lb([0,1])}$ be the function provided by \cite[Proposition 5.1]{odell-rosen}. Then $f\circ p\circ \psi$ witnesses that $A_1\not\subset\overline{A_s}$ and hence $\overline{A_s}\subsetneqq A_b\cap\Bo_1$.
The same function witnesses that
$M(A_1)\not\subset M(\overline{A_s})$
and $M(\overline{A_s})\subsetneqq M(A_b\cap\Bo_1)$. 
We also deduce that $\overline{A_s}\subsetneqq(A_s)^\mu$ and $M(\overline{A_s})\subsetneqq M((A_s)^\mu)$.

Let $x\in L$ be arbitrary. Since $x$ is not a $G_\delta$ point, the function $1_{\{x\}}\circ \psi$ witnesses that  $\overline{A_s}\not\subset (A_c)^\mu$ and hence $A_c\subsetneqq \overline{A_s}$, $A_1\subsetneqq A_b\cap\Bo_1$ and  $(A_c)^\mu\subsetneqq (A_s)^\mu$. The same function shows that 
$M(\overline{A_s})\not\subset M((A_c)^\mu)$, $M(A_c)\subsetneqq M(\overline{A_s})$, $M(A_1)\subsetneqq M(A_b\cap\Bo_1)$ and $M((A_c)^\mu)\subsetneqq M((A_s)^\mu)$.

The function $1_{L\times\{1\}}-1_{L\times\{-1\}}$ witnesses that  $(A_s)^\mu\subsetneqq (A_b\cap \Bo_1)^\mu$. (Note that it says nothing on the respective spaces of multipliers.)

Further, the space $L$ is separable, so we may take a countable dense set $C\subset L$.
 The function $1_C\circ\psi$ witnesses that $(A_c)^\mu\not\subset A_f$ and $A_b\cap\Bo_1\subsetneqq(A_b\cap\Bo_1)^\mu$. The same function witnesses  that $M((A_c)^\mu)\not\subset M(A_f)$ and $M(A_b\cap\Bo_1)\subsetneqq M((A_b\cap\Bo_1)^\mu)$.

There is also $D\subset L$ homeomorphic to the ordinal interval $[0,\omega_1]$ and $N\subset D$ non-Borel (this follows from \cite[Lemma 8.8]{jech-book} and \cite[Lemma 1]{raorao71} as remarked in the proof of Proposition~\ref{P:dikous-Bo1-new}$(d)$) . Then $1_N\circ \psi$ witnesses that $A_b\cap\Bo_1\subsetneqq A_f$ and $A_f\not\subset (A_b\cap \Bo_1)^\mu$.  The same function shows that  $M(A_b\cap\Bo_1)\subsetneqq M(A_f)$ and $M(A_f)\not\subset M((A_b\cap \Bo_1)^\mu)$. 
Moreover, $(1_N+1_C)\circ\psi$ witnesses that $A_f\subsetneqq (A_f)^\mu$ and $(A_b\cap\Bo_1)^\mu\subsetneqq (A_f)^\mu$. The same function shows that  $M(A_f)\subsetneqq M((A_f)^\mu)$ and $M((A_b\cap\Bo_1)^\mu)\subsetneqq M((A_f)^\mu)$.
Finally, if $S\subset [0,1]$ is an analytic non-Borel set, then $1_S\circ p\circ\psi$ witnesses that $(A_f)^\mu\subsetneqq A_{sa}$ and $M((A_f)^\mu)\subsetneqq M(A_{sa})$.

$(c)$: The case of $A_c$ and $A_1$ follows from Proposition~\ref{P:Baire-srovnani}(iii).
The remaining cases follow Propositions~\ref{P:dikous-Asmu-new}, \ref{P:dikous-bo1-mu-new} and~\ref{P:dikous-Afmu-new}.

 $(d)$: Equality $M(H)=M^s(H)$ in all the cases follows from the above propositions. Equality $Z(H)=M(H)$ follows from \cite[Theorem II.7.10]{alfsen} for $A_c$, from Proposition~\ref{p:postacproa1}$(a)$ for $A_1$ and $(A_c)^\mu$, 
 from Corollary~\ref{cor:iotax} and Corollary~\ref{cor:rovnost na ext}$(b)$ for the remaining spaces.
 
 $(e)$: This follows easily from the above characterizations.
\end{proof}

We point out that  assertion $(d)$ of the previous proposition provides a partial positive answer to Question~\ref{q:m=ms}. On the other hand,  the following question seems to be open:

\begin{ques}
    Let $X=X_{L,A}$. Is it true that
   $$\begin{gathered}
         M((A_s(X))^\mu)=M((A_b(X)\cap\Bo_1(X))^\mu)\mbox{ and }\\ M((A_b(X)\cap\Bo_1(X))^\mu)=(M(A_b(X)\cap\Bo_1(X)))^\mu \ ? \end{gathered}$$   
\end{ques}

In view of the above characterizations this question is related to the following topological problem which is, up to our knowledge, open.

\begin{ques}
    Let $L$ be a compact space and let $B\subset L$ be a hereditarily Borel set. Is $B$ necessarily $\sigma$-isolated?
\end{ques}

Such a set $B$ must be $\sigma$-scattered by Lemma~\ref{L:dedicneres}$(b)$, but we do not know whether this stronger conclusion is valid.

We continue by an overview of the important special case of metrizable starting space $L$.

\begin{prop}
 Let $K=K_{L,A}$, $E=E_{L,A}$ and $X=X_{L,A}$. Assume that $L$ is metrizable. 
 Then the following assertions hold (we omit the letter $X$).
\begin{enumerate}[$(a)$]
    \item We have the following inclusions between intermediate function spaces:
\begin{equation}
    \label{eq:table-ifs-metriz}
\begin{array}{ccccccccc}
          A_c&   \subset     & \overline{A_s}       &\subset &A_1          &\subset  &(A_{sa}\cap \Ba^b)& =& (A_{c})^\mu \\
          &        &       &        &\cap            &         &   \cap             &  & \parallel \\
          &        && &A_b\cap \Bo_1      & \subset & (A_b\cap\Bo_1)^\mu            &  & (A_s)^\mu\\  
          &        &    &        &\parallel            &         &\parallel         &  &   \\
 & &     & &A_f&\subset &(A_f)^\mu      & \subset & A_{sa}.
\end{array}\end{equation}
Moreover, it may happen that all the inclusions are proper and no more inclusions hold.
\item  We have the following inclusions between the respective spaces of multipliers:
    \[\begin{aligned}
    M(A_c)&\subset M(\overline{A_s})\subset M(A_1)=M(A_b\cap \Bo_1)=M(A_f)
    \subset A_1\\&\subset  M((A_c)^\mu)=M((A_b\cap \Bo_1)^\mu)=M((A_f)^\mu)=(A_c)^\mu\subset M(A_{sa}).\end{aligned}
    \]  
  Moreover, it may happen that all the inclusions are proper.  
\end{enumerate}
  \end{prop}

\begin{proof} 
Since $L$ is metrizable, any lower semicontinuous function on $L$ is of the first Baire class and any isolated set in $L$ is countable. It follows that $\overline{A_s}\subset A_1$. Further, $A_c\subset A_s\subset A_1$ and $(A_c)^\mu=(A_1)^\mu$ implies that $(A_s)^\mu=(A_c)^\mu$. Moreover, on $L$ metrizable, Borel one and Baire one functions coincide (see Theorem~\ref{T:b}), thus $A_f=A_b\cap\Bo_1$. These facts together with table \eqref{eq:table-ifs} yield the validity of \eqref{eq:table-ifs-metriz}.

Let us continue by proving the inclusions and equalities from assertion $(b)$. To this end we use the above-mentioned fact that fragmented, Borel one and Baire one functions on $L$ coincide and that any scattered subset of a  $L$ is countable and $G_\delta$. Indeed, assume $B$ is a scattered subset of $L$. If $B$ were uncountable, $B$ would contain a uncountable set without isolated points, which is impossible. Hence $B$ is countable.  If it were not $G_\delta$,  using the Hurewicz theorem (see \cite[Theorem 21.18]{kechris}) we would find a relatively closed subset of $B$ homeomorphic to $\qe$. This is again impossible, and thus $B$ is of type $G_\delta$. 

Using the two mentioned fact and the above characterizations we easily see that the chain of inclusions and equalities in $(b)$ is valid.

Finally, let $L=A=[0,1]$. We will check that no more inclusions hold.
The function $1_{\{0\}}\circ\psi$ witnesses that $A_c\subsetneqq \overline{A_s}$ and $M(A_c)\subsetneqq M(\overline{A_s})$. There is a subset $D\subset[0,1]$ homeomorphic to the ordinal interval $[0,\omega^\omega]$. Then the function $1_{D\times \{1\}}-1_{D\times\{-1\}}$ witnesses that $\overline{A_s}\subsetneqq A_1$ and  $M(\overline{A_s})\subsetneqq M(A_1)$.
The function $1_{L\times\{1\}}-1_{L\times\{-1\}}$ witnesses that $A_1\subsetneqq A_b\cap \Bo_1$. The function $1_{\qe\times\{1\}}-1_{\qe\times\{-1\}}$ shows that $M(A_1)\subsetneqq A_1$. The function $1_{\qe}\circ\psi$ witnesses that $A_1\subsetneqq (A_c)^\mu$. The function $1_{L\times\{1\}}-1_{L\times\{-1\}}+1_{\qe}\circ\psi$ 
witnesses that $(A_c)^\mu\subsetneqq (A_b\cap\Bo_1)^\mu$ and  $A_b\cap\Bo_1\subsetneqq (A_b\cap\Bo_1)^\mu$. If $S\subset [0,1]$ is an analytic non-Borel set, the function $1_S\circ \psi$ witnesses that $(A_f)^\mu\subsetneqq A_{sa}$ and $(A_c)^\mu\subsetneqq M(A_{sa})$.
\end{proof}

We continue by the following example showing, in particular, difference between multipliers and strong multipliers. 

\begin{example}
    \label{ex:dikous-mezi-new}
Let $K=K_{L,A}$, $E=E_{L,A}$ and $X=X_{L,A}$.
 For any bounded function $f$ on $K$ let us define three functions on $L$:
 $$\begin{gathered}
     f_0(t)=f(t,0),\ t\in L,\quad f_1(t)=\begin{cases}
     f(t,1), & t\in A,\\
     f(t,0), & t\in L\setminus A,
 \end{cases} \\
  f_{-1}(t)=\begin{cases}
     f(t,-1), & t\in A,\\
     f(t,0), & t\in L\setminus A.
 \end{cases} \end{gathered}$$
We set
 $$\wt{H}=\{f\in \ell^\infty (K)\setsep f_0,f_1,f_{-1}\mbox{ are Borel functions and }f_0=\tfrac12(f_1+f_{-1})\}$$
 and $H=V(\wt{H})$.
 Then we have the following:
 \begin{enumerate}[$(a)$]
     \item $H$ is an intermediate function space such that $H^\mu=H$.
     \item $(A_s(X))^\mu\subset H\subset (A_f(X))^\mu$.
     \item $M(H)=H$.
     \item $M^s(H)=V(\{f\in \wt{H}\setsep\{t\in A\setsep f(t,1)\ne f(t,-1)\}\mbox{ is $\sigma$-scattered}\})$.
     \item The following assertions are equivalent:
     \begin{enumerate}[$(i)$]
         \item $M^s(H)=M(H)$;
         \item $M(H)\subset M((A_f(X)^\mu)$;
         \item $A$ contains no compact perfect subset.
     \end{enumerate}
          In particular, if, for example, $L=A=[0,1]$, we have $M(H)\not\subset M((A_f(X))^\mu)$ and $M^s(H)\subsetneqq M(H)$.
     \item $\begin{aligned}[t]
        \A_H&=\{\phi(B)\setsep B\subset\Ch_E K, \psi(B) \mbox{ and }\psi (\Ch_E K\setminus B)\mbox{ are Borel subsets of }L\}\\ &= \Z_H,\\
     \A_H^s&=\{\phi(B)\setsep B\subset\Ch_E K, \phi(B)\in\A_H, \\&\qquad\{t\in A\setsep B\cap\{(t,1),(t,-1)\}\mbox{ contains exactly one point$\}$ is $\sigma$-scattered}\} 
     \\&=\{\phi(B)\setsep B\subset\Ch_E K, \psi(B) \mbox{ and }\psi (\Ch_E K\setminus B)\mbox{ are Borel subsets of }L\\&\qquad
     \psi(B)\cap\psi(\Ch_E K\setminus B)\mbox{ is $\sigma$-scattered}\}=\ms_H.    
     \end{aligned}$
      \end{enumerate}
\end{example}

\begin{proof}
Assertions $(a)$ and $(c)$ are obvious. Assertion $(b)$ follows by combining Propositions~\ref{P:dikous-Asmu-new} and~\ref{P:dikous-Afmu-new} with Lemma~\ref{L:sigma isolated}.

$(d)$: `$\supset$': Assume $B=\{t\in A\setsep f(t,1)\ne f(t,-1)\}$ is $\sigma$-scattered.  It follows that for each $g\in\wt{H}$ we have $fg\circ\jmath=\wt{fg}\circ\jmath$ for $t\in L\setminus B$. Since $B$ is universal null, we deduce that $V(\wt{fg})=V(f)V(g)$ $\mu$-almost everywhere for any maximal measure $\mu$. It follows that $V(g)\in M^s(H)$,

`$\subset$': Assume that the set $B=\{t\in A\setsep f(t,1)\ne f(t,-1)\}$ is not $\sigma$-scattered. Since it is Borel, by Lemma~\ref{L:dedicneres}$(b)$ it contains a compact perfect set $D$. Then there is a continuous probability measure $\mu$ carried by $D$. By Lemma~\ref{L:maxmiry-dikous} $\phi(\mu)$ is maximal and $\{t\in L:(f(t,0))^2=\wt{f^2}(t,0)\}=L\setminus B$ has $\mu$-measure zero. Thus $V(f)\notin M^s(H)$.

$(e)$: Assume $A$ contains no compact perfect set. Let $f\in \wt{H}$. Then 
$$B=\{t\in A\setsep f(t,1)\ne f(t,-1)\}$$ is a Borel subset of $A$. Since $A$ contains no compact perfect set, by Lemma~\ref{L:dedicneres} we deduce that $B$ is $\sigma$-scattered. Thus
$V(f)\in M^s(H)$ by $(d)$ and $V(f)\in M((A_f(X)^\mu)$ by Proposition~\ref{P:dikous-Afmu-new}.

Conversely, assume that $A$ contains a compact perfect set $D$. Let $f=1_{D\times\{1\}}-1_{D\times\{-1\}}$. Then $f\in \wt{H}$, so $V(f)\in M(H)=H$. By $(d)$ we get $V(f)\notin M^s(H)$ and by Proposition~\ref{P:dikous-Afmu-new} we deduce $V(f)\notin M((A_f(X))^\mu)$.

The `in particular part' is obvious.

$(f)$: Since $M(H)=H$, $\A_H=\Z_H$ by the very definition. Further, it is easy to see that the second and the third sets coincide.

The formulas for $\A_H^s$ follow easily from the proof of $(d)$,  \eqref{eq:podmny ChEK} and Lemma~\ref{L:dikous-split-new}.
\end{proof}

We continue by an overview of the relationship of the systems $\A_H$, $\A_H^s$, $\ms_H$ and $\Z_H$. It is contained in the following proposition.

\begin{prop}\label{P:shrnutidikousu}
    Let $K=K_{L,A}$, $E=E_{L,A}$ and $X=X_{L,A}$. Then the following assertions hold.
    \begin{enumerate}[$(a)$]
        \item If $H$ is one of the spaces 
        $$A_1,\overline{A_s},A_f, (A_1)^\mu,(A_s)^\mu,(A_f)^\mu,A_{sa},$$
        then $\A_H=\A^s_H=\ms_H$.
       \item If $H$ is one of the spaces
       $$A_b\cap\Bo_1, (A_b\cap \Bo_1)^\mu,$$
       then $\A_H=\A^s_H\subset\ms_H$. The inclusion may be proper even if $X$ is standard.
  
      \item  If $H$ is one of the spaces $A_1$, $(A_c)^\mu$, $\overline{A_s}$, $(A_s)^\mu$  then $\A_H=\A^s_H=\ms_H=\mathcal{Z}_H$.    
      \item If $H$ is one of the spaces
      $$ A_b\cap\Bo_1, A_f, (A_b\cap \Bo_1)^\mu,(A_f)^\mu,A_{sa},$$
      Then $\ms_H=\mathcal{Z}_H$ if and only if $X$ is a standard compact convex set. Moreover, it may happen that this is not satisfied.
      \item If $H$ is the space from Example~\ref{ex:dikous-mezi-new}, then
      $$\A^s_H=\ms_H\subset\A_H=\Z_H.$$
      The equality holds  if and only if $X$ is a standard compact convex set. Moreover, it may happen that this is not satisfied.
    \end{enumerate}
\end{prop}

\begin{proof} This follows from the above  propositions on the individual intermediate function spaces and Lemma~\ref{L:maxmiry-dikous}$(2)$. As an example of a non-standard $X$ may serve the case $L=A=[0,1]$. An example of a standard $X$ illustrating the proper inclusion in $(b)$ is provided by choosing, for example, $L=A=[0,\omega_1]$.
\end{proof}

The previous proposition show that the assumptions in Lemma~\ref{L:SH=ZH} are natural and cannot be simply dropped. It also illustrates that the assumptions of Proposition~\ref{P: silnemulti-inkluze} are not satisfied automatically. Further, it provides a partial positive answer to Question~\ref{q:A1-split} and Question~\ref{q:fr-split}.

We continue by examples illustrating limits of possible extensions of Theorem~\ref{T:Assigma-reprez}. The first one shows that the quoted theorem cannot be extended to functions of the first Borel class.

\begin{example}\label{ex:teleman-ne}
    Let $L=A=[0,1]$ and $X=X_{L,A}$. Then there are a function $f\in A_b(X)\cap\Bo_1(X)$ and a maximal measure $\mu\in M_1(X)$ such that
    $f|_{\ext X}$ is not $\mu'$-measurable (using the notation from Lemma~\ref{L:miry na ext}).
\end{example}

\begin{proof} Let $K=K_{L,A}$ and $E=E_{L,A}$. Let $h=1_{L\times \{1\}}-1_{L\times\{-1\}}$. By Proposition~\ref{P:dikous-Bo1-new}$(a)$ we get that $f=V(h)\in A_b(X)\cap\Bo_1(X)$. Let $\lambda$ be the Lebesgue measure on $L=[0,1]$. Then $\mu=\phi(\jmath(\lambda))$ is a maximal measure on $X$ (by Lemma~\ref{L:maxmiry-dikous}). 
Let us prove that the set $U=[f>0]\cap\ext X=\phi(L\times\{1\})$ is not $\mu'$-measurable. 

Let $B\subset X$ be a closed extremal set with $B\cap\ext X\subset U$. Let $F=\phi^{-1}(B)$. Then $F$ is a closed subset of $K$ contained in $L\times\{0,1\}$. The extremality of $B$ implies that $F\subset L\times\{1\}$, so $F$ is finite.
By \eqref{eq:mira1} we deduce that $\mu'(B\cap \ext X)=0$. Since $B$ is arbitrary, $\mu'(U)=0$. In the same way we see that $\mu'(\ext X\setminus U)=0$, so $\mu'=0$, a contradiction. 
\end{proof}

In the second example we address the weaker version of the representation addressed in the end of Section~\ref{s:determined}. It appears that it depends on additional axioms of the set theory.

\begin{example}\label{ex:AMC}
    Let $L=A=[0,1]$ and $X=X_{L,A}$.
    \begin{enumerate}[$(1)$]
        \item Assume the continuum hypothesis. Then there is $x\in X$ such that there is no probability measure $\mu$ defined on some $\sigma$-algebra on $\ext X$ such that
        $$\forall f\in A_b(X)\cap\Bo_1(X)\colon f|_{\ext X}\mbox{ is $\mu$-measurable and } f(x)=\int f\di\mu.$$
        \item Assume there  is an atomless measurable cardinal. Then for any $x\in X$ there is a 
        probability measure $\mu$ defined on some $\sigma$-algebra on $\ext X$ such that
        $$\forall f\in A_{sa}(X)\colon f|_{\ext X}\mbox{ is $\mu$-measurable and } f(x)=\int f\di\mu.$$
    \end{enumerate}
\end{example}

\begin{proof}
    Let $K=K_{L,A}$. 

    $(1)$: Let $B\subset L$ be arbitrary. Then $V(1_{B\times\{1\}}-1_{B\times\{-1\}})\in A_b(X)\cap\Bo_1(X)$. It follows that the domain $\sigma$-algebra of any $\mu$ with the required property should be the power set of $\ext X$. But, by the continuum hypothesis this set has cardinality $\aleph_1$ and thus by Ulam's theorem (see, e.g., \cite[Lemma 10.13]{jech-book}) any measure defined on this $\sigma$-algebra is discrete. Let $\lambda$ be the Lebesgue measure on $L=[0,1]$ and let $x$ be the barycenter of $\phi(\jmath(\lambda))$. Assume that $\mu$ is a measure with the required properties. Since it is discrete, there is some $t\in L$ such that $\mu(\{\phi(t,1),\phi(t,-1)\})>0$. Let $f=1_{\{t\}}\circ\psi$. Then $V(f)\in A_1(X)\subset A_b(X)\cap\Bo_1(X)$ and
    $$0<\int V(f)\di\mu= V(f)(x)=\int V(f)\di\phi(\jmath(\lambda))=\int f\di\jmath(\lambda) =\int 1_{\{t\}}\di\lambda=0,$$
    a contradiction.

    $(2)$: Under this assumption any continuous Borel measure $\mu$ on $[0,1]$ admits a $\sigma$-additive extension $\widehat{\mu}$ defined on the power set of $[0,1]$. Indeed, let $\mu$ be any continuous measure on $[0,1]$. Then the function $t\mapsto\mu([0,t])$ is a continuous non-decreasing mapping of $[0,1]$ onto $[0,1]$. Thus we may define a kind of inverse:
    $$m(t)=\min\{ s\in[0,1]\setsep \mu([0,s])=t\}, \quad t\in [0,1].$$
    Then $m$ is increasing and it is easily seen to be left continuous and upper semicontinuous. In particular it is a Borel injection of $[0,1]$ into $[0,1]$. Denote by $\lambda$ the Lebesgue measure on $[0,1]$. We claim that $\mu=m(\lambda)$. Indeed, for any $t\in(0,1]$
    $$\begin{aligned}
      m(\lambda)([0,t])&=\lambda(m^{-1}([0,t]))=  \sup\{s\in[0,1]\setsep m(s)\le t\}
     \\&=\sup\{s\in[0,1]\setsep m(s)< t\} = \sup\{s\in [0,1]\setsep \mu([0,s])<t\}=\mu([0,t]),
    \end{aligned}$$
    hence the equality holds for all intervals and thus for all Borel sets. Further, it is proved in \cite[p. 123]{jech-book} that there is the required extension $\widehat{\lambda}$ for the Lebesgue measure. It remains to set $\widehat{\mu}=m(\widehat{\lambda})$.
    
    Next, given $x\in X$ let $\delta_x$ be the unique maximal representing measure. It follows from Lemma~\ref{L:maxmiry-dikous} that
    $$\delta_x=\phi(\jmath(\nu_x)+\sigma_x),$$
    where $\nu_x$ is a continuous measure on $L=[0,1]$ and $\sigma_x$ is a discrete measure on $L\times\{-1,1\}$. Define 
    $$\mu_x(F\times\{1\}\cup G\times\{-1\})=\tfrac12(\widehat{\nu_x}(F)+\widehat{\nu_x}(G)),\quad F,G\subset L.$$
    We claim that $\phi(\mu_x+\sigma_x)$ is the right choice.
    Indeed, let $f$ be a function such that $V(f)\in A_{sa}(X)$. Then
    $$\begin{aligned}
            \int V(f)\di \phi(\mu_x)&=\frac12\left(\int f(t,1)\di\widehat{\nu_x}(t)+  \int f(t,-1)\di\widehat{\nu_x}(t)\right)=
    \int f(t,0)\di\widehat{\nu_x}(t)\\&=\int f(t,0)\di\nu_x(t)=
    \int V(f)\di\phi(\jmath(\nu_x)),\end{aligned}$$
    hence 
    $$\int V(f)\di \phi(\mu_x+\sigma_x)=\int V(f)\di \phi(\jmath(\nu_x)+\sigma_x)=\int V(f)\di\delta_x=V(f)(x).$$
    This completes the proof.
    \end{proof}

The previous example inspires the following question.

\begin{ques}
    Assume that there is an atomless measurable cardinal.
    Let $X$ be a compact convex set.
    \begin{enumerate}[$(a)$]
        \item Is it true that for each $x\in X$ there is a probability measure $\mu_x$ defined on some $\sigma$-algebra on $\ext X$ such that
        $$\forall f\in (A_{f}(X))^\mu\colon f|_{\ext X}\mbox{ is $\mu_x$-measurable and } f(x)=\int f\di\mu_x\ ?$$
        \item Assume  moreover that $A_{sa}(X)$ is determined by extreme points. Is the representation from $(a)$ valid for all strongly affine functions?
    \end{enumerate}
\end{ques}

%\section*{References}

 %\renewcommand\refname{\relax} 
\bibliography{krs-multip}\bibliographystyle{acm}

\newpage

%\section*{List of notation}
\printglossary[title=List of notation,type=symbols,style=list,nogroupskip]

%%%%%%%%%%%%%%%%%%%%%%%%%%%%%%%%%%%%%%%%%%%%%%%%%%%%%%%%%%%%%%%%%

\printindex

%%%%%%%%%%%

\end{document}